\documentclass[11pt, a4paper]{article}

\usepackage[a4paper, total={6.5in, 10in}]{geometry}
\usepackage{amsmath, amssymb, amsthm} 
\usepackage{graphicx, subfig}
\usepackage{caption}

\newtheorem{alg}{Algorithm}
\newtheorem{thm}{Theorem}

\newtheorem{lem}{Lemma}

\newtheorem{remark}{Remark}
\newcommand{\unit}{\texttt{1}\!\!\texttt{l}}
\newcommand{\Int}{\int_{0}^{2\pi}}
\newcommand{\tr}{\mbox{tr}}
\newcommand{\M}{\mathcal{M}}

\def \D #1{\underline{D}_{#1}}

\makeatletter
\let\gplgaddtomacro\g@addto@macro
\gdef\gplbacktext{}%
\gdef\gplfronttext{}%
\makeatother

\numberwithin{equation}{section}

\author{Charles M. Elliott\footnotemark[1]~ and Hans Fritz\footnotemark[1]}

\title{On Approximations of the Curve Shortening Flow and of the Mean Curvature Flow based on the DeTurck trick}

\date{}

\begin{document}
\maketitle
\renewcommand{\thefootnote}{\fnsymbol{footnote}}
\footnotetext[1]{Mathematics Institute, Zeeman Building, University of Warwick, Coventry. CV4 7AL. UK\\ 
C.M.Elliott@warwick.ac.uk, H.Fritz@warwick.ac.uk}
\abstract{
In this paper we discuss novel numerical schemes for the computation of the 
curve shortening and mean curvature flows that are based on 
special reparametrizations. The main idea is to use special solutions 
to the harmonic map heat flow in order to reparametrize the equations of motion. 
This idea is widely known from the Ricci flow as the DeTurck trick. 
By introducing a variable time scale for the harmonic map heat flow, we obtain 
families of numerical schemes for the reparametrized flows. 
For the curve shortening flow this family unveils a surprising geometric connection 
between the numerical schemes in \cite{BGN11} and \cite{DD94}. 
For the mean curvature flow we obtain families of schemes with good mesh properties
similar to those in \cite{BGN08}.
We prove error estimates for the semi-discrete scheme of the curve shortening flow.
The behaviour of the fully-discrete schemes with respect to the 
redistribution of mesh points is studied in numerical experiments.
We also discuss possible generalizations of our ideas to other extrinsic flows.
}\\

\textbf{Key words.} 
Curve shortening flow, mean curvature flow, harmonic map heat flow, DeTurck trick,
parametric finite elements, error estimates, tangential redistributions, mesh properties. 
\\

\textbf{AMS subject classifications.} 65M60, 65M15, 35K93, 53C44, 58E20. 
\\

\section{Introduction}
\subsection*{Motivation}
The numerical analysis and the approximation of geometric flows have made
significant progress during the last decades, see \cite{DDE05, DE13} and references therein. Different numerical schemes for geometric flows 
such as the mean curvature flow and the Willmore flow 
have been proposed by several authors - all very appealing for different kinds of reason.
In \cite{Dz91}, for example, Dziuk presented a discretization of the mean curvature flow
based on the fact that the mean curvature flow is a kind of diffusion equation for the surface embedding.
On the other hand, Barrett, Garcke and N\"urnberg introduced in
a series of papers, see \cite{BGN07,BGN08,BGN08b, BGN11,BGN12}, 
numerical schemes with good properties
with respect to the redistribution of mesh points.
Algorithms providing time-dependent rearrangements of mesh points
that prevent mesh degenerations are indeed very desirable. 
In fact, it can be stated that the formation of degenerate meshes 
in the simulation of geometric flows is the Achilles heel of many 
state of the art algorithms that are based on the parametric approach. 
Numerical simulations usually have to be stopped,
when the mesh degenerates and some sophisticated machinery for remeshing
the polyhedral surfaces, for example, using harmonic maps between surfaces \cite{St14}, 
has to be applied. It seems therefore to be a far better solution of this problem
to use algorithms that already induce tangential motions 
that lead to good redistributions of mesh points, see \cite{BGN07,BGN08,BGN08b, BGN11,BGN12}.
On the other hand, introducing schemes that lead to artificial tangential motions 
seems to be problematic too.  
Obviously, discrete solutions with non-vanishing tangential motions 
cannot converge to smooth solutions
with vanishing tangential velocities. Hence, such solutions cannot converge to the
surface parametrization evolving according to 
the original (that is non-reparametrized) system of PDEs.  
It is thus unclear whether there is at all a well-defined evolution of 
the surface parametrization that is approximated by the discrete solution. 
Since the numerical analysis of geometric flows is usually based on the analysis of
non-degenerate evolution equations for the surface parametrization
rather than on the analysis of evolving shapes only determined by their normal velocity,
the numerical analysis of such schemes seems to be very difficult.

\subsection*{Our approach}
An obvious possibility to tackle this problem is to replace the original evolution equations 
by suitable reparametrizations that lead
to the desired tangential motions in the discrete setting. 
An interesting reparametrization of the curve shortening flow can be found in the paper
of Deckelnick and Dziuk \cite{DD94}. 
Surprisingly, it turns out that this reparametrization can be linked to the so-called DeTurck trick.
This trick refers to an idea, which was originally introduced by DeTurck for the
Ricci flow for purely analytical reasons, see \cite{DeT83} and \cite{Ham95} for details. 
Reparametrizing the Ricci flow by solutions to the harmonic map heat flow
leads to a strongly parabolic PDE, 
which is now known as the Ricci-DeTurck flow or as the dual Ricci-harmonic map heat flow. 
Fortunately, this idea is not restricted to the Ricci flow.
Rather, it is possible to apply this idea also to other geometric
flows such as the curve shortening and mean curvature flows, see \cite{Ba10, Ham89}. 
Below, we will explicitly derive the reparametrized evolution equations of these flows
by using solutions to harmonic map heat flows. 
For the curve shortening flow, we will obtain the evolution equations 
considered by Deckelnick and Dziuk in \cite{DD94} for numerical reasons as a special 
case of our more 
general approach.
As we will see, the reparametrization by the harmonic map heat flow gives rise to tangential motions, 
which we here aim to exploit for purely numerical reasons,
namely for the tangential redistribution of the mesh points.
Since the numerical scheme in \cite{DD94} is fully based on a consistent discretization
of a non-degenerate system of PDEs, 
it is possible to prove rigorous error estimates. 
This still holds in our more general setting.
Unfortunately, in general, the scheme in \cite{DD94} does not provide  
sufficiently large redistributions that are able to keep the mesh points 
approximately equidistributed. The reason for this lack seems to be that
the time scale on which the tangential redistributions take place
is just too large. 

In this paper, we hence generalize the idea to use the DeTurck trick 
as a tool of deriving useful reparametrizations of geometric evolution equations
by introducing a variable time scale.
Our aim is to develop novel algorithms which can be analysed rigorously
and which also provide good mesh properties. 
We here say that a discrete curve has good mesh properties if
all mesh segments have approximately the same length.
For higher-dimensional hypersurfaces a triangulation is said to be a good mesh 
if the quotient of the diameter of a simplex and of the radius of the largest ball contained in it
is reasonable small for all simplices of the triangulation, see also definition (\ref{definition_sigma_max}) below.   
This seems to be a good quantity to evaluate the mesh quality, since it plays an important role in the numerical analysis of PDEs.
An interesting open question is the rigorous proof
that an algorithm maintains this mesh quality.

Our approach leads to families of numerical schemes 
for the approximation
of the curve shortening and mean curvature flows
depending on a parameter $\alpha > 0$,
which determines the time scale for the tangential motions of the surface
parametrization. 
Indeed, it is not very surprising that for a special choice of this parameter,
here for $\alpha = 1$, 
we recover the semi-discrete problem
studied in \cite{DD94},
that is
$$
 \Int \hat{X}_{ht} \cdot \varphi_h |\hat{X}_{h\theta}|^2 d\theta +  
 \Int \hat{X}_{h\theta} \cdot \varphi_{h\theta} d\theta = 0,
 \quad \forall \varphi_{h} \in \mathcal{S}^2_h, 0 < t < T,
$$
where $\hat{X}_h \in H^{1,2}((0,T), \mathcal{S}^2_h)$ is an approximation
to the reparametrized curve shortening flow.
Here, $\mathcal{S}_h$ denotes the space of piecewise linear, continuous 
functions $\varphi_h: [0,2\pi] \rightarrow \mathbb{R}$
with $\varphi_h(0) = \varphi_h(2 \pi)$ on a given grid in $[0,2 \pi]$ with grid size $h$. 
We are also able to recover
a scheme similar to the curve shortening flow scheme in \cite{BGN11}
if we formally choose $\alpha = 0$.
This means that we are able to connect
the schemes in \cite{BGN11} and in \cite{DD94} by a family of numerical schemes depending on the time scale parameter $\alpha$. 
To be more precise, for $\alpha = 0$, Algorithm \ref{algo_CSF} proposed below
simplifies to
\begin{align*}
	&\Int \left(  
	\left(\frac{\hat{X}^{m+1}_{h} - \hat{X}^{m}_{h}}{\tau} \cdot \nu^m_h \right)
				(\nu^m_h \cdot \varphi_h) \right) 
				|\hat{X}^m_{h\theta}|^2 d\theta 
	+ \Int \hat{X}^{m+1}_{h\theta} \cdot \varphi_{h\theta} d\theta 
	= 0 ,
	\quad \forall \varphi_h \in \mathcal{S}_h^2,			
\end{align*}
while the solution of algorithm (2.16a) from \cite{BGN11} satisfies 
$$
 \Int  I_h \left( \left( \frac{\hat{X}_h^{m+1} - \hat{X}_h^m}{\tau} \cdot \rho^{m+1}_h \right)
 \left( \rho^{m+1}_h \cdot \varphi_h \right) 
 \right) d\theta
 + \Int \hat{X}^{m+1}_{h \theta} \cdot \varphi_{h\theta} 
 d\theta = 0,
 \quad \forall \varphi_{h} \in \mathcal{S}^2_h,
$$
where $I_h$ denotes the Lagrange interpolation operator and
$\rho_h^{m+1} \in \mathcal{S}_h^2$
is defined such that in each node $\theta_j \in [0,2\pi]$ of the grid
it is given by the mean value of the piecewise constant vector field 
$\nu_h^{m+1} |\hat{X}^{m+1}_{h \theta}|$.

The important similarities between both schemes are the appearance of a discrete
normal projection acting on the discrete time derivative and the multiplication of this term by the square of the length element.
Furthermore, the elliptic operator in both schemes only depends linearly on the curve parametrization,
which is a rather surprising result, since the curve shortening flow itself is non-linear with respect to the curve parametrization. 
In \cite{BGN11} the normal projection is motivated by the fact that the geometric problem associated with the curve shortening flow only describes the normal velocity,
whereas the tangential motion is undetermined by the geometry, and therefore, regarded as free. 
It is hence a bit surprising that the algorithm based on this view gives rise to desirable tangential motions that lead 
to the advantageous redistribution of the mesh points. In our derivation,
the normal projection is a formal limit of a certain map which has its origin
in the reparametrization by the DeTurck trick.
This observation leads to the interesting question whether the tangential redistributions 
observed in \cite{BGN11} could be explained as the limiting behaviour of the DeTurck reparametrization process.
Although, we are able to formally choose $\alpha = 0$ in our algorithms,
which in the one-dimensional case $n=1$ leads to the above scheme, 
it has to be stated that the derivation of the 
reparametrized evolution equations is only valid for $\alpha > 0$. 
Unfortunately, 
it is not only the derivation, which becomes problematic for the choice
$\alpha = 0$, but also the system of PDEs for the surface parametrization
itself. This means that the reparametrized evolution equation 
that underlies Algorithm \ref{algo_CSF}
changes from a non-degenerate system of parabolic differential equations 
for $\alpha > 0$ to a system of PDEs where the normal component of the system looks parabolic and 
its tangential part seems to be elliptic.
We, therefore, restrict the analytical part of this paper to the case $\alpha >0$.
Nevertheless, the relation between Algorithm \ref{algo_CSF} for $\alpha=0$
and the scheme proposed in \cite{BGN11} strengthens the idea, 
that for a suitable choice of the parameter $\alpha > 0$
a sufficiently good redistribution behaviour of the mesh points should be
achievable.
In order to confirm this assumption we study the limiting behaviour
$\alpha \searrow 0$ of our algorithms in numerical experiments.

In contrast to the curve shortening flow, our numerical schemes for the mean curvature flow 
seem to be totally new. Although, it is still possible to connect them loosely 
to the scheme in \cite{BGN08}, the main difference between both approaches
is that our schemes are based on the consistent discretization of a non-degenerate system
of parabolic PDEs for the surface parametrization, 
whereas the scheme in \cite{BGN08} is based on the evolution equation for the normal velocity
of the surface, and hence on a degenerate equation. This difference becomes manifest in the appearance
of an additional (second order) term in our numerical scheme for the mean curvature flow. 

The here presented approach is based on the DeTurck trick and the harmonic map heat flow.
At first glance this approach seems to be rather ad hoc and it does not seem to be clear why this approach should lead to schemes with good mesh properties.
Yet, the following observations strengthen the idea that it is indeed possible to produce nice meshes by using DeTurck reparametrizations.  
\begin{enumerate}
\item[1.] Under certain assumptions it is possible to prove that solutions to the harmonic map heat flow converge to harmonic maps for long times; see \cite{ES64}.
\item[2.] Under certain assumptions, harmonic maps of surfaces are conformal maps; see, for example, in \cite{EW76}.
We therefore expect that the map $y_h^m$ in Algorithms \ref{algo_moving_hypersurface} and \ref{algorithm_harmonic_MCF} of this paper
will also approximate a conformal map -- at least for very small $\alpha$.
\item[3.] Suppose that there is a conformal map between two surfaces with different Riemannian metrics. Furthermore, suppose that one surface 
		is approximated by a simplicial mesh such that the triangles do not have any sharp angles with respect to the corresponding metric.
		The image of this mesh under the conformal map, or more precisely, under a good approximation thereof should then give a good mesh for the other surface.    
\end{enumerate}  
Any rigorous results in this direction are far beyond the scope of this paper. 
In the above arguments we have not made use of any properties of the curve shortening and mean curvature flows.
Hence, our approach should be applicable in a much wider context;
see Section \ref{Section_generalization_to_other_geometric_flows}.
In the main part of this paper, however, we will focus on the curve shortening and mean curvature flows.  
The applicability of our approach might be restricted by the fact that, in the general case, solutions to the harmonic map heat flow can generate singularities in finite time; 
see, for example, in \cite{CDY92}. It is unclear whether these singularities then also arise in the curve shortening-DeTurck and mean curvature-DeTurck flows.

\subsection*{Related work}
The redistribution of mesh points in order to prevent mesh degenerations has been studied for quite a long time.
In \cite{HLS94}, a non-local equation for the tangential velocity functional of curves in $\mathbb{R}^2$ was introduced and utilized for the first time.
Later, this functional has been studied in detail in \cite{MS01} showing that the redistribution preserves relative local lengths of curve segments.
The method was then generalized in \cite{MS04} in order to achieve asymptotically uniform redistribution of grid points for evolving curves.  
In \cite{MS04b} this approach was extended by the addition of a diffusive term to the equation of the tangential velocity.
The advantages of adding a diffusive term to the curve shortening flow have already been exploited in \cite{DD94}.
The scheme proposed in \cite{MS04b} is suitably chosen in order to uniformly redistribute the mesh points.

Only recently, a numerical scheme for the tangential redistribution of mesh points on higher-dimensional manifolds has been proposed in \cite{MRSS14}.
The approach of this paper is based on an appropriate variation of the surface velocities. The additional tangential velocities are chosen in such a way that the volume density
of the surface parametrization can be controlled.
In order to obtain well-defined problems, the authors assume 
that the tangential velocities are in fact gradient fields.
This assumption then leads to elliptic problems, in which time is 
an additional parameter. The authors demonstrate the effectiveness of their approach
in numerical experiments. 
It is an interesting question whether controlling relative volumes during the evolution can really prevent the formation of mesh degenerations in general.
Reparametrizations by harmonic maps, see \cite{St14}, might be an interesting alternative, since it allows for conformal remeshing.

We believe that introducing additional equations into numerical schemes in order to control
certain mesh quantities not only increases the computational costs, but also makes the numerical analysis of the scheme much more involved.
Hence, we aim to use a kind of built-in reparametrization based on 
the DeTurck trick such that we do not have to solve other problems 
than the (reparametrized) evolution equation of the surface parametrization. 
This approach differs from the schemes in \cite{HLS94, MS01, MS04, MS04b, MRSS14, St14}, where a larger system of PDEs is considered in order to
improve or maintain the mesh quality. 
We therefore only compare our numerical results to the schemes in \cite{BGN08, BGN08b, BGN11}, which are also in the spirit of a
built-in approach. We expect that such an approach can reduce the computational errors, 
while still providing sufficiently good mesh behaviour.

In \cite{St14}, the DeTurck trick has already 
been used to derive a family of numerical schemes
for the approximation of the curve shortening flow; compare equation (4.5) in \cite{St14} to (\ref{intermediate_result_CSF_2}) in the present paper.
However, the author does not use an important trick, which
we introduce below and which is crucial for different kinds of reasons.
Firstly, without this trick, it is not possible to consider the limit $\alpha \searrow 0$ and to see the connection between the schemes in \cite{BGN11} and \cite{DD94}.
Furthermore, in \cite{St14} a variable
for the mean curvature vector has to be introduced
in order to be able to discretize the weak formulation by piecewise linear finite elements.
In contrast to this result, it is not necessary in our formulation to introduce any further variables
for the computation of the curve shortening flow. 
Finally, the author in \cite{St14} does not prove any error estimates for his scheme, which we will do in Theorem \ref{approximation_theorem}.
We are not aware of any further publications, where the 
reparametrization of the evolution equations by solutions 
to the harmonic map heat flow has been used
to develop numerical schemes based on surface finite elements. 
We would like to emphasize that the ideas developed in this paper are
not restricted to any dimension $n$ of the hypersurface, even if the error analysis
in Section $3$ is only valid for the one-dimensional case, that is $n=1$.

\subsubsection*{Outline of the paper}
This paper is organised as follows. In Section $2$, we 
introduce the DeTurck trick and apply it to the $n$-dimensional
mean curvature flow. We derive the reparametrized evolution equations in detail
for any dimension $n \in \mathbb{N}$ of the hypersurface.  
In Section $3$, we discretize a weak formulation of the reparametrized
evolution equations in space for the one-dimensional problem, that is for
the curve shortening flow. We then show that it is possible to adapt the proof of
the error estimates in \cite{DD94} to our novel schemes with only minor changes. In Section \ref{Numerical_results_CSF}, numerical tests for the computation of the curve shortening
flow are presented with a special focus on the behaviour of the mesh properties.
In Section $5$, we derive a weak formulation of the reparametrized mean curvature
flow of $n$-dimensional hypersurfaces for arbitrary $n \in \mathbb{N}$.   
These equations are then lifted onto the moving 
hypersurface and discretized using surface finite elements. 
In Section $6$, we introduce a variant of the DeTurck trick that leads to a formulation,
which only contains terms that can be related to the first variation of some energy functionals.
In Section $7$, numerical experiments for the mean curvature flow are presented with a special focus on mesh properties. 
We compare the performance of our schemes to the behaviour of the scheme proposed in \cite{BGN08}.
Possible generalizations of the ideas developed in this paper to other geometric flows 
are explained in Section $8$. In Section $9$, we discuss the results of this paper
and compare our approach to previous works.

The main results of this paper are Algorithm \ref{algo_CSF} for the computation of the 
reparametrized curve shortening flow and Algorithms \ref{algo_moving_hypersurface} and \ref{algorithm_harmonic_MCF} 
for the computation of the
reparametrized mean curvature flow.

\section{Reparametrizations via the DeTurck trick}
\subsection{Notation}
Henceforward, let $\M$ be a closed (that is compact and without boundary), 
connected, orientable, $n$-dimensional smooth manifold $\M$
(that is a topological space which is locally homeomorphic to open subsets of $\mathbb{R}^n$
via the so-called coordinate charts 
$\mathcal{C}_i: U_i \subset \M \rightarrow \Omega_i \subset \mathbb{R}^n$,
where the transition maps $\mathcal{C}_i \circ \mathcal{C}_j^{-1}$ are supposed to be smooth).
We denote the identity map on $\M$ by $id(p) = p$.
Let $g$ be a smooth Riemannian metric on $\M$, that is a smooth map which defines an inner product on all tangent spaces of $\M$.
The components of $g$ with respect to a local coordinate system are denoted by $g_{ij}$. 
The components of the inverse of the matrix $(g_{ij})_{i,j= 1, \ldots,n}$ are denoted by $g^{ij}$. 
In the following we will make use of the convention to sum over repeated indices.
 
For a $C^2$-function $f: \mathcal{M} \rightarrow \mathbb{R}$ the Laplace operator
$\Delta_g$ with respect to the metric $g$ is defined by
\begin{equation}
\label{defi_laplacian}
	(\Delta_g f) \circ \mathcal{C}^{-1}_1  := g^{ij} \left( \frac{\partial^2 F}{\partial \theta^i \theta^j}
				- \Gamma(g)^k_{ij} \frac{\partial F}{\partial \theta^k} \right),
\end{equation}
where $F := f \circ \mathcal{C}_1^{-1}$ and $\Gamma(g)^k_{ij}$ are
the Christoffel symbols of $g$ defined by 
\begin{align}
  \Gamma(g)^k_{ij} := \frac{1}{2} g^{kl} \left(
  	\frac{\partial g_{lj}}{\partial \theta^i} 
  	+ \frac{\partial g_{li}}{\partial \theta^j} 
  	- \frac{\partial g_{ij}}{\partial \theta^l} 
  \right).
\label{definition_Christoffel_symbols}
\end{align} 
If $\M$ is a hypersurface of the Euclidean space $\mathbb{R}^{n+1}$
and $g$ the corresponding induced metric, the Laplace operator $\Delta_g$ coincides
with the surface Laplacian defined in Definition (2.3) of \cite{DDE05}. 
Because of Jacobi's formula for the derivative of the
determinant we obtain
\begin{align*}
	g^{ij} \Gamma(g)^k_{ij} 
	&= g^{ij} g^{kl} 	\left( \frac{\partial g_{li}}{\partial \theta^j}
							 - \frac{1}{2} \frac{\partial g_{ij}}{\partial \theta^l}
						\right)	
	= - \frac{\partial g^{jk}}{\partial \theta^j}
		- g^{kl} \frac{1}{\sqrt{|g|}} \frac{\partial \sqrt{|g|}}{\partial \theta^l}
	\\	
	&= - \frac{1}{\sqrt{|g|}} \frac{\partial}{\partial \theta^j} 
		\left( \sqrt{|g|} g^{jk} \right),
\end{align*}
where $|g| = \det (g_{ij})$, and hence,
\begin{align*}
	(\Delta_g f) \circ \mathcal{C}^{-1}_1 
	&= \frac{1}{\sqrt{|g|}} \frac{\partial}{\partial \theta^i}
				\left( \sqrt{|g|} g^{ij} \frac{\partial F}{\partial \theta^j} \right),
\end{align*}
which is sometimes used as an alternative definition of the Laplace operator.

The map Laplacian $\Delta_{g,h} \psi$ of a twice-differentiable map $\psi: \M \rightarrow \M$ is defined by 
\begin{equation}
\label{map_Laplacian}
(\mathcal{C}_2 \circ (\Delta_{g, h} \psi) \circ \mathcal{C}^{-1}_1)^q := g^{ij}
	\left( \frac{\partial^2 \Psi^q}{\partial \theta^i \theta^j} 
		- \Gamma(g)^k_{ij} \frac{\partial \Psi^q}{\partial \theta^k}
		+ (\Gamma(h)^q_{mn} \circ \Psi ) \frac{\partial \Psi^m}{\partial \theta^i}
						   \frac{\partial \Psi^n}{\partial \theta^j}		
		\right).
\end{equation}
Here, $g^{ij}$ and $\Gamma(g)^k_{ij}$ are the components of the inverse of
the matrix $(g_{ij})_{i,j= 1, \ldots,n}$, and respectively, of the Christoffel symbols of $g$ with respect to
the coordinate chart $\mathcal{C}_1$. 
The quantities $h^{mn}$ and $\Gamma(h)^q_{mn}$ denote the corresponding quantities
of the metric $h$ with respect to the coordinate chart $\mathcal{C}_2$.
The local representation of the map $\psi$ with respect
to the charts $\mathcal{C}_1$ and $\mathcal{C}_2$ is denoted by 
$\Psi := \mathcal{C}_2 \circ \psi \circ \mathcal{C}^{-1}_1$. 
In order to keep notation simple yet concise,
we make the convention 
that whenever a term depends on different local coordinates, 
then the indices $i,j,k,l$ refers to the coordinates with respect to the chart $\mathcal{C}_1$,
whereas the indices $m,n,p,q$ refers to the coordinates with respect to the chart $\mathcal{C}_2$
-- if not otherwise stated.

The differential $\nabla f $ of a differentiable function
$f: \M \rightarrow \mathbb{R}$ on $\M$ is defined by
$$
	(\nabla f) (w) \circ \mathcal{C}^{-1}_1
	:= W^j \frac{\partial F}{\partial \theta^j},
$$
where $F := f \circ \mathcal{C}^{-1}_1$ and $w = W^j \frac{\partial}{\partial \theta^j}$
is an arbitrary tangent vector field on $\M$.

For a function $f$ on an $n$-dimensional smooth hypersurface 
$\mathcal{N} \subset \mathbb{R}^{n+1}$ differentiable at $p \in \mathcal{N}$,
the tangential gradient $\nabla_{\mathcal{N}} f (p)$ is defined by
$$
	\nabla_{\mathcal{N}} f(p) 
	:= \nabla \overline{f}(p) - (\nu \cdot \nabla \overline{f})(p) 
	\nu(p).
$$
Here, $\nu(p)$ is a unit normal to $\mathcal{N}$ at the point $p$
and $\overline{f}$ is a differentiable extension of $f$ to an open neighbourhood
$U \subset \mathbb{R}^{n+1}$ of $p$, such that 
$\overline{f}_{|\mathcal{N} \cap U} = f_{|\mathcal{N} \cap U}$. 
The tangential gradient is well-defined, since the above definition only depends
on the values of $f$ on $\mathcal{N}$, see \cite{DDE05} for more details.
The components of the tangential gradient are denoted by
$$
	\left(
	\begin{array}{c}
		\D 1 f
		\\
		\vdots
		\\
		\D {n+1} f
	\end{array}
	\right)
	:= \nabla_\mathcal{N} f.
$$

We define the integral on a Riemannian manifold $(\mathcal{M}, g)$ 
with respect to a Riemannian metric $g$ by
$$
	\int_{\mathcal{M}} f do_{g} := \int_{\Omega} f \circ \mathcal{C}^{-1}_1 \sqrt{\det(g_{ij})} d^n\theta,
$$
where $\mathcal{C}_1: U \subset \mathcal{M} \rightarrow \Omega \subset \mathbb{R}^n$
denotes a local coordinate chart of $\mathcal{M}$ and $f : \mathcal{M} \rightarrow \mathbb{R}$ is an integrable function on $\mathcal{M}$ 
with support ${supp} f \subset U$.
Using a partition of unity, this definition easily generalizes to
arbitrary integrable functions on $\mathcal{M}$.
On $n$-dimensional hypersurfaces in $\mathbb{R}^{n+1}$ the volume form that is induced by the $n$-dimensional Hausdorff measure will be denoted by $d\sigma$. 
Below, we will make use of the matrix scalar product $A:B$
defined by 
$A:B = \sum_{\alpha, \beta = 1}^{n+1} A_{\alpha \beta} B_{\alpha \beta}$
for $A,B \in \mathbb{R}^{(n+1) \times (n+1)}$.

Henceforward, let $x: \M \times [0,T) \rightarrow \mathbb{R}^{n+1}$
be a time-dependent embedding of $\M$ 
(that is an immersion on $\M$ which is a homeomorphism of $\M$ onto $x(\M)$) of at least class $C^2$.
The local representation of the embedding $x$ with respect to a coordinate chart
$\mathcal{C}_1$ is denoted by $X:= x \circ \mathcal{C}_1^{-1}$. 
The embedding $x$ induces a Riemannian metric on $\M$ given by the the pull-back of the Euclidean metric $\mathfrak{e}$ in $\mathbb{R}^{n+1}$.  
In the following we will denote this metric by $g(t) := x(t)^\ast \mathfrak{e}$.
In local coordinates the pull-back metric $g(t)$ is given by
\begin{equation}
	\label{induced_metric}
	g_{ij}(\theta, t) 
			:= \mathfrak{e}\left( \frac{\partial X}{\partial \theta^i}(\theta, t) ,
				 \frac{\partial X}{\partial \theta^j}(\theta, t)  \right)
			:= \frac{\partial X}{\partial \theta^i}(\theta, t) 
					\cdot \frac{\partial X}{\partial \theta^j}(\theta, t),
			\quad \forall i,j \in \{1, \ldots, n \}.
\end{equation}
The Euclidean metric $\mathfrak{e}$ will also be denoted by $\cdot$, where convenient.
For the sake of convenience, we will omit the full dependency of the metric
$g(p,t)(\cdot, \cdot)$ with
$(p,t) \in \M \times [0,T)$ where appropriate. 

\subsection{The mean curvature flow}
We next introduce the (non-reparametrized) evolution equations of the mean curvature flow.
Very readable surveys on the mean curvature flow are \cite{Ec04} and \cite{Ma10}. 
The embedding $x$ is said to evolve according to the mean curvature flow if 
\begin{equation}
	\frac{\partial}{\partial t} x = - (H \nu) \circ x.
	\label{MCF_equation}
\end{equation}
Here, $\nu$ denotes a unit normal vector field 
on the embedded hypersurface $\Gamma(t) := x(\M,t) \subset \mathbb{R}^{n+1}$
and $H$ is the corresponding mean curvature, that is
\begin{equation}
\label{mean_curvature}
	H \circ X := g^{ij} \frac{\partial X}{\partial \theta^i}
					\cdot \frac{\partial (\nu \circ X)}{\partial \theta^j} 
			= - g^{ij} \frac{\partial^2 X}{\partial \theta^i \partial \theta^j} 
					\cdot  (\nu \circ X).
\end{equation}
The definition of the mean curvature flow does not depend on 
the choice of the unit normal field. Please note that our definition of the mean curvature $H$
differs form the more common one by a factor of $n$. For example, the mean curvature of the $n$-dimensional unit sphere with unit normal pointing outwards is $n$.
From the definition of the induced metric (\ref{induced_metric}) one directly obtains that
$$
	\frac{\partial^2 X}{\partial \theta^i \partial \theta^j}
	\cdot \frac{\partial X}{\partial \theta^l}
	= \frac{1}{2} \left( 
			\frac{\partial g_{jl}}{\partial \theta^i} 
			+ \frac{\partial g_{il}}{\partial \theta^j} 
			- \frac{\partial g_{ij}}{\partial \theta^l}				
		\right) 
	= g_{kl} \Gamma(g)^k_{ij}, 
$$
and hence,
$$
	\Gamma(g)^k_{ij} \frac{\partial X}{\partial \theta^k}
	= \frac{\partial^2 X}{\partial \theta^i \partial \theta^j}
		\cdot \frac{\partial X}{\partial \theta^l} g^{lk}
		\frac{\partial X}{\partial \theta^k}.
$$
The following identity 
$$
	\unit = (\nu \circ X) \otimes (\nu \circ X) 
	+ g^{lk} \frac{\partial X}{\partial \theta^l} \otimes \frac{\partial X}{\partial \theta^k}
$$
then gives the decomposition
\begin{equation}
\label{decomposition}
	\frac{\partial^2 X}{\partial \theta^i \partial \theta^j}
	= \frac{\partial^2 X}{\partial \theta^i \partial \theta^j} \cdot (\nu \circ X)
	 (\nu \circ X) + \Gamma(g)^k_{ij} \frac{\partial X}{\partial \theta^k}.
\end{equation}
From this identity, the definition of the Laplace operator (\ref{defi_laplacian}) and
the definition of the mean curvature in (\ref{mean_curvature}) it follows that
$$
	(\Delta_{g(t)} x) \circ \mathcal{C}_1^{-1}
	= g^{ij} \frac{\partial^2 X}{\partial \theta^i \partial \theta^j} \cdot (\nu \circ X)
	 (\nu \circ X) 
	= - (H \circ X)  (\nu \circ X). 
$$
This means that
\begin{equation}
	\label{Laplace_of_the_identity}
	\Delta_{g(t)} x = - (H \nu) \circ x.
\end{equation}
The mean curvature flow is therefore given by the following non-linear heat equation
$$
	\frac{\partial}{\partial t} x = \Delta_{g(t)} x,
$$
where $g(t) = x(t)^\ast \mathfrak{e}$.

A straightforward calculation gives the evolution of the metric $g(t)$
under the mean curvature flow, see, for example, in \cite{Ma10},
$$
	\frac{\partial}{\partial t} g_{ij} 
			= 2 (H \circ X) \frac{\partial^2 X}{\partial \theta^i \partial \theta^j} 
				\cdot (\nu \circ X)
			= - 	 2 (H \circ X) \mathcal{H}_{ij},
$$
where $\mathcal{H}_{ij}$ denotes the components of the second fundamental form $\mathcal{H}$
defined by
$$
	\mathcal{H}_{ij} := \frac{\partial X}{\partial \theta^i}
					\cdot \frac{\partial (\nu \circ X)}{\partial \theta^j}. 
$$

\subsection{The harmonic map heat flow}
We now introduce the harmonic map heat flow on the manifold $\M$.
A map $\psi: \M \times [0,T) \rightarrow \M$ is said to evolve according to the
harmonic map heat flow if 
\begin{align*}
	\frac{\partial}{\partial t} \psi = \Delta_{g,h} \psi.
\end{align*}
Here, $\Delta_{g,h}$ denotes the map Laplacian (\ref{map_Laplacian}) on $\M$ with
respect to the smooth metrics $g$ and $h$ on $\M$.
In the following we choose $g$ to be the time-dependent metric $g(t) = x(t)^\ast \mathfrak{e}$ induced by the embedding $x(t)$ on $\M$.
This choice is motivated by our aim to use the harmonic map heat flow $\psi$
for the reparametrization of the mean curvature flow $x(t)$.
In contrast to the metric $g(t)$, we keep the metric $h$ arbitrary
but fixed in time. 
In local coordinates, the harmonic map heat flow is given by
$$
	\frac{\partial}{\partial t} \Psi^q = g^{ij}
	\left( \frac{\partial^2 \Psi^q}{\partial \theta^i \theta^j} 
		- \Gamma(g)^k_{ij} \frac{\partial \Psi^q}{\partial \theta^k}
		+ (\Gamma(h)^q_{mn} \circ \Psi ) \frac{\partial \Psi^m}{\partial \theta^i}
						   \frac{\partial \Psi^n}{\partial \theta^j}		
		\right).
$$
Short-time existence and uniqueness results for this flow can be found in \cite{ES64}.
In the following, we choose the initial conditions for the harmonic map heat flow
to be the identity on $\M$, that is $\psi(\cdot, 0) = id(\cdot)$.

\subsection{The mean curvature-DeTurck flow}
The DeTurck trick was first introduced in \cite{DeT83}
in order to prove existence and uniqueness of solutions to the Ricci flow \cite{Ham82}.
Later, it was also used to prove existence and uniqueness
for the mean curvature flow, see for example \cite{Ba10, Ham89}.
However, to our knowledge, it has never been considered in a numerical setting so far.

The basic idea of the trick is to reparametrize the original evolution equations
by a smooth family of diffeomorphisms solving the harmonic map heat flow.
Since this is a rather special concept,
we here provide a detailed derivation of the reparametrized evolution equations.
In the first step, we combine the mean curvature flow with the harmonic map heat flow by
\[
(P)=
\left\{
\begin{aligned}
	& \frac{\partial}{\partial t} x = - (H \nu) \circ x,
	\quad \textnormal{with $x(\cdot,0) = x_0$ on $\M$,}
	\\
	& \frac{\partial}{\partial t} \psi_\alpha = \frac{1}{\alpha} \Delta_{g(t),h} \psi_\alpha,
	\quad \textnormal{with $g(t) := x(t)^\ast \mathfrak{e}$  
		and $\psi_\alpha(\cdot,0) = id(\cdot)$ on $\M$,}
\end{aligned}
\right.
\]
where $h$ is a fixed yet arbitrary smooth Riemannian metric on $\M$.
We have here introduced the inverse diffusion constant $\alpha > 0$
in the harmonic map heat flow.
As we will see below, this parameter determines the time scale on which
the tangential motions of the reparametrized flow take place. 
Since $\psi_\alpha(\cdot, 0) = id(\cdot)$ is initially the identity, the map $\psi_\alpha(t)$ remains
a diffeomorphism at least for short times. We can thus reparametrize the mean curvature flow by
\begin{equation}
	\hat{x}_\alpha(t) := (\psi_\alpha(t)^{-1})^\ast x(t) := x(t) \circ \psi_\alpha(t)^{-1}.
	\label{pull_back_of_MCF}
\end{equation}
The idea to use this reparametrization is usually called the DeTurck trick. Two properties of this reparametrization are of particular importance:
Firstly, it turns out that $\hat{x}_\alpha(t)$ is the solution to a strongly parabolic PDE, see \cite{Ba10}.
Secondly, this PDE does not depend on the solution $\psi(t)$ of the harmonic map heat flow.
The latter point, in particular, means that it will not be necessary to solve the harmonic map heat flow numerically, although
we use the above reparametrization to derive our numerical schemes. We now state the reparametrized evolution equations.  
\begin{lem}[Mean curvature-DeTurck flow]
\label{Lemma_MCF_DeTurck}
The evolution equation for the reparametrized embedding 
$\hat{x}_\alpha: \M \times [0,T) \rightarrow \mathbb{R}^{n+1}$ defined in 
(\ref{pull_back_of_MCF}) is given by
\begin{equation}
	\label{reparam_equation}
	\frac{\partial}{\partial t} \hat{x}_\alpha =
	\Delta_{\hat{g}_{\alpha}(t)} \hat{x}_\alpha - \frac{1}{\alpha} \nabla_{} 
			\hat{x}_\alpha (V_\alpha).
\end{equation}
In local coordinates this equation looks like
$$
	\frac{\partial}{\partial t} \hat{X}_\alpha =
	\hat{g}_\alpha^{ij} 
	\left( \frac{\partial^2 \hat{X}_\alpha}{\partial \theta^i \partial \theta^j}
	- \Gamma(\hat{g}_\alpha)^k_{ij} \frac{\partial \hat{X}_\alpha}{\partial \theta^k} \right)
	- \frac{1}{\alpha} V^j_\alpha \frac{\partial \hat{X}_\alpha}{\partial \theta^j}.
$$
Here, $\hat{g}_\alpha(t) := (\hat{x}_\alpha (t))^\ast \mathfrak{e}$ is the metric induced by the embedding $\hat{x}_\alpha(t)$,
and $V_\alpha$ is the vector field locally defined by  
\begin{equation}
\label{dual_V_alpha}
	V^j_{\alpha} := \hat{g}_\alpha^{kl} (\Gamma(h)^j_{kl} - \Gamma(\hat{g}_\alpha)^{j}_{kl}).
\end{equation}
\end{lem}
\begin{proof}
Below, we will make use of the fact that the pull-back metrics $(\psi_\alpha(t)^{-1})^\ast g(t)$ and $(\hat{x}_\alpha (t))^\ast \mathfrak{e}$
on $\M$ are equal
\begin{equation}
	\label{metric_is_pull_back_metric}
	(\psi_\alpha(t)^{-1})^\ast g(t)
	= (\psi_\alpha(t)^{-1})^\ast ( x(t)^\ast \mathfrak{e} )
	= (x(t) \circ \psi_\alpha(t)^{-1})^\ast \mathfrak{e}
	= (\hat{x}_\alpha(t))^\ast \mathfrak{e} = \hat{g}_\alpha(t).
\end{equation}
For the time derivative of $\hat{x}_\alpha$ we obtain
\begin{align*}
	\frac{\partial}{\partial t} \hat{x}_\alpha(t)
	&= \frac{\partial}{\partial t} x(t) \circ \psi_\alpha(t)^{-1}
		+ (\nabla_{} x \circ \psi_\alpha(t)^{-1})
				\left(\frac{\partial}{\partial t} ( \psi_\alpha(t)^{-1}) \right)
\\		
	&= 	(\Delta_{g(t)} x) \circ \psi_\alpha(t)^{-1}
		+ (\nabla_{} x \circ \psi_\alpha(t)^{-1})
				\left(\frac{\partial}{\partial t} ( \psi_\alpha(t)^{-1}) \right),
\end{align*}
where $\nabla x$ denotes the differential of $x$. 
From the identity 
$\psi_\alpha(t)^{-1} \circ \psi_\alpha(t) = id$,
it follows that
$$
	\frac{\partial (\Psi_\alpha^{-1})^j}{\partial t} \circ \Psi_\alpha  
	= - \left( \frac{\partial (\Psi_\alpha^{-1})^j}{\partial \theta^q}
		\right) \circ \Psi_\alpha
		\frac{\partial \Psi_\alpha^q}{\partial t},
$$
where $\Psi_\alpha = \mathcal{C}_2 \circ \psi_\alpha \circ \mathcal{C}_1^{-1}$.
And hence,
\begin{align*}
	\frac{\partial (\Psi_\alpha^{-1})^j}{\partial t} =
	- \frac{\partial (\Psi_\alpha^{-1})^j}{\partial \theta^q}
	\left(
	\frac{\partial \Psi_\alpha^q}{\partial t} \circ \Psi^{-1}_\alpha
	\right).
\end{align*}
For $X := x \circ \mathcal{C}_1^{-1}$
and $\hat{X}_\alpha := \hat{x}_\alpha \circ \mathcal{C}_2^{-1}
= x \circ \psi_\alpha(t)^{-1} \circ \mathcal{C}_2^{-1} = X \circ \Psi_\alpha^{-1}$,
we then obtain
\begin{align*}
	\frac{\partial (\Psi_\alpha^{-1})^j}{\partial t}
	\left(
	\frac{\partial X}{\partial \theta^j} \circ \Psi_\alpha^{-1}
	\right)
	&= - \left(
	\frac{\partial X}{\partial \theta^j} \circ \Psi_\alpha^{-1}
	\right) \frac{\partial (\Psi_\alpha^{-1})^j}{\partial \theta^q}
	\left(
	\frac{\partial \Psi_\alpha^q}{\partial t} \circ \Psi^{-1}_\alpha
	\right)
	\\
	&= - \frac{\partial \hat{X}_\alpha}{\partial \theta^q} \left(
	\frac{\partial \Psi_\alpha^q}{\partial t} \circ \Psi^{-1}_\alpha
	\right).
\end{align*}
We now conclude that
\begin{align}
	(\nabla_{} x \circ \psi_\alpha(t)^{-1})
				\left(\frac{\partial}{\partial t} ( \psi_\alpha(t)^{-1}) \right)
		&= - (\nabla \hat{x}_\alpha) 
			\left( \left( \frac{\partial}{\partial t} \psi_\alpha(t) \right)	
						\circ \psi_\alpha(t)^{-1} \right)
\nonumber						
\\		
		&= - \frac{1}{\alpha} (\nabla \hat{x}_\alpha)
			\left( (\Delta_{g(t),h} \psi_\alpha(t)) 	
						\circ \psi_\alpha(t)^{-1} \right).	
\label{interim_result}									
\end{align}
According to Remark 2.46 in \cite{CLN06}, the following identity holds
\begin{equation}
\label{transformation_map_Laplacian}
	(\Delta_{g(t),h} \psi_\alpha(t)) 	
						\circ \psi_\alpha(t)^{-1}
					= \Delta_{(\psi_\alpha(t)^{-1})^\ast g(t),h} id.
\end{equation}
Using the fact that $(\psi_\alpha(t)^{-1})^\ast g(t) = (\hat{x}_\alpha(t))^\ast \mathfrak{e} = \hat{g}_\alpha(t)$, we obtain
$$
	(\Delta_{g(t),h} \psi_\alpha(t)) 	
						\circ \psi_\alpha(t)^{-1} = 
						\Delta_{\hat{g}_\alpha(t),h} id.
$$
For the sake of completeness we here give a short proof of the identity
(\ref{transformation_map_Laplacian}). In local coordinates we have
\begin{align*}
	& \frac{\partial (\Psi^q_\alpha \circ \Psi_\alpha^{-1})}{\partial \theta^m}
	=  \frac{\partial \Psi^q_\alpha}{\partial \theta^i} \circ \Psi_\alpha^{-1}
		\frac{\partial (\Psi_\alpha^{-1})^i}{\partial \theta^m}, 
	\\
	& \frac{\partial^2 (\Psi^q_\alpha \circ \Psi_\alpha^{-1})}{\partial \theta^m \partial \theta^n}
	=  \frac{\partial^2 \Psi^q_\alpha}{\partial \theta^i \partial \theta^j}
		 \circ \Psi_\alpha^{-1}
		\frac{\partial (\Psi_\alpha^{-1})^i}{\partial \theta^m} 
		\frac{\partial (\Psi_\alpha^{-1})^j}{\partial \theta^n} 
		+ \frac{\partial \Psi^q_\alpha}{\partial \theta^i} \circ \Psi_\alpha^{-1}
		\frac{\partial^2 (\Psi_\alpha^{-1})^i}{\partial \theta^m \partial \theta^n}, 
	\\
	& \hat{g}_{\alpha mn} = (g_{ij} \circ \Psi_\alpha^{-1})
	 \frac{\partial (\Psi^{-1}_\alpha)^i }{\partial \theta^m}
	 \frac{\partial (\Psi^{-1}_\alpha)^j }{\partial \theta^n},
	\\
	& (g^{ij} \circ \Psi_\alpha^{-1}) = \hat{g}_\alpha^{mn}
		 \frac{\partial (\Psi^{-1}_\alpha)^i }{\partial \theta^m}
	 \frac{\partial (\Psi^{-1}_\alpha)^j }{\partial \theta^n}.
\end{align*}
A straightforward calculation also shows that
\begin{align*}
	\hat{g}_\alpha^{mn} \Gamma(\hat{g}_\alpha)^p_{mn} 
	\frac{\partial (\Psi_\alpha^{-1})^k}{\partial \theta^p}
	= (g^{ij} \circ \Psi_\alpha^{-1}) \Gamma(g)^k_{ij} \circ \Psi_\alpha^{-1}
		+ \hat{g}^{mn}_\alpha 
		\frac{\partial^2 (\Psi_\alpha^{-1})^k}{\partial \theta^m \partial \theta^n}.
\end{align*}
Using these formulas, we obtain
\begin{align*}
& (\mathcal{C}_2 \circ ( \Delta_{\hat{g}_\alpha(t),h} id ) \circ \mathcal{C}_1^{-1})^q
\\
&=	\hat{g}_\alpha^{mn}
	\left( 
	\frac{\partial^2 (\Psi^q_\alpha \circ \Psi_\alpha^{-1})}{\partial \theta^m \theta^n} 
	- \Gamma(\hat{g}_\alpha)^p_{mn} 
		\frac{\partial (\Psi^q_\alpha \circ \Psi_\alpha^{-1})}{\partial \theta^p}
		+   \Gamma(h)^q_{pr} 
			\frac{\partial (\Psi^p_\alpha \circ \Psi_\alpha^{-1})}{\partial \theta^m}
			\frac{\partial (\Psi^r_\alpha \circ \Psi_\alpha^{-1})}{\partial \theta^n}		
		\right)
\\
&= 	(g^{ij} \circ \Psi_\alpha^{-1})  
	\frac{\partial^2 \Psi^q_\alpha}{\partial \theta^i \partial \theta^j}
		 \circ \Psi_\alpha^{-1}	
	+  
	\frac{\partial \Psi^q_\alpha}{\partial \theta^i} \circ \Psi_\alpha^{-1}
	\frac{\partial^2 (\Psi_\alpha^{-1})^i}{\partial \theta^m \partial \theta^n}
	\hat{g}_\alpha^{mn}
\\
& \quad	
	- (g^{ij} \circ \Psi_\alpha^{-1}) \Gamma(g)^k_{ij} \circ \Psi_\alpha^{-1}
		\frac{\partial \Psi^q_\alpha}{\partial \theta^k} \circ \Psi_\alpha^{-1}
	- \hat{g}^{mn}_\alpha 
		\frac{\partial^2 (\Psi_\alpha^{-1})^k}{\partial \theta^m	\partial \theta^n}
		\frac{\partial \Psi^q_\alpha}{\partial \theta^k} \circ \Psi_\alpha^{-1}
\\
& \quad	
	+ \Gamma(h)^q_{pr} 
		\frac{\partial \Psi^p_\alpha}{\partial \theta^i} \circ \Psi_\alpha^{-1}
		\frac{\partial \Psi^r_\alpha}{\partial \theta^j} \circ \Psi_\alpha^{-1}
		(g^{ij} \circ \Psi_\alpha^{-1})
\\
&=	(g^{ij} \circ \Psi_\alpha^{-1}) \left(
	\frac{\partial^2 \Psi^q_\alpha}{\partial \theta^i \partial \theta^j}
	-  \Gamma(g)^k_{ij} 
		\frac{\partial \Psi^q_\alpha}{\partial \theta^k}
	+ ( \Gamma(h)^q_{pr}  \circ \Psi_\alpha )
		\frac{\partial \Psi^p_\alpha}{\partial \theta^i} 
		\frac{\partial \Psi^r_\alpha}{\partial \theta^j} 
	\right)  \circ \Psi_\alpha^{-1}.
\end{align*}
Together with the definition (\ref{map_Laplacian}) of the map Laplacian
this shows (\ref{transformation_map_Laplacian}).
On the other hand the map Laplacian of the identity is just given by
$$
	(\mathcal{C}_2 \circ ( \Delta_{\hat{g}_\alpha(t),h} id ) \circ \mathcal{C}_1^{-1})^j
	= \hat{g}_\alpha^{kl} ( \Gamma(h)^j_{kl} -\Gamma(\hat{g}_\alpha)^{j}_{kl}) 
	= V^j_\alpha,
$$
where we have obtained the components $V_\alpha^j$ of the vector field $V_\alpha$
defined in (\ref{dual_V_alpha}). We summarize that
$$
	(\Delta_{g(t),h} \psi_\alpha(t)) 	
						\circ \psi_\alpha(t)^{-1} = V_\alpha,
$$
and with (\ref{interim_result}) we obtain
$$
   	(\nabla_{} x \circ \psi_\alpha(t)^{-1})
				\left(\frac{\partial}{\partial t} ( \psi_\alpha(t)^{-1}) \right)
	= - \frac{1}{\alpha} (\nabla \hat{x}_\alpha) (V_\alpha).			
$$
From identity (\ref{Laplace_of_the_identity}) it follows that
\begin{align*}
	(\Delta_{g(t)} x) \circ \psi_\alpha(t)^{-1}
	&= - ((H \nu) \circ x ) \circ \psi_\alpha(t)^{-1}	
	= - ( H \nu) \circ (x \circ \psi_\alpha(t)^{-1})
	= - ( H \nu) \circ \hat{x}_{\alpha},
\end{align*}
and thus,
\begin{equation*}
	(\Delta_{g(t)} x) \circ \psi_\alpha(t)^{-1} = \Delta_{\hat{g}_\alpha(t)} \hat{x}_\alpha,
\end{equation*}
where we have used that $g(t) = x(t)^\ast \mathfrak{e}$ and $\hat{g}_\alpha(t) = (\hat{x}_\alpha(t))^\ast \mathfrak{e}$.
\end{proof}
\begin{remark}
Please note that one can, in principle, recover the solution $x(t)$ to (\ref{MCF_equation}) from the solution to
(\ref{reparam_equation}) by solving the ODE 
$\frac{\partial}{\partial t} \psi_\alpha = \frac{1}{\alpha} V_\alpha \circ \psi_\alpha$ and setting 
$x(t) = \hat{x}_\alpha(t) \circ \psi_\alpha(t)$. Actually, this is how one can prove short-time existence of solutions
to the mean curvature flow by using the existence result for the mean curvature-DeTurck flow, see, for example, in \cite{Ba10}.
The reason why it is easier to establish short-time existence for the mean curvature-DeTurck flow than for the mean curvature flow
itself is that the DeTurck flow is strongly parabolic whereas the original flow is not. 
\end{remark}

\subsection{The time separation trick}
We next introduce an idea, which will turn out to be advantageous for the spatial discretization and its numerical analysis. 

In order to motivate this idea we will first apply the results of Lemma \ref{Lemma_MCF_DeTurck} to the curve shortening flow.
In this case the reference manifold $\M$ is one-dimensional. Without loss of generality
we can assume that $\M$ is parametrized in such a way that
$\hat{X}_\alpha: [0, 2\pi] \rightarrow \mathbb{R}^2$ with $\hat{X}_\alpha(0) = \hat{X}_\alpha(2\pi)$.
The metric $\hat{g}_\alpha$ is then given by 
$\hat{g}_{\alpha 11} = |\frac{\partial \hat{X}_\alpha}{\partial \theta}|^2$,
and hence,
$$
	\Gamma(\hat{g})^1_{11} = \frac{1}{2 } \bigg|\frac{\partial \hat{X}_\alpha}{\partial \theta}\bigg|^{-2} \frac{\partial }{\partial \theta} \bigg|\frac{\partial \hat{X}_\alpha}{\partial \theta}\bigg|^2
	= \bigg|\frac{\partial \hat{X}_\alpha}{\partial \theta}\bigg|^{-2} 		\frac{\partial \hat{X}_\alpha}{\partial \theta} 
	\cdot \frac{\partial^2 \hat{X}_\alpha}{\partial \theta^2}.
$$
Since $h$ is a fixed yet arbitrary metric, we are allowed to choose $h$ 
such that $h_{11}$ is constant on $[0,2\pi]$ and thus $\Gamma(h)^1_{11}=0$.
The reparametrized evolution equations of the curve shortening flow
in local coordinates are then given by
\begin{equation}
	\frac{\partial}{\partial t} \hat{X}_\alpha =
	\bigg|\frac{\partial \hat{X}_\alpha}{\partial \theta}\bigg|^{-2}
	\left( \frac{\partial^2 \hat{X}_\alpha}{\partial \theta^2}
		- \bigg|\frac{\partial \hat{X}_\alpha}{\partial \theta}\bigg|^{-2} 			
				\frac{\partial \hat{X}_\alpha}{\partial \theta} 
				\cdot \frac{\partial^2 \hat{X}_\alpha}{\partial \theta^2} 
				\frac{\partial X}{\partial \theta} \right)
		+ \frac{1}{\alpha}
			\bigg|\frac{\partial \hat{X}_\alpha}{\partial \theta}\bigg|^{-4} 
				\frac{\partial \hat{X}_\alpha}{\partial \theta} 
				\cdot \frac{\partial^2 \hat{X}_\alpha}{\partial \theta^2} \frac{\partial 
				\hat{X}_\alpha}{\partial \theta},
\label{intermediate_result_CSF}
\end{equation}
which directly simplifies to
\begin{equation}
	\frac{\partial}{\partial t} \hat{X}_\alpha =
	\bigg|\frac{\partial \hat{X}_\alpha}{\partial \theta}\bigg|^{-2}
	\frac{\partial^2 \hat{X}_\alpha}{\partial \theta^2}
		+ \frac{1 - \alpha}{\alpha}
			\bigg|\frac{\partial \hat{X}_\alpha}{\partial \theta}\bigg|^{-4} 
				\frac{\partial \hat{X}_\alpha}{\partial \theta} 
				\cdot \frac{\partial^2 \hat{X}_\alpha}{\partial \theta^2} \frac{\partial 
				\hat{X}_\alpha}{\partial \theta}.
	\label{intermediate_result_CSF_2}			
\end{equation}
For $\alpha = 1$ 
this equation was the starting point of the analysis in \cite{DD94}. 
In \cite{St14} it was later derived for arbitrary $\alpha > 0$.
However, the author in \cite{St14} did not use the following trick
described below, which is crucial
for two kinds of reasons. Firstly, the trick makes it possible to choose, 
at least formally, $\alpha = 0$ in the reparametrized equations.
This choice unveils the origin of the tangential redistributions 
in the scheme proposed in \cite{BGN11}. 
Secondly, our trick
leads to an equation that can be directly discretized in space. 
In contrast to this result, the author in \cite{St14} 
had to introduce a variable for the curvature vector 
$$
- (H \nu) \circ \hat{X}_\alpha 
= \bigg|\frac{\partial \hat{X}_\alpha}{\partial \theta}\bigg|^{-1}
  \frac{\partial}{\partial \theta} \left( 
  	\bigg|\frac{\partial \hat{X}_\alpha}{\partial \theta}\bigg|^{-1}
  	\frac{\partial \hat{X}_\alpha}{\partial \theta} 
  \right)	
$$ 
in order to be able to discretize the reparametrized equations with piecewise linear
finite elements, see Problem 4.1.7 in \cite{St14}. As a result, the author obtained a system of equations
for the computation of the mean curvature vector and of the curve shortening flow.
However, for the computation of the curve shortening flow,
such a system seems to be a bit exaggerated.
We would like to emphasize that in contrast to the results in \cite{St14}, 
the above derivations are 
also valid for the higher-dimensional case, that is for the mean curvature flow. 

We now continue with the general case. Our trick is based on the following two 
observations. Firstly, as we have seen in (\ref{Laplace_of_the_identity}),
the Laplace operator with respect to metric $\hat{g}_\alpha$ of the map
$\hat{x}_{\alpha}$ is equal to $- (H\nu) \circ \hat{x}_\alpha$. 
Secondly, the Laplace operator with respect to $\hat{g}_\alpha$
satisfies the following identity
$$
	\Delta_{\hat{g}_\alpha} \hat{x}_\alpha
	= \tr_{\hat{g}_\alpha}(\nabla^{\hat{g}_\alpha} \nabla \hat{x}_\alpha)
	= \tr_{\hat{g}_\alpha}(\nabla^{h} \nabla \hat{x}_\alpha)
		+ \nabla_{} \hat{x}_\alpha (V_\alpha),
$$
where $\tr_{\hat{g}_\alpha}$ denotes the trace with respect 
to the metric $\hat{g}_\alpha$.
This easily follows from definition (\ref{dual_V_alpha})
and the formulas for the covariant derivatives
of the differential $\nabla \hat{x}_\alpha$ with
respect to the metrics $\hat{g}_\alpha$ and $h$, which are
\begin{align*}
	& (\nabla^{\hat{g}_\alpha}_i \nabla_j \hat{x}_\alpha) \circ \mathcal{C}_1^{-1}
		= \frac{\partial^2 \hat{X}_\alpha}{\partial \theta^i \partial \theta^j}
	- \Gamma(\hat{g}_\alpha)^k_{ij} \frac{\partial \hat{X}_\alpha}{\partial \theta^k},
	\\
	& (\nabla^{h}_i \nabla_j \hat{x}_\alpha) \circ \mathcal{C}_1^{-1}
		= \frac{\partial^2 \hat{X}_\alpha}{\partial \theta^i \partial \theta^j}
	- \Gamma(h)^k_{ij} \frac{\partial \hat{X}_\alpha}{\partial \theta^k}.
\end{align*} 
Please note that the vector field
$(\nabla_{V_\alpha} \hat{x}_\alpha) \circ \mathcal{C}_1^{-1} 
= V_\alpha^j \frac{\partial \hat{X}_\alpha}{\partial \theta^j}$ 
is tangential 
to the embedded hypersurface $\Gamma(t)$.
We thus obtain the following decomposition of the operator
$\tr_{\hat{g}_\alpha}(\nabla^h \nabla \hat{x}_\alpha)$ in its normal and tangential parts
\begin{align}
&	(\nu \circ \hat{x}_\alpha) \otimes (\nu \circ \hat{x}_\alpha)
	 \tr_{\hat{g}_\alpha}(\nabla^h \nabla \hat{x}_\alpha)
	= \Delta_{\hat{g}_\alpha} \hat{x}_\alpha,
	\label{normal_part}
	\\
&	(P \circ \hat{x}_\alpha) \tr_{\hat{g}_\alpha}(\nabla^h \nabla \hat{x}_\alpha)
	= - \nabla \hat{x}_\alpha (V_\alpha),
	\label{tangential_part}
\end{align}
where $P = \unit - \nu \otimes \nu$ is the projection onto the tangent bundle of $\Gamma$. Inserting this decomposition into (\ref{reparam_equation}) then gives
\begin{align*}
\frac{\partial}{\partial t} \hat{x}_\alpha &=
	(\nu \circ \hat{x}_\alpha) \otimes (\nu \circ \hat{x}_\alpha)
	 \tr_{\hat{g}_\alpha}(\nabla^h \nabla \hat{x}_\alpha) 
	 + \frac{1}{\alpha} 
	 (P \circ \hat{x}_\alpha) \tr_{\hat{g}_\alpha}(\nabla^h \nabla \hat{x}_\alpha)
\\	 
	&= \left( (\nu \circ \hat{x}_\alpha) \otimes (\nu \circ \hat{x}_\alpha)
	   +  \frac{1}{\alpha} P \circ \hat{x}_\alpha \right)
	 \tr_{\hat{g}_\alpha}(\nabla^h \nabla \hat{x}_\alpha).
\end{align*}
Applying the inverse map 
\begin{equation}
\left( (\nu \circ \hat{x}_\alpha) \otimes (\nu \circ \hat{x}_\alpha) + \alpha P \circ \hat{x}_\alpha \right)
 = \left(\alpha \unit + (1 - \alpha) (\nu \circ \hat{x}_\alpha) \otimes (\nu \circ \hat{x}_\alpha) \right)
 \label{inverse_map}
\end{equation}
leads to the following evolution equations.
\begin{thm}
The mean curvature-DeTurck flow $\hat{x}_\alpha: \M \times [0,T) \rightarrow \mathbb{R}^{n+1}$
defined in (\ref{pull_back_of_MCF}) satisfies
\begin{align}
\label{reparam_MCF}
\left( \alpha \unit + (1 - \alpha) (\nu \circ \hat{x}_\alpha) \otimes (\nu \circ \hat{x}_\alpha) \right) 
\frac{\partial}{\partial t}\hat{x}_\alpha
	= \tr_{\hat{g}_\alpha} (\nabla^h \nabla \hat{x}_\alpha),
	\quad \textnormal{with $\hat{x}_\alpha(\cdot,0) = x_0(\cdot)$ on $\M$},
\end{align}
and in local coordinates respectively,
\begin{equation*}
	\left( \alpha \unit + (1 - \alpha) (\nu \circ \hat{X}_\alpha) \otimes (\nu \circ \hat{X}_\alpha) \right) 
	\frac{\partial}{\partial t} \hat{X}_\alpha =
	\hat{g}_\alpha^{ij} 
	\left( \frac{\partial^2 \hat{X}_\alpha}{\partial \theta^i \partial \theta^j}
		- \Gamma(h)^k_{ij} \frac{\partial \hat{X}_\alpha}{\partial \theta^k} \right),
\end{equation*}
with $\hat{X}_\alpha(\cdot, 0) = X_0(\cdot)$.
\end{thm}
Please note that the initial values of the reparametrized flow 
are just given by the initial values of the original flow.
The reason for this is that the map $\psi(t)$ satisfies $\psi(\cdot,0) = id(\cdot)$ on $\M$.
For the curve shortening flow with the metric $h$ chosen such that $h_{11}$ is constant, we obtain
\begin{equation}
	\label{reparam_CSF}
	\left( \alpha \unit + (1 - \alpha) (\nu \circ \hat{X}_\alpha) \otimes (\nu \circ \hat{X}_\alpha) \right) 
	\frac{\partial}{\partial t} \hat{X}_\alpha = 
	\bigg|\frac{\partial \hat{X}_\alpha}{\partial \theta}\bigg|^{-2}
	\frac{\partial^2 \hat{X}_\alpha}{\partial \theta^2}.
\end{equation}
This equation will be the basis of the next section.
Although the inverse diffusion constant $\alpha$ has to be positive
in the derivation of (\ref{reparam_MCF}),
it is yet possible to formally choose $\alpha =0$ in this equation.
This would lead to
\begin{align*}
(\nu \circ \hat{x}_\alpha) (\nu \circ \hat{x}_\alpha)\cdot
\frac{\partial}{\partial t}\hat{x}_\alpha
	= \tr_{\hat{g}_\alpha} (\nabla^h \nabla \hat{x}_\alpha).
\end{align*}

\section{The reparametrized curve shortening flow}
\subsection{Weak formulation and discretization}
In this section, we develop an algorithm for the computation of the
reparametrized curve shortening flow (\ref{reparam_CSF}).
Our algorithm is based on a straightforward discretization of the weak formulation
of the reparametrized flow with piecewise linear finite elements.
Henceforward, we omit the subscript $\alpha$ and write $\hat{X}_t$
and $\hat{X}_\theta$ instead of $\frac{\partial \hat{X}}{\partial t}$
and  $\frac{\partial \hat{X}}{\partial \theta}$.
Furthermore, we use the notation $\nu$ for the local parametrization
of the unit normal field $\nu \circ \hat{X}$. 
Multiplying (\ref{reparam_CSF}) by a test function 
$\varphi \in H^{1,2}(\mathbb{R}/2\pi,\mathbb{R}^2)$ as well as by the density function 
$|\hat{X}_\theta|^2$ and integrating by parts
yield 
\begin{equation}
\label{weak_formulation_reparam_CSF}
	\Int \left( \alpha \hat{X}_t \cdot \varphi 
	+ (1 - \alpha) (\hat{X}_t \cdot \nu) (\nu \cdot \varphi) \right) 
	 |\hat{X}_\theta|^2 d\theta 
	+ \Int \hat{X}_\theta \cdot \varphi_\theta d\theta =0,
	\quad 0 < t < T.
\end{equation}
This weak formulation is now discretized in space by linear finite elements.
In the following we consider 
the finite element mesh $\theta_j \in [0, 2\pi)$, $j = 1, \ldots, N$, 
with grid size $h = |\theta_{j+1} - \theta_j|$
where $\theta_{N+1} := 2 \pi + \theta_1$.
The space of continuous 
functions $\varphi_h: \mathbb{R}/ 2\pi \rightarrow \mathbb{R}$ 
that are linear on $[\theta_j, \theta_{j+1}]$, 
$\forall j = 1, \ldots, N$, is denoted by $\mathcal{S}_h$.
The basis functions $\phi_j \in \mathcal{S}_h$ are defined such that
$\phi_j(\theta_i) = \delta_{ij}$, $\forall i,j = 1, \ldots, N$.
The linear Lagrange interpolation for a continuous function $f$ on 
$\mathbb{R}/ 2\pi$ is defined by
$$
	I_h f := \sum_{j=1}^N f(\theta_j) \phi_j.
$$ 
For a curve $\Gamma_h = \hat{X}_h([0,2\pi))$ with $\hat{X}_h \in \mathcal{S}_h^2$, a 
piecewise constant vector field $\nu_h: [0,2\pi] \rightarrow \mathbb{R}^2$ with $|\nu_h| = 1$ and
\begin{equation*}
	\nu_h \cdot (\hat{X}_h(\theta_{j+1}) - \hat{X}_h(\theta_{j})) = 0 \quad \textnormal{on $[\theta_j, \theta_{j+1})$ for $j=1, \ldots, N,$}
\end{equation*}
is called a unit normal vector field to $\hat{X}_h$.
\subsection{ Convergence results }
\begin{thm}
\label{approximation_theorem}
Let $\alpha \in (0,1]$ 
and suppose that $\hat{X} \in C^{2,1}(\mathbb{R}/ 2\pi \times [0,T],\mathbb{R}^2)$ is a solution of
\begin{align*}
	& \alpha \hat{X}_t + (1 - \alpha) ( \nu \cdot \hat{X}_t ) \nu 
	= | \hat{X}_\theta |^{-2} \hat{X}_{\theta \theta},
	\quad \textnormal{in $\mathbb{R}/ 2\pi \times (0,T)$,}
	\\
	& \hat{X}(\cdot, 0) = X_0(\cdot),
	\quad \textnormal{on $\mathbb{R}/ 2\pi$,} 
\end{align*}
with
\begin{align}
	& \hat{X}_t \in L^\infty((0,T), H^{1,2}(\mathbb{R}/ 2 \pi,\mathbb{R}^2))
				\cap L^2((0,T), H^{2,2}(\mathbb{R}/ 2\pi,\mathbb{R}^2)),
	\nonumber			
	\\
	& | \hat{X}_\theta | \geq c_0 > 0,
	  \quad \textnormal{in $\mathbb{R} / 2 \pi \times [0,T]$.}  
	\label{lower_bound}  
\end{align}
Then there exists a constant $h_0 > 0$ depending on $\hat{X}$, $T$ and $\alpha$ such that for
every $0 < h \leq h_0$ there is a unique solution $\hat{X}_h \in H^{1,2}((0,T), \mathcal{S}_h^2)$ of the 
non-linear, semi-discrete problem
\begin{align}
\label{semi_discrete_problem_CSF}
	\Int \left( \alpha \hat{X}_{ht} \cdot \varphi_h 
		+ (1 - \alpha) (\hat{X}_{ht} \cdot \nu_h)(\nu_h \cdot \varphi_h) \right) 
				|\hat{X}_{h\theta}|^2 d\theta 
	+ \Int \hat{X}_{h\theta} \cdot \varphi_{h\theta} d\theta = 0,
\end{align}
$\forall \varphi_h \in \mathcal{S}_h^2, 0 < t < T$, 
with initial data $\hat{X}_h(\cdot ,0) = (I_h X_0)(\cdot)$ on $\mathbb{R} / 2 \pi$,
and
\begin{align*}
	& \alpha \int_0^T \|( \hat{X}_t - \hat{X}_{ht} )(t) \|^2_{L^2(0,2\pi)} dt
	 	+ (1 - \alpha) 	
	 \int_0^T   \| \nu_h \cdot ( \hat{X}_t - \hat{X}_{ht} )(t) \|^2_{L^2(0,2\pi)} dt
\\ 
	& + \max_{t \in [0,T]} \| ( \hat{X} - \hat{X}_h )(t) \|^2_{H^{1,2}(0,2\pi)}  
	\leq C e^{\frac{M}{\alpha}T} h^2.
\end{align*}
The constants $C$ and $M$ depend on the continuous solution $\hat{X}$ and on $T$.
\end{thm}
\proof

The following proof is adopted from \cite{DD94},
where the special case $\alpha = 1$ is considered.
It is based on the Schauder fixed point theorem. For this,
we first introduce the Banach space $\mathcal{Z}_h := C^0([0,T], \mathcal{S}^2_h)$
of time-continuous functions with values in $\mathcal{S}^2_h$ equipped
with the norm 
$$
	\| \hat{X}_h \|_{\mathcal{Z}_h}
	:= \sup_{t \in [0,T]} \| \hat{X}_h(t) \|_{L^2},
$$  
and the convex subset $\mathcal{B}_h$ defined by
\begin{equation}
	\mathcal{B}_h :=
	\left\{ \hat{X}_h \in \mathcal{Z}_h ~|~
		\sup_{t \in [0,T]} e^{- \frac{M}{\alpha} t} 
		\| (\hat{X}_\theta - \hat{X}_{h \theta})(t) \|^2_{L^2}
		\leq K^2 h^2	
		\ \textnormal{and $\hat{X}_h(\cdot,0) = (I_h X_0)(\cdot)$}
	 \right\},
\end{equation}
where we will choose the constants $K, M > 0$ below.
Please note that the space $\mathcal{S}_h$ is finite-dimensional.
For sufficiently large $K$ the subset $\mathcal{B}_h$ is non-empty, since then
$I_{h} \hat{X} \in \mathcal{B}_h$. Moreover, $\mathcal{B}_h$ is closed.
For $\hat{X}_h \in \mathcal{B}_h$, $0 \leq t \leq T$, we obtain
\begin{align*}
	\| (\hat{X}_\theta - \hat{X}_{h\theta})(t) \|_{L^\infty} &\leq
	\| (\hat{X}_\theta - (I_h \hat{X})_\theta )(t) \|_{L^\infty} 
	+ \| ((I_h \hat{X})_\theta - \hat{X}_{h \theta})(t) \|_{L^\infty}
	\\
	&\leq Ch 
	+ \frac{C}{\sqrt{h}} \| ((I_h \hat{X})_\theta - \hat{X}_{h \theta})(t) \|_{L^2} 
	\\
	&\leq Ch 
	+ \frac{C}{\sqrt{h}} ( \| ((I_h \hat{X})_\theta - \hat{X}_{\theta})(t) \|_{L^2} 
	+  \| (\hat{X}_\theta - \hat{X}_{h \theta})(t) \|_{L^2} )
	\\
	&\leq C \sqrt{h}(1 + e^{\frac{M}{2\alpha}T} K),
\end{align*}
where we have used interpolation and inverse inequalities.
Because of assumption (\ref{lower_bound}), we can thus assume
that for $h_0 = h_0(M,K,\hat{X},T,\alpha)$ sufficiently small the following lower and upper bounds hold
\begin{align}
\label{lower_upper_bounds}
	|\hat{X}_{h \theta}| \geq \frac{1}{2} c_0,
	\quad \textnormal{and} \quad
	| \hat{X}_{h \theta} | \leq C
	\quad \textnormal{in $\mathbb{R} / 2\pi \times [0,T]$.}
\end{align}
We now consider the operator $F$ defined by
$$
	F: \hat{X}_h \in \mathcal{B}_h \mapsto \hat{Y}_h \in \mathcal{Z}_h, 
$$
where $\hat{Y}_h \in \mathcal{Z}_h$ is the unique solution of the following
linear system of ODEs
\begin{align}
\label{ODE}
	\Int \left( \alpha \hat{Y}_{ht} \cdot \varphi_h 
		+ (1 - \alpha) (\hat{Y}_{ht} \cdot \nu_h)(\nu_h \cdot \varphi_h) \right) 
				|\hat{X}_{h\theta}|^2 d\theta 
	+ \Int \hat{Y}_{h\theta} \cdot \varphi_{h\theta} d\theta = 0
\end{align}
$\forall \varphi_h \in \mathcal{S}^2_h, 0 < t < T$, 
with initial data $\hat{Y}_h(\cdot ,0) = (I_h X_0)(\cdot)$.
Here, $\nu_h$ is a piecewise constant unit normal vector field to
the embedding $\hat{X}_h$.
The operator $F$ is continuous, since the solution $\hat{Y}_h$ continuously depends on 
$\hat{X}_h$.
Furthermore, we show below that $\hat{Y}_h \in \mathcal{B}_h$
and that $\| \hat{Y}_h \|_{H^{1,2}((0,T), \mathcal{S}^2_h)} \leq C$.
Hence, we have $F(\mathcal{B}_h) \subset \mathcal{B}_h$, and since the embedding 
$H^{1,2}(0,T) \hookrightarrow C^0([0,T])$ is compact,
$F(\mathcal{B}_h) \subset H^{1,2}((0,T), \mathcal{S}^2_h)$ is a precompact
subset of $\mathcal{B}_h \subset C^0([0,T], \mathcal{S}^2_h)$. 
The Schauder fixed point theorem therefore implies the existence of a fixed point
$F(\hat{X}_h) = \hat{X}_h$, that is a solution of (\ref{semi_discrete_problem_CSF}).
The uniqueness of solutions to ODEs of the form
$ x_t = f(x,t)$, for $f : G \subset \mathbb{R}^{m+1} \rightarrow \mathbb{R}^m$
locally Lipschitz continuous with respect to the variable $x$, 
implies that the solution $\hat{X}_h$ to (\ref{semi_discrete_problem_CSF})
is unique. In order to see that (\ref{semi_discrete_problem_CSF}) is equivalent
to such an equation, one chooses $\varphi_h = \phi_i e_\beta$, $i=1,\ldots,N$,
$\beta =1,2$, with $e_1 = (1,0)^T$ and $e_2 = (0,1)^T$,
and inserts 
$\hat{X}_h = \sum_{\gamma=1}^2 \sum_{j=1}^N \mathbf{\hat{X}}^{j\gamma} \phi_j e_\gamma$.
The resulting non-linear ODE then is
$$
	\sum_{\gamma=1}^2 \sum_{j=1}^N M_{ij \beta \gamma}(\mathbf{\hat{X}}) \mathbf{\hat{X}}^{j\gamma}_t 
	+ \sum_{\gamma=1}^2 \sum_{j=1}^N S_{ij \beta \gamma} \mathbf{\hat{X}}^{j\gamma} = 0,
	\quad
	\forall i=1, \ldots, N, \beta = 1,2,
$$
where the mass matrix 
$M := (M_{ij \beta \gamma}(\mathbf{\hat{X}}) ) \in \mathbb{R}^{(2N) \times (2N) }$ 
non-linearly depends on the vector 
$\mathbf{\hat{X}} := (\mathbf{\hat{X}}^{j\gamma}) \in \mathbb{R}^{2N}$,
whereas the components of the stiffness matrix 
$S := (S_{ij \beta \gamma} ) \in \mathbb{R}^{(2N) \times (2N) }$ are constants. 
The mass matrix is invertible
if $(\mathbf{\hat{X}}^{j+1, \gamma})_{\gamma=1,2} 
\neq (\mathbf{\hat{X}}^{j, \gamma})_{\gamma=1,2}$, 
for all $j= 1, \ldots, N$, 
where 
$(\mathbf{\hat{X}}^{N+1, \gamma})_{\gamma=1,2} 
:= (\mathbf{\hat{X}}^{1, \gamma})_{\gamma=1,2}$.
Furthermore, its inverse is locally Lipschitz continuous
in $G^\ast := \{ \mathbf{\hat{X}} \in \mathbb{R}^{2N} |~ 
(\mathbf{\hat{X}}^{j+1, \gamma} )_{\gamma=1,2} 
\neq (\mathbf{\hat{X}}^{j \gamma})_{\gamma=1,2} \}$.
Since $X_0$ is an embedding and $\hat{X}_h(\cdot,0) := (I_h X_0)(\cdot)$,
we have $\mathbf{\hat{X}}(0) \in G^\ast$.
Hence, the standard uniqueness theorem
for non-linear ODEs applies. 

We now show that $\hat{Y}_h \in \mathcal{B}_h$. First, we obviously have 
$\hat{Y}_h \in \mathcal{Z}_h$, and from the weak formulations
(\ref{weak_formulation_reparam_CSF}) and (\ref{ODE}) we obtain
\begin{align*}
  &\Int \left( \alpha (\hat{X}_t - \hat{Y}_{ht}) \cdot \varphi_h 
  		+ (1 - \alpha) (\hat{X}_t - \hat{Y}_{ht}) \cdot \nu_h (\nu_h \cdot \varphi_h) \right)
  				|\hat{X}_{h\theta}|^2 d\theta
  + \Int (\hat{X}_\theta - \hat{Y}_{h\theta}) \cdot \varphi_{h\theta} d\theta
  \\
  &=	 \Int \left( \alpha \hat{X}_t \cdot \varphi_h + (1- \alpha) (\hat{X}_t \cdot \nu_h)(\nu_h \cdot \varphi_h) \right) 
  				|\hat{X}_{h\theta}|^2 d\theta	
  \\
  & \quad - \Int \left( \alpha \hat{X}_t \cdot \varphi_h + (1- \alpha) (\hat{X}_t \cdot \nu)(\nu \cdot \varphi_h) \right) 
  				|\hat{X}_{\theta}|^2 d\theta		
  \\
  &=	 \Int (|\hat{X}_{h\theta}|^2 - |\hat{X}_{\theta}|^2) \left( \alpha \hat{X}_t \cdot \varphi_h + (1- \alpha) (\hat{X}_t \cdot \nu_h)(\nu_h \cdot \varphi_h) \right) 
  				 d\theta	
  \\
  & \quad + \Int |\hat{X}_{\theta}|^2 (1- \alpha) \left(
  			(\hat{X}_t \cdot \nu_h)(\nu_h \cdot \varphi_h) - (\hat{X}_t \cdot \nu)(\nu \cdot \varphi_h) \right) d\theta	
   \\
  &=	 \Int (|\hat{X}_{h\theta}|^2 - |\hat{X}_{\theta}|^2) \left( \alpha \hat{X}_t \cdot \varphi_h + (1- \alpha) (\hat{X}_t \cdot \nu_h)(\nu_h \cdot \varphi_h) \right) 
  				 d\theta	
  \\
  & \quad +  (1- \alpha) \Int |\hat{X}_{\theta}|^2 \left(
  			\hat{X}_t \cdot (\nu_h - \nu)(\nu_h \cdot \varphi_h) 
  						+ (\hat{X}_t \cdot \nu)(\nu_h - \nu) \cdot \varphi_h \right) d\theta.
\end{align*}
We now choose $\varphi_h = I_h \hat{X}_t - \hat{Y}_{ht}$ and get
\begin{align*}
	&\Int \left( \alpha |\hat{X}_t - \hat{Y}_{ht}|^2 
	+ (1 - \alpha) |\nu_h \cdot (\hat{X}_t - \hat{Y}_{ht})|^2 \right) |\hat{X}_{h\theta}|^2 d\theta
	+ \Int (\hat{X}_\theta - \hat{Y}_{h\theta}) 
						\cdot (\hat{X}_{t\theta} - \hat{Y}_{ht\theta}) d\theta
\\
	&= \Int \left( \alpha (\hat{X}_t - \hat{Y}_{ht}) \cdot (\hat{X}_t - I_h \hat{X}_t)
	 	+ (1-\alpha)(\hat{X}_t - \hat{Y}_{ht}) \cdot \nu_h 
	 	\nu_h \cdot (\hat{X}_t - I_h \hat{X}_t) \right) |\hat{X}_{h\theta}|^2 d\theta	
	\\ 			
	& \quad + \Int (\hat{X}_\theta - \hat{Y}_{h\theta})
				\cdot(\hat{X}_{t\theta} - (I_h \hat{X}_t)_\theta) d\theta			
	\\
	& \quad + \Int (|\hat{X}_{h\theta}|^2 - |\hat{X}_{\theta}|^2) 
		\left( \alpha \hat{X}_t \cdot (I_h \hat{X}_t - \hat{Y}_{ht}) 
	+ (1- \alpha) (\hat{X}_t \cdot \nu_h) \nu_h \cdot (I_h \hat{X}_t - \hat{Y}_{ht}) \right) 
  	d\theta	
  \\
  & \quad +  (1- \alpha) \Int |\hat{X}_{\theta}|^2 \left(
  			\hat{X}_t \cdot (\nu_h - \nu) \nu_h \cdot (I_h \hat{X}_t - \hat{Y}_{ht}) 
+ (\hat{X}_t \cdot \nu)(\nu_h - \nu) \cdot (I_h \hat{X}_t - \hat{Y}_{ht}) \right) d\theta	.
\end{align*}
Using the lower and upper bounds in (\ref{lower_upper_bounds}), we estimate
\begin{align*}
	&\frac{c_0^2}{4} \Int \alpha |\hat{X}_t - \hat{Y}_{ht}|^2 
		+ (1 - \alpha) |\nu_h \cdot (\hat{X}_t - \hat{Y}_{ht})|^2 d\theta
		+ \frac{1}{2} \frac{d}{dt} \Int |\hat{X}_\theta - \hat{Y}_{h\theta} |^2 d\theta
	\\
	&\leq C \Int \alpha |\hat{X}_t - \hat{Y}_{ht}||\hat{X}_t - I_h \hat{X}_t| 
			+ (1 - \alpha) |\nu_h \cdot (\hat{X}_t - \hat{Y}_{ht})|
					|\nu_h \cdot (\hat{X}_t - I_h \hat{X}_t)| d\theta	
	\\
	& \quad + \|(\hat{X}_t - I_h \hat{X}_t)_\theta \|_{L^2} 
				\| \hat{X}_\theta - \hat{Y}_{h\theta} \|_{L^2}		
		+ \alpha \Int ||\hat{X}_{h\theta}| - |\hat{X}_\theta|| 
					(|\hat{X}_{h\theta}| + |\hat{X}_\theta|) |\hat{X}_t| 
					|I_h \hat{X}_t - \hat{Y}_{ht}| d\theta
	\\
	& \quad + (1 - \alpha)	\Int ||\hat{X}_{h\theta}| - |\hat{X}_\theta|| 
					(|\hat{X}_{h\theta}| + |\hat{X}_\theta|) |\hat{X}_t| 
					|\nu_h \cdot (I_h \hat{X}_t - \hat{Y}_{ht})| d\theta						
	\\
	& \quad + C (1 - \alpha) \Int |\hat{X}_t \cdot (\nu_h - \nu)| 
				|\nu_h \cdot (I_h \hat{X}_t - \hat{Y}_{ht})| d\theta	
	\\
	& \quad			
			+ C (1 - \alpha) \Int |\hat{X}_t \cdot \nu| 
				|(\nu_h - \nu) \cdot (I_h \hat{X}_t - \hat{Y}_{ht})| d\theta	
	\\
	&\leq C \alpha  \|\hat{X}_t - I_h \hat{X}_t \|_{L^2} \|\hat{X}_t - \hat{Y}_{ht}\|_{L^2}
			+ C (1 - \alpha) \|\nu_h \cdot (\hat{X}_t - I_h \hat{X}_t)\|_{L^2} 
							\|\nu_h \cdot (\hat{X}_t - \hat{Y}_{ht})\|_{L^2}	
	\\
	& \quad + Ch \|\hat{X}_t \|_{H^{2,2}} 
				\| \hat{X}_\theta - \hat{Y}_{h\theta} \|_{L^2}		
		+ C \alpha \| \hat{X}_t \|_{L^\infty}	 
				\Int |\hat{X}_{h\theta} - \hat{X}_\theta| 
					|I_h \hat{X}_t - \hat{Y}_{ht}| d\theta
	\\
	& \quad + C (1 - \alpha) \| \hat{X}_t \|_{L^\infty}	
					\Int |\hat{X}_{h\theta} - \hat{X}_\theta| 
					|\nu_h \cdot (I_h \hat{X}_t - \hat{Y}_{ht})| d\theta						
	\\
	& \quad + C (1 - \alpha) \Int |\hat{X}_t \cdot (\nu_h - \nu)| 
				|\nu_h \cdot (I_h \hat{X}_t - \hat{Y}_{ht})| d\theta	
	\\
	& \quad			
			+ C (1 - \alpha) \Int |\hat{X}_t \cdot \nu| 
				|(\nu_h - \nu) \cdot (I_h \hat{X}_t - \hat{Y}_{ht})| d\theta.		
\end{align*}
Applying the interpolation estimate
$\| \hat{X}_t - I_h \hat{X}_t \|_{L^2} \leq Ch \| \hat{X}_t \|_{H^{1,2}} \leq C h$ then gives
\begin{align*}
&A :=\frac{c_0^2}{4} \alpha \|\hat{X}_t - \hat{Y}_{ht}\|^2_{L^2} 
		+ \frac{c_0^2}{4} (1 - \alpha) \|\nu_h \cdot (\hat{X}_t - \hat{Y}_{ht})\|^2_{L^2}
		+ \frac{1}{2} \frac{d}{dt} \|\hat{X}_\theta - \hat{Y}_{h\theta} \|^2_{L^2}
	\\
	&\leq Ch \alpha \| \hat{X}_t \|_{H^{1,2}}\|\hat{X}_t - \hat{Y}_{ht}\|_{L^2}
			+ Ch (1 - \alpha) \| \hat{X}_t \|_{H^{1,2}}
				 \|\nu_h \cdot (\hat{X}_t - \hat{Y}_{ht})\|_{L^2}	
	\\
	& \quad + Ch \|\hat{X}_t \|_{H^{2,2}} 
		 \| \hat{X}_\theta - \hat{Y}_{h\theta} \|_{L^2}		
		+ C \alpha \|\hat{X}_{h\theta} - \hat{X}_\theta \|_{L^2} 
				   \| I_h \hat{X}_t - \hat{Y}_{ht} \|_{L^2}
	\\
	& \quad + C (1 - \alpha) \|\hat{X}_{h\theta} - \hat{X}_\theta \|_{L^2} 
			\| \nu_h \cdot (I_h \hat{X}_t - \hat{Y}_{ht}) \|_{L^2}						
	\\
	& \quad + C (1 - \alpha) \| \hat{X}_t \|_{L^\infty} \| \nu_h - \nu \|_{L^2} 
			\| \nu_h \cdot (I_h \hat{X}_t - \hat{Y}_{ht}) \|_{L^2}	
	\\
	& \quad			
			+ C (1 - \alpha) \| \hat{X}_t \|_{L^\infty}
			\|(\nu_h - \nu) \cdot (I_h \hat{X}_t - \hat{Y}_{ht})\|_{L^1}
	\\		
	& \leq C h \alpha \|\hat{X}_t - \hat{Y}_{ht}\|_{L^2}
			+ C h (1 - \alpha) \|\nu_h \cdot (\hat{X}_t - \hat{Y}_{ht})\|_{L^2}	
	\\
	& \quad + Ch \|\hat{X}_t \|_{H^{2,2}} 
				\| \hat{X}_\theta - \hat{Y}_{h\theta} \|_{L^2}		
	    + C h \alpha \|\hat{X}_t \|_{H^{1,2}} 
	    			\|\hat{X}_{h\theta} - \hat{X}_\theta \|_{L^2} 
		+ C \alpha \|\hat{X}_{h\theta} - \hat{X}_\theta \|_{L^2} 
				   \|\hat{X}_t - \hat{Y}_{ht} \|_{L^2}
	\\
	& \quad + C h (1 - \alpha) \|\hat{X}_t \|_{H^{1,2}}
			\|\hat{X}_{h\theta} - \hat{X}_\theta \|_{L^2} 
		+ C (1 - \alpha) \|\hat{X}_{h\theta} - \hat{X}_\theta \|_{L^2} 
			\| \nu_h \cdot (\hat{X}_t - \hat{Y}_{ht}) \|_{L^2}						
	\\
	& \quad + C h (1 - \alpha) \|\hat{X}_t \|_{H^{1,2}} \| \nu_h - \nu \|_{L^2} 
		+ C (1 - \alpha) \| \nu_h - \nu \|_{L^2} 
			\| \nu_h \cdot (\hat{X}_t - \hat{Y}_{ht}) \|_{L^2}	
	\\
	& \quad	 + Ch (1 - \alpha) \|\hat{X}_t \|_{H^{1,2}}
			\|\nu_h - \nu \|_{L^2}					
			+ C (1 - \alpha)
			\|(\nu_h - \nu) \cdot (\hat{X}_t - \hat{Y}_{ht})\|_{L^1}.
\end{align*}
Hence,
\begin{align*}
	&A \leq C h \alpha \|\hat{X}_t - \hat{Y}_{ht}\|_{L^2}
			+ C h (1 - \alpha) \|\nu_h \cdot (\hat{X}_t - \hat{Y}_{ht})\|_{L^2}	
	\\
	& \quad + Ch \|\hat{X}_t \|_{H^{2,2}} 
					\| \hat{X}_\theta - \hat{Y}_{h\theta} \|_{L^2}		
	    + C h \alpha \|\hat{X}_{h\theta} - \hat{X}_\theta \|_{L^2} 
		+ C \alpha \|\hat{X}_{h\theta} - \hat{X}_\theta \|_{L^2} 
				   \|\hat{X}_t - \hat{Y}_{ht} \|_{L^2}
	\\
	& \quad + C h (1 - \alpha) \|\hat{X}_{h\theta} - \hat{X}_\theta \|_{L^2} 
		+ C (1 - \alpha) \|\hat{X}_{h\theta} - \hat{X}_\theta \|_{L^2} 
			\| \nu_h \cdot (\hat{X}_t - \hat{Y}_{ht}) \|_{L^2}						
	\\
	& \quad + C h (1 - \alpha) \| \nu_h - \nu \|_{L^2} 
		+ C (1 - \alpha) \| \nu_h - \nu \|_{L^2} 
			\| \nu_h \cdot (\hat{X}_t - \hat{Y}_{ht}) \|_{L^2}	
	\\
	& \quad	 + C (1 - \alpha)
	\|(\nu_h - \nu) \cdot \tau_h \|_{L^2}  \| \tau_h \cdot (\hat{X}_t - \hat{Y}_{ht})\|_{L^2},			
\end{align*}
where $\tau_h$ denotes the vector field given by $\tau_h := \hat{X}_{h \theta}/ |\hat{X}_{h\theta}|$.
Since $\hat{X}_h \in \mathcal{B}_h$, the estimate  
$\| \hat{X}_{\theta} - \hat{X}_{h \theta} \|_{L^2} \leq K e^{\frac{M}{2\alpha}t} h$
holds for all $0 \leq t \leq T$. It is then not difficult to estimate
$\| \nu_h - \nu \|_{L^2} \leq C K e^{\frac{M}{2\alpha}t} h$.
Using Young's inequality we deduce
\begin{align*}
A &\leq C h \alpha \|\hat{X}_t - \hat{Y}_{ht}\|_{L^2}
			+ C h (1 - \alpha) \|\nu_h \cdot (\hat{X}_t - \hat{Y}_{ht})\|_{L^2}	
	\\
	& \quad + Ch \|\hat{X}_t \|_{H^{2,2}}  
				\| \hat{X}_\theta - \hat{Y}_{h\theta} \|_{L^2}		
	    + C h^2 K e^{\frac{M}{2\alpha}t} 
		+ C h \alpha K e^{\frac{M}{2\alpha}t}
				   \|\hat{X}_t - \hat{Y}_{ht} \|_{L^2}
	\\
	& \quad + C h (1 - \alpha) K e^{\frac{M}{2\alpha}t}
			\| \nu_h \cdot (\hat{X}_t - \hat{Y}_{ht}) \|_{L^2}						
	\\
	& \quad	 + C (1 - \alpha)
	\|(\nu_h - \nu) \cdot \tau_h \|_{L^2}  \| \tau_h \cdot (\hat{X}_t - \hat{Y}_{ht})\|_{L^2}		\\
	&\leq \frac{C h^2 \alpha}{2 \delta_1} 
	+ \frac{\delta_1}{2} \alpha  \|\hat{X}_t - \hat{Y}_{ht}\|^2_{L^2}
			+ \frac{C h^2 (1 - \alpha)}{2 \delta_2}
			+ \frac{\delta_2 }{2} (1 - \alpha)
					\|\nu_h \cdot (\hat{X}_t - \hat{Y}_{ht})\|^2_{L^2}	
	\\
	& \quad + \frac{Ch^2}{2 \delta_3} \|\hat{X}_t \|^2_{H^{2,2}} 
		+ \frac{\delta_3}{2} \| \hat{X}_\theta - \hat{Y}_{h\theta} \|^2_{L^2}		
	    + C h^2 K e^{\frac{M}{2\alpha}t} 
		+ \frac{C h^2 \alpha K^2 e^{\frac{M}{\alpha} t }}{2 \delta_4} 
		+ \frac{\delta_4}{2} \alpha \|\hat{X}_t - \hat{Y}_{ht} \|^2_{L^2}
	\\
	& \quad + \frac{C h^2 (1 - \alpha) K^2 e^{\frac{M}{\alpha} t}}{ 2 \delta_5}
	+ \frac{\delta_5}{2} (1- \alpha) \| \nu_h \cdot (\hat{X}_t - \hat{Y}_{ht}) \|^2_{L^2}		
	\\
	& \quad	 + C (1 - \alpha)
	\|(\nu_h - \nu) \cdot \tau_h \|_{L^2}  \| \tau_h \cdot (\hat{X}_t - \hat{Y}_{ht})\|_{L^2}.	
\end{align*}
By choosing the constants $\delta_i > 0$, for $i=1, \ldots, 5$, appropriately,
it follows that
\begin{align*}
&\frac{c_0^2}{8} \alpha \|\hat{X}_t - \hat{Y}_{ht}\|^2_{L^2} 
		+ \frac{c_0^2}{8} (1 - \alpha) \|\nu_h \cdot (\hat{X}_t - \hat{Y}_{ht})\|^2_{L^2}
		+ \frac{1}{2} \frac{d}{dt} \|\hat{X}_\theta - \hat{Y}_{h\theta} \|^2_{L^2}
	\\	
	&\leq Ch^2 ( 1 + \|\hat{X}_t \|^2_{H^{2,2}} ) 
		+ \frac{1}{2} \| \hat{X}_\theta - \hat{Y}_{h\theta} \|^2_{L^2}		
	  + C h^2 K^2 e^{\frac{M}{\alpha} t} 
	\\  
	  & \quad + C (1 - \alpha)
	\|(\nu_h - \nu) \cdot \tau_h \|_{L^2}  \| \tau_h \cdot (\hat{X}_t - \hat{Y}_{ht})\|_{L^2} .
\end{align*}
We now estimate the last term on the right hand side.
This term is the reason for the dependence of the approximation error on $\alpha$.
Using the estimate $\| \nu_h - \nu \|_{L^2} \leq C K e^{\frac{M}{2\alpha}t} h$
and Young's inequality, we obtain
\begin{align*}
&\frac{c_0^2}{4} \alpha \|\hat{X}_t - \hat{Y}_{ht}\|^2_{L^2} 
		+ \frac{c_0^2}{4} (1 - \alpha) \|\nu_h \cdot (\hat{X}_t - \hat{Y}_{ht})\|^2_{L^2}
		+ \frac{d}{dt} \|\hat{X}_\theta - \hat{Y}_{h\theta} \|^2_{L^2}
	\\	
	&\leq Ch^2 ( 1 + \|\hat{X}_t \|^2_{H^{2,2}} ) 
		 + \| \hat{X}_\theta - \hat{Y}_{h\theta} \|^2_{L^2}		
	  + C h^2 K^2 e^{\frac{M}{\alpha} t} 
	  + C (1 - \alpha) K e^{\frac{M}{2\alpha}t} h \| \hat{X}_t - \hat{Y}_{ht} \|_{L^2}
	\\
	&\leq Ch^2 ( 1 + \|\hat{X}_t \|^2_{H^{2,2}} ) 
		 + \| \hat{X}_\theta - \hat{Y}_{h\theta} \|^2_{L^2}		
	  + C h^2 K^2 e^{\frac{M}{\alpha} t} 
	  + \frac{C h^2 K^2 e^{\frac{M}{\alpha} t}}{\alpha} 
	  + \frac{c_0^2}{8} \alpha \| \hat{X}_t - \hat{Y}_{ht} \|^2_{L^2} . 
\end{align*}
Hence, we have 	
\begin{align}
&\frac{c_0^2}{8} \alpha \|\hat{X}_t - \hat{Y}_{ht}\|^2_{L^2} 
		+ \frac{c_0^2}{4} (1 - \alpha) \|\nu_h \cdot (\hat{X}_t - \hat{Y}_{ht})\|^2_{L^2}
		+ \frac{d}{dt} \|\hat{X}_\theta - \hat{Y}_{h\theta} \|^2_{L^2}
	\nonumber	
	\\
	&\leq Ch^2 ( 1 + \|\hat{X}_t \|^2_{H^{2,2}} ) 
		 + \| \hat{X}_\theta - \hat{Y}_{h\theta} \|^2_{L^2}		
	  + \frac{C h^2 K^2 e^{\frac{M}{\alpha} t}}{\alpha} .
	\label{error_estimate_intermediate_result}  
\end{align}
Integrating with respect to time and using the fact that 
$\hat{X}_t \in L^2((0,T), H^{2,2}(\mathbb{R}/ 2\pi))$
leads to
\begin{align*}
	\| (\hat{X}_\theta - \hat{Y}_\theta)(t)\|^2_{L^2}
	&\leq \| (\hat{X}_\theta - \hat{Y}_\theta)(0) \|^2_{L^2} + Ch^2 
		+ \frac{C h^2 K^2 e^{\frac{M}{\alpha} t}}{M}
	  	+ \int_{0}^{t} \| (\hat{X}_\theta - \hat{Y}_{h\theta})(s) \|^2 ds 
	\\
	&\leq  Ch^2 
		+ \frac{C h^2 K^2 e^{\frac{M}{\alpha} t}}{M}
	  	+ \int_{0}^{t} \| (\hat{X}_\theta - \hat{Y}_{h\theta})(s) \|^2 ds, 
\end{align*}
where we have used $X(\cdot,0) = X_0(\cdot)$ 
and $\hat{Y}_h(\cdot, 0) = (I_h X_0) (\cdot)$
as well as $\| (X_0 - (I_h X_0))_\theta \|_{L^2} \leq C h$.
We infer from Gronwall's lemma that  
\begin{align*}
	\| (\hat{X}_\theta - \hat{Y}_\theta)(t)\|^2_{L^2}
	&\leq Ch^2 e^t
		+ \frac{C h^2 K^2 e^{(\frac{M}{\alpha} + 1) t}}{M},
\end{align*}
and hence,
\begin{align*}
	\sup_{t \in [0,T]} e^{- \frac{M}{\alpha} t} 
			\| (\hat{X}_\theta - \hat{Y}_{h\theta})(t) \|_{L^2}^2
	&\leq C h^2 + \frac{C h^2 K^2}{M},
\end{align*}
where $C$ depends on $\hat{X}$ and $T$. 
Choosing 
$K^2 \geq 2C$ and $M \geq 2C$
we can finally conclude that
\begin{align*}
	\sup_{t \in [0,T]} e^{- \frac{M}{\alpha} t} 
			\| (\hat{X}_\theta - \hat{Y}_{h\theta})(t) \|_{L^2}^2
	&\leq \frac{K^2 h^2}{2} + \frac{C}{M} K^2 h^2 \leq K^2 h^2
\end{align*}
and thus, $\hat{Y}_h \in \mathcal{B}_h$.
The maximal grid size $h_0$ then only depends on $\hat{X},T$ and $\alpha$.
Inserting the estimate $\| (\hat{X}_\theta - \hat{Y}_{h\theta})(t) \|_{L^2}^2
\leq K^2 h^2 e^{\frac{M}{\alpha} t}$ into (\ref{error_estimate_intermediate_result})
gives
\begin{align*}
&\frac{c_0^2}{8} \alpha \|\hat{X}_t - \hat{Y}_{ht}\|^2_{L^2} 
		+ \frac{c_0^2}{4} (1 - \alpha) \|\nu_h \cdot (\hat{X}_t - \hat{Y}_{ht})\|^2_{L^2}
		+ \frac{d}{dt} \|\hat{X}_\theta - \hat{Y}_{h\theta} \|^2_{L^2}	
	\\
	&\leq Ch^2 ( 1 + \|\hat{X}_t \|^2_{H^{2,2}} ) 		
	  + \frac{C h^2 K^2 e^{\frac{M}{\alpha} t}}{\alpha} .
\end{align*}
The same procedure as above then shows that
\begin{align*}
&\frac{c_0^2}{8} \alpha \int_0^t \|(\hat{X}_t - \hat{Y}_{ht})(t)\|^2_{L^2}  dt
		+ \frac{c_0^2}{4} (1 - \alpha)
		\int_0^t \|\nu_h \cdot (\hat{X}_t - \hat{Y}_{ht})(t)\|^2_{L^2} dt
		+  \| (\hat{X}_\theta - \hat{Y}_{h\theta})(t) \|^2_{L^2}	
	\\
	&\leq Ch^2	
	  + \frac{C h^2 K^2 e^{\frac{M}{\alpha} t}}{M} .
\end{align*}
and hence,
\begin{align}
& \alpha \int_0^T \|(\hat{X}_t - \hat{Y}_{ht})(t)\|^2_{L^2}  dt
		+  (1 - \alpha)
		\int_0^T \|\nu_h \cdot (\hat{X}_t - \hat{Y}_{ht})(t)\|^2_{L^2} dt
		+ \max_{t \in [0,T]} \| (\hat{X}_\theta - \hat{Y}_{h\theta})(t) \|^2_{L^2}	
	\nonumber \\
	&\leq C e^{\frac{M}{\alpha} T} h^2,
	\label{final_error_estimate}
\end{align}
where $C$ and $M$ depends on $\hat{X}$ and $T$.
In particular, we have 
$\int_0^T \| \hat{Y}_{ht} \|^2_{L^2} dt \leq C(\alpha, h, T, \hat{X})$.
Together with $\hat{Y}_h(\cdot ,0) = (I_h X_0)(\cdot)$ on $[0,2\pi]$,
this implies that
$\| \hat{Y}_h \|_{H^{1,2}( (0,T), \mathcal{S}_h^2 )} \leq C(\alpha, h, T, \hat{X})$.

This completes the proof of the fact that $F(\mathcal{B}_h) \subset \mathcal{B}_h$
is precompact. We can now apply the Schauder fixed point theorem.
The error estimate for the solution to the semi-discrete flow 
(\ref{semi_discrete_problem_CSF}), 
that is for the fixed point $\hat{X}_h = F(\hat{X}_h)$, finally infers from
(\ref{final_error_estimate}) and the fact that
\begin{align*}
	\max_{t \in [0,T]} \| (\hat{X} - \hat{X}_h)(t) \|^2_{L^2}
	&\leq C \| (\hat{X} - \hat{X}_h)(0) \|^2_{L^2}
		+ C T \int_{0}^{T} \| (\hat{X}_t - \hat{X}_{ht})(t) \|_{L^2}^2 dt
	\\
	&\leq C \| (\hat{X} - I_h \hat{X})(0) \|^2_{L^2}
		+ C T \int_{0}^{T} \| (\hat{X}_t - \hat{X}_{ht})(t) \|_{L^2}^2 dt	
	\\
	&\leq C h^4 + C T \int_{0}^{T} \| (\hat{X}_t - \hat{X}_{ht})(t) \|_{L^2}^2 dt	.
\end{align*}
\hfill $\Box$

\begin{remark}
In the above theorem we have excluded the case $\alpha > 1$, since we are mainly interested
in the behaviour of the reparametrized flow for small $\alpha$ anyway.
However, this restriction is clearly only a formal one.
To be more precise, the proof for $\alpha > 1$ works by writing the map $\alpha \unit + (1 - \alpha) \nu_h \otimes \nu_h$ as
$\alpha \tau_h \otimes \tau_h + \nu_h \otimes \nu_h$, where $\tau_h = \hat{X}_{h\theta}/|\hat{X}_{h\theta}|$ is a unit tangential vector field.
\end{remark}
\begin{remark}
Unfortunately, for the case $\alpha \searrow 0$ the error estimate in Theorem \ref{approximation_theorem}
becomes unbounded. It is therefore not clear whether for the choice $\alpha =0$, the approximation error still converges to zero for $h \searrow 0$.  
This is an open question.
\end{remark}
\begin{remark}
If we choose $\varphi_h = \hat{X}_{ht}$ in (\ref{semi_discrete_problem_CSF}) and integrate in time,
we obtain the following stability estimate for the semi-discrete scheme
$$
	\frac{1}{2}\Int | \hat{X}_{h \theta}(\cdot, T) |^2 d\theta
	+ \int_0^T \Int \left( \alpha |\hat{X}_{ht}|^2 + (1 - \alpha) |\hat{X}_{ht} \cdot \nu_h|^2 \right) | \hat{X}_{h \theta} |^2 d\theta
	= \frac{1}{2} \Int |\hat{X}_{h \theta}(\cdot, 0)|^2 d\theta.
$$
\end{remark}

\subsection{Numerical scheme for the curve shortening flow}
Time discretization of the semi-discrete scheme (\ref{semi_discrete_problem_CSF})
leads to a family of algorithms for the computation of the curve shortening. 
Using the notation $f^m = f(\cdot,m\tau)$ for the discrete time levels
$\{ m\tau ~|~ m = 0 , \ldots, M_\tau \in \mathbb{N} \}$ with time step size
$\tau > 0$ and $M_\tau \tau < T$, we propose the following semi-implicit schemes.
\begin{alg}
\label{algo_CSF}
Let $\alpha \in (0,\infty)$. For a given initial polygonal curve 
$\Gamma_h^0 = \hat{X}_h^0([0,2\pi])$ with $\hat{X}^0_h \in \mathcal{S}_h^2$,
determine for $m=0, \ldots, M_\tau -1$ the periodic solution 
$\hat{X}_h^{m+1} \in \mathcal{S}_h^2$ of 
\begin{align*}
	&\Int \left( \frac{\alpha}{\tau} \hat{X}^{m+1}_{h} \cdot \varphi_h 
		+ \frac{1 - \alpha}{\tau} 
			(\hat{X}^{m+1}_{h} \cdot \nu^m_h)(\nu^m_h \cdot \varphi_h) \right) 
				|\hat{X}^m_{h\theta}|^2 d\theta 
	+ \Int \hat{X}^{m+1}_{h\theta} \cdot \varphi_{h\theta} d\theta 
\\	
	&= \Int \left( \frac{\alpha}{\tau} \hat{X}^m_{h} \cdot \varphi_h 
		+ \frac{1 - \alpha}{\tau} 
			(\hat{X}^m_{h} \cdot \nu^m_h)(\nu^m_h \cdot \varphi_h) \right) 
				|\hat{X}^m_{h\theta}|^2 d\theta ,
	\quad \forall \varphi \in \mathcal{S}_h^2,			
\end{align*}
where $\nu_h^m$ is a piecewise constant unit normal field to the polygonal curve $\Gamma_h^m$,
and set
$$
	\Gamma_h^{m+1} := \hat{X}_h^{m+1}([0,2\pi]).
$$
\end{alg}
\begin{remark}
If we choose $\varphi_h = \hat{X}^{m+1}_h - \hat{X}^m_h$, we obtain after a short calculation
\begin{align*}
&\frac{1}{2} \Int |\hat{X}^{M}_{h\theta}|^2 d\theta
+ \sum_{m=0}^{M-1} \Int \left( 
\frac{\alpha}{\tau} |\hat{X}^{m+1}_h - \hat{X}^m_h|^2
+ \frac{1 - \alpha}{\tau} |(\hat{X}^{m+1}_h - \hat{X}^m_h) \cdot \nu_h^m|^2
\right) |\hat{X}^m_{h \theta}|^2 d\theta
\\
&\leq 
\frac{1}{2} \Int |\hat{X}^0_{h\theta}|^2 d\theta
\end{align*}
This holds for all mesh sizes $h > 0$ and time steps $\tau>0$.
\end{remark}

It is formally possible to choose
$\alpha = 0$ in Algorithm \ref{algo_CSF},
although this case was excluded in the derivation of the reparametrized
curve shortening flow. The algorithm (2.16a)
in \cite{BGN11}, which we have cited in the introduction, is in this spirit. 

\section{Numerical results for the curve shortening flow}
\label{Numerical_results_CSF}

In order to implement Algorithm \ref{algo_CSF} we solve the following linear system
of equations within the Finite Element Toolbox ALBERTA, see \cite{SS05},
\begin{equation}
  \frac{1}{\tau}\sum_{\gamma = 1}^2 \sum_{j=1}^N M_{ij \beta\gamma} \mathbf{\hat{X}}^{j\gamma}
  + \sum_{\gamma = 1}^2 \sum_{j=1}^N S_{ij \beta\gamma} \mathbf{\hat{X}}^{j\gamma}
  = \frac{1}{\tau} \sum_{\gamma = 1}^2 \sum_{j=1}^N M_{ij \beta\gamma} \mathbf{\hat{X}}^{j\gamma}_{old},
  \quad \forall i = 1, \ldots, N, \beta = 1, 2,
  \label{algebraic_equation_CSF}
\end{equation}
where $\hat{X}^{m+1}_h = \sum_{\gamma =1}^2 \sum_{j=1}^N \mathbf{\hat{X}}^{j \gamma} \phi_j e_\gamma$
is the unknown parametrization of the polygonal curve
and $\hat{X}^{m}_h = \sum_{\gamma =1}^2 \sum_{j=1}^N \mathbf{\hat{X}}^{j \gamma}_{old} \phi_j e_\gamma$
is the solution from the previous time step. The mass matrix 
$M := (M_{ij \beta \gamma}) \in \mathbb{R}^{(2N) \times (2N)}$ and the stiffness matrix 
$S := (S_{ij \beta \gamma}) \in \mathbb{R}^{(2N) \times (2N)}$ are assembled by summing up
all simplex matrices $(M_{ij \beta \gamma}(T))$
and $(S_{ij \beta \gamma}(T))$  given by
\begin{align*}
& M_{ij \beta \gamma}(T) = 
	(\alpha \delta_{\beta \gamma} |\rho_h^m(T)|^2 
		+ (1 - \alpha) \rho_{h\beta}^m(T) \rho^m_{h\gamma}(T)) 
	\int_{T} \phi_i(\theta) \phi_j(\theta) d\theta,
	\\
& S_{ij \beta \gamma}(T) 
	= \delta_{\beta \gamma} \int_{T} \phi_{i\theta}(\theta) \phi_{j\theta}(\theta) d\theta.
\end{align*}
Here, $\phi_{i}$ and $\phi_j$ denote the local basis functions of the simplex $T$
and $\rho_h^m(T)$ is the constant vector field $\hat{X}_{h\theta|T}^m$ rotated by $90$ degrees.
The linear system (\ref{algebraic_equation_CSF}) can be solved by the conjugate gradient method.
The initial polygonal curve of the simulation
is constructed by mapping the vertices of a triangulation of the unit sphere
onto an initial smooth curve via a problem dependent map. 
We compare the performance of Algorithm \ref{algo_CSF}
to the benchmark scheme (2.16a) of \cite{BGN11}. This fully-implicit scheme is solved by the fixed point iteration defined in (3.3a) of \cite{BGN11}
with the suggested stopping criteria $\| \hat{X}^{m+1,i+1}_h - \hat{X}^{m+1,i}_h \|_\infty < 10^{-8}$.
In order to solve the linear system (3.3a) in \cite{BGN11}, we apply the conjugate gradient method.

\subsection*{Example 1:}
The first example is presented in Figure \ref{Figure_curve_CSF_example_1}. 
The initial curve is given by the parametrization
\begin{equation*}
	X_0(\theta) := \left(
	\begin{array}{c}
	 \cos \theta \\
	 (0.9 \cos^2 \theta + 0.1)\sin \theta
	\end{array}
	\right), \quad \theta \in [0, 2\pi).
\end{equation*}
Under the curve shortening flow, this curve shrinks to a round circle, which is clearly visible
in Figure \ref{CSF_round_curve_example_1}. 
We here want to demonstrate that our scheme is not only able to maintain the mesh quality but also to improve it significantly.
We therefore start with non-equidistributed meshes. Whether it is appropriate for the BGN-scheme (2.16a) of \cite{BGN11} to start
with a non-equidistributed mesh is discussed in Example 3.
Figure \ref{Figure_length_decrease_CSF_example_1}
shows the decrease of the length of the evolving curve. 
The BGN-scheme (2.16a) of \cite{BGN11} seems to lead to a slightly stronger 
drop of the curve length 
in the first time step. This is probably due to the fact that 
the BGN-scheme changes the mesh in the first time step
such that the segments of the polygonal curve are equally long. 
The ratio of the maximal to the minimal segment length
is therefore equal to $1$ after the first time step.
However, also Algorithm \ref{algo_CSF} leads to length ratios that are close to $1$,
provided that the parameter $\alpha$ is chosen sufficiently small, see Figure \ref{Figure_length_ratio_CSF_example_1}. 
The main difference between the BGN-scheme and the $\alpha$-scheme is
that the redistributions of the mesh vertices 
do not occur instantaneously under Algorithm \ref{algo_CSF}.
Moreover, the mesh ratio does not exactly stick to $1$, which gives Algorithm \ref{algo_CSF}
a bit more flexibility. In Example 2, it will turn out that this can be 
advantageous in certain circumstances. Since the BGN-scheme is a fully-implicit scheme,
we have to solve a non-linear system of algebraic equations in each time step.
We solve this system by the fixed point iteration proposed in (3.3a) of \cite{BGN11}. Apart from the first time step
and the time step at the end of the simulation, 
where the round circle in Figure \ref{CSF_round_curve_example_1}
actually drops to a point, the fixed point iteration converges rather fast, see Figure \ref{number_of_iteration_steps_BGN_example_1}.
The initial redistribution of the vertices in both schemes is associated with a large initial
(tangential) velocity. Since for Algorithm \ref{algo_CSF}
the redistribution of the mesh vertices occurs on a
time scale determined by the parameter $\alpha$,
smaller values of $\alpha$ lead to a larger maximal initial velocity, 
see Figure \ref{Figure_maximal_velocity_CSF_example_1}. 
However, in contrast to the BGN-scheme, the maximal initial velocity of Algorithm
\ref{algo_CSF} is bounded for different choices of the time step size $\tau$,
see Figure \ref{Figure_maximal_velocity_BGN_scheme}. The fact that the
maximal initial velocity in the BGN-scheme depends linearly 
on the inverse time step size $\tau^{-1}$ is associated with the relatively large jumps of the mesh vertices 
in the first time step, see Figure \ref{Figure_motion_of_the_curve_vertices_CSF}.
Since these jumps must lead to an equidistributed mesh, their size cannot become small even for small
time step sizes $\tau$. This issue will be further discussed in Example 3.
Interestingly, the $\alpha$-scheme seems to interpolate the initial jump of the vertices 
in the BGN-scheme, see Figure \ref{Figure_motion_of_the_curve_vertices_CSF}.

\begin{figure}
\begin{center}
\subfloat[][Time $t=0.0$. The curve length is $5.44$.]
{\includegraphics[width=0.4\textwidth]{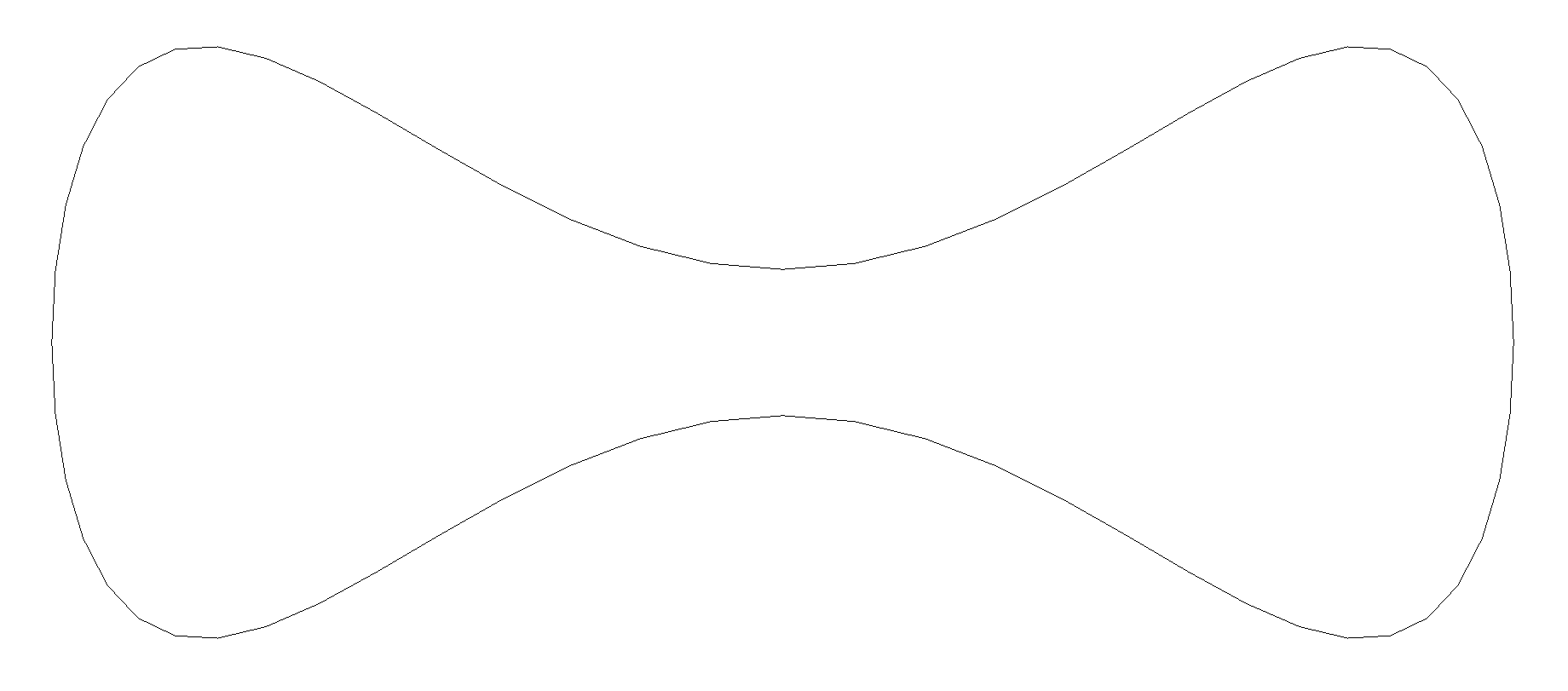}
\label{CSF_initial_curve_example_1}}
\subfloat[][Time $t=0.05$. The curve length is $4.05$.]
{\includegraphics[width=0.4\textwidth]{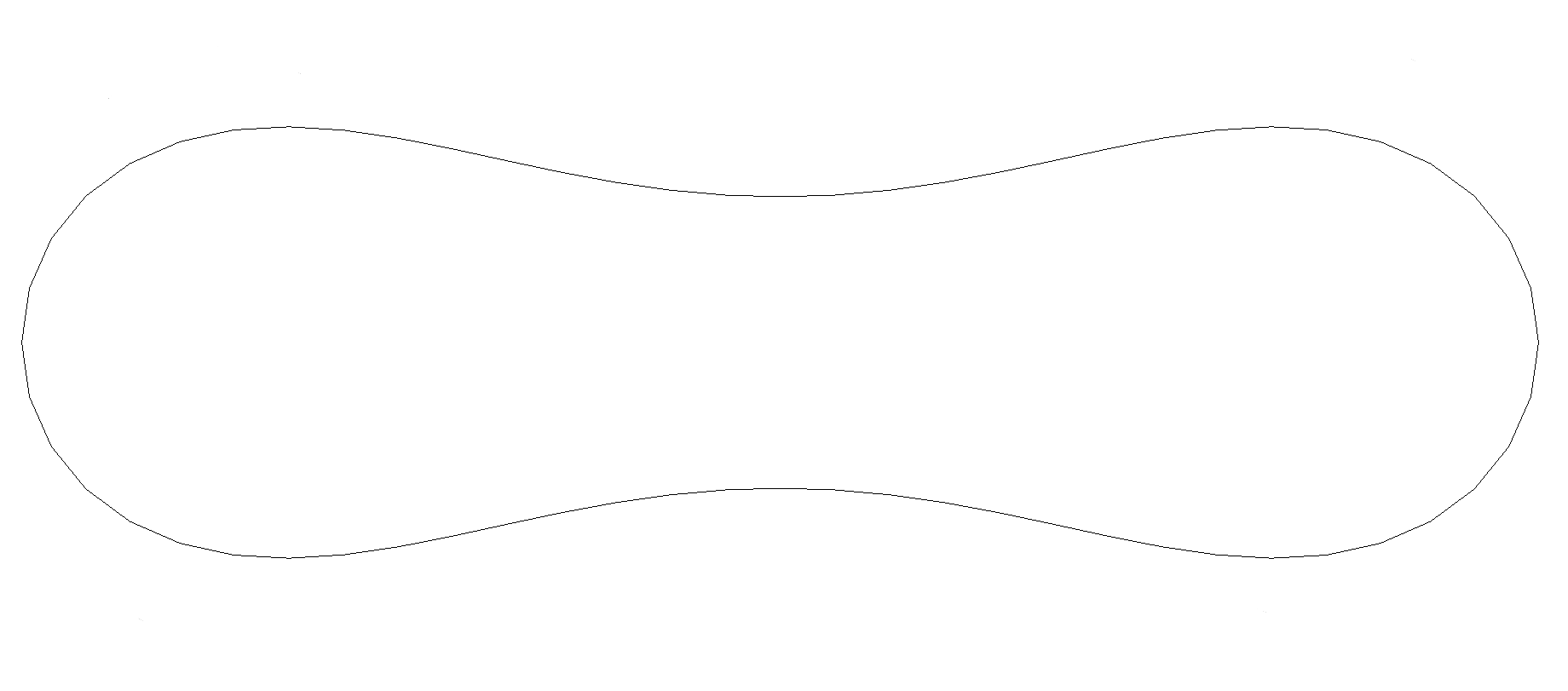}} \\
\subfloat[][Time $t=0.1$. The curve length is $2.64$.]
{\includegraphics[width=0.4\textwidth]{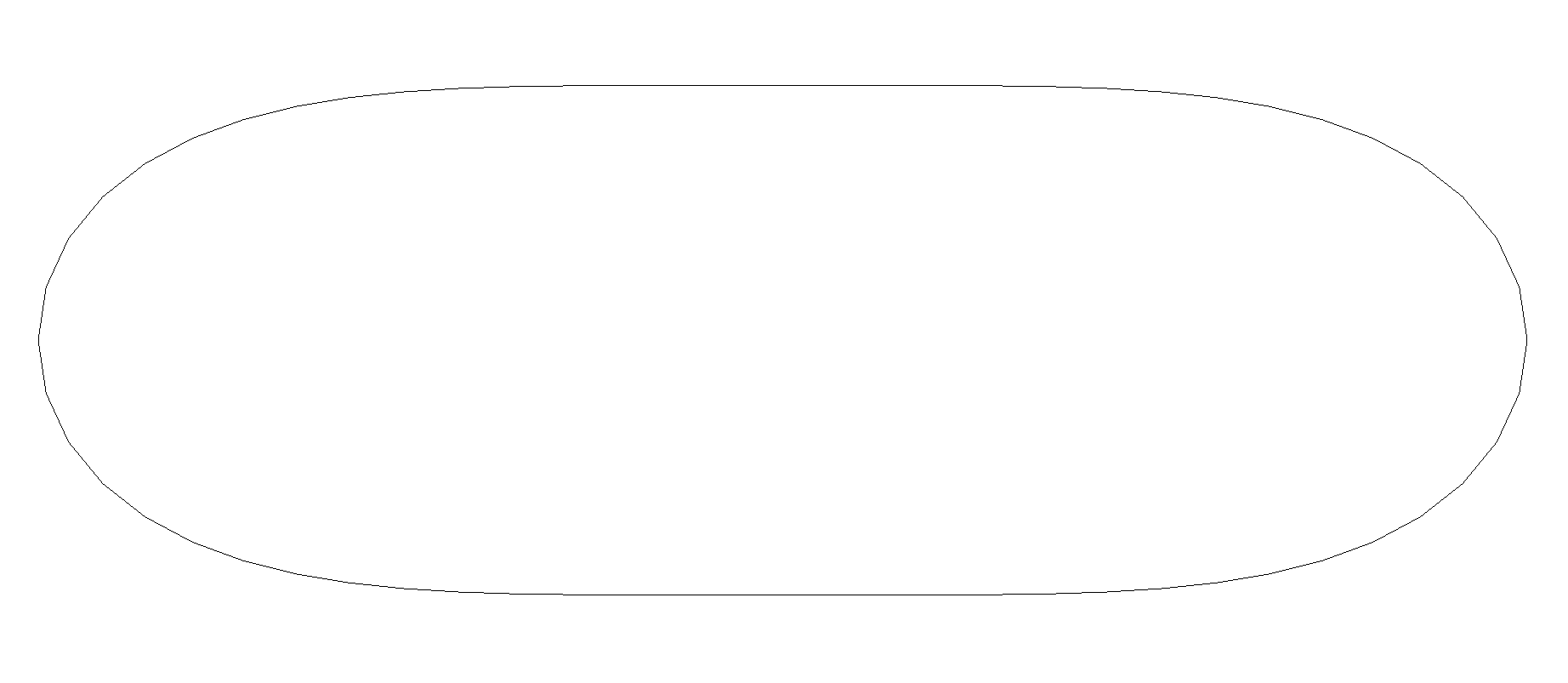}} 
\subfloat[][Time $t=0.15$. The curve length is $0.98$.]{
\includegraphics[width=0.4\textwidth]{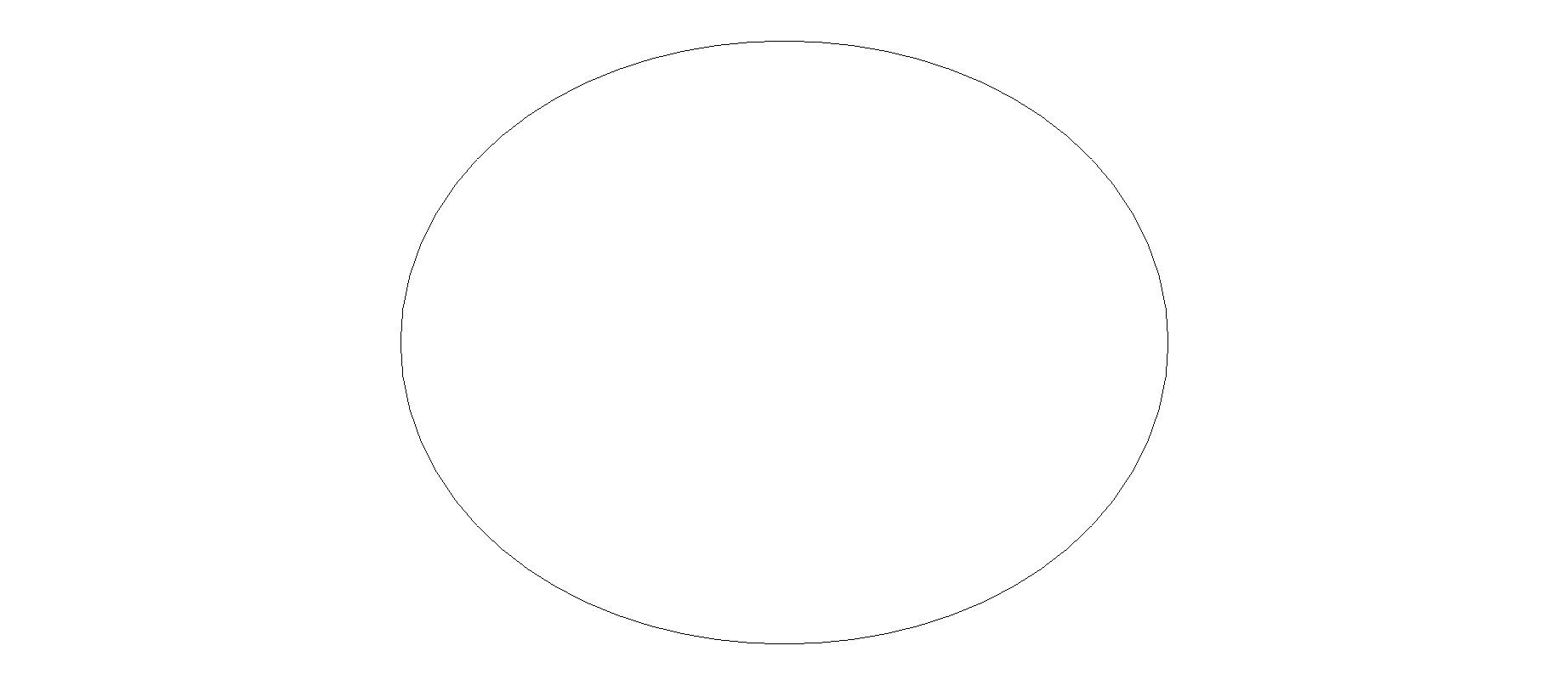}
\label{CSF_round_curve_example_1}
}
\caption{Simulation of the curve shortening flow with Algorithm \ref{algo_CSF} for $\alpha = 10^{-3}$
and time step size $\tau = 10^{-4}$. The computational mesh had $64$ vertices. 
The simulation clearly shows that the curve shrinks to a round 
circle whose length converges to zero. The images are rescaled.
See Example 1 of Section \ref{Numerical_results_CSF} for further details.}
\label{Figure_curve_CSF_example_1}
\end{center}
\end{figure}

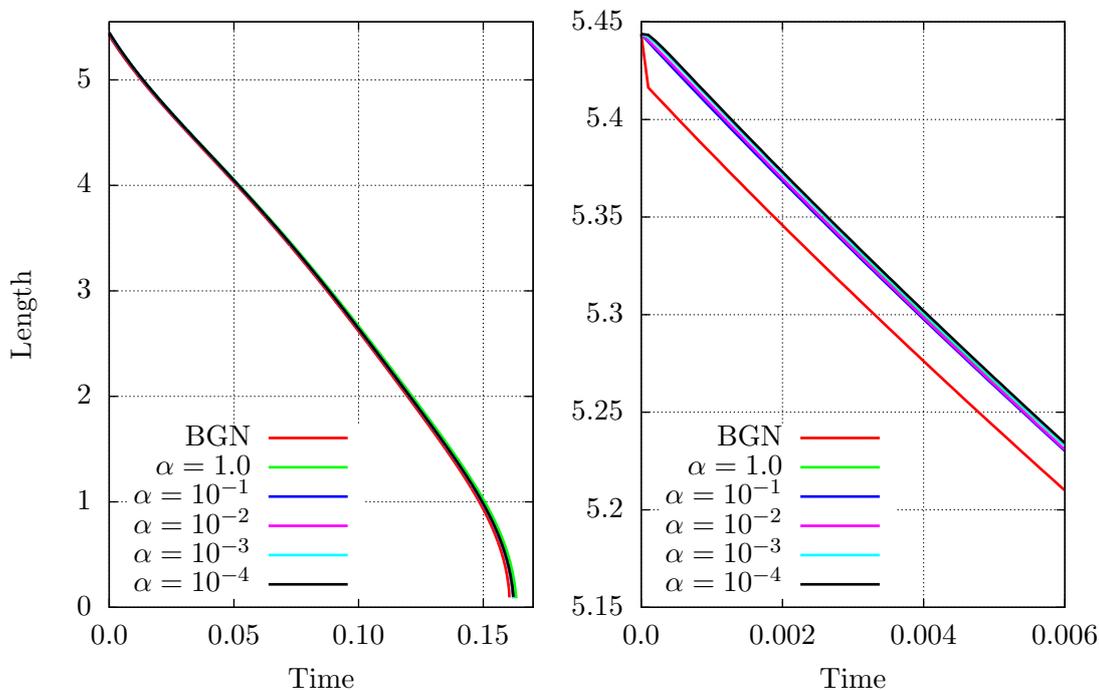
\begin{figure}
\centering
 \begin{picture}(7936.00,4534.00)%
    \gplgaddtomacro\gplbacktext{%
      \csname LTb\endcsname%
      \put(660,110){\makebox(0,0)[r]{\strut{} 0}}%
      \csname LTb\endcsname%
      \put(660,903){\makebox(0,0)[r]{\strut{} 1}}%
      \csname LTb\endcsname%
      \put(660,1696){\makebox(0,0)[r]{\strut{} 2}}%
      \csname LTb\endcsname%
      \put(660,2489){\makebox(0,0)[r]{\strut{} 3}}%
      \csname LTb\endcsname%
      \put(660,3282){\makebox(0,0)[r]{\strut{} 4}}%
      \csname LTb\endcsname%
      \put(660,4075){\makebox(0,0)[r]{\strut{} 5}}%
      \csname LTb\endcsname%
      \put(792,-110){\makebox(0,0){\strut{}0.0}}%
      \csname LTb\endcsname%
      \put(1722,-110){\makebox(0,0){\strut{}0.05}}%
      \csname LTb\endcsname%
      \put(2652,-110){\makebox(0,0){\strut{}0.10}}%
      \csname LTb\endcsname%
      \put(3582,-110){\makebox(0,0){\strut{}0.15}}%
      \put(154,2310){\rotatebox{-270}{\makebox(0,0){\strut{}Length}}}%
      \put(2373,-440){\makebox(0,0){\strut{}Time}}%
    }%
    \gplgaddtomacro\gplfronttext{%
      \csname LTb\endcsname%
      \put(1848,1383){\makebox(0,0)[r]{\strut{}BGN}}%
      \csname LTb\endcsname%
      \put(1848,1163){\makebox(0,0)[r]{\strut{}$\alpha =1.0$}}%
      \csname LTb\endcsname%
      \put(1848,943){\makebox(0,0)[r]{\strut{}$\alpha =10^{-1}$}}%
      \csname LTb\endcsname%
      \put(1848,723){\makebox(0,0)[r]{\strut{}$\alpha =10^{-2}$}}%
      \csname LTb\endcsname%
      \put(1848,503){\makebox(0,0)[r]{\strut{}$\alpha =10^{-3}$}}%
      \csname LTb\endcsname%
      \put(1848,283){\makebox(0,0)[r]{\strut{}$\alpha =10^{-4}$}}%
    }%
    \gplgaddtomacro\gplbacktext{%
      \csname LTb\endcsname%
      \put(4628,110){\makebox(0,0)[r]{\strut{} 5.15}}%
      \csname LTb\endcsname%
      \put(4628,843){\makebox(0,0)[r]{\strut{} 5.2}}%
      \csname LTb\endcsname%
      \put(4628,1577){\makebox(0,0)[r]{\strut{} 5.25}}%
      \csname LTb\endcsname%
      \put(4628,2310){\makebox(0,0)[r]{\strut{} 5.3}}%
      \csname LTb\endcsname%
      \put(4628,3044){\makebox(0,0)[r]{\strut{} 5.35}}%
      \csname LTb\endcsname%
      \put(4628,3777){\makebox(0,0)[r]{\strut{} 5.4}}%
      \csname LTb\endcsname%
      \put(4628,4511){\makebox(0,0)[r]{\strut{} 5.45}}%
      \csname LTb\endcsname%
      \put(4760,-110){\makebox(0,0){\strut{}0.0}}%
      \csname LTb\endcsname%
      \put(5814,-110){\makebox(0,0){\strut{}0.002}}%
      \csname LTb\endcsname%
      \put(6867,-110){\makebox(0,0){\strut{}0.004}}%
      \csname LTb\endcsname%
      \put(7921,-110){\makebox(0,0){\strut{}0.006}}%
      \put(6340,-440){\makebox(0,0){\strut{}Time}}%
    }%
    \gplgaddtomacro\gplfronttext{%
      \csname LTb\endcsname%
      \put(5816,1383){\makebox(0,0)[r]{\strut{}BGN}}%
      \csname LTb\endcsname%
      \put(5816,1163){\makebox(0,0)[r]{\strut{}$\alpha =1.0$}}%
      \csname LTb\endcsname%
      \put(5816,943){\makebox(0,0)[r]{\strut{}$\alpha =10^{-1}$}}%
      \csname LTb\endcsname%
      \put(5816,723){\makebox(0,0)[r]{\strut{}$\alpha =10^{-2}$}}%
      \csname LTb\endcsname%
      \put(5816,503){\makebox(0,0)[r]{\strut{}$\alpha =10^{-3}$}}%
      \csname LTb\endcsname%
      \put(5816,283){\makebox(0,0)[r]{\strut{}$\alpha =10^{-4}$}}%
    }%
    \gplbacktext
    \put(0,0){\includegraphics{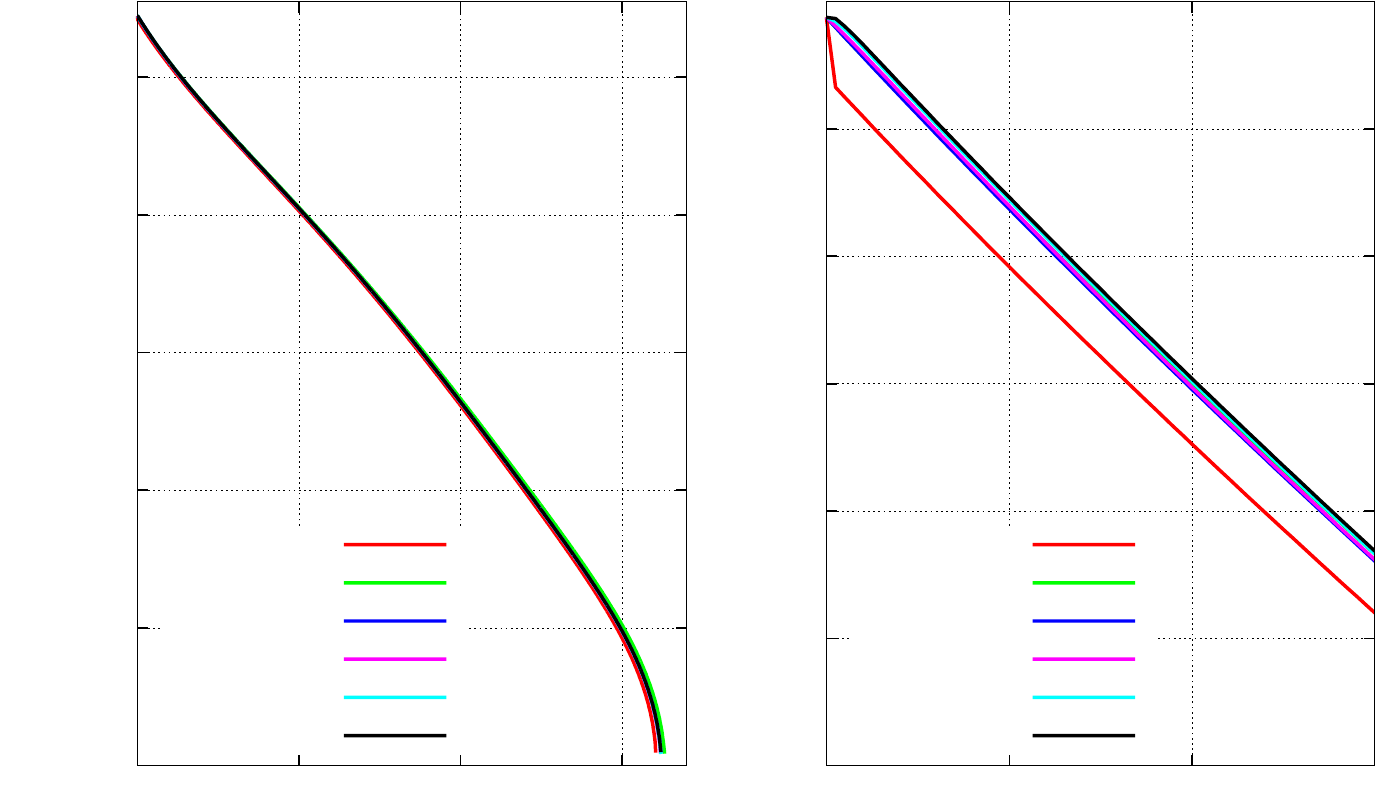}}%
    \gplfronttext
  \end{picture}%
  \vspace*{12mm}
  \caption{The images show the decrease of the curve length under the curve shortening flow
  for the BGN-scheme (2.16a) in \cite{BGN11} and for Algorithm \ref{algo_CSF} 
  for different choices of $\alpha$. 
  The initial curve is shown in Figure \ref{CSF_initial_curve_example_1}.  
  The time step size was chosen as 
  $\tau = 10^{-4}$. The right image shows an enlarged section for small times $t$.
  Due to the equidistribution property of the BGN-scheme, this scheme leads to 
  a slightly stronger drop of 
  the curve length in the first time step. See Example 1 of Section \ref{Numerical_results_CSF}
  for further details.}
  \label{Figure_length_decrease_CSF_example_1}
\end{figure}
\gdef\gplbacktext{}%
\gdef\gplfronttext{}%
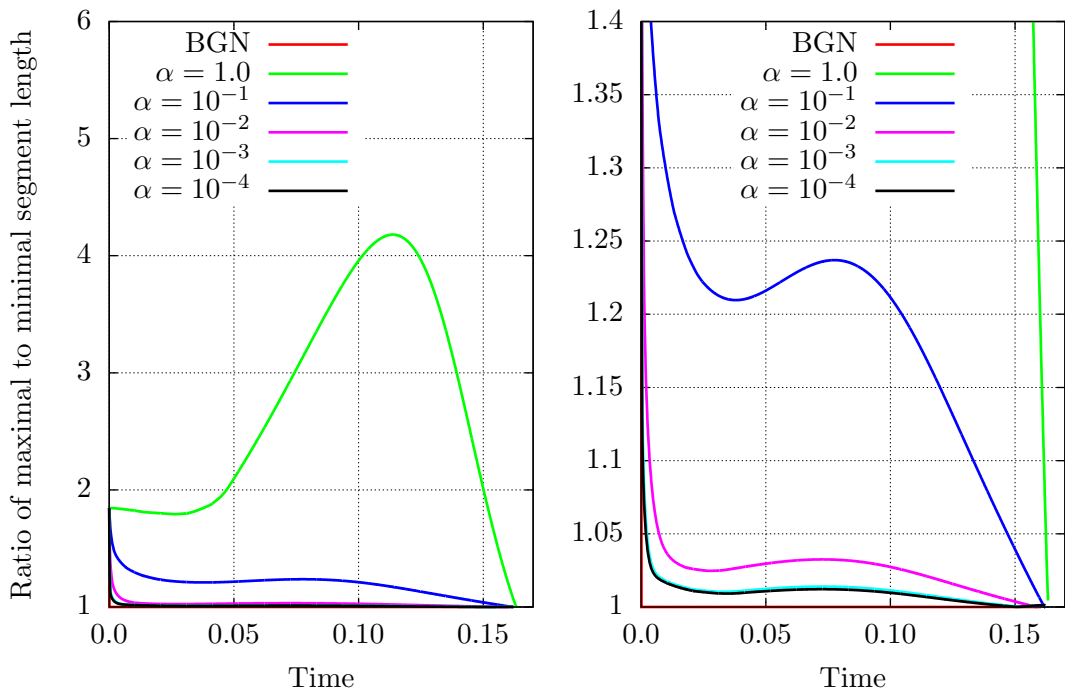
\begin{figure}
\centering
 \setlength{\unitlength}{0.0500bp}%
  \begin{picture}(7936.00,4534.00)%
    \gplgaddtomacro\gplbacktext{%
      \csname LTb\endcsname%
      \put(660,110){\makebox(0,0)[r]{\strut{} 1}}%
      \csname LTb\endcsname%
      \put(660,990){\makebox(0,0)[r]{\strut{} 2}}%
      \csname LTb\endcsname%
      \put(660,1870){\makebox(0,0)[r]{\strut{} 3}}%
      \csname LTb\endcsname%
      \put(660,2751){\makebox(0,0)[r]{\strut{} 4}}%
      \csname LTb\endcsname%
      \put(660,3631){\makebox(0,0)[r]{\strut{} 5}}%
      \csname LTb\endcsname%
      \put(660,4511){\makebox(0,0)[r]{\strut{} 6}}%
      \csname LTb\endcsname%
      \put(792,-110){\makebox(0,0){\strut{}0.0}}%
      \csname LTb\endcsname%
      \put(1722,-110){\makebox(0,0){\strut{}0.05}}%
      \csname LTb\endcsname%
      \put(2652,-110){\makebox(0,0){\strut{}0.10}}%
      \csname LTb\endcsname%
      \put(3582,-110){\makebox(0,0){\strut{}0.15}}%
      \put(154,2310){\rotatebox{-270}{\makebox(0,0){\strut{}Ratio of maximal to minimal segment length}}}%
      \put(2373,-440){\makebox(0,0){\strut{}Time}}%
    }%
    \gplgaddtomacro\gplfronttext{%
      \csname LTb\endcsname%
      \put(1848,4338){\makebox(0,0)[r]{\strut{}BGN}}%
      \csname LTb\endcsname%
      \put(1848,4118){\makebox(0,0)[r]{\strut{}$\alpha =1.0$}}%
      \csname LTb\endcsname%
      \put(1848,3898){\makebox(0,0)[r]{\strut{}$\alpha =10^{-1}$}}%
      \csname LTb\endcsname%
      \put(1848,3678){\makebox(0,0)[r]{\strut{}$\alpha =10^{-2}$}}%
      \csname LTb\endcsname%
      \put(1848,3458){\makebox(0,0)[r]{\strut{}$\alpha =10^{-3}$}}%
      \csname LTb\endcsname%
      \put(1848,3238){\makebox(0,0)[r]{\strut{}$\alpha =10^{-4}$}}%
    }%
    \gplgaddtomacro\gplbacktext{%
      \csname LTb\endcsname%
      \put(4628,110){\makebox(0,0)[r]{\strut{} 1}}%
      \csname LTb\endcsname%
      \put(4628,660){\makebox(0,0)[r]{\strut{} 1.05}}%
      \csname LTb\endcsname%
      \put(4628,1210){\makebox(0,0)[r]{\strut{} 1.1}}%
      \csname LTb\endcsname%
      \put(4628,1760){\makebox(0,0)[r]{\strut{} 1.15}}%
      \csname LTb\endcsname%
      \put(4628,2311){\makebox(0,0)[r]{\strut{} 1.2}}%
      \csname LTb\endcsname%
      \put(4628,2861){\makebox(0,0)[r]{\strut{} 1.25}}%
      \csname LTb\endcsname%
      \put(4628,3411){\makebox(0,0)[r]{\strut{} 1.3}}%
      \csname LTb\endcsname%
      \put(4628,3961){\makebox(0,0)[r]{\strut{} 1.35}}%
      \csname LTb\endcsname%
      \put(4628,4511){\makebox(0,0)[r]{\strut{} 1.4}}%
      \csname LTb\endcsname%
      \put(4760,-110){\makebox(0,0){\strut{}0.0}}%
      \csname LTb\endcsname%
      \put(5690,-110){\makebox(0,0){\strut{}0.05}}%
      \csname LTb\endcsname%
      \put(6619,-110){\makebox(0,0){\strut{}0.10}}%
      \csname LTb\endcsname%
      \put(7549,-110){\makebox(0,0){\strut{}0.15}}%
      \put(6340,-440){\makebox(0,0){\strut{}Time}}%
    }%
    \gplgaddtomacro\gplfronttext{%
      \csname LTb\endcsname%
      \put(6375,4338){\makebox(0,0)[r]{\strut{}BGN}}%
      \csname LTb\endcsname%
      \put(6375,4118){\makebox(0,0)[r]{\strut{}$\alpha =1.0$}}%
      \csname LTb\endcsname%
      \put(6375,3898){\makebox(0,0)[r]{\strut{}$\alpha =10^{-1}$}}%
      \csname LTb\endcsname%
      \put(6375,3678){\makebox(0,0)[r]{\strut{}$\alpha =10^{-2}$}}%
      \csname LTb\endcsname%
      \put(6375,3458){\makebox(0,0)[r]{\strut{}$\alpha =10^{-3}$}}%
      \csname LTb\endcsname%
      \put(6375,3238){\makebox(0,0)[r]{\strut{}$\alpha =10^{-4}$}}%
    }%
    \gplbacktext
    \put(0,0){\includegraphics{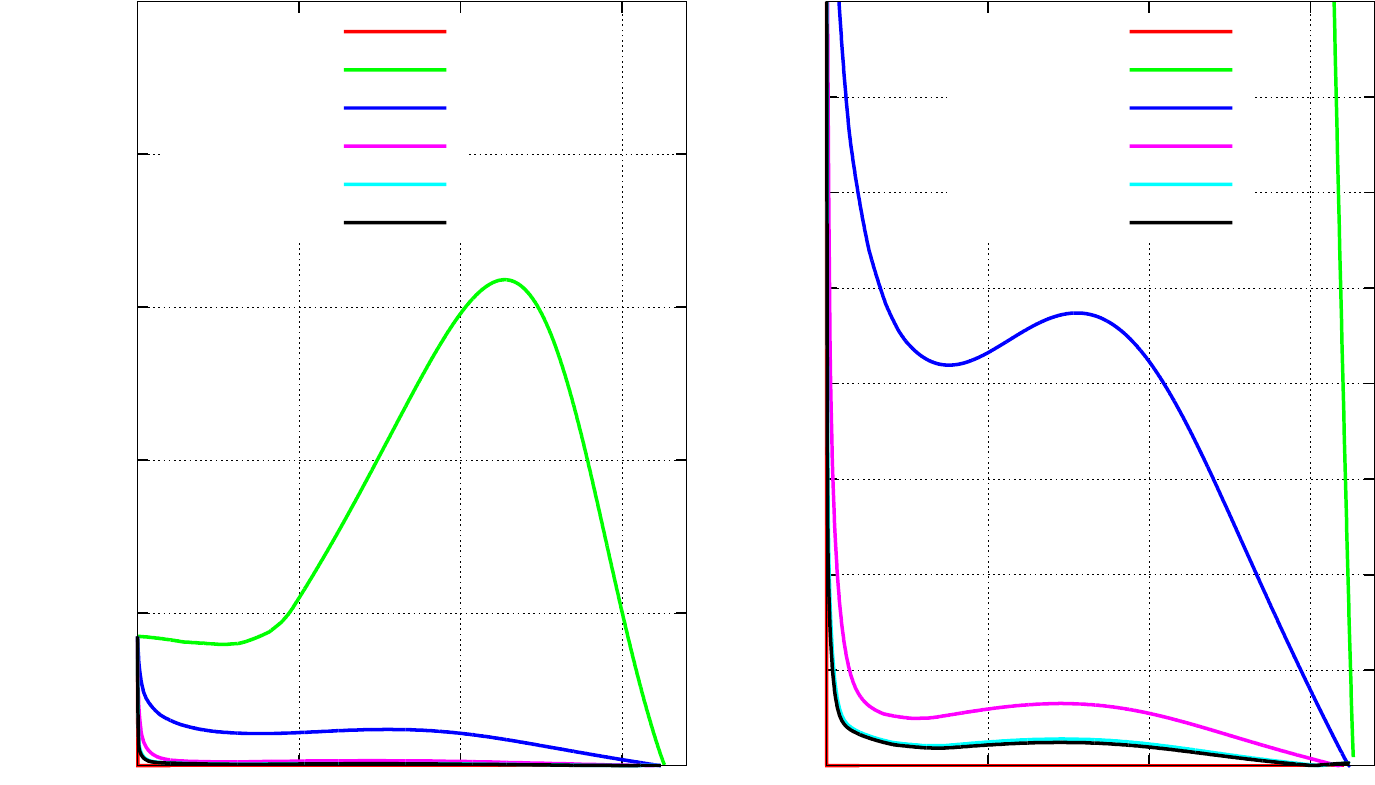}}%
    \gplfronttext
  \end{picture}%
\vspace*{12mm}
\caption{The images show the ratio of the maximal to the minimal segment length 
for the BGN-scheme (2.16a) in \cite{BGN11} and for Algorithm \ref{algo_CSF} for different choices
of $\alpha$. The initial curve is shown in Figure \ref{CSF_initial_curve_example_1}.
The right image shows an enlarged section.
Due to the equidistribution property of the BGN-scheme, the ratio of the
maximal to the minimal segment length is equal to $1$ after the first time step. 
However, the images clearly show that also Algorithm \ref{algo_CSF} has good properties
with respect to the mesh quality provided that $\alpha$ is chosen sufficiently small. 
The mesh properties of Algorithm \ref{algo_CSF} seemed to be (almost) independent of the choice of $\tau$
if $\tau \leq 10^{-4}$ (not shown in the picture).
See Example 1 of Section \ref{Numerical_results_CSF} for further details.}
\label{Figure_length_ratio_CSF_example_1}
\end{figure}

\gdef\gplbacktext{}%
\gdef\gplfronttext{}%
\begin{figure}
\centering
\begin{picture}(5102.00,3400.00)%
    \gplgaddtomacro\gplbacktext{%
      \csname LTb\endcsname%
      \put(396,110){\makebox(0,0)[r]{\strut{} 0}}%
      \csname LTb\endcsname%
      \put(396,763){\makebox(0,0)[r]{\strut{} 10}}%
      \csname LTb\endcsname%
      \put(396,1417){\makebox(0,0)[r]{\strut{} 20}}%
      \csname LTb\endcsname%
      \put(396,2070){\makebox(0,0)[r]{\strut{} 30}}%
      \csname LTb\endcsname%
      \put(396,2724){\makebox(0,0)[r]{\strut{} 40}}%
      \csname LTb\endcsname%
      \put(396,3377){\makebox(0,0)[r]{\strut{} 50}}%
      \csname LTb\endcsname%
      \put(528,-110){\makebox(0,0){\strut{} 0}}%
      \csname LTb\endcsname%
      \put(1064,-110){\makebox(0,0){\strut{} 0.02}}%
      \csname LTb\endcsname%
      \put(1601,-110){\makebox(0,0){\strut{} 0.04}}%
      \csname LTb\endcsname%
      \put(2137,-110){\makebox(0,0){\strut{} 0.06}}%
      \csname LTb\endcsname%
      \put(2673,-110){\makebox(0,0){\strut{} 0.08}}%
      \csname LTb\endcsname%
      \put(3210,-110){\makebox(0,0){\strut{} 0.1}}%
      \csname LTb\endcsname%
      \put(3746,-110){\makebox(0,0){\strut{} 0.12}}%
      \csname LTb\endcsname%
      \put(4282,-110){\makebox(0,0){\strut{} 0.14}}%
      \csname LTb\endcsname%
      \put(4819,-110){\makebox(0,0){\strut{} 0.16}}%
      \put(-242,1743){\rotatebox{-270}{\makebox(0,0){\strut{}Number of iteration steps}}}%
      \put(2807,-440){\makebox(0,0){\strut{}Time}}%
    }%
    \gplgaddtomacro\gplfronttext{%
      \csname LTb\endcsname%
      \put(4100,3204){\makebox(0,0)[r]{\strut{}BGN}}%
    }%
    \gplbacktext
    \put(0,0){\includegraphics{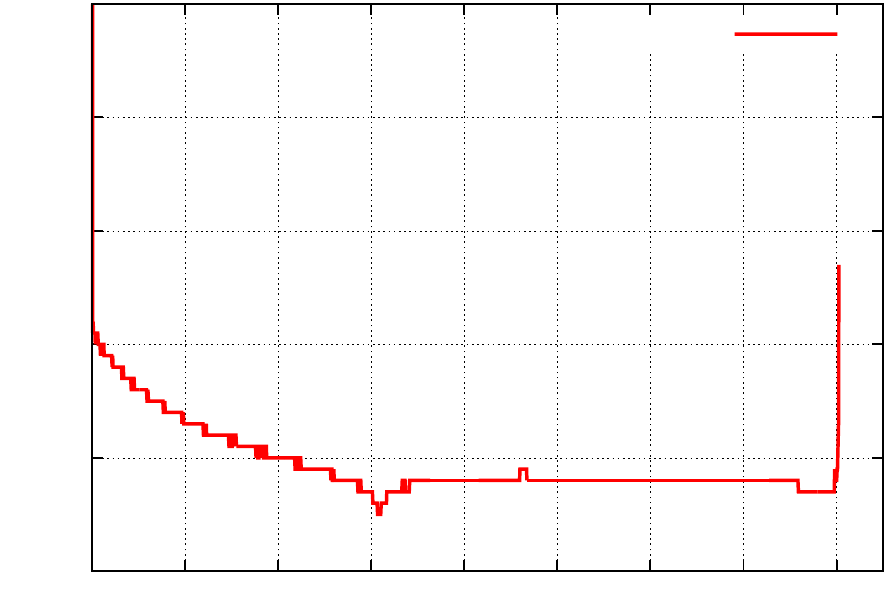}}%
    \gplfronttext
  \end{picture}%
\vspace*{12mm}
\caption{The image shows the number of iteration steps of the fixed point iteration 
that are necessary to solve the non-linear system of equations which arises in 
the fully-implicit BGN-scheme (2.16) in \cite{BGN11}.
The fixed point iteration was stopped if for all vertices of the discrete curve
the distance between the position vectors does not change more than $10^{-8}$
in one iteration step. The image shows that apart from the first time step
and the time step at the end of the simulation the fixed point iteration converges rather fast.
See Example 1 of Section \ref{Numerical_results_CSF} for further details.}
\label{number_of_iteration_steps_BGN_example_1}
\end{figure}

\gdef\gplbacktext{}%
\gdef\gplfronttext{}%
\begin{figure}
\centering
\begin{picture}(7936.00,4534.00)%
    \gplgaddtomacro\gplbacktext{%
      \csname LTb\endcsname%
      \put(660,110){\makebox(0,0)[r]{\strut{} 0}}%
      \csname LTb\endcsname%
      \put(660,990){\makebox(0,0)[r]{\strut{} 10}}%
      \csname LTb\endcsname%
      \put(660,1870){\makebox(0,0)[r]{\strut{} 20}}%
      \csname LTb\endcsname%
      \put(660,2751){\makebox(0,0)[r]{\strut{} 30}}%
      \csname LTb\endcsname%
      \put(660,3631){\makebox(0,0)[r]{\strut{} 40}}%
      \csname LTb\endcsname%
      \put(660,4511){\makebox(0,0)[r]{\strut{} 50}}%
      \csname LTb\endcsname%
      \put(792,-110){\makebox(0,0){\strut{}0.0}}%
      \csname LTb\endcsname%
      \put(1722,-110){\makebox(0,0){\strut{}0.05}}%
      \csname LTb\endcsname%
      \put(2652,-110){\makebox(0,0){\strut{}0.10}}%
      \csname LTb\endcsname%
      \put(3582,-110){\makebox(0,0){\strut{}0.15}}%
      \put(22,2310){\rotatebox{-270}{\makebox(0,0){\strut{}Maximal velocity}}}%
      \put(2373,-440){\makebox(0,0){\strut{}Time}}%
    }%
    \gplgaddtomacro\gplfronttext{%
      \csname LTb\endcsname%
      \put(1848,4338){\makebox(0,0)[r]{\strut{}BGN}}%
      \csname LTb\endcsname%
      \put(1848,4118){\makebox(0,0)[r]{\strut{}$\alpha =1.0$}}%
      \csname LTb\endcsname%
      \put(1848,3898){\makebox(0,0)[r]{\strut{}$\alpha =10^{-1}$}}%
      \csname LTb\endcsname%
      \put(1848,3678){\makebox(0,0)[r]{\strut{}$\alpha =10^{-2}$}}%
      \csname LTb\endcsname%
      \put(1848,3458){\makebox(0,0)[r]{\strut{}$\alpha =10^{-3}$}}%
      \csname LTb\endcsname%
      \put(1848,3238){\makebox(0,0)[r]{\strut{}$\alpha =10^{-4}$}}%
    }%
    \gplgaddtomacro\gplbacktext{%
      \csname LTb\endcsname%
      \put(4628,110){\makebox(0,0)[r]{\strut{} 0}}%
      \csname LTb\endcsname%
      \put(4628,910){\makebox(0,0)[r]{\strut{} 200}}%
      \csname LTb\endcsname%
      \put(4628,1710){\makebox(0,0)[r]{\strut{} 400}}%
      \csname LTb\endcsname%
      \put(4628,2511){\makebox(0,0)[r]{\strut{} 600}}%
      \csname LTb\endcsname%
      \put(4628,3311){\makebox(0,0)[r]{\strut{} 800}}%
      \csname LTb\endcsname%
      \put(4628,4111){\makebox(0,0)[r]{\strut{} 1000}}%
      \csname LTb\endcsname%
      \put(5023,-110){\makebox(0,0){\strut{}0.0}}%
      \csname LTb\endcsname%
      \put(5682,-110){\makebox(0,0){\strut{}0.25}}%
      \csname LTb\endcsname%
      \put(6341,-110){\makebox(0,0){\strut{}0.50}}%
      \csname LTb\endcsname%
      \put(6999,-110){\makebox(0,0){\strut{}0.75}}%
      \csname LTb\endcsname%
      \put(7658,-110){\makebox(0,0){\strut{}1.00}}%
      \put(6340,-440){\makebox(0,0){\strut{}$\alpha$}}%
    }%
    \gplgaddtomacro\gplfronttext{%
      \csname LTb\endcsname%
      \put(6934,4338){\makebox(0,0)[r]{\strut{}BGN}}%
      \csname LTb\endcsname%
      \put(6934,4118){\makebox(0,0)[r]{\strut{}$\alpha =1.0$}}%
      \csname LTb\endcsname%
      \put(6934,3898){\makebox(0,0)[r]{\strut{}$\alpha =10^{-1}$}}%
      \csname LTb\endcsname%
      \put(6934,3678){\makebox(0,0)[r]{\strut{}$\alpha =10^{-2}$}}%
      \csname LTb\endcsname%
      \put(6934,3458){\makebox(0,0)[r]{\strut{}$\alpha =10^{-3}$}}%
      \csname LTb\endcsname%
      \put(6934,3238){\makebox(0,0)[r]{\strut{}$\alpha =10^{-4}$}}%
      \csname LTb\endcsname%
      \put(6934,3018){\makebox(0,0)[r]{\strut{}$\alpha =10^{-5}$}}%
      \csname LTb\endcsname%
      \put(6934,2798){\makebox(0,0)[r]{\strut{}$\alpha =0.0$}}%
    }%
    \gplbacktext
    \put(0,0){\includegraphics{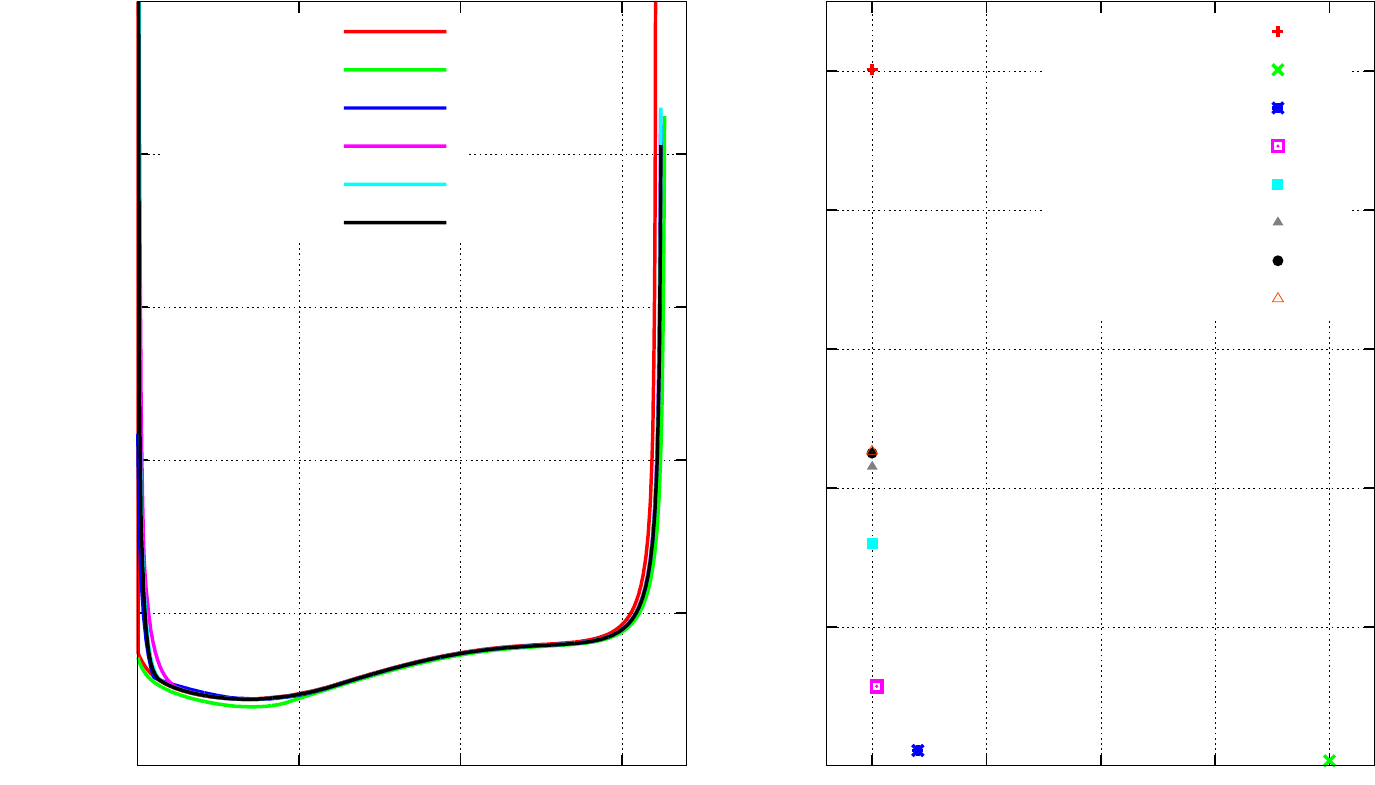}}%
    \gplfronttext
  \end{picture}%
\vspace*{12mm}
\caption{Comparison of the maximal velocity of the vertices for the BGN-scheme 
  (2.16a) in \cite{BGN11} and for Algorithm \ref{algo_CSF} for different choices of $\alpha$. 
  The initial curve of the simulation is shown in Figure \ref{CSF_initial_curve_example_1}.
  The time step size was $\tau = 10^{-4}$.  
  The left image shows the maximal velocity of all vertices as a function of time $t$.
  The right image shows how the maximal velocity of the vertices at time $t=0$ depends on
  the parameter $\alpha$. The maximal initial velocity of the BGN-scheme is clearly
  higher than the maximal initial velocity of the $\alpha$-schemes. 
  In fact, Figure \ref{Figure_maximal_velocity_BGN_scheme} shows
  that the maximal initial velocity of the BGN-scheme 
  as a function of the inverse time step size $\tau^{-1}$ is unbounded.
  See Example 1 of Section \ref{Numerical_results_CSF} for further details.}
  \label{Figure_maximal_velocity_CSF_example_1}
\end{figure}
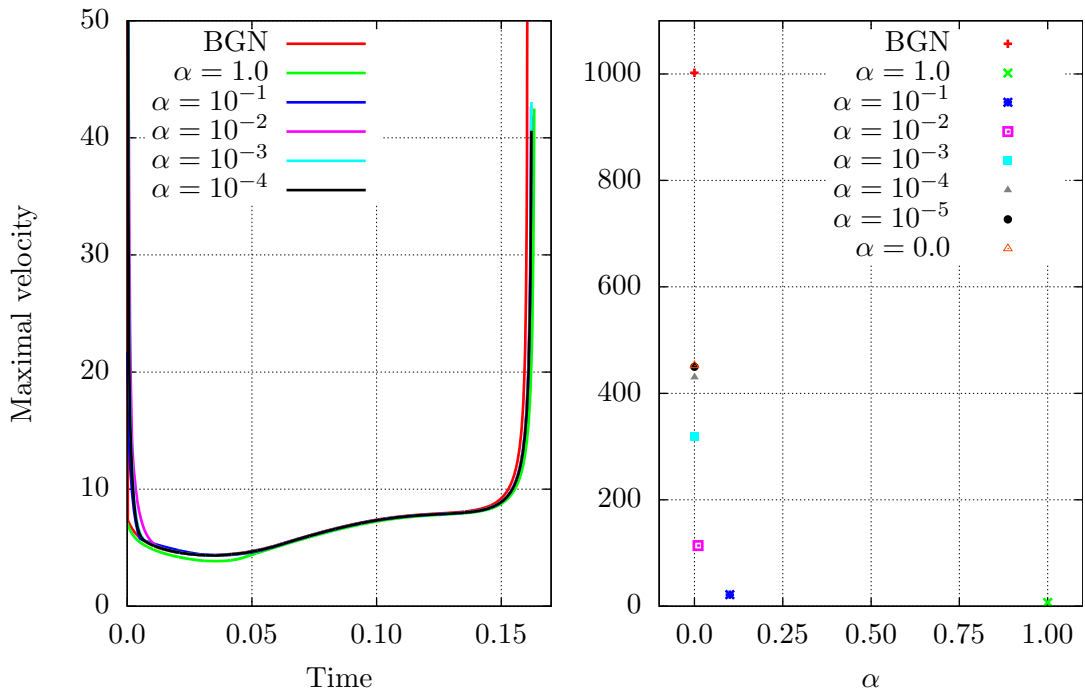

\gdef\gplbacktext{}%
\gdef\gplfronttext{}%
\begin{figure}
\centering
\begin{picture}(5102.00,3400.00)%
    \gplgaddtomacro\gplbacktext{%
      \csname LTb\endcsname%
      \put(396,110){\makebox(0,0)[r]{\strut{} 10}}%
      \csname LTb\endcsname%
      \put(396,927){\makebox(0,0)[r]{\strut{} 100}}%
      \csname LTb\endcsname%
      \put(396,1744){\makebox(0,0)[r]{\strut{} 1000}}%
      \csname LTb\endcsname%
      \put(396,2560){\makebox(0,0)[r]{\strut{} 10000}}%
      \csname LTb\endcsname%
      \put(396,3377){\makebox(0,0)[r]{\strut{} 100000}}%
      \csname LTb\endcsname%
      \put(528,-110){\makebox(0,0){\strut{}100}}%
      \csname LTb\endcsname%
      \put(1668,-110){\makebox(0,0){\strut{}1000}}%
      \csname LTb\endcsname%
      \put(2807,-110){\makebox(0,0){\strut{}10000}}%
      \csname LTb\endcsname%
      \put(3947,-110){\makebox(0,0){\strut{}100000}}%
      \csname LTb\endcsname%
      \put(5087,-110){\makebox(0,0){\strut{}1000000}}%
      \put(-770,1743){\rotatebox{-270}{\makebox(0,0){\strut{}Maximal velocity}}}%
      \put(2807,-440){\makebox(0,0){\strut{}$\tau^{-1}$}}%
    }%
    \gplgaddtomacro\gplfronttext{%
      \csname LTb\endcsname%
      \put(1584,3204){\makebox(0,0)[r]{\strut{}BGN}}%
      \csname LTb\endcsname%
      \put(1584,2984){\makebox(0,0)[r]{\strut{}$\alpha = 10^{-2}$}}%
    }%
    \gplbacktext
    \put(0,0){\includegraphics{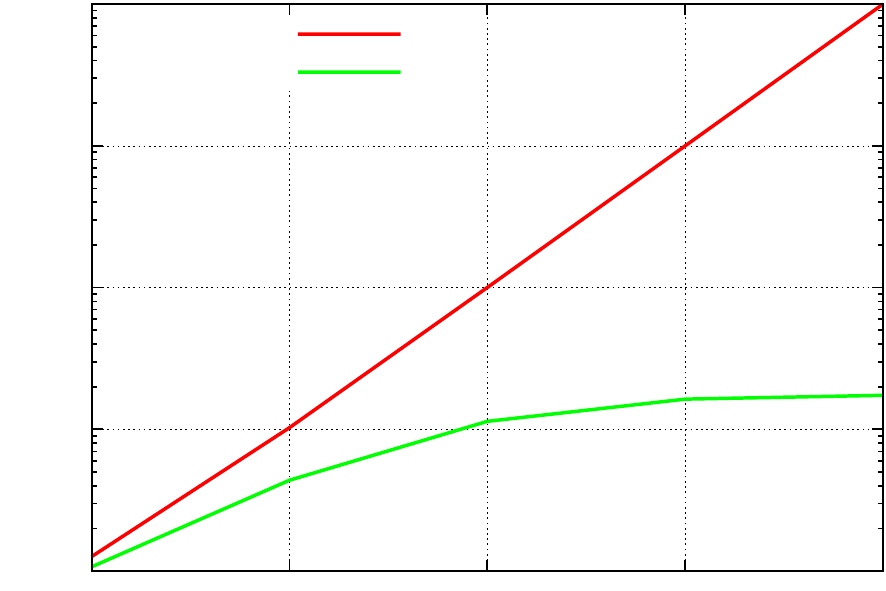}}%
    \gplfronttext
  \end{picture}%
  \vspace*{12mm}
\caption{The image shows the maximal initial velocity of the BGN-scheme (2.16a) in \cite{BGN11} 
	and of Algorithm \ref{algo_CSF} with $\alpha = 10^{-2}$ as a function of the inverse 
	time step size $\tau^{-1}$. 
	The linear growth of the maximal initial velocity of the BGN-scheme is due to the 
	finite jumps of the vertices in the first time step, 
	see Figure \ref{Figure_motion_of_the_curve_vertices_CSF}.
	In contrast, the maximal initial velocity of the $\alpha$-scheme is bounded.
	See Example 1 of Section \ref{Numerical_results_CSF} for further details.
	}
\label{Figure_maximal_velocity_BGN_scheme}
\end{figure}
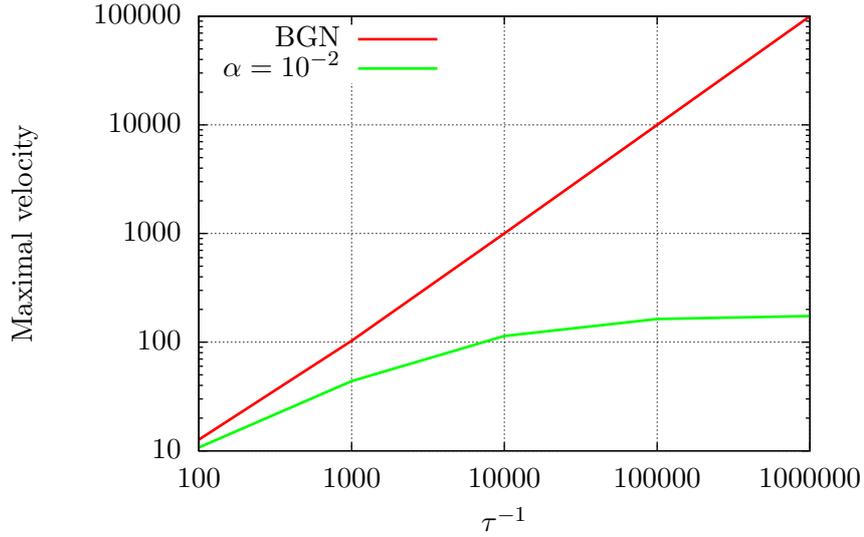

\gdef\gplbacktext{}%
\gdef\gplfronttext{}%
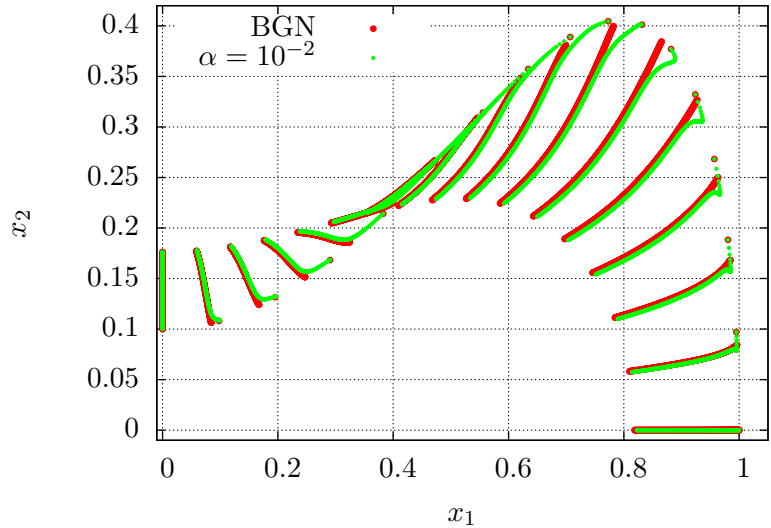
\begin{figure}
\centering
 \begin{picture}(5102.00,3400.00)%
    \gplgaddtomacro\gplbacktext{%
      \csname LTb\endcsname%
      \put(396,186){\makebox(0,0)[r]{\strut{} 0}}%
      \csname LTb\endcsname%
      \put(396,566){\makebox(0,0)[r]{\strut{} 0.05}}%
      \csname LTb\endcsname%
      \put(396,946){\makebox(0,0)[r]{\strut{} 0.1}}%
      \csname LTb\endcsname%
      \put(396,1326){\makebox(0,0)[r]{\strut{} 0.15}}%
      \csname LTb\endcsname%
      \put(396,1706){\makebox(0,0)[r]{\strut{} 0.2}}%
      \csname LTb\endcsname%
      \put(396,2085){\makebox(0,0)[r]{\strut{} 0.25}}%
      \csname LTb\endcsname%
      \put(396,2465){\makebox(0,0)[r]{\strut{} 0.3}}%
      \csname LTb\endcsname%
      \put(396,2845){\makebox(0,0)[r]{\strut{} 0.35}}%
      \csname LTb\endcsname%
      \put(396,3225){\makebox(0,0)[r]{\strut{} 0.4}}%
      \csname LTb\endcsname%
      \put(571,-110){\makebox(0,0){\strut{} 0}}%
      \csname LTb\endcsname%
      \put(1431,-110){\makebox(0,0){\strut{} 0.2}}%
      \csname LTb\endcsname%
      \put(2291,-110){\makebox(0,0){\strut{} 0.4}}%
      \csname LTb\endcsname%
      \put(3152,-110){\makebox(0,0){\strut{} 0.6}}%
      \csname LTb\endcsname%
      \put(4012,-110){\makebox(0,0){\strut{} 0.8}}%
      \csname LTb\endcsname%
      \put(4872,-110){\makebox(0,0){\strut{} 1}}%
      \put(-506,1743){\rotatebox{-270}{\makebox(0,0){\strut{}$x_2$}}}%
      \put(2807,-440){\makebox(0,0){\strut{}$x_1$}}%
    }%
    \gplgaddtomacro\gplfronttext{%
      \csname LTb\endcsname%
      \put(1716,3204){\makebox(0,0)[r]{\strut{}BGN}}%
      \csname LTb\endcsname%
      \put(1716,2984){\makebox(0,0)[r]{\strut{}$\alpha = 10^{-2}$}}%
    }%
    \gplbacktext
    \put(0,0){\includegraphics{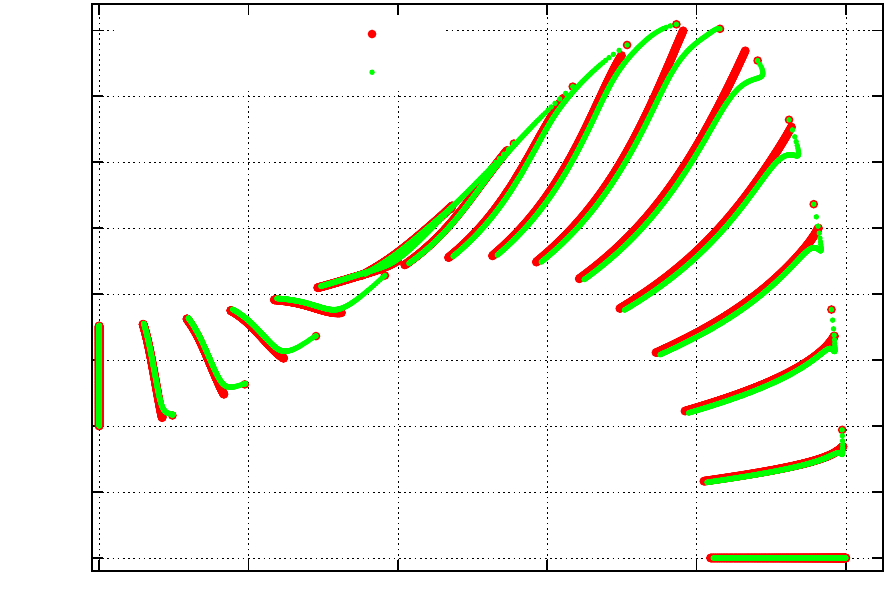}}%
    \gplfronttext
  \end{picture}%
\vspace*{12mm}
\caption{The image shows the motion of the curve vertices for the BGN-scheme (2.16a) in \cite{BGN11} 
	and for Algorithm \ref{algo_CSF} with $\alpha = 10^{-2}$. 
	Only the first quadrant is shown for the first time steps of the simulation.
	The whole initial curve is shown in Figure \ref{CSF_initial_curve_example_1}.
    The time step size was chosen as $\tau = 10^{-4}$. 
    One can clearly see that the BGN-scheme leads to large jumps of the vertices 
    in the first time step.
    The $\alpha$-scheme seems to interpolate between these jumps.
    See Example 1 of Section \ref{Numerical_results_CSF} for further details.}
\label{Figure_motion_of_the_curve_vertices_CSF}
\end{figure}

\subsection*{Example 2:}
We now consider the curve shortening flow for the initial curve given by the parametrization
\begin{equation*}
	X_0(\theta):= \left(
		\begin{array}{c}
			\cos(2\theta) \cos \theta
			\\
			\cos(2\theta) \sin \theta 
		\end{array}
	\right),
	\quad \theta \in [0,2\pi),
\end{equation*}
see Figure \ref{CSF_initial_curve_example_2} for a visualization.
The simulation based on Algorithm $1$ with $\alpha = 10^{-3}$ shows that the curve shortening 
flow develops a singularity, see Figure \ref{CSF_singularity}, at time $t \approx 0.0828$. 
Beyond this singularity,
see Figures \ref{CSF_not_yet_a_round_circle_example_2} and \ref{CSF_round_circle_example_2},
the curve shrinks to a round circle. 
Note that this example does not satisfy the regularity assumptions
made in Theorem \ref{approximation_theorem}.
Since the fixed point iteration used
to solve the non-linear system of equations arising in the BGN-scheme stops to converge at
the curve singularity shown in Figure \ref{CSF_singularity}, see Figure \ref{Figure_iteration_steps_CSF_example_2} for the number of iteration steps in
the fixed point iteration, it is not possible to compute the curve shortening flow through this
singularity by the fixed point iteration we used for the BGN-scheme; also note the remark below.
Figure \ref{Figure_length_ratio_CSF_example_2} shows that at this singularity
the ratio between the maximal and minimal segment length of the $\alpha$-scheme increases
very fast before it decreases again. It is this flexibility which seems to be advantageous 
in this example. Apart from the singularities the $\alpha$-scheme shows again good mesh
properties provided that $\alpha$ is sufficiently small. 
The decrease of the curve length is presented in Figure \ref{Figure_decrease_curve_length_CSF_example_2}.    
We here only report that by using a damped fixed point iteration in the BGN-scheme, the
solver converges for mild damping parameters also at the singularity of the curve shortening
flow; see also the remark at the end of Example 3.
\begin{figure}
\begin{center}
\subfloat[][Time $t = 0.0$. The curve length is $9.66$.]
{\includegraphics[width=0.4\textwidth]{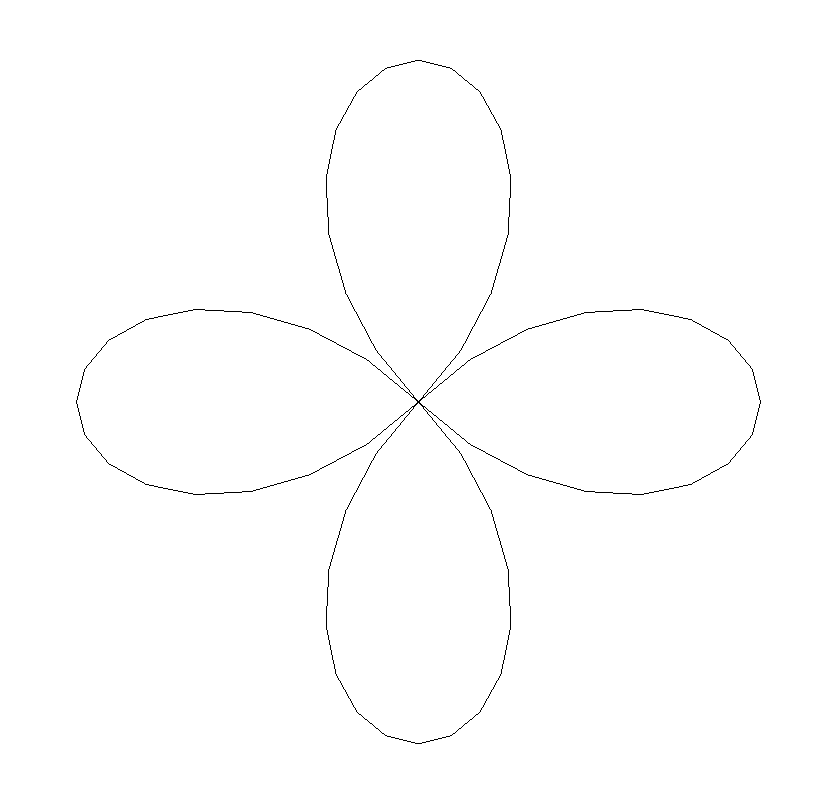}
\label{CSF_initial_curve_example_2}} 
\subfloat[][Time $t = 0.02$. The curve length is $8.56$.]
{\includegraphics[width=0.4\textwidth]{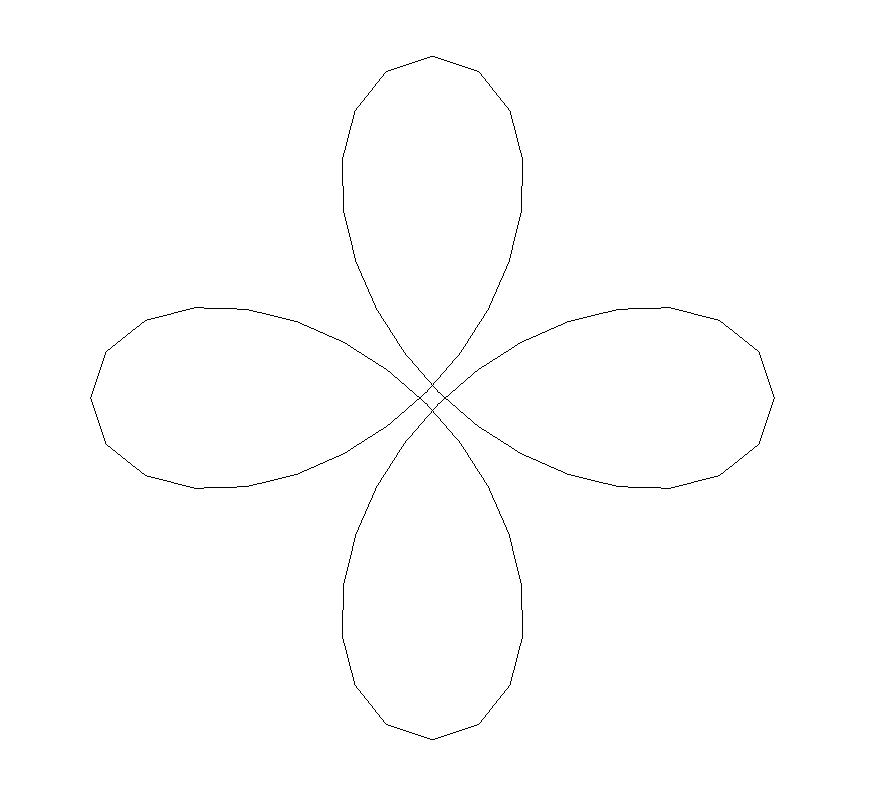}} \\
\subfloat[][Time $t = 0.08$. The curve length is $2.66$.]
{\includegraphics[width=0.4\textwidth]{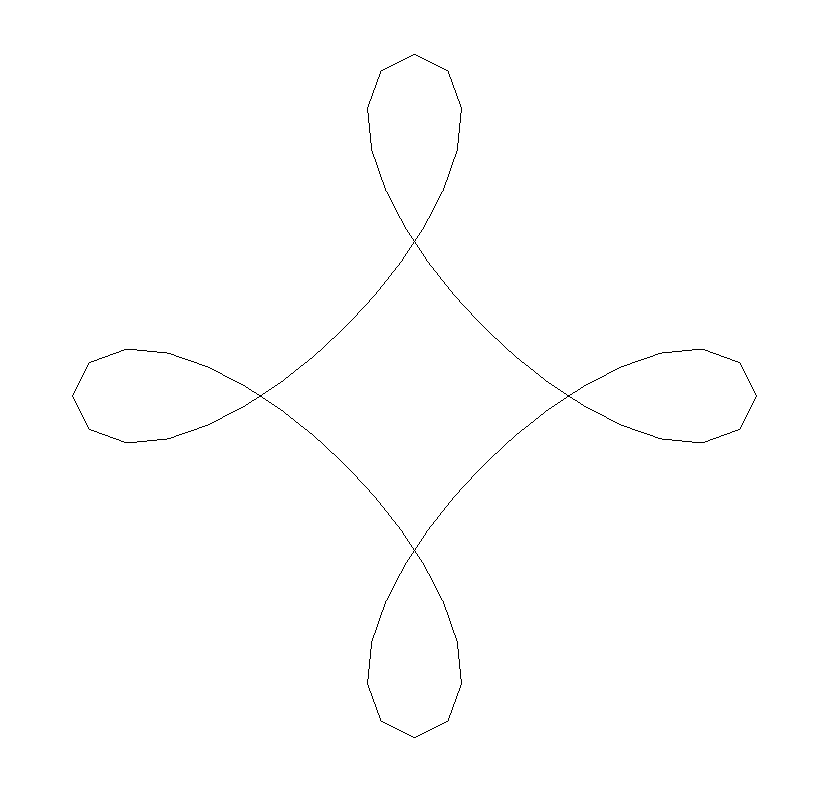}} 
\subfloat[][Time $t = 0.0828$. The curve length is $1.12$.]
{\includegraphics[width=0.4\textwidth]{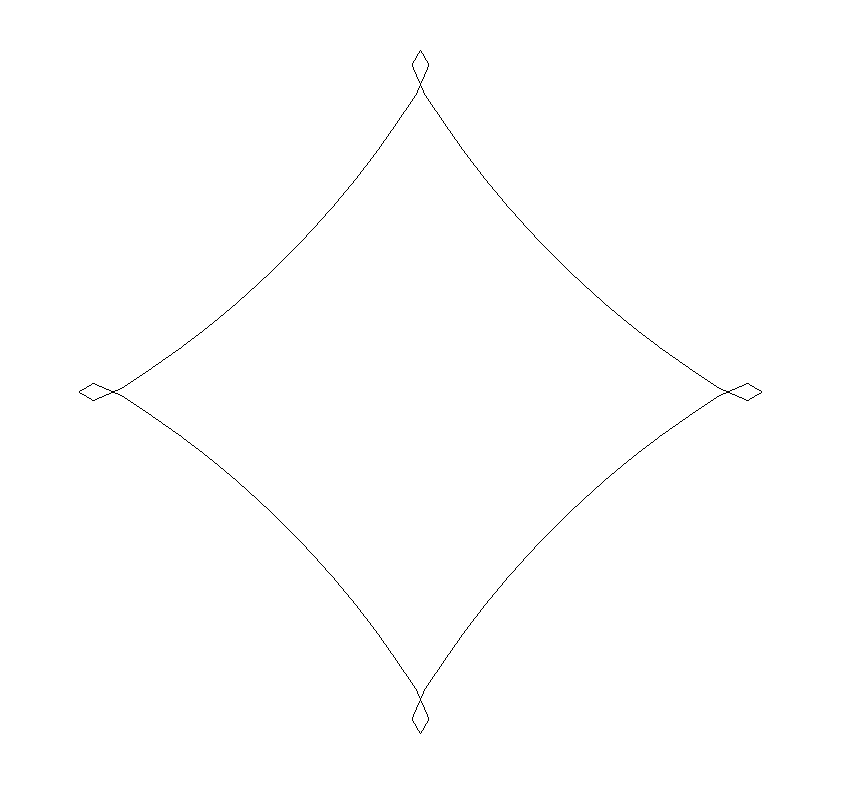}
\label{CSF_singularity}}     \\
\subfloat[][Time $t = 0.0829$. The curve length is $0.86$.]
{\includegraphics[width=0.4\textwidth]{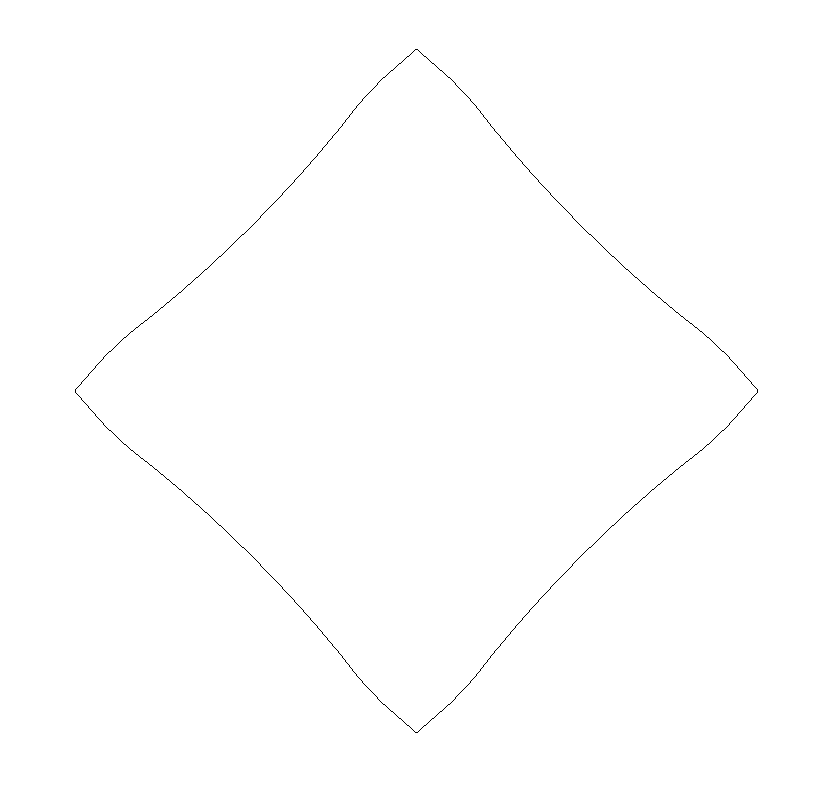}
\label{CSF_not_yet_a_round_circle_example_2}
}
\subfloat[][Time $t = 0.086$. The curve length is $0.56$.]
{\includegraphics[width=0.4\textwidth]{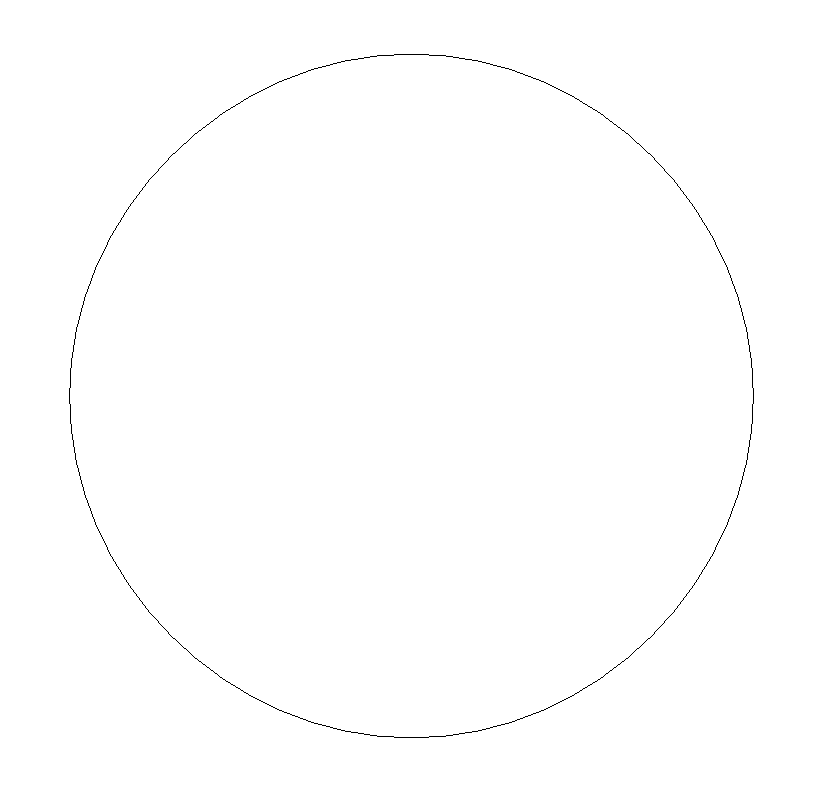}
\label{CSF_round_circle_example_2}
}
\caption{Simulation of the curve shortening flow with Algorithm \ref{algo_CSF} for $\alpha = 10^{-3}$
and time step size $\tau = 10^{-4}$. The computational mesh had $64$ vertices. 
The images are rescaled. A comparison with the BGN-scheme (2.16a) in \cite{BGN11} shows that the
fixed point iteration used to solve the fully-implicit BGN-scheme stops to converge at the singularity
shown in Figure \ref{CSF_singularity}.}  
\end{center}
\label{Figure_curve_evolution_CSF_example_2}
\end{figure}

\gdef\gplbacktext{}%
\gdef\gplfronttext{}%
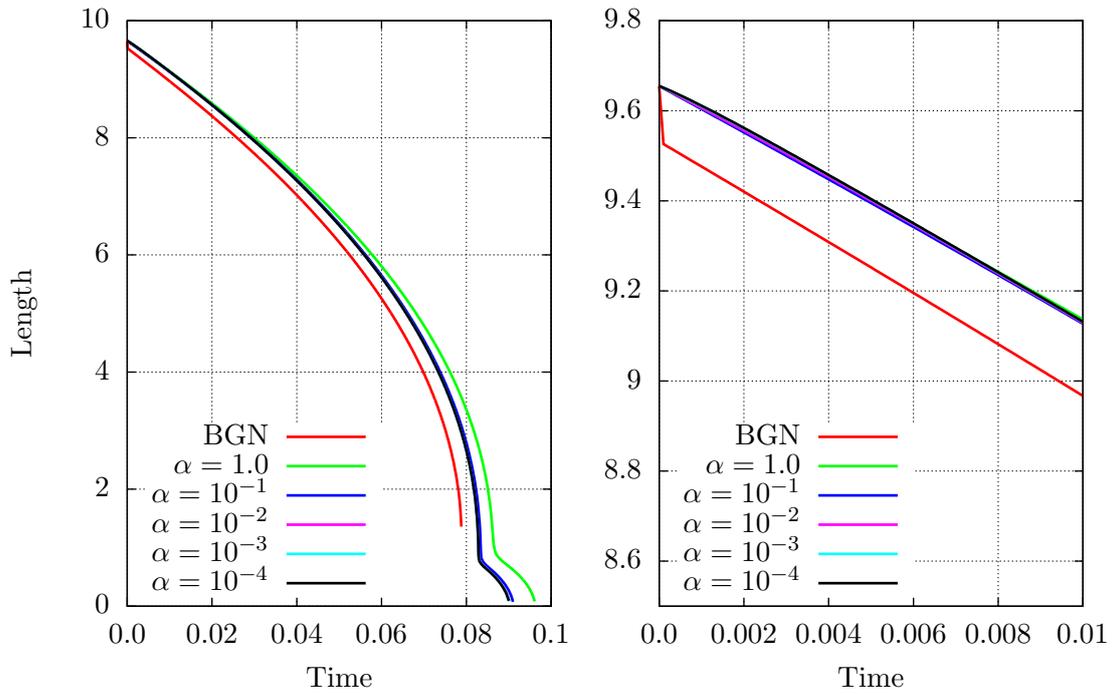
\begin{figure}
\centering
 \begin{picture}(7936.00,4534.00)%
    \gplgaddtomacro\gplbacktext{%
      \csname LTb\endcsname%
      \put(660,110){\makebox(0,0)[r]{\strut{} 0}}%
      \csname LTb\endcsname%
      \put(660,990){\makebox(0,0)[r]{\strut{} 2}}%
      \csname LTb\endcsname%
      \put(660,1870){\makebox(0,0)[r]{\strut{} 4}}%
      \csname LTb\endcsname%
      \put(660,2751){\makebox(0,0)[r]{\strut{} 6}}%
      \csname LTb\endcsname%
      \put(660,3631){\makebox(0,0)[r]{\strut{} 8}}%
      \csname LTb\endcsname%
      \put(660,4511){\makebox(0,0)[r]{\strut{} 10}}%
      \csname LTb\endcsname%
      \put(792,-110){\makebox(0,0){\strut{}0.0}}%
      \csname LTb\endcsname%
      \put(1424,-110){\makebox(0,0){\strut{}0.02}}%
      \csname LTb\endcsname%
      \put(2057,-110){\makebox(0,0){\strut{}0.04}}%
      \csname LTb\endcsname%
      \put(2689,-110){\makebox(0,0){\strut{}0.06}}%
      \csname LTb\endcsname%
      \put(3322,-110){\makebox(0,0){\strut{}0.08}}%
      \csname LTb\endcsname%
      \put(3954,-110){\makebox(0,0){\strut{}0.1}}%
      \put(22,2310){\rotatebox{-270}{\makebox(0,0){\strut{}Length}}}%
      \put(2373,-440){\makebox(0,0){\strut{}Time}}%
    }%
    \gplgaddtomacro\gplfronttext{%
      \csname LTb\endcsname%
      \put(1848,1383){\makebox(0,0)[r]{\strut{}BGN}}%
      \csname LTb\endcsname%
      \put(1848,1163){\makebox(0,0)[r]{\strut{}$\alpha =1.0$}}%
      \csname LTb\endcsname%
      \put(1848,943){\makebox(0,0)[r]{\strut{}$\alpha =10^{-1}$}}%
      \csname LTb\endcsname%
      \put(1848,723){\makebox(0,0)[r]{\strut{}$\alpha =10^{-2}$}}%
      \csname LTb\endcsname%
      \put(1848,503){\makebox(0,0)[r]{\strut{}$\alpha =10^{-3}$}}%
      \csname LTb\endcsname%
      \put(1848,283){\makebox(0,0)[r]{\strut{}$\alpha =10^{-4}$}}%
    }%
    \gplgaddtomacro\gplbacktext{%
      \csname LTb\endcsname%
      \put(4628,449){\makebox(0,0)[r]{\strut{} 8.6}}%
      \csname LTb\endcsname%
      \put(4628,1126){\makebox(0,0)[r]{\strut{} 8.8}}%
      \csname LTb\endcsname%
      \put(4628,1803){\makebox(0,0)[r]{\strut{} 9}}%
      \csname LTb\endcsname%
      \put(4628,2480){\makebox(0,0)[r]{\strut{} 9.2}}%
      \csname LTb\endcsname%
      \put(4628,3157){\makebox(0,0)[r]{\strut{} 9.4}}%
      \csname LTb\endcsname%
      \put(4628,3834){\makebox(0,0)[r]{\strut{} 9.6}}%
      \csname LTb\endcsname%
      \put(4628,4511){\makebox(0,0)[r]{\strut{} 9.8}}%
      \csname LTb\endcsname%
      \put(4760,-110){\makebox(0,0){\strut{}0.0}}%
      \csname LTb\endcsname%
      \put(5392,-110){\makebox(0,0){\strut{}0.002}}%
      \csname LTb\endcsname%
      \put(6024,-110){\makebox(0,0){\strut{}0.004}}%
      \csname LTb\endcsname%
      \put(6657,-110){\makebox(0,0){\strut{}0.006}}%
      \csname LTb\endcsname%
      \put(7289,-110){\makebox(0,0){\strut{}0.008}}%
      \csname LTb\endcsname%
      \put(7921,-110){\makebox(0,0){\strut{}0.01}}%
      \put(6340,-440){\makebox(0,0){\strut{}Time}}%
    }%
    \gplgaddtomacro\gplfronttext{%
      \csname LTb\endcsname%
      \put(5816,1383){\makebox(0,0)[r]{\strut{}BGN}}%
      \csname LTb\endcsname%
      \put(5816,1163){\makebox(0,0)[r]{\strut{}$\alpha =1.0$}}%
      \csname LTb\endcsname%
      \put(5816,943){\makebox(0,0)[r]{\strut{}$\alpha =10^{-1}$}}%
      \csname LTb\endcsname%
      \put(5816,723){\makebox(0,0)[r]{\strut{}$\alpha =10^{-2}$}}%
      \csname LTb\endcsname%
      \put(5816,503){\makebox(0,0)[r]{\strut{}$\alpha =10^{-3}$}}%
      \csname LTb\endcsname%
      \put(5816,283){\makebox(0,0)[r]{\strut{}$\alpha =10^{-4}$}}%
    }%
    \gplbacktext
    \put(0,0){\includegraphics{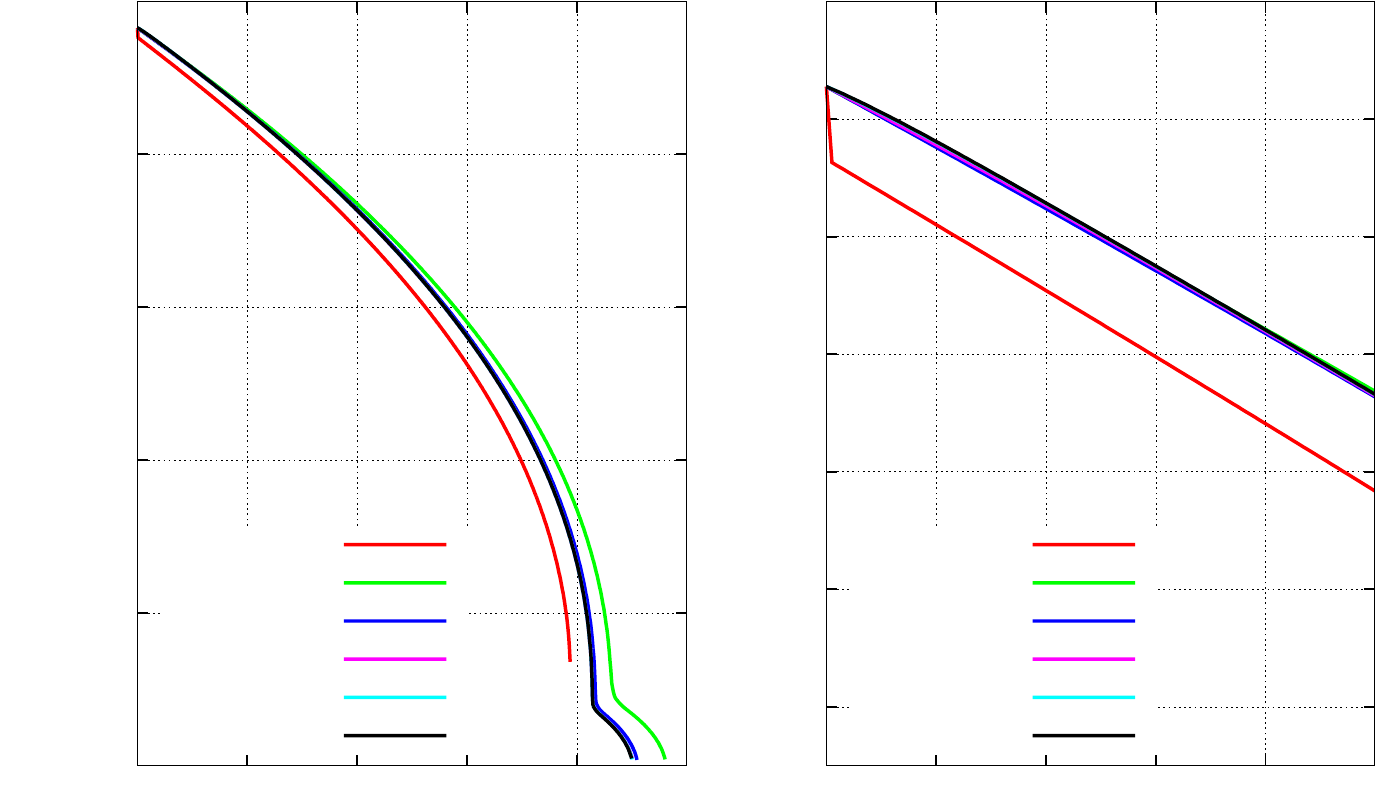}}%
    \gplfronttext
  \end{picture}%
\vspace*{12mm}
\caption{The images show the decrease of the curve length under the curve shortening flow
  for the BGN-scheme (2.16a) in \cite{BGN11} and for Algorithm \ref{algo_CSF} 
  for different choices of $\alpha$. 
  The initial curve is shown in Figure \ref{CSF_initial_curve_example_2}.   
  The time step size was chosen as $\tau = 10^{-4}$. 
  The fixed point iteration used to solve the BGN-scheme stops to converge
  at time $t= 0.0789$, see Figure \ref{Figure_iteration_steps_CSF_example_2}. 
  The right image shows an enlarged section for small times $t$.
  See Example 2 of Section \ref{Numerical_results_CSF}
  for further details.}  
  \label{Figure_decrease_curve_length_CSF_example_2}
\end{figure}

\gdef\gplbacktext{}%
\gdef\gplfronttext{}%
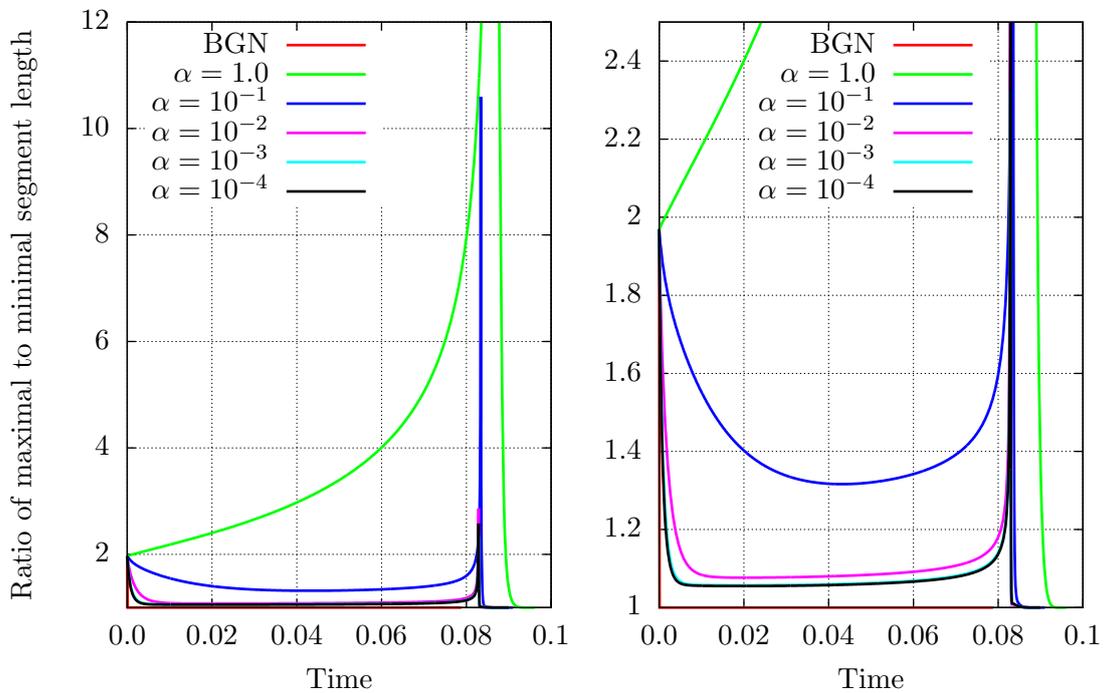
\begin{figure}
\centering
\begin{picture}(7936.00,4534.00)%
    \gplgaddtomacro\gplbacktext{%
      \csname LTb\endcsname%
      \put(660,510){\makebox(0,0)[r]{\strut{} 2}}%
      \csname LTb\endcsname%
      \put(660,1310){\makebox(0,0)[r]{\strut{} 4}}%
      \csname LTb\endcsname%
      \put(660,2110){\makebox(0,0)[r]{\strut{} 6}}%
      \csname LTb\endcsname%
      \put(660,2911){\makebox(0,0)[r]{\strut{} 8}}%
      \csname LTb\endcsname%
      \put(660,3711){\makebox(0,0)[r]{\strut{} 10}}%
      \csname LTb\endcsname%
      \put(660,4511){\makebox(0,0)[r]{\strut{} 12}}%
      \csname LTb\endcsname%
      \put(792,-110){\makebox(0,0){\strut{}0.0}}%
      \csname LTb\endcsname%
      \put(1424,-110){\makebox(0,0){\strut{}0.02}}%
      \csname LTb\endcsname%
      \put(2057,-110){\makebox(0,0){\strut{}0.04}}%
      \csname LTb\endcsname%
      \put(2689,-110){\makebox(0,0){\strut{}0.06}}%
      \csname LTb\endcsname%
      \put(3322,-110){\makebox(0,0){\strut{}0.08}}%
      \csname LTb\endcsname%
      \put(3954,-110){\makebox(0,0){\strut{}0.1}}%
      \put(22,2310){\rotatebox{-270}{\makebox(0,0){\strut{}Ratio of maximal to minimal segment length}}}%
      \put(2373,-440){\makebox(0,0){\strut{}Time}}%
    }%
    \gplgaddtomacro\gplfronttext{%
      \csname LTb\endcsname%
      \put(1848,4338){\makebox(0,0)[r]{\strut{}BGN}}%
      \csname LTb\endcsname%
      \put(1848,4118){\makebox(0,0)[r]{\strut{}$\alpha =1.0$}}%
      \csname LTb\endcsname%
      \put(1848,3898){\makebox(0,0)[r]{\strut{}$\alpha =10^{-1}$}}%
      \csname LTb\endcsname%
      \put(1848,3678){\makebox(0,0)[r]{\strut{}$\alpha =10^{-2}$}}%
      \csname LTb\endcsname%
      \put(1848,3458){\makebox(0,0)[r]{\strut{}$\alpha =10^{-3}$}}%
      \csname LTb\endcsname%
      \put(1848,3238){\makebox(0,0)[r]{\strut{}$\alpha =10^{-4}$}}%
    }%
    \gplgaddtomacro\gplbacktext{%
      \csname LTb\endcsname%
      \put(4628,110){\makebox(0,0)[r]{\strut{} 1}}%
      \csname LTb\endcsname%
      \put(4628,697){\makebox(0,0)[r]{\strut{} 1.2}}%
      \csname LTb\endcsname%
      \put(4628,1284){\makebox(0,0)[r]{\strut{} 1.4}}%
      \csname LTb\endcsname%
      \put(4628,1870){\makebox(0,0)[r]{\strut{} 1.6}}%
      \csname LTb\endcsname%
      \put(4628,2457){\makebox(0,0)[r]{\strut{} 1.8}}%
      \csname LTb\endcsname%
      \put(4628,3044){\makebox(0,0)[r]{\strut{} 2}}%
      \csname LTb\endcsname%
      \put(4628,3631){\makebox(0,0)[r]{\strut{} 2.2}}%
      \csname LTb\endcsname%
      \put(4628,4218){\makebox(0,0)[r]{\strut{} 2.4}}%
      \csname LTb\endcsname%
      \put(4760,-110){\makebox(0,0){\strut{}0.0}}%
      \csname LTb\endcsname%
      \put(5392,-110){\makebox(0,0){\strut{}0.02}}%
      \csname LTb\endcsname%
      \put(6024,-110){\makebox(0,0){\strut{}0.04}}%
      \csname LTb\endcsname%
      \put(6657,-110){\makebox(0,0){\strut{}0.06}}%
      \csname LTb\endcsname%
      \put(7289,-110){\makebox(0,0){\strut{}0.08}}%
      \csname LTb\endcsname%
      \put(7921,-110){\makebox(0,0){\strut{}0.1}}%
      \put(6340,-440){\makebox(0,0){\strut{}Time}}%
    }%
    \gplgaddtomacro\gplfronttext{%
      \csname LTb\endcsname%
      \put(6375,4338){\makebox(0,0)[r]{\strut{}BGN}}%
      \csname LTb\endcsname%
      \put(6375,4118){\makebox(0,0)[r]{\strut{}$\alpha =1.0$}}%
      \csname LTb\endcsname%
      \put(6375,3898){\makebox(0,0)[r]{\strut{}$\alpha =10^{-1}$}}%
      \csname LTb\endcsname%
      \put(6375,3678){\makebox(0,0)[r]{\strut{}$\alpha =10^{-2}$}}%
      \csname LTb\endcsname%
      \put(6375,3458){\makebox(0,0)[r]{\strut{}$\alpha =10^{-3}$}}%
      \csname LTb\endcsname%
      \put(6375,3238){\makebox(0,0)[r]{\strut{}$\alpha =10^{-4}$}}%
    }%
    \gplbacktext
    \put(0,0){\includegraphics{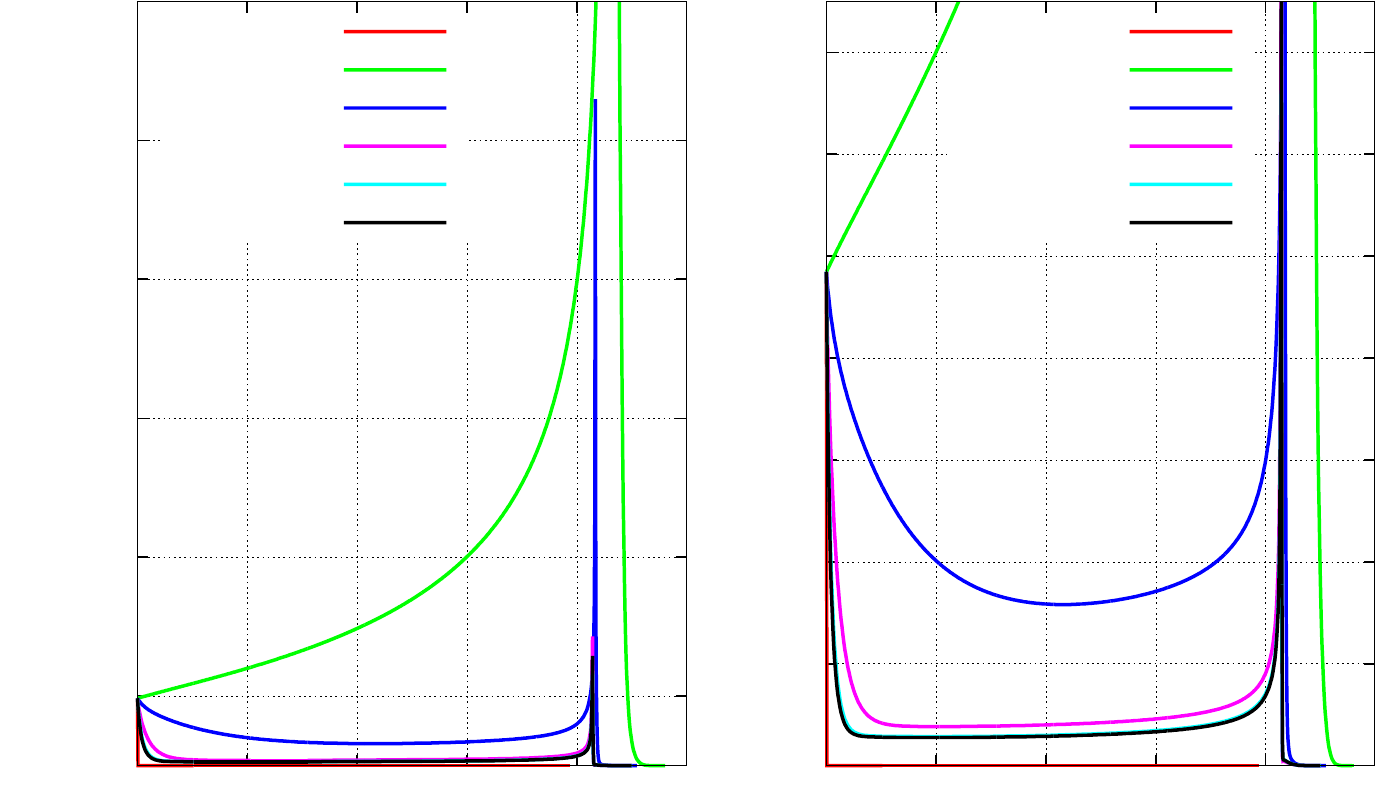}}%
    \gplfronttext
  \end{picture}%
  \vspace*{12mm}
\caption{The images show the ratio of the maximal to the minimal segment length 
for the BGN-scheme (2.16a) in \cite{BGN11} and for Algorithm \ref{algo_CSF} for different choices
of $\alpha$. The initial curve is shown in Figure \ref{CSF_initial_curve_example_2}.
The right image shows an enlarged section.
The simulation shows that Algorithm \ref{algo_CSF} has good mesh properties 
for $\alpha$ sufficiently small. 
See Example 2 of Section \ref{Numerical_results_CSF} for further details.}
\label{Figure_length_ratio_CSF_example_2}
\end{figure}

\gdef\gplbacktext{}%
\gdef\gplfronttext{}%
\begin{figure}
\centering
\begin{picture}(5102.00,3400.00)%
    \gplgaddtomacro\gplbacktext{%
      \csname LTb\endcsname%
      \put(396,110){\makebox(0,0)[r]{\strut{} 0}}%
      \csname LTb\endcsname%
      \put(396,763){\makebox(0,0)[r]{\strut{} 10}}%
      \csname LTb\endcsname%
      \put(396,1417){\makebox(0,0)[r]{\strut{} 20}}%
      \csname LTb\endcsname%
      \put(396,2070){\makebox(0,0)[r]{\strut{} 30}}%
      \csname LTb\endcsname%
      \put(396,2724){\makebox(0,0)[r]{\strut{} 40}}%
      \csname LTb\endcsname%
      \put(396,3377){\makebox(0,0)[r]{\strut{} 50}}%
      \csname LTb\endcsname%
      \put(528,-110){\makebox(0,0){\strut{}0.0}}%
      \csname LTb\endcsname%
      \put(1440,-110){\makebox(0,0){\strut{}0.02}}%
      \csname LTb\endcsname%
      \put(2352,-110){\makebox(0,0){\strut{}0.04}}%
      \csname LTb\endcsname%
      \put(3263,-110){\makebox(0,0){\strut{}0.06}}%
      \csname LTb\endcsname%
      \put(4175,-110){\makebox(0,0){\strut{}0.08}}%
      \csname LTb\endcsname%
      \put(5087,-110){\makebox(0,0){\strut{}0.1}}%
      \put(-242,1743){\rotatebox{-270}{\makebox(0,0){\strut{}Number of iteration steps}}}%
      \put(2807,-440){\makebox(0,0){\strut{}Time}}%
    }%
    \gplgaddtomacro\gplfronttext{%
      \csname LTb\endcsname%
      \put(1056,3204){\makebox(0,0)[r]{\strut{}BGN}}%
    }%
    \gplbacktext
    \put(0,0){\includegraphics{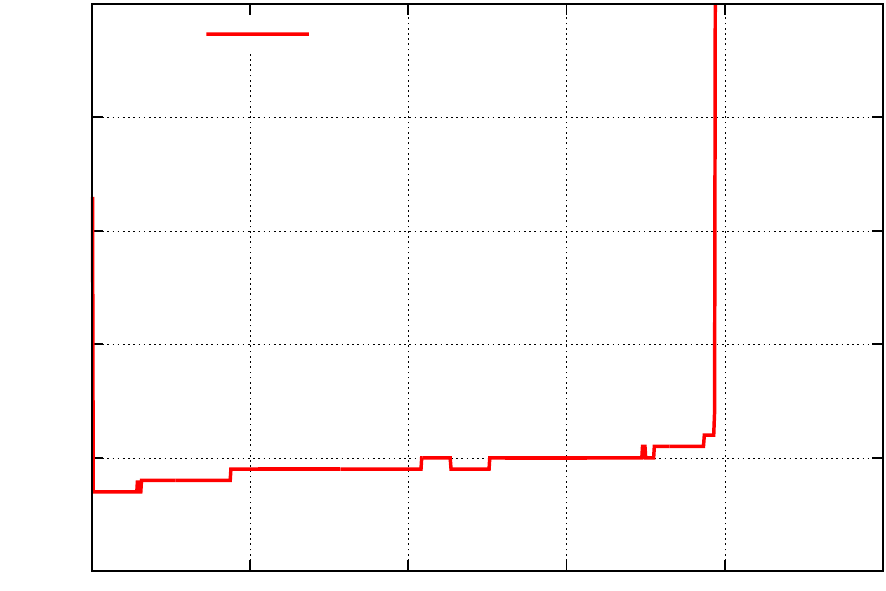}}%
    \gplfronttext
  \end{picture}%
\vspace*{12mm}
\caption{The figure shows the number of iteration steps of the fixed point iteration
that are necessary
to solve the non-linear system of equations which arises in 
the fully-implicit BGN-scheme (2.16) in \cite{BGN11}.
The fixed point iteration was stopped if for all vertices 
of the discrete curve the distance between 
the position vectors does not change more than $10^{-8}$ in one iteration step.
The initial curve of the simulation is shown in Figure \ref{CSF_initial_curve_example_2}.
The fixed point iteration stops to converge at time $t=0.0789$ 
which corresponds to the situation shown in Figure \ref{CSF_singularity}.
In Example 3, the fixed point iteration also does not converge
for a non-singular curve if the time step size is below a critical value.  
See Example 2 of Section \ref{Numerical_results_CSF} for further details. } 
\label{Figure_iteration_steps_CSF_example_2} 
\end{figure}

\subsection*{Example 3:}
In this example, we demonstrate that the employed fixed point iteration for the BGN-scheme
might not only fail at curve singularities as in Example 2. 
In fact, the problem also occurs for the most simple case, that is the unit circle, 
if the segment length of the initial triangulation is not constant 
and if the time step size $\tau$ is smaller than a critical value.
To start with non-equidistributed meshes might be indeed desirable for certain applications.
Other algorithms like the scheme in \cite{DD94} and the scheme (2.3) in \cite{BGN07},
are expected to be more appropriate to handle non-equidistributed initial meshes than the BGN-scheme (2.16a) from \cite{BGN11}.
In fact, it is an interesting question how the latter scheme behaves for initially non-equidistributed meshes.
Since the BGN-scheme (2.16a) plays an important role in this paper, we will address this question now in more detail.
In this example, the curve shortening flow is computed for the initial triangulation shown in 
Figure \ref{Figure_initial_triangulation_of_the_unit_sphere}. The vertices of this
triangulation are distributed in such a way that the segment length 
slowly decreases anti-clockwise.
Figure \ref{Figure_motion_of_the_vertices_CSF_unit_circle} shows
that the initial jumps of the mesh vertices under the BGN-scheme 
strongly depend on the time step size $\tau$.
We experimentally observed that 
for time step sizes $\tau \leq 10^{-4}$ the fixed point iteration
for solving the first time step of the BGN-scheme does not converge any more.
Interestingly, Algorithm \ref{algo_CSF} with $\alpha = 10^{-2}$ seems to interpolate
the motion of the vertices computed by the BGN-scheme with $\tau = 10^{-2}$,
see also Figure \ref{Figure_motion_of_the_curve_vertices_CSF} for a similar behaviour.
We finally note that by employing a damped fixed point iteration for the
BGN-scheme it is possible to circumvent the reported difficulties. However, we have observed that the
solution then strongly depends on the damping parameter.

\gdef\gplbacktext{}%
\gdef\gplfronttext{}%
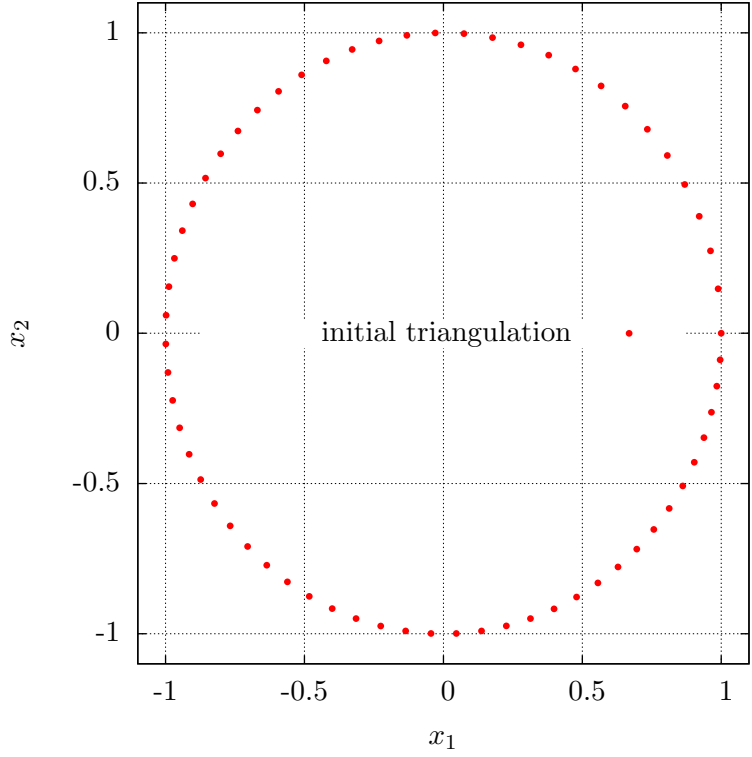
\begin{figure}
\centering
\begin{picture}(5102.00,5102.00)%
    \gplgaddtomacro\gplbacktext{%
      \csname LTb\endcsname%
      \put(396,336){\makebox(0,0)[r]{\strut{}-1}}%
      \csname LTb\endcsname%
      \put(396,1465){\makebox(0,0)[r]{\strut{}-0.5}}%
      \csname LTb\endcsname%
      \put(396,2595){\makebox(0,0)[r]{\strut{} 0}}%
      \csname LTb\endcsname%
      \put(396,3724){\makebox(0,0)[r]{\strut{} 0.5}}%
      \csname LTb\endcsname%
      \put(396,4853){\makebox(0,0)[r]{\strut{} 1}}%
      \csname LTb\endcsname%
      \put(735,-110){\makebox(0,0){\strut{}-1}}%
      \csname LTb\endcsname%
      \put(1771,-110){\makebox(0,0){\strut{}-0.5}}%
      \csname LTb\endcsname%
      \put(2808,-110){\makebox(0,0){\strut{} 0}}%
      \csname LTb\endcsname%
      \put(3844,-110){\makebox(0,0){\strut{} 0.5}}%
      \csname LTb\endcsname%
      \put(4880,-110){\makebox(0,0){\strut{} 1}}%
      \put(-374,2594){\rotatebox{-270}{\makebox(0,0){\strut{}$x_2$}}}%
      \put(2807,-440){\makebox(0,0){\strut{}$x_1$}}%
    }%
    \gplgaddtomacro\gplfronttext{%
      \csname LTb\endcsname%
      \put(3766,2594){\makebox(0,0)[r]{\strut{}initial triangulation}}%
    }%
    \gplbacktext
    \put(0,0){\includegraphics{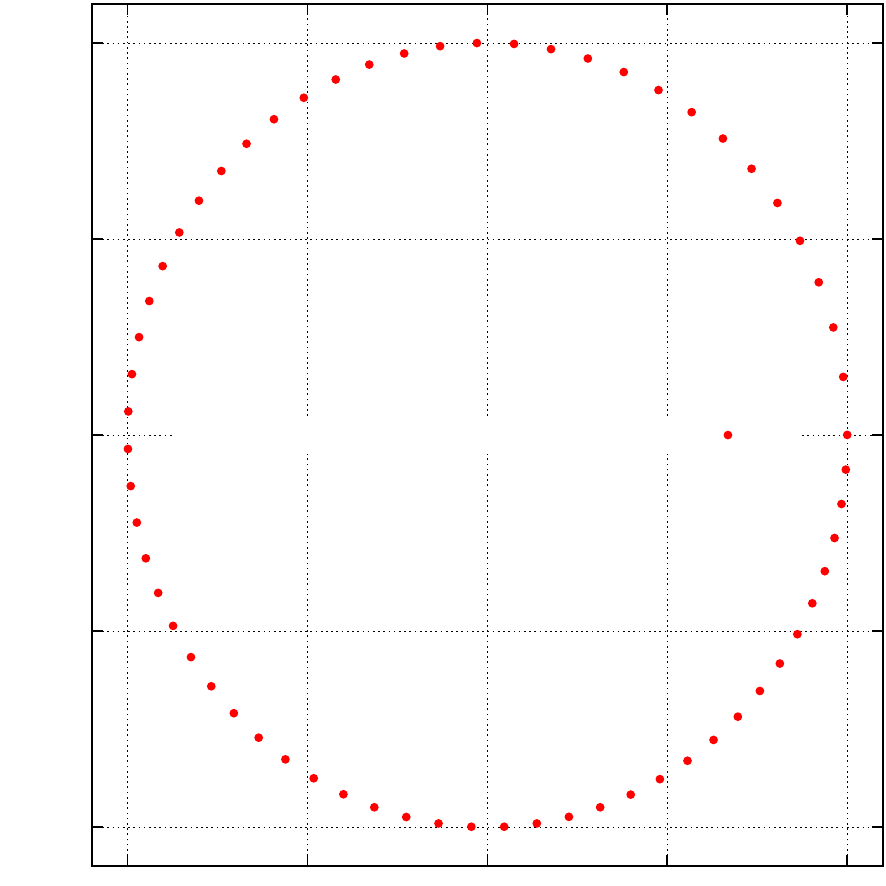}}%
    \gplfronttext
  \end{picture}%
\vspace*{12mm}
\caption{
The image shows the $64$ vertices of a triangulation of the unit circle.
Starting at the point $(x_1,x_2) = (1,0)$, 
the segment length, that is the length between two vertices, is slowly decreasing anti-clockwise.
This leads to a relatively large jump of the segment length between 
the two segments belonging to the point $(x_1, x_2) = (1,0)$. 
See Example 3 of Section \ref{Numerical_results_CSF} for further details.}  
\label{Figure_initial_triangulation_of_the_unit_sphere}
\end{figure}

\gdef\gplbacktext{}%
\gdef\gplfronttext{}%
\begin{figure}
\centering
\begin{picture}(5102.00,5102.00)%
    \gplgaddtomacro\gplbacktext{%
      \csname LTb\endcsname%
      \put(396,155){\makebox(0,0)[r]{\strut{} 0}}%
      \csname LTb\endcsname%
      \put(396,1050){\makebox(0,0)[r]{\strut{} 0.2}}%
      \csname LTb\endcsname%
      \put(396,1945){\makebox(0,0)[r]{\strut{} 0.4}}%
      \csname LTb\endcsname%
      \put(396,2841){\makebox(0,0)[r]{\strut{} 0.6}}%
      \csname LTb\endcsname%
      \put(396,3736){\makebox(0,0)[r]{\strut{} 0.8}}%
      \csname LTb\endcsname%
      \put(396,4631){\makebox(0,0)[r]{\strut{} 1}}%
      \csname LTb\endcsname%
      \put(569,-110){\makebox(0,0){\strut{} 0}}%
      \csname LTb\endcsname%
      \put(1391,-110){\makebox(0,0){\strut{} 0.2}}%
      \csname LTb\endcsname%
      \put(2212,-110){\makebox(0,0){\strut{} 0.4}}%
      \csname LTb\endcsname%
      \put(3033,-110){\makebox(0,0){\strut{} 0.6}}%
      \csname LTb\endcsname%
      \put(3855,-110){\makebox(0,0){\strut{} 0.8}}%
      \csname LTb\endcsname%
      \put(4676,-110){\makebox(0,0){\strut{} 1}}%
      \put(-374,2594){\rotatebox{-270}{\makebox(0,0){\strut{}$x_2$}}}%
      \put(2807,-440){\makebox(0,0){\strut{}$x_1$}}%
    }%
    \gplgaddtomacro\gplfronttext{%
      \csname LTb\endcsname%
      \put(4100,4906){\makebox(0,0)[r]{\strut{}BGN, $\tau = 10^{-2}$}}%
      \csname LTb\endcsname%
      \put(4100,4686){\makebox(0,0)[r]{\strut{}BGN, $\tau = 10^{-3}$}}%
      \csname LTb\endcsname%
      \put(4100,4466){\makebox(0,0)[r]{\strut{}$\alpha = 10^{-2}, \tau = 10^{-3}$}}%
    }%
    \gplbacktext
    \put(0,0){\includegraphics{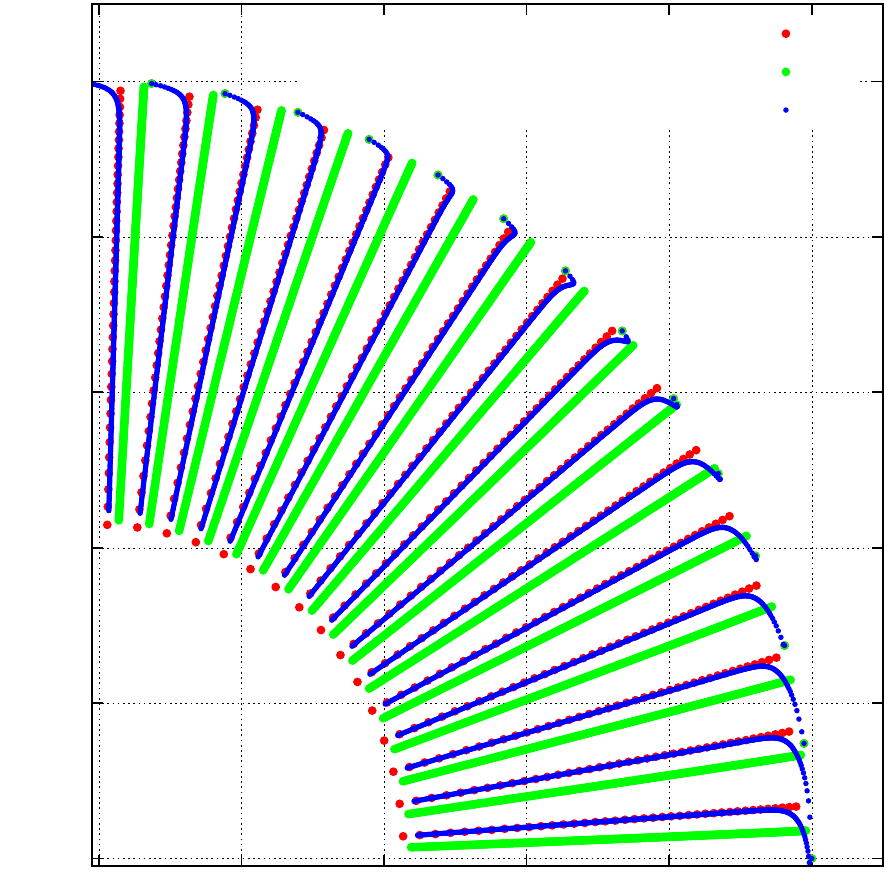}}%
    \gplfronttext
  \end{picture}%
\vspace*{12mm}
\caption{The image shows the motion of the triangulation vertices 
	for the BGN-scheme (2.16a) in \cite{BGN11} 
	and for Algorithm \ref{algo_CSF} with $\alpha = 10^{-2}$. 
	Only the first quadrant is shown.
	The vertices of the initial triangulation are presented in Figure
	\ref{Figure_initial_triangulation_of_the_unit_sphere}.
    The time step size was chosen as $\tau = 10^{-2}$ and $\tau = 10^{-3}$. 
    The BGN-scheme leads to relatively large jumps of the vertices in the first time step.
    The direction of these jumps seems to depend on the time step size (red and green dots).
    We also observed numerically that
    the fixed point iteration for the first time step of the BGN-scheme does not converge
    when the time step size is smaller than a critical value, 
    here for $\tau \leq 10^{-4}$.
    Compare to Figure \ref{Figure_iteration_steps_CSF_example_2}, where the fixed point iteration
    stops to converge at a singularity of the curve.
    See Example 3 of Section \ref{Numerical_results_CSF} for further details.} 
    \label{Figure_motion_of_the_vertices_CSF_unit_circle} 
\end{figure}

\section{The reparametrized mean curvature flow}
\label{reparametrizations_by_the_original_DeTurck_trick}

\subsection{Weak formulation on the reference manifold $\M$}
In this section, we derive a weak formulation of the reparametrized mean curvature
flow (\ref{reparam_MCF}).
Unfortunately, this flow is not in divergence form. We hence decompose the elliptic operator of (\ref{reparam_MCF})
into a divergence and into a non-divergence part.
This can be simply achieved by using identity (\ref{normal_part})
\begin{align*}
\left( \alpha \unit + (1 - \alpha) (\nu \circ \hat{x}_\alpha) \otimes (\nu \circ \hat{x}_\alpha) \right) 
\frac{\partial}{\partial t}\hat{x}_\alpha
	= \Delta_{\hat{g}_\alpha} \hat{x}_\alpha
	+ (P \circ \hat{x}_\alpha) \tr_{\hat{g}_\alpha} (\nabla^h \nabla \hat{x}_\alpha).
\end{align*}
The reason why the second term on the right hand side is not in divergence form
is that the trace $\tr_{\hat{g}_\alpha}$ 
has to be taken with respect to the metric $\hat{g}_\alpha(t)$,
whereas $\nabla^h$ is the covariant derivative with respect to the metric $h$.
By choosing $h$ appropriately, it is possible to derive an expression
for $(P \circ \hat{x}_\alpha) \tr_{\hat{g}_\alpha} (\nabla^h \nabla \hat{x}_\alpha)$,
which is the product of a first order term and of a second order term
that is in divergence form. In this section, we will assume that the reference manifold $\M$ is
an $n$-dimensional hypersurface in $\mathbb{R}^{n+1}$.

In our weak formulation of the mean curvature-DeTurck flow we will make use of the following representation of the metric $\hat{g}_\alpha(t)$.
We define $\hat{G}_\alpha : \M \times [0,T) \rightarrow \mathbb{R}^{(n+1) \times (n+1)}$ by 
\begin{equation}
	\label{defi_G}
	\hat{G}_\alpha(t) :=
	(\nabla_\M \hat{x}_\alpha(t))^T \nabla_\M \hat{x}_\alpha(t)
	+ \mu \otimes \mu,
\end{equation}
where $\mu = (\mu_1, \ldots, \mu_{n+1})^T$ is a unit normal field to $\M$.
That this map is indeed a representation of the metric $\hat{g}_\alpha(t)$ is stated in (\ref{G_beta_gamma}) below. 
The reason, why we have introduced $\hat{G}_\alpha(t)$, is the fact that it allows us to represent the metric $\hat{g}_\alpha(t)$ in global coordinates instead of local ones;
see \cite{Fr13} for more details on this kind of representation of metric tensors. Using the global coordinates system of the ambient space
will make the spatial discretization of the weak formulation much easier. 

\begin{lem}
\label{Lemma_weak_formulation_on_M}
Let $\alpha \in (0, \infty)$.
Suppose $\M$ is a smooth, $n$-dimensional, closed, connected hypersurface in $\mathbb{R}^{n+1}$.
Furthermore, let $h$ be the metric on $\M \subset \mathbb{R}^{n+1}$ induced by the Euclidean metric $\mathfrak{e}$ of the ambient space, that is
$$
	h_{ij} := \frac{\partial \mathcal{C}^{-1}_1}{\partial \theta^i} 
				\cdot \frac{\partial \mathcal{C}^{-1}_1}{\partial \theta^j},
	\quad
	\textnormal{and}
	\quad
	(h^{ij}) := (h_{ij})^{-1}_{i,j = 1, \ldots, n},			
$$
where $\mathcal{C}_1$ is a local coordinate chart of $\M$.
The mean curvature-DeTurck flow (\ref{pull_back_of_MCF}) then satisfies the following weak formulation
\begin{align}
0 &= \alpha \int_\M \hat{x}_{\alpha t} \cdot \chi \sqrt{\det \hat{G}_\alpha} do_h
	 + (1- \alpha) \int_\M 
	(\hat{x}_{\alpha t} \cdot (\nu \circ \hat{x}_\alpha))
	( (\nu \circ \hat{x}_\alpha) \cdot \chi) \sqrt{\det \hat{G}_\alpha} do_{h}	
\nonumber
\\	
&  \quad + \int_{\M} \hat{G}_\alpha^{-1} \nabla_\M \hat{x}_\alpha : \nabla_\M \chi
 \sqrt{\det \hat{G}_\alpha} do_{h}
  +	\int_\M \big( (\nabla_\M \hat{x}_\alpha) v_\alpha \big) \cdot \chi 
  \sqrt{\det \hat{G}_\alpha} do_{h},
\label{weak_formula_MCF_reference_M_1}  
  \\
0 &= \int_\M v_\alpha \cdot \xi \sqrt{ \det \hat{G}_\alpha} do_{h}
  + \int_\M  \hat{G}_\alpha^{-1} \nabla_\M id : \nabla_\M \xi 
  \sqrt{\det \hat{G}_\alpha} do_{h},
\label{weak_formula_MCF_reference_M_2}    
\end{align}
for all $\chi, \xi \in H^{1,2}(\M,\mathbb{R}^{n+1})$ and $t \in (0,T)$.
\end{lem}
\begin{proof}
A short calculation shows that
\begin{align*}
	\hat{g}^{ij}_\alpha \Gamma(h)^k_{ij}
	= \hat{g}^{ij}_\alpha h^{kl} \left(
		\frac{\partial h_{li}}{\partial \theta^j}
		- \frac{1}{2} \frac{\partial h_{ij}}{\partial \theta^l}
		\right)
	= \hat{g}_\alpha^{ij} h^{kl} \frac{\partial \mathcal{C}_1^{-1}}{\partial \theta^l}
	\cdot \frac{\partial^2 \mathcal{C}^{-1}_1}{\partial \theta^i \partial \theta^j}	.
\end{align*}
From the decomposition (\ref{decomposition}) we infer that
\begin{equation}
	( P \circ \hat{X}_\alpha ) 
	\frac{\partial^2 \hat{X}_\alpha}{\partial \theta^i \partial \theta^j}
	= \Gamma(\hat{g}_\alpha)^k_{ij} \frac{\partial \hat{X}_\alpha}{\partial \theta^k} , 
	\label{formula_Christoffel_symbols}
\end{equation}
and hence,
\begin{align*}
	( P \circ \hat{X}_\alpha ) 
	\hat{g}_\alpha^{ij} 
	\frac{\partial^2 \hat{X}_\alpha}{\partial \theta^i \partial \theta^j}
	= \hat{g}_\alpha^{ij} 
	\Gamma(\hat{g}_\alpha)^m_{ij} \delta^k_m
		\frac{\partial \hat{X}_\alpha}{\partial \theta^k}
	= \hat{g}_\alpha^{ij} 
	\Gamma(\hat{g}_\alpha)^m_{ij} h^{kl} 
		\frac{\partial \mathcal{C}^{-1}_1}{\partial \theta^l} \cdot
		\frac{\partial \mathcal{C}^{-1}_1}{\partial \theta^m}
		\frac{\partial \hat{X}_\alpha}{\partial \theta^k}.
\end{align*}
Altogether, it follows that
\begin{align*}
	(P \circ \hat{X}_\alpha) (\tr_{\hat{g}_\alpha} (\nabla^h \nabla \hat{x}_\alpha))
	\circ \mathcal{C}_1^{-1}
	&= (P \circ \hat{X}_\alpha) \hat{g}_\alpha^{ij}
	\left( 
		\frac{\partial^2 \hat{X}_\alpha}{\partial \theta^i \partial \theta^j}
		- \Gamma(h)^k_{ij} \frac{\partial \hat{X}_\alpha}{\partial \theta^k}	
	\right)
	\\	
	&= \hat{g}_\alpha^{ij} \frac{\partial \hat{X}_\alpha}{\partial \theta^k} h^{kl}
		\frac{\partial \mathcal{C}^{-1}_1}{\partial \theta^l} \cdot
		\left(
			\Gamma(\hat{g}_\alpha)^m_{ij} 
			\frac{\partial \mathcal{C}^{-1}_1}{\partial \theta^m}
			- \frac{\partial^2 \mathcal{C}^{-1}_1}{\partial \theta^i \partial \theta^j}
		\right)	
	\\
	&= 	- \frac{\partial \hat{X}_\alpha}{\partial \theta^k} h^{kl}
		\frac{\partial \mathcal{C}^{-1}_1}{\partial \theta^l} \cdot
		(\Delta_{\hat{g}_\alpha} id) \circ \mathcal{C}^{-1}_1.
\end{align*}
The tangential gradient of a function $f$ on $\M$
satisfies the formula
\begin{align}
	(\nabla_\M f) \circ \mathcal{C}^{-1}_1 := h^{ij} 
		\frac{\partial F}{\partial \theta^i} 
		\frac{\partial \mathcal{C}^{-1}_1}{\partial \theta^j},	
\label{tangential_gradient_coords}		
\end{align}
where $F$ denotes $F:= f \circ \mathcal{C}^{-1}_1$, see \cite{DE13} for more details.
Please note that a similar identity also holds for the tangential gradient on $\Gamma(t)$
and the Riemannian metric $\hat{g}_\alpha(t)$.
We can now deduce that
$$
	((P \circ \hat{x}_\alpha) \tr_{\hat{g}_\alpha} (\nabla^h \nabla \hat{x}_\alpha))
	\circ \mathcal{C}_1^{-1}
	= - (\nabla_\M \hat{x}_\alpha) \circ \mathcal{C}^{-1}_1 
	(\Delta_{\hat{g}_\alpha} id) \circ \mathcal{C}^{-1}_1,
$$
and thus,
$$
	\left( \alpha \unit 
	+ (1 - \alpha) (\nu \circ \hat{x}_\alpha) \otimes (\nu \circ \hat{x}_\alpha) \right) 
	\frac{\partial}{\partial t}\hat{x}_\alpha
	= \Delta_{\hat{g}_\alpha} \hat{x}_\alpha
	- (\nabla_\M \hat{x}_\alpha) (\Delta_{\hat{g}_\alpha} id).
$$
We introduce the vector field $v_\alpha := \Delta_{\hat{g}_\alpha} id$ on $\M$
and get
\begin{equation}
\label{reparam_MCF_2}
	\left( \alpha \unit 
	+ (1 - \alpha) (\nu \circ \hat{x}_\alpha) \otimes (\nu \circ \hat{x}_\alpha) \right) 
	\frac{\partial}{\partial t}\hat{x}_\alpha
	= \Delta_{\hat{g}_\alpha} \hat{x}_\alpha
	- (\nabla_\M \hat{x}_\alpha) v_\alpha.
\end{equation}
Henceforward, we write $\hat{x}_{\alpha t}$ instead of 
$\frac{\partial}{\partial t} \hat{x}_\alpha$. 
We now multiply (\ref{reparam_MCF_2}) and the definition 
$v_\alpha := \Delta_{\hat{g}_\alpha} id$ on $\M$ 
by test functions $\chi, \xi \in H^{1,2}(\M,\mathbb{R}^{n+1})$
and integrate with respect to the volume form $do_{\hat{g}_\alpha}$.
Integration by parts directly yields
\begin{align}
0 &= \alpha \int_\M \hat{x}_{\alpha t} \cdot \chi do_{\hat{g}_\alpha} 
	+ (1- \alpha) \int_\M 
	(\hat{x}_{\alpha t} \cdot (\nu \circ \hat{x}_\alpha))
	( (\nu \circ \hat{x}_\alpha) \cdot \chi ) do_{\hat{g}_\alpha} 	
\nonumber	
\\	
&  \quad + \int_{\M} \hat{G}_\alpha^{-1} \nabla_\M \hat{x}_\alpha : \nabla_\M \chi do_{\hat{g}_\alpha}
  +	\int_\M \big( (\nabla_\M \hat{x}_\alpha) v_\alpha \big)
  				\cdot \chi  do_{\hat{g}_\alpha},
\label{weak_formulation_reparam_MCF_1}  				
  \\
0 &= \int_\M v_\alpha \cdot \xi do_{\hat{g}_\alpha}
  + \int_\M \hat{G}_\alpha^{-1} \nabla_\M id : \nabla_\M \xi do_{\hat{g}_\alpha},
\label{weak_formulation_reparam_MCF_2}  
\end{align}
where we have made use of the following identity
\begin{equation}
\label{identity_1}
	\hat{g}_\alpha^{ij} \frac{\partial F}{\partial \theta^i}
	 \frac{\partial W}{\partial \theta^j}
	 = (\hat{G}_\alpha^{-1} \nabla_\M f \cdot \nabla_\M w) \circ \mathcal{C}_1^{-1}.
\end{equation}
Here, $f,w: \M \rightarrow \mathbb{R}$ denote differentiable functions on $\M$,
and $F,W$ are given by $F := f \circ \mathcal{C}^{-1}_1$ and 
$W := w \circ \mathcal{C}^{-1}_1$, respectively. 
From (\ref{defi_G}) and (\ref{tangential_gradient_coords}) we immediately see that
\begin{align}
	(\hat{G}_\alpha)_{\beta \gamma} \circ \mathcal{C}^{-1}_1
	&= \frac{\partial (\mathcal{C}^{-1}_{1})_\beta}{\partial \theta^i} h^{ij} 
		\frac{\partial \hat{X}_\alpha}{\partial \theta^j}
		\cdot \frac{\partial \hat{X}_\alpha}{\partial \theta^k} h^{kl} 
		\frac{\partial (\mathcal{C}^{-1}_{1})_\gamma}{\partial \theta^l} 
	+ (\mu_\beta  \mu_\gamma) \circ \mathcal{C}^{-1}_1 	
	\nonumber
	\\
	&= \frac{\partial (\mathcal{C}^{-1}_{1})_\beta}{\partial \theta^i} h^{ij} 
		\hat{g}_{\alpha jk} h^{kl} 
		\frac{\partial (\mathcal{C}^{-1}_{1})_\gamma}{\partial \theta^l} 
	+ (\mu_\beta  \mu_\gamma) \circ \mathcal{C}^{-1}_1 	
	\label{G_beta_gamma}
\end{align}
and hence,
\begin{equation}
	(\hat{G}^{-1}_\alpha)^{\beta \gamma} \circ \mathcal{C}^{-1}_1
	= \frac{\partial (\mathcal{C}^{-1}_{1})^\beta}{\partial \theta^i} 
		\hat{g}_{\alpha}^{ij} 
		\frac{\partial (\mathcal{C}^{-1}_{1})^\gamma}{\partial \theta^j} 
	+ (\mu_\beta  \mu_\gamma) \circ \mathcal{C}^{-1}_1. 	
	\label{G_inv_beta_gamma}
\end{equation}
The identity (\ref{identity_1}) can then be obtained as follows
\begin{align*}
 	(\hat{G}_\alpha^{-1} \nabla_\M f \cdot \nabla_\M w) \circ \mathcal{C}_1^{-1}
 	&= (\hat{G}^{-1}_\alpha)^{\beta \gamma} \circ \mathcal{C}^{-1}_1
 	\frac{\partial (\mathcal{C}^{-1}_{1})_\beta}{\partial \theta^k} 
 	h^{kl} \frac{\partial F}{\partial \theta^l}
 	\frac{\partial (\mathcal{C}^{-1}_{1})_\gamma}{\partial \theta^s} 
 	h^{st} \frac{\partial W}{\partial \theta^t}
 	\\
 	&= 	\hat{g}_{\alpha}^{ij} 
		\frac{\partial (\mathcal{C}^{-1}_{1})}{\partial \theta^i} \cdot 
		\frac{\partial (\mathcal{C}^{-1}_{1})}{\partial \theta^k} 
 	h^{kl} \frac{\partial F}{\partial \theta^l}
 	\frac{\partial (\mathcal{C}^{-1}_{1})}{\partial \theta^j} \cdot
 	\frac{\partial (\mathcal{C}^{-1}_{1})}{\partial \theta^s} 
 	h^{st} \frac{\partial W}{\partial \theta^t}
 	\\
 	&= \hat{g}_{\alpha}^{ij} h_{ik}
 	h^{kl} \frac{\partial F}{\partial \theta^l} h_{js}
 	h^{st} \frac{\partial W}{\partial \theta^t}
 	\\
 	&= \hat{g}_\alpha^{ij} \frac{\partial F}{\partial \theta^i}
	 \frac{\partial W}{\partial \theta^j}.
\end{align*}
Please be aware that the expressions
$\hat{G}_\alpha^{-1} \nabla_\M \hat{x}_\alpha : \nabla_\M \chi$
in (\ref{weak_formulation_reparam_MCF_1}) 
and $\hat{G}_\alpha^{-1} \nabla_\M id : \nabla_\M \xi$
in (\ref{weak_formulation_reparam_MCF_2})
are meant to be the following sums
\begin{align*}
	&\hat{G}_\alpha^{-1} \nabla_\M \hat{x}_\alpha : \nabla_\M \chi
	= (\hat{G}_\alpha^{-1})^{\beta \gamma}
		\D \beta \hat{x}_\alpha \cdot \D \gamma \chi,
\\		
	&\hat{G}_\alpha^{-1} \nabla_\M id : \nabla_\M \xi
	= (\hat{G}_\alpha^{-1})^{\beta \gamma}
		\D \beta id \cdot \D \gamma \xi.		
\end{align*}
It is not difficult to show that the volume form $do_{\hat{g}_\alpha}$ satisfy the following identity
\begin{equation}
	do_{\hat{g}_{\alpha}} = \sqrt{\det \hat{G}_\alpha} do_h.
	\label{identity_of_volume_forms}
\end{equation}
Without loss of generality we can suppose that $\mu = (0, \ldots, 0, 1)^T$.
It then follows that
$$
	\mathbb{R}^{(n+1) \times n} \ni \left(
		\frac{\partial (\mathcal{C}^{-1}_1)}{\partial \theta^i}
	\right)_{i= 1, \ldots, n}
	=
	\left(
	\begin{array}{c}
		S  
		\\
		0  
	\end{array}
	\right)
$$
for some $S \in \mathbb{R}^{n \times n}$, and together with (\ref{G_beta_gamma}),
\begin{equation}
	\det \hat{G}_\alpha \circ \mathcal{C}^{-1}_1 = \det S \det (h_{ij})^{-1} \det (\hat{g}_{\alpha jk})
						\det (h_{kl})^{-1} \det S
						= \frac{\det  (\hat{g}_{\alpha jk})}{\det (h_{jk})}, 
	\label{det_G}					
\end{equation}
where we have made use of $(\det S)^2 = \det (h_{ij})$.
\end{proof}

\subsection{Weak formulation on the moving hypersurface $\Gamma(t)$}
The weak formulation in Lemma \ref{Lemma_weak_formulation_on_M} could, in principle, be used for developing an algorithm for the computation of the
mean curvature-DeTurck flow.
However, we will not follow this route here, since numerical schemes based on surface finite elements are usually formulated on the moving
hypersurface $\Gamma(t) := \hat{x}_\alpha(\M, t) \subset \mathbb{R}^{n+1}$ 
rather than on the reference manifold. 
In this section we therefore reformulate the problem on the moving hypersurface $\Gamma(t)$.

In the following we derive a weak formulation for the map
$u: \Gamma(t) \times [0,T) \rightarrow \mathbb{R}^{n+1}$
defined by
$u:= \hat{x}_\alpha \circ \hat{x}_\alpha^{-1}$.
Obviously, we have $u = id_{| \Gamma(t)}$. We define the material derivative of a
differentiable function $f$ on $\Gamma(t)$ by
$$
	(\partial^\bullet f) \circ \hat{x}_\alpha 
	= \frac{\partial}{\partial t} (f \circ \hat{x}_\alpha).
$$ 
The material derivative of $u$ is thus given by
\begin{equation}
	\partial^\bullet u = \frac{ \partial \hat{x}_\alpha}{\partial t} 
	\circ \hat{x}^{-1}_\alpha.
	\label{material_derivative_u}
\end{equation}
We want to recall that in this section the reference manifold $\M$ is assumed to be an 
$n$-dimensional hypersurface in $\mathbb{R}^{n+1}$.
Similar to the definition of the map $\hat{G}_\alpha(t)$ in (\ref{defi_G}), 
we next introduce the global representation
$\hat{H}_\alpha : \bigcup_{t \in [0,T]} \Gamma(t) \times \{t\} \rightarrow \mathbb{R}^{(n+1) \times (n+1)}$ 
of the Riemannian metric 
$\hat{h}_\alpha(t) := (\hat{x}_\alpha^{-1}(t))^\ast \mathfrak{e}$ on $\Gamma(t)$.
The map $\hat{H}_\alpha(t)$ is defined by 
\begin{equation}
	\label{defi_H}
	\hat{H}_\alpha(t) :=
	(\nabla_{\Gamma(t)} \hat{x}_\alpha^{-1}(t))^T \nabla_{\Gamma(t)} \hat{x}^{-1}_\alpha(t)
	+ \nu(t) \otimes \nu(t).
\end{equation}
Here, $\nu(t)$ is a unit normal field to $\Gamma(t)$. That $\hat{H}_\alpha(t)$
is indeed a global representation of the metric $\hat{h}_\alpha(t)$ is shown in (\ref{H_is_representation_of_h}). 
Using the material derivative, the weak formulation in Lemma \ref{Lemma_weak_formulation_on_M} can be lifted onto the moving hypersurface $\Gamma(t)$.
This is summarized in the following statement.
\begin{thm}
\label{Theorem_weak_formulation_on_Gamma}
Under the same assumptions as in Lemma \ref{Lemma_weak_formulation_on_M}, the identity map $u = id_{\Gamma(t)}$ on $\Gamma(t)$ satisfies 
\begin{align*}
0 &= \alpha \int_{\Gamma(t)} \partial^\bullet u \cdot \eta d\sigma
	+ (1- \alpha) \int_{\Gamma(t)}
		(\partial^\bullet u \cdot \nu)
		( \nu \cdot \eta) d\sigma 	
\\	
&  \quad + \int_{\Gamma(t)} 
		  \nabla_{\Gamma(t)} u : \nabla_{\Gamma(t)} \eta d\sigma
  +	\int_{\Gamma(t)} w_\alpha \cdot
  			 \big( \nabla_{\Gamma(t)} \hat{x}_\alpha^{-1}
  				\hat{H}^{-1}_\alpha \eta \big) d\sigma,
  \\
0 &= \int_{\Gamma(t)} w_\alpha \cdot \zeta d\sigma
  + \int_{\Gamma(t)} 
  		\nabla_{\Gamma(t)} \hat{x}_\alpha^{-1} : \nabla_{\Gamma(t)} \zeta d\sigma,
\end{align*}		
for all $\eta, \zeta \in H^{1,2}(\Gamma(t),\mathbb{R}^{n+1})$ and $t \in [0,T)$.
\end{thm}
\begin{proof}
In order to be able to distinguish between the components of the tangential gradients
on $\M$ and on $\Gamma(t)$, we introduce the following notation
\begin{align*}
	 \left(  
		\begin{array}{c}
			\D 1 f
			\\
				\vdots
			\\
			\D {n+1} f
		\end{array}			
	 \right)
	 := \nabla_{\M} f,
	 \quad
	 \textnormal{and}
	 \quad
	 \left(  
		\begin{array}{c}
			\D 1' f
			\\
				\vdots
			\\
			\D {n+1}' f
		\end{array}			
	 \right)
	 := \nabla_{\Gamma(t)} f,
\end{align*}
where $f$ is a differentiable function on $\M$, or on $\Gamma(t)$,
respectively. 
For differentiable functions 
$f,w: \Gamma(t) \times [0,T) \rightarrow \mathbb{R}^{n+1}$,
we obtain the identity
\begin{equation}
	\D \beta (f \circ \hat{x}_\alpha) 
	= (\D \kappa' f) \circ \hat{x}_\alpha \D \beta (\hat{x}_\alpha)^\kappa,
	\label{chain_rule_tangential_gradients}
\end{equation}
and hence,
\begin{align*}
	(\hat{G}_\alpha^{-1})^{\beta \gamma} \D \beta (f \circ \hat{x}_\alpha) 
									\D \gamma (w \circ \hat{x}_\alpha) 
	&= (\hat{G}_\alpha^{-1})^{\beta \gamma}
		(\D \kappa' f) \circ \hat{x}_\alpha \D \beta (\hat{x}_\alpha)^\kappa
		(\D \iota' w) \circ \hat{x}_\alpha \D \gamma (\hat{x}_\alpha)^\iota
	\\
	&= (\hat{G}_\alpha^{-1})^{\beta \gamma} \D \beta (\hat{x}_\alpha)^\kappa
		\D \gamma (\hat{x}_\alpha)^\iota
		(\D \kappa' f \D \iota' w) \circ \hat{x}_\alpha 
	\\
	&= P^{\kappa \iota} \circ \hat{x}_\alpha 
		(\D \kappa' f \D \iota' w) \circ \hat{x}_\alpha 	
	= 	(\D \kappa' f \D \kappa' w) \circ \hat{x}_\alpha, 	
\end{align*}
where we have made use of the identity 
$(\hat{G}_\alpha^{-1})^{\beta \gamma} \D \beta (\hat{x}_\alpha)^\kappa
		\D \gamma (\hat{x}_\alpha)^\iota = P^{\kappa \iota} \circ \hat{x}_\alpha$.
Applying this result, it is easy to see that
\begin{align}
	& \hat{G}_\alpha^{-1} \nabla_\M \hat{x}_\alpha : \nabla_\M \chi
	= \left(
			\nabla_{\Gamma(t)} u : 
						\nabla_{\Gamma(t)} (\chi \circ \hat{x}^{-1}_\alpha)
		\right) \circ \hat{x}_\alpha,
 \label{result_1}		
\\	
	& \hat{G}_\alpha^{-1} \nabla_\M id : \nabla_\M \xi
	= \left(
			\nabla_{\Gamma(t)} \hat{x}_\alpha^{-1} : 
						\nabla_{\Gamma(t)} (\xi \circ \hat{x}^{-1}_\alpha)
		\right) \circ \hat{x}_\alpha.
 \label{result_2}		
\end{align}	
As in (\ref{G_beta_gamma}) and (\ref{G_inv_beta_gamma}), we obtain
the following relations between the components of the map $\hat{H}_{\alpha}(t)$
and the components $\hat{h}_{\alpha ij}(t)$ of the Riemannian metric 
$\hat{h}_\alpha(t) := (\hat{x}_\alpha^{-1}(t))^\ast \mathfrak{e}$ on $\Gamma(t)$
\begin{align}
& (\hat{H}_\alpha)_{\beta \gamma} \circ \hat{X}_\alpha
	= \frac{\partial (\hat{X}_\alpha)_\beta}{\partial \theta^i} \hat{g}_\alpha^{ij} 
		\hat{h}_{\alpha jk} \hat{g}^{kl}_\alpha 
		\frac{\partial (\hat{X}_\alpha)_\gamma}{\partial \theta^l} 
	+ (\nu_\beta  \nu_\gamma) \circ \hat{X}_\alpha,
\label{H_is_representation_of_h}	
\\
& (\hat{H}^{-1}_\alpha)^{\beta \gamma} \circ \hat{X}_\alpha
	= \frac{\partial (\hat{X}_\alpha)^\beta}{\partial \theta^i} 
		\hat{h}_{\alpha}^{ij} 
		\frac{\partial (\hat{X}_\alpha)^\gamma}{\partial \theta^j} 
	+ (\nu_\beta  \nu_\gamma) \circ \hat{X}_\alpha,
\end{align}
where
\begin{align*}
	& \hat{h}_{\alpha ij} = \frac{\partial (\hat{x}^{-1}_\alpha \circ
									 \hat{X}_\alpha)}{\partial \theta^i}
		\cdot \frac{\partial (\hat{x}^{-1}_\alpha \circ
									 \hat{X}_\alpha)}{\partial \theta^j},
		\quad
		\textnormal{and}
		\quad
		(\hat{h}^{ij}_\alpha)_{i,j=1, \ldots, n}
		:= (\hat{h}_{\alpha ij})^{-1}_{i,j=1, \ldots, n},
	\\
	& \hat{g}_{\alpha ij} = \frac{\partial \hat{X}_\alpha}{\partial \theta^i}
		\cdot \frac{\partial \hat{X}_\alpha}{\partial \theta^j},
		\quad
		\textnormal{and}
		\quad
		(\hat{g}^{ij}_\alpha)_{i,j=1, \ldots, n}
		:= (\hat{g}_{\alpha ij})^{-1}_{i,j=1, \ldots, n}.
\end{align*}
Please note that $\hat{h}_\alpha(t)$ is a metric on $\Gamma(t)$,
whereas $h(t)$ is a metric on $\M$.
However, since $\hat{X}_\alpha := \hat{x}_\alpha \circ \mathcal{C}^{-1}_1$,
we have 
$$
	\hat{h}_{\alpha ij} = \frac{\partial \mathcal{C}^{-1}_1}{\partial \theta^i}
		\cdot \frac{\partial \mathcal{C}^{-1}_1}{\partial \theta^j} = h_{ij}.
$$
It follows that
\begin{align}
& (\hat{H}_\alpha)_{\beta \gamma} \circ \hat{X}_\alpha
	= \frac{\partial (\hat{X}_\alpha)_\beta}{\partial \theta^i} \hat{g}_\alpha^{ij} 
		h_{jk} \hat{g}^{kl}_\alpha 
		\frac{\partial (\hat{X}_\alpha)_\gamma}{\partial \theta^l} 
	+ (\nu_\beta  \nu_\gamma) \circ \hat{X}_\alpha,
\label{H_in_local_coordinates}	
\\
& (\hat{H}^{-1}_\alpha)^{\beta \gamma} \circ \hat{X}_\alpha
	= \frac{\partial (\hat{X}_\alpha)^\beta}{\partial \theta^i} 
		h^{ij} 
		\frac{\partial (\hat{X}_\alpha)^\gamma}{\partial \theta^j} 
	+ (\nu_\beta  \nu_\gamma) \circ \hat{X}_\alpha.
\label{H_inverse_in_local_coordinates}	
\end{align}
Using the latter identity and (\ref{tangential_gradient_coords}), we deduce that
\begin{align}
	( \nabla_\M (\hat{x}_\alpha)^\beta ) \circ \mathcal{C}^{-1}_1
	 &= \frac{\partial (\hat{X}_\alpha)^\beta}{\partial \theta^i}  h^{ij} 
	 	\frac{\partial \mathcal{C}^{-1}_1}{\partial \theta^j}
	 =  \frac{\partial (\hat{X}_\alpha)^\beta}{\partial \theta^i}  h^{ij} 
	 	\frac{\partial (\hat{x}^{-1}_\alpha \circ \hat{X}_\alpha)}
	 									{\partial \theta^j}
	\nonumber	 
	 \\
	 &= 	\frac{\partial (\hat{X}_\alpha)^\beta}{\partial \theta^i}  h^{ij} 
	 	\hat{g}_{\alpha jl} \hat{g}_\alpha^{lk}
	 	\frac{\partial (\hat{x}^{-1}_\alpha \circ \hat{X}_\alpha)}
	 									{\partial \theta^k}
	\nonumber	 
	 \\
	 &= 	\frac{\partial (\hat{X}_\alpha)^\beta}{\partial \theta^i}  h^{ij} 
	 	\frac{\partial \hat{X}_\alpha}{\partial \theta^j}
	 	\cdot
	 	\frac{\partial \hat{X}_\alpha}{\partial \theta^l}
	 	\hat{g}_\alpha^{lk}
	 	\frac{\partial (\hat{x}^{-1}_\alpha \circ \hat{X}_\alpha)}
	 									{\partial \theta^k}
	\nonumber	 
	 \\
	 &= 	\frac{\partial (\hat{X}_\alpha)^\beta}{\partial \theta^i}  h^{ij} 
	 	\frac{\partial (\hat{X}_\alpha)^\gamma}{\partial \theta^j}
	 	(\D \gamma' \hat{x}^{-1}_\alpha) \circ \hat{X}_\alpha		
	\nonumber	 
	 \\
	 &=	(\hat{H}_\alpha^{-1})^{\beta \gamma} \circ \hat{X}_\alpha 
	 	(\D \gamma' \hat{x}^{-1}_\alpha) \circ \hat{X}_\alpha,
	 \nonumber	
\end{align}
and hence,
\begin{equation}
	 \nabla_\M (\hat{x}_\alpha)^\beta = 
		\left(	 		
	 		(\hat{H}_\alpha^{-1})^{\beta \gamma} 
	 		\D \gamma' \hat{x}^{-1}_\alpha
	 	\right) \circ \hat{x}_\alpha.	
	 \label{result_3}	
\end{equation}		
The integral on $\M$ with respect to the volume form $do_{\hat{g}_\alpha}$
is transformed into an integral on $\Gamma(t)$ with respect to the volume form $d\sigma$ induced by the Hausdorff measure
in the following way
\begin{align*}
	 \int_\M f do_{\hat{g}_\alpha}
 &= \int_{\Omega} f \circ \mathcal{C}^{-1}_1 
 \sqrt{\det (\hat{g}_{\alpha ij})} d^n \theta
\\
 &= \int_{\Omega} (f \circ \hat{x}^{-1}_\alpha) \circ \hat{X}_\alpha
   \sqrt{
	\det \left(\frac{\partial \hat{X}_\alpha}{\partial \theta^i}
		\cdot \frac{\partial \hat{X}_\alpha}{\partial \theta^j} \right)   
   }
   d^n \theta
\\  
  &= \int_{\Gamma(t)} f \circ \hat{x}_\alpha^{-1} d\sigma, 
\end{align*}
where $f: \M \rightarrow \mathbb{R}$ denotes an integrable function
with $supp f \subset U$ and 
$\mathcal{C}_1: U \subset \M \rightarrow \Omega \subset \mathbb{R}^{n}$ is
a local coordinate chart of $\M$. This result can be easily generalized using a partition of unity.
From (\ref{weak_formulation_reparam_MCF_1}) and (\ref{weak_formulation_reparam_MCF_2}) we then infer that
\begin{align*}
0 &= \alpha \int_{\Gamma(t)} (\hat{x}_{\alpha t} \cdot \chi) 
	\circ \hat{x}_\alpha^{-1} d\sigma
	+ (1- \alpha) \int_{\Gamma(t)}
	\big(	
	(\hat{x}_{\alpha t} \cdot (\nu \circ \hat{x}_\alpha))
	( (\nu \circ \hat{x}_\alpha) \cdot \chi ) \big) 
		\circ \hat{x}_\alpha^{-1} d\sigma 	
\\	
&  \quad + \int_{\Gamma(t)} 
	\left(
		\hat{G}_\alpha^{-1} \nabla_\M \hat{x}_\alpha : \nabla_\M \chi 
	\right) \circ \hat{x}_\alpha^{-1} d\sigma
  +	\int_{\Gamma(t)} 
  		\Big(
  			 \big( (\nabla_\M \hat{x}_\alpha) v_\alpha \big) \cdot \chi 
		\Big) \circ \hat{x}_\alpha^{-1} d\sigma,
  \\
0 &= \int_{\Gamma(t)} ( v_\alpha \cdot \xi ) \circ \hat{x}_\alpha^{-1} d\sigma
  + \int_{\Gamma(t)} 
  	\left(
  		\hat{G}_\alpha^{-1} \nabla_\M id : \nabla_\M \xi 
  	\right) \circ \hat{x}_\alpha^{-1} d\sigma.
\end{align*}	
Using the notations
$w_\alpha := v_\alpha \circ \hat{x}^{-1}_\alpha$,
$\eta := \chi \circ \hat{x}^{-1}_\alpha$ and 
$\zeta := \xi \circ \hat{x}^{-1}_\alpha$,
and applying the identities (\ref{material_derivative_u}), 
(\ref{result_1}), (\ref{result_2}) and (\ref{result_3}) finally proves the claim. 
\end{proof}

\subsection{Numerical scheme on moving hypersurfaces}
We will now discretize the weak formulation which we received in Theorem \ref{Theorem_weak_formulation_on_Gamma} in space and time. 
The reference hypersurface $\M \subset \mathbb{R}^{n+1}$ 
is supposed to be approximated by a piecewise linear,
polyhedral hypersurface 
$$
	\M_h := \bigcup_{T \in \mathcal{T}_h } T \subset \mathbb{R}^{n+1}.
$$	 
Here, $\mathcal{T}_h$ is an admissible triangulation consisting of 
non-degenerate $n$-dimensional simplices in $\mathbb{R}^{n+1}$. 
The finite element space $\mathcal{S}(\M_h)$
is the set of piecewise linear, continuous functions
\begin{equation}
	\mathcal{S}(\M_h) := 
	\left\{ 
		\chi_h \in C^0(\M_h) ~|~ \chi_{h|T} ~ \textnormal{is a linear polynomial
		for all} ~ T \in \mathcal{T}_h
	\right\}.
	\label{finite_element_space}
\end{equation}
$\mathcal{S}(\M_h)$ is a linear space of dimension $N$, where $N$
is the number of the vertices $p_j \in \M_h$, $j=1, \ldots , N$,
of the triangulation. It is spanned by the basis functions 
$\phi_i \in \mathcal{S}(\M_h)$ defined by $\phi_i (p_j) = \delta_{ij}$,
$\forall i,j=1, \ldots,N$.

For the time discretization we introduce the notation 
$f^m = f(\cdot, m \tau)$ for the discrete time levels 
$\{m \tau ~|~ m=0, \ldots, M_\tau \in \mathbb{N} \}$
with time step size $\tau >0$ and $M_\tau \tau < T$.
The approximation of the moving hypersurface
$\Gamma(t)$ at time $m\tau$ will be denoted by
$$
	\Gamma_h^m := \bigcup_{T \in \mathcal{T}_h^m} T \subset \mathbb{R}^{n+1},
$$ 
where $\mathcal{T}^m_h$ is an admissible triangulation of $n$-simplices in 
$\mathbb{R}^{n+1}$. 
The finite element spaces $\mathcal{S}(\Gamma^m_h)$
are defined in accordance to (\ref{finite_element_space}).
The tangential gradient $\nabla_{\Gamma_{h}^m}$ on $\Gamma_h^m$ is defined
piecewise on each $n$-simplex $T \in \mathcal{T}_h^m$.
Taking the geometric quantities from the previous time step,
it is possible to linearize the problem in each time step
and to obtain a semi-implicit scheme 
in the spirit of \cite{Dz91}.
In order to make the scheme more implicit, we observe that $\nabla_\Gamma u = P$
and write 
\begin{align*}
	w_\alpha \cdot
  			 \big( \nabla_{\Gamma(t)} \hat{x}_\alpha^{-1}
  				\hat{H}^{-1}_\alpha \eta \big)
  	&= w_\alpha \cdot
  			 \big( \nabla_{\Gamma(t)} \hat{x}_\alpha^{-1}
  				\hat{H}^{-1}_\alpha P \eta \big)		
\\ 
  	&= P 	\hat{H}^{-1}_\alpha	(\nabla_{\Gamma(t)} \hat{x}_\alpha^{-1})^T w_\alpha
  		\cdot \eta
\\
  	&= (\nabla_\Gamma u) 
  		\hat{H}^{-1}_\alpha	(\nabla_{\Gamma(t)} \hat{x}_\alpha^{-1})^T w_\alpha
  		\cdot \eta.			
\end{align*}
This leads to the following scheme.
	
\begin{alg}
\label{algo_moving_hypersurface}
Let $\alpha \in (0,\infty)$.
For a given initial polyhedral hypersurface $\Gamma_h^0 = \hat{x}_h^0 (\M_h)$ 
with $\hat{x}_h^0 \in \mathcal{S}(\M_h)^{n+1}$, 
set $y_h^0 := (\hat{x}_h^0)^{-1} \in \mathcal{S}(\Gamma_h^0)^{n+1}$ 
and determine for $m=0, \ldots, M_\tau -1$
solutions $u_h^{m+1} \in \mathcal{S}(\Gamma_h^m)^{n+1}$
and $w_h^m \in \mathcal{S}(\Gamma_h^m)^{n+1}$ such that
\begin{align*}
& \frac{\alpha}{\tau} \int_{\Gamma^m_h} u^{m+1}_h \cdot \eta_h d\sigma
	+ \frac{1- \alpha}{\tau} \int_{\Gamma^m_h}
		(u^{m+1}_h \cdot \nu^m_h)
		( \nu^m_h \cdot \eta_h ) d\sigma	
	 + \int_{\Gamma^m_h} 
		  \nabla_{\Gamma^m_h} u^{m+1}_h : \nabla_{\Gamma^m_h} \eta_h d\sigma
\\	
&  \ \ 	  
  +	\int_{\Gamma^m_h} 
  		\nabla_{\Gamma^m_h} u^{m+1}_h (\hat{H}^m_h)^{-1} 
  				(\nabla_{\Gamma^m_h} y^m_h)^T w^m_h \cdot \eta_h d\sigma
  =  \int_{\Gamma^m_h} \frac{\alpha}{\tau} (\tilde{u}^m_h \cdot \eta_h) 
	+ \frac{1- \alpha}{\tau}
		(\tilde{u}^m_h \cdot \nu^m_h)
		( \nu^m_h \cdot \eta_h ) d\sigma	,
  \\
& \int_{\Gamma^m_h} w^m_h \cdot \zeta_h d\sigma
  + \int_{\Gamma^m_h} 
  		\nabla_{\Gamma^m_h} y^m_h : \nabla_{\Gamma^m_h} \zeta_h d\sigma
  = 0,
\end{align*}	
for all $ \eta_h, \zeta_h \in \mathcal{S}(\Gamma_h^m)^{n+1}$, where
$\nu_h^m$ is a unit normal to $\Gamma_h^m$, and
\begin{align*}
	& \tilde{u}^{m}_h := id_{| \Gamma^{m}_h},
\\	
	& \hat{H}^{m}_h 
		:= (\nabla_{\Gamma_h^{m}} y^{m}_h)^T \nabla_{\Gamma^{m}_h} y^{m}_h
			+ \nu^{m}_h \otimes \nu^{m}_h.
\end{align*}
The hypersurface $\Gamma_h^{m+1}$ is defined by
$$
	\Gamma^{m+1}_h := u^{m+1}_h(\Gamma^m_h),
$$
and $y^{m+1}_h \in \mathcal{S}(\Gamma_h^{m+1})^{n+1}$ is set to be
\begin{align*}
	& y^{m+1}_h := y^m_h \circ (u^{m+1}_h)^{-1}.
\end{align*}
\end{alg}
\begin{remark}
Please note that we do not have to keep track of two different triangulations in the above algorithm. The reference mesh $\M_h$ only enters
into the scheme via the maps $y_h^m$. As soon as the map $y_h^0$ is initialized, the reference mesh $\M_h$ is not needed any more.
Moreover, as we will describe in Section \ref{Numerical_results_MCF}, the representation vector $\mathbf{Y}$ of the maps $y_h^m$
does not depend on the discrete time levels $m$. So in the computer code, the only remnant of the reference manifold is a constant vector.
\end{remark}
\begin{remark}
In numerical experiments we observe that the above algorithm is able to redistribute the mesh points of the evolving surface $\Gamma_h^m$ in such a way
that, at least for small $\alpha$, the discrete surfaces $\Gamma_h^m$ and $\M_h$ are of similar mesh quality. 
It is therefore crucial to have a reference mesh $\M_h$ of sufficiently high quality.
\end{remark}
\begin{remark}
Please note that the term $\nabla_{\Gamma^m_h} u^{m+1}_h (\hat{H}^m_h)^{-1} (\nabla_{\Gamma^m_h} y^m_h)^T w^m_h$
in the above algorithm is the product of the first order term $\nabla_{\Gamma^m_h} u^{m+1}_h$ and the second order term $w^m_h$,
which is the weak Laplacian of $y_h^m$.
Since first order terms are, roughly speaking, more critical in numerical simulations with respect to stability,
an alternative approach that only leads to a (pure) second order term might be desirable in certain cases.   
As we will see below, a slight variation of the original DeTurck trick indeed gives rise to a scheme where the elliptic operator only consists of second order terms. 
By this means, we will obtain an algorithm for the computation of a variant of the mean curvature-DeTurck flow.
\end{remark}

\section{Reparametrizations via a variant of the DeTurck trick}
\label{reparametrizations_via_a_generalized_DeTurck_trick}

In the following we introduce a variant of the DeTurck trick.
We start by changing the system of equations $(P)$ in Section $2$ in the following way  
\[
(P')=
\left\{
\begin{aligned}
	& \frac{\partial}{\partial t} x = - (H \nu) \circ x,
	\quad \textnormal{with $x(\cdot,0) = x_0$ on $\M$,}
	\\
	& \frac{\partial}{\partial t} \psi_\alpha = \frac{1}{\alpha} \Delta_{h,g(t)} \psi_\alpha,
	\quad \textnormal{with $g(t) := x(t)^\ast \mathfrak{e}$  
		and $\psi_\alpha(\cdot,0) = id(\cdot)$ on $\M$,}
\end{aligned}
\right.
\]
where $h$ is again a fixed yet arbitrary smooth Riemannian metric on $\M$.
The difference between problem $(P)$ in Section $2$ and $(P')$ is that the metrics
$h$ and $g(t) := x(t)^\ast \mathfrak{e}$ in the map Laplacian have been permuted.
By this means we will obtain a reparametrized flow that is (almost) in divergence form
if we define the reparametrization of the mean curvature flow by
\begin{equation}
	\hat{x}_\alpha(t) := (\psi_\alpha(t)^{-1})_\ast x(t) := (\psi_\alpha(t))^\ast x(t)
	:= x(t) \circ \psi_\alpha(t).
	\label{alternative_reparametrisation_of_MCF} 
\end{equation}
In contrast to definition (\ref{pull_back_of_MCF}) the reparametrization is here defined
by using the push-forward instead of the pull-back. 
Similar as in (\ref{metric_is_pull_back_metric}),
we first observe that the the push-forward metric $(\psi_\alpha^{-1}(t))_\ast g(t)$
and the induced metric $(\hat{x}_\alpha(t))^\ast \mathfrak{e}$ are equal
\begin{equation}
	(\psi_\alpha^{-1}(t))_\ast g(t)
	= \psi_\alpha(t)^\ast g(t)
	= \psi_\alpha(t)^\ast (x(t)^\ast \mathfrak{e})
	= (x(t) \circ \psi_\alpha(t))^\ast \mathfrak{e}
	= (\hat{x}_\alpha(t))^\ast \mathfrak{e} =: \hat{g}_\alpha (t).
	\label{push_forward_metric}
\end{equation}
The same procedure as in Section $2$ then leads to the following evolution equation of the
reparametrized flow
\begin{align*}
	\frac{\partial}{\partial t} \hat{x}_\alpha(t)
	&= \frac{\partial}{\partial t} x(t) \circ \psi_\alpha(t) 
		+ (\nabla x \circ \psi_\alpha(t))\left( \frac{\partial }{\partial t} \psi_\alpha(t) \right)
	\\
	&= (\Delta_{g(t)} x) \circ \psi_\alpha(t) + \frac{1}{\alpha}(\nabla x \circ \psi_\alpha(t))
	 \left( \Delta_{h, g(t)} \psi_\alpha (t) \right)
	 \\
	&= - (H \nu) \circ \hat{x}_\alpha(t) + \frac{1}{\alpha}(\nabla x \circ \psi_\alpha(t))
	 \left( \Delta_{h, g(t)} \psi_\alpha (t) \right)
	 \\
	&=  \Delta_{\hat{g}_\alpha(t)} \hat{x}_\alpha + \frac{1}{\alpha}(\nabla x \circ \psi_\alpha(t))
	 \left( \Delta_{h, g(t)} \psi_\alpha (t) \right).
\end{align*}
In the following we will show that the second term on the right hand side is given by
\begin{equation}
	(\nabla x \circ \psi_\alpha(t))
	 \left( \Delta_{h, g(t)} \psi_\alpha (t) \right)
	 = (P \circ \hat{x}_\alpha) \Delta_h \hat{x}_\alpha,
	 \label{formula_tangential_part}
\end{equation}
and hence,
$$
	\frac{\partial}{\partial t} \hat{x}_\alpha(t) = \Delta_{\hat{g}_\alpha(t)} \hat{x}_\alpha
	+ \frac{1}{\alpha} (P \circ \hat{x}) \Delta_h \hat{x}_\alpha.
$$ 
We are not aware that the above equation has yet been considered elsewhere.
For $\Psi_\alpha = \mathcal{C}_2 \circ \psi_\alpha  \circ \mathcal{C}_1^{-1}$,
$X := x \circ \mathcal{C}_2^{-1}$ and $\hat{X}_\alpha := \hat{x}_\alpha \circ \mathcal{C}_1^{-1}$
we obtain that
\begin{align*}
	\frac{\partial \hat{X}_\alpha}{\partial \theta^k} 
	\left( \frac{\partial (\Psi_\alpha^{-1})^{k}}{\partial \theta^l} \circ \Psi_\alpha \right)
	= \frac{\partial X}{\partial \theta^j} \circ \Psi_\alpha \frac{\partial \Psi_\alpha^j}{\partial \theta^k} 
	\left( \frac{\partial (\Psi_\alpha^{-1})^{k}}{\partial \theta^l} \circ \Psi_\alpha \right)
	= \frac{\partial X}{\partial \theta^l} \circ \Psi_\alpha.
\end{align*}
Therefore, in local coordinates $\mathcal{C}_1$ the term $\nabla x \circ \psi_\alpha(t) \left( \Delta_{h, g(t)} \psi_\alpha (t) \right)$ is given by
\begin{align*}
	&\frac{\partial \hat{X}_\alpha}{\partial \theta^k} 
	\left( \frac{\partial (\Psi_\alpha^{-1})^{k}}{\partial \theta^l} \circ \Psi_\alpha \right)
	h^{mn} \left( \frac{\partial^2 \Psi_\alpha^l}{\partial \theta^m \partial \theta^n}
	 - \Gamma(h)_{mn}^p \frac{\partial \Psi_\alpha^l}{\partial \theta^p}
	 + \Gamma(g)_{ij}^l \circ \Psi_\alpha \frac{\Psi_\alpha^i}{\partial \theta^m}
	 \frac{\Psi_\alpha^j}{\partial \theta^n}
	\right)
\\	
	&= - \frac{\partial \hat{X}_\alpha}{\partial \theta^k} h^{mn} \Gamma(h)^k_{mn}
	  + \frac{\partial \hat{X}_\alpha}{\partial \theta^k} h^{mn} 
	\left( \frac{\partial (\Psi_\alpha^{-1})^{k}}{\partial \theta^l} \circ \Psi_\alpha \right)
	\left( \frac{\partial^2 \Psi_\alpha^l}{\partial \theta^m \partial \theta^n}
	 + \Gamma(g)_{ij}^l \circ \Psi_\alpha \frac{\Psi_\alpha^i}{\partial \theta^m}
	 \frac{\Psi_\alpha^j}{\partial \theta^n}
	\right),
\end{align*}
where in this case the indices $k,m,n,p$ refer to the coordinate chart $\mathcal{C}_1$ and the indices $i,j,l$
to the coordinate chart $\mathcal{C}_2$.
From (\ref{push_forward_metric}) it follows that
\begin{align*}
	& \hat{g}_{\alpha mn} 
		= (g_{ij} \circ \Psi_\alpha) \frac{\partial \Psi^i_\alpha}{\partial \theta^m}
		\frac{\partial \Psi^j_\alpha}{\partial \theta^n}, \quad \textnormal{and}
	\quad g^{ij} \circ \Psi_\alpha 
		= (\hat{g}_\alpha^{mn} ) \frac{\partial \Psi^i_\alpha}{\partial \theta^m}
		\frac{\partial \Psi^j_\alpha}{\partial \theta^n},
	\\
	&\textnormal{as well as}
	\quad  \frac{\partial (\Psi_\alpha)^l}{\partial \theta^k} \Gamma(\hat{g})^k_{mn} 
	= \left(
	    \frac{\partial^2 \Psi_\alpha^l }{\partial \theta^m \partial \theta^n}
		+ \Gamma(g)^l_{ij} \circ \Psi_\alpha
		\frac{\partial \Psi^i_\alpha}{\partial \theta^m} 
		\frac{\partial \Psi^j_\alpha}{\partial \theta^n}
	  \right),
\end{align*}
and hence,
$$
	\left( \nabla x \circ \psi_\alpha(t) \left( \Delta_{h, g(t)} \psi_\alpha (t) \right) \right)
	\circ \mathcal{C}_1^{-1} = 
	- \frac{\partial \hat{X}_\alpha}{\partial \theta^k} h^{mn} \Gamma(h)^k_{mn}
	  + \frac{\partial \hat{X}_\alpha}{\partial \theta^k} h^{mn} \Gamma(\hat{g})^k_{mn}. 
$$
Since identity (\ref{formula_Christoffel_symbols}) is still true, we finally obtain
\begin{align*}
	\left( \nabla x \circ \psi_\alpha(t) \left( \Delta_{h, g(t)} \psi_\alpha (t) \right) \right)
	\circ \mathcal{C}_1^{-1} 
	&= - \frac{\partial \hat{X}_\alpha}{\partial \theta^k} h^{mn} \Gamma(h)^k_{mn} 
+ h^{mn} (P \circ \hat{X}_\alpha) \frac{\partial^2 \hat{X}_\alpha}{\partial \theta^m \partial \theta^n}
\\
	&= (P \circ \hat{X}_\alpha) h^{mn} \left(
		 \frac{\partial^2 \hat{X}_\alpha}{\partial \theta^m \partial \theta^n} 
		 - \Gamma(h)^k_{mn} \frac{\partial \hat{X}_\alpha}{\partial \theta^k}
	\right),
\end{align*}
which proves (\ref{formula_tangential_part}). 
Instead of (\ref{inverse_map}) we now apply the map 
\begin{equation*}
\left( (\nu \circ \hat{x}_\alpha) \otimes (\nu \circ \hat{x}_\alpha) 
	+ \alpha \hat{\rho} P \circ \hat{x}_\alpha \right)
 = \left(\alpha \hat{\rho} \unit + (1 - \alpha \hat{\rho}) 
 	(\nu \circ \hat{x}_\alpha) \otimes (\nu \circ \hat{x}_\alpha) \right),
 \label{inverse_map_2}
\end{equation*}
where $\hat{\rho} := \sqrt{\det (h_{ij})}/ \sqrt{\det (\hat{g}_{\alpha ij})}$. 
Please note that the definition
of $\hat{\rho}$ does not dependent on the local coordinates in which it is evaluated.
This yields
\begin{thm}
The reparametrized mean curvature flow $\hat{x}_\alpha: \M \times [0,T) \rightarrow \mathbb{R}^{n+1}$
defined in (\ref{alternative_reparametrisation_of_MCF}) satisfies
\begin{equation}
\left( \alpha \hat{\rho} \unit + (1 - \alpha \hat{\rho}) (\nu \circ \hat{x}_\alpha) \otimes (\nu \circ \hat{x}_\alpha) \right) \frac{\partial}{\partial t} \hat{x}_\alpha
= \Delta_{\hat{g}_\alpha(t)} \hat{x}_\alpha + \hat{\rho} (P \circ \hat{x}_\alpha) \Delta_h \hat{x}_\alpha,
\label{reparam_MCF_3}
\end{equation}
with $\hat{x}_\alpha(\cdot, 0) = x_0(\cdot)$ on $\M$.
\end{thm}
Using the above identities, a short calculation shows that the inverse $\psi_\alpha^{-1}$ solves an 
ODE, which in local coordinates is given by
$$
	\frac{\partial}{\partial t} (\Psi_\alpha^{-1})^k	
	= \frac{1}{\alpha} \left( 
		h^{mn} ( \Gamma(h)_{mn}^k  - \Gamma(\hat{g}_\alpha)_{mn}^k)	
	\right)	\circ \Psi_\alpha^{-1}.
$$ 
Solving this equation and setting $x(t) = \hat{x}_\alpha(t) \circ \psi_\alpha(t)^{-1}$,
one can recover the solution to (\ref{MCF_equation}) from the solution to (\ref{reparam_MCF_3}).

\subsection{Weak formulation}
We now assume again that $\M$ is a hypersurface in $\mathbb{R}^{n+1}$
and we choose the metric $h$ on the reference hypersurface $\M \subset \mathbb{R}^{n+1}$ to be the metric that is induced by the 
Euclidean metric $\mathfrak{e}$ of the ambient space.
We multiply (\ref{reparam_MCF_3}) by a test function $\chi \in H^{1,2}(\M,\mathbb{R}^{n+1})$.
Integrating with respect to the
volume form $do_{\hat{g}_\alpha}$ then gives
\begin{align}
	0 &= \int_{\M} (\alpha \hat{\rho}) \hat{x}_{\alpha t} \cdot \chi do_{\hat{g}_\alpha}
	+ \int_{\M} ( 1 - \alpha \hat{\rho}) (\hat{x}_{\alpha t} \cdot (\nu \circ \hat{x}_\alpha)) 
							 (( \nu \circ \hat{x}_\alpha) \cdot \chi)							
							do_{\hat{g}_\alpha}
	\nonumber \\						
	& \quad + \int_\M \hat{G}_\alpha^{-1} \nabla_\M \hat{x}_\alpha : \nabla_\M \chi do_{\hat{g}_\alpha}
	+ \int_\M \nabla_{\M} \hat{x}_\alpha : \nabla_\M ( (P \circ \hat{x}_\alpha) \chi) do_{h},	
	\label{weak_formula_Harmonic_MCF_reference_M}		
\end{align}	 
where we used the fact that $do_{h} = \hat{\rho} do_{\hat{g}_\alpha}$ in the last term.
In order to derive the equivalent formulation of (\ref{weak_formula_Harmonic_MCF_reference_M})
on the moving hypersurface $\Gamma(t) := \hat{x}_\alpha (\M,t) \subset \mathbb{R}^{n+1}$,
we set $u:= \hat{x}_\alpha \circ \hat{x}_\alpha^{-1}$, $\rho := \hat{\rho} \circ \hat{x}_\alpha^{-1}$ 
and $\eta := \chi \circ \hat{x}_\alpha^{-1}$ and obtain the following weak formulation on 
$\Gamma(t)$.
\begin{thm}
\label{theorem_harmonic_MCF_weak_formulation}
Let $\alpha \in (0, \infty)$ and $\M$ be a smooth, $n$-dimensional, closed, connected hypersurface in $\mathbb{R}^{n+1}$.
Furthermore, let $h$ be the metric on $\M \subset \mathbb{R}^{n+1}$ which is induced by the Euclidean metric $\mathfrak{e}$ of the ambient space
and let $\hat{x}_\alpha: \M \times [0,T) \rightarrow \mathbb{R}^{n+1}$ evolve according
to the reparametrized mean curvature flow (\ref{reparam_MCF_3}). 
Then the identity map $u:= id_{\Gamma(t)}$ on $\Gamma(t)$ satisfies
\begin{align*}
	0 &= \int_{\Gamma(t)} (\alpha {\rho}) \partial^\bullet u \cdot \eta d\sigma
	+ \int_{\Gamma(t)} ( 1 - \alpha {\rho}) (\partial^\bullet u \cdot \nu) (\nu \cdot \eta) d\sigma
	\\						
	& \quad + 
	\int_{\Gamma(t)} \nabla_{\Gamma(t)} u : \nabla_{\Gamma(t)} \eta d\sigma
	+ \int_{\Gamma(t)} \hat{H}^{-1}_{\alpha}
	\nabla_{\Gamma(t)} u : \nabla_{\Gamma(t)}( P \eta) \rho d\sigma,
\end{align*}
where $\hat{H}_\alpha(t)$ is defined as in (\ref{defi_H}) and 
$\hat{H}_\alpha^{-1} \nabla_{\Gamma(t)} u : \nabla_{\Gamma(t)} (P \eta)
:=  (\hat{H}_\alpha^{-1})^{\gamma \kappa } \D \gamma' u_\beta \D \kappa' (P \eta)_\beta$. 
The weight function $\rho$ is given by 
$\rho = \sqrt{\det \hat{H}_\alpha}$.
\end{thm}	
\proof
By applying (\ref{material_derivative_u}) as well as (\ref{result_1}) we obtain
the first three integrals.
For the last term we combine (\ref{chain_rule_tangential_gradients}) and (\ref{result_3}),
which leads to
\begin{align*}
	\nabla_{\M} \hat{x}_\alpha : \nabla_{\M} \chi 
	&= \left( (\hat{H}_\alpha^{-1})^{\beta \gamma} \D \gamma'  (\hat{x}_\alpha^{-1})^\rho 
		\right) \circ \hat{x}_\alpha \D \rho \chi_\beta
	\\
	&= \left( (\hat{H}_\alpha^{-1})^{\beta \gamma} \D \gamma'  (\hat{x}_\alpha^{-1})^\rho 
		\right) \circ \hat{x}_\alpha (\D \kappa' \eta_\beta) \circ \hat{x}_\alpha \D \rho 
		(\hat{x}_\alpha)^\kappa 
	\\
	&= \left( (\hat{H}_\alpha^{-1})^{\beta \kappa} \D \kappa' \eta_\beta \right) \circ \hat{x}_\alpha
	= 	\left( (\hat{H}_\alpha^{-1})^{\beta \gamma} P_{\gamma \kappa} 
		\D \kappa' \eta_\beta \right) \circ \hat{x}_\alpha
	\\
	&= 	\left( (\hat{H}_\alpha^{-1})^{\kappa \gamma} P_{\gamma \beta} 
		\D \kappa' \eta_\beta \right) \circ \hat{x}_\alpha	
	= 	\left( (\hat{H}_\alpha^{-1})^{\gamma \kappa } \D \gamma' u_\beta 
		\D \kappa' \eta_\beta \right) \circ \hat{x}_\alpha	
	\\	
	&= \left( \hat{H}_\alpha^{-1} \nabla_\Gamma u : \nabla_{\Gamma} \eta \right) 
	\circ \hat{x}_\alpha,
\end{align*}
where we have used $\hat{H}_\alpha^{-1} P = P \hat{H}_\alpha^{-1}$,
which can be easily seen from (\ref{H_inverse_in_local_coordinates}), 
and $P = \nabla_\Gamma u$. 
Similar as in (\ref{det_G}), it follows from (\ref{H_in_local_coordinates}) that
$$
	\det \hat{H}_\alpha \circ \hat{X}_\alpha = \frac{\det h_{ij}}{\det \hat{g}_{\alpha ij}},
$$
and thus $\rho = \sqrt{\det \hat{H}_\alpha}$. 
\hfill $\Box$

\begin{remark}
The reparametrized mean curvature flow (\ref{weak_formula_Harmonic_MCF_reference_M})
is a gradient flow in the sense that
\begin{align*}
	& \int_\M \hat{G}_\alpha^{-1} \nabla_\M \hat{x}_\alpha : \nabla_\M \chi 
		d o_{\hat{g}_\alpha} 
		= \frac{d}{d \epsilon} A(\hat{x}_\alpha + \epsilon \chi) \big|_{\epsilon = 0}
	\\
	& \int_\M \nabla_\M \hat{x}_\alpha : \nabla_\M ((P\circ \hat{x}_\alpha) \chi) do_h
		= \frac{d}{d \epsilon} E(a(\hat{x}_\alpha + \epsilon \chi)) \big|_{\epsilon = 0}
\end{align*}
where $A(\hat{x}_\alpha) := \int_\M \sqrt{\det \hat{G}_\alpha} do_h = \int_\M 1 do_{\hat{g}_\alpha}$ is the area functional of $\Gamma := \hat{x}_\alpha (\M)$,
and $E(\hat{x}_\alpha) := \frac{1}{2} \int_\M | \nabla_\M \hat{x}_\alpha |^2 do_h$ is
the Dirichlet energy of the map $\hat{x}_\alpha: \M \rightarrow \Gamma$.
Please note that in the second equation only tangential variations 
$a(\hat{x}_\alpha + \epsilon \chi)$ are considered. Here, $a: \Gamma_\delta \rightarrow \Gamma$ denotes the orthogonal projection onto $\Gamma$, that is
$$
	a(x) = x - d(x) \nu(a(x)) \in \Gamma,
$$
where $x$ is a point in the tubular neighbourhood 
$\Gamma_\delta := \{ x \in \mathbb{R}^{n+1} 
~|~ |d(x)| < \delta \}$ of width $\delta > 0$ about $\Gamma$ and $d(\cdot)$
denotes the oriented distance function to $\Gamma$, see \cite{DDE05} for more details. 
\end{remark}

\subsection{Numerical scheme for a variant of the DeTurck trick}
We now present an algorithm for the computation of the reparametrized mean curvature flow (\ref{reparam_MCF_3}) 
based on the weak formulation in Theorem \ref{theorem_harmonic_MCF_weak_formulation}.
\begin{alg}
\label{algorithm_harmonic_MCF}
Let $\alpha \in (0,\infty)$.
For a given initial polyhedral hypersurface $\Gamma_h^0 = \hat{x}_h^0 (\M_h)$
with $\hat{x}_h^0 \in \mathcal{S}(\M_h)^{n+1}$, 
set $y_h^0 := (\hat{x}_h^0)^{-1} \in \mathcal{S}(\Gamma_h^0)^{n+1}$ and  
determine for $m=0, \ldots, M_\tau -1$
solutions $u_h^{m+1} \in \mathcal{S}(\Gamma_h^m)^{n+1}$
such that
\begin{align*}
& \frac{1}{\tau} \int_{\Gamma^m_h} (\alpha \rho^m_h) I_h (u^{m+1}_h \cdot \eta_h) d\sigma
	+ \frac{1}{\tau} \int_{\Gamma^m_h} (1- \alpha \rho^m_h)
		I_h ((u^{m+1}_h \cdot \tilde{\nu}^m_h)
		    ( \tilde{\nu}^m_h \cdot \eta_h )) d\sigma	
\\
&	 + \int_{\Gamma^m_h} 
		  \nabla_{\Gamma^m_h} u^{m+1}_h : \nabla_{\Gamma^m_h} \eta_h d\sigma
	+ \int_{\Gamma^m_h}
		(\hat{H}^m_h)^{-1} \nabla_{\Gamma_h^m} u^{m+1} : \nabla_{\Gamma_h^m} (I_h(P^m_h \eta_h))	
		\rho_h^m d\sigma
\\	
&  =  \frac{1}{\tau} \int_{\Gamma^m_h} (\alpha \rho^m_h) 
		I_h (\tilde{u}^m_h \cdot \eta_h)  d\sigma
	+ \frac{1}{\tau} \int_{\Gamma^m_h} (1- \alpha \rho^m_h)
		I_h ((\tilde{u}^m_h \cdot \tilde{\nu}^m_h)
		( \tilde{\nu}^m_h \cdot \eta_h )) d\sigma	,
		\quad \forall \eta_h \in \mathcal{S}(\Gamma_h^m)^{n+1},	
\end{align*}	
where
$\nu_h^m$ is the piecewise constant outward unit normal to 
$\Gamma_h^m$, $\tilde{u}^{m}_h := id_{| \Gamma^{m}_h}$,
and
\begin{align*}
	& \hat{H}^{m}_h 
		:= (\nabla_{\Gamma_h^{m}} y^{m}_h)^T \nabla_{\Gamma^{m}_h} y^{m}_h
			+ \nu^{m}_h \otimes \nu^{m}_h,
     ~ \textnormal{and} ~ \rho_h^m 
		:= \sqrt{\det \hat{H}^m_h }, 		
\\
	& \tilde{\nu}^m_h \in \mathcal{S}(\Gamma_h^m)^{n+1} ~ \textnormal{such that} ~
	\tilde{\nu}^m_h(p_j) 
	= \frac{\sum_{T \in \sigma_j} \nu_{h |T}^m |T|}{|\sum_{T \in \sigma_j} \nu_{h |T}^m |T||},
	~ \textnormal{with} ~ \sigma_j = \{ T \in \mathcal{T}^m_h |~ p_j \in T \}, 
\\
	& P_h^m \in \mathcal{S}(\Gamma_h^m)^{(n+1)\times(n+1)} ~ \textnormal{such that} ~ 
	P_h^m(p_j):= \unit - \tilde{\nu}_h^m(p_j) \otimes \tilde{\nu}_h^m(p_j)
	~ \textnormal{for all vertices $p_j \in \Gamma_h^m$.}				
\end{align*}
The hypersurface $\Gamma_h^{m+1}$ is defined by
$$
	\Gamma^{m+1}_h := u^{m+1}_h(\Gamma^m_h),
$$
and $y^{m+1}_h \in \mathcal{S}(\Gamma_h^{m+1})^{n+1}$ is set to be
\begin{align*}
	& y^{m+1}_h := y^m_h \circ (u^{m+1}_h)^{-1}.
\end{align*}
\end{alg}

\section{Numerical results for the mean curvature flow}
\label{Numerical_results_MCF}
We implemented Algorithms \ref{algo_moving_hypersurface} and \ref{algorithm_harmonic_MCF} 
as well as the benchmark algorithm
(2.25) in \cite{BGN08} for the computation of the mean curvature flow within the
Finite Element Toolbox ALBERTA, see \cite{SS05}. 
For Algorithm \ref{algo_moving_hypersurface} one has to solve the linear system 
\begin{align}
	\frac{1}{\tau} \sum_{\gamma=1}^{n+1} \sum_{j=1}^N \widetilde{M}_{ij \beta \gamma} \mathbf{U}^{j \gamma}
	+ \sum_{\gamma=1}^{n+1} \sum_{j=1}^N S_{ij \beta \gamma} \mathbf{U}^{j \gamma}
    + \sum_{\gamma=1}^{n+1} \sum_{j=1}^N B_{ij \beta \gamma} \mathbf{U}^{j \gamma}	
	= \frac{1}{\tau} 
		\sum_{\gamma=1}^{n+1} \sum_{j=1}^N \widetilde{M}_{ij \beta \gamma} \mathbf{U}_{old}^{j \gamma},
	\label{system_of_linear_equations_MCF_original_DeTurck}	
\end{align}
for $i=1, \ldots, N$, $\beta = 1,2$, whereas for Algorithm \ref{algorithm_harmonic_MCF} one has to assemble the
following system,
\begin{align}
	\frac{1}{\tau} \sum_{\gamma=1}^{n+1} \sum_{j=1}^N M_{ij \beta \gamma} \mathbf{U}^{j \gamma}
	+ \sum_{\gamma=1}^{n+1} \sum_{j=1}^N S_{ij \beta \gamma} \mathbf{U}^{j \gamma}
    + \sum_{\gamma=1}^{n+1} \sum_{j=1}^N D_{ij \beta \gamma} \mathbf{U}^{j \gamma}	
	= \frac{1}{\tau} 
		\sum_{\gamma=1}^{n+1} \sum_{j=1}^N M_{ij \beta \gamma} \mathbf{U}_{old}^{j \gamma}.
	\label{system_of_linear_equations_MCF}	
\end{align}
Here, $u_h^{m+1} = \sum_{\gamma=1}^{n+1} \sum_{j=1}^N \mathbf{U}^{j \gamma} \phi_j e_\gamma$
is the unknown parametrization of the discrete surface $\Gamma_h^{m+1}$ and 
$\tilde{u}_h^m = \sum_{\gamma=1}^{n+1} \sum_{j=1}^N \mathbf{U}_{old}^{j \gamma} \phi_j e_\gamma$
is the identity function on $\Gamma_h^m$.
Please note that the basis functions are defined on the changing polyhedral surfaces $\Gamma^m_h$, that is
$\phi_j = \phi_j^m$. We usually drop the superscript $m$ for the sake of convenience.
The vector $\mathbf{U}_{old}$ is just given by the solution vector $\mathbf{U}$ from the previous time step.
The matrices 
$\widetilde{M} := (\widetilde{M}_{ij \beta \gamma}), S:= (S_{ij \beta \gamma}), B:= (B_{ij \beta \gamma})
 \in \mathbb{R}^{((n+1)N)\times((n+1)N)}$, 
and respectively, the matrices
$M: = (M_{ij \beta \gamma}), S:= (S_{ij \beta \gamma})$ and 
$D:= (D_{ij \beta \gamma})$ can be assembled
by summing up the non-vanishing components coming from the element matrices, which are
\begin{align*}
	& \widetilde{M}_{ij \beta \gamma}(T) = \left( \alpha \delta_{\beta \gamma} 
	+ (1 - \alpha) \nu_{h \beta}^m(T) \nu_{h \gamma}^m(T) \right) 
	\int_T \phi_i \phi_j d\sigma,
	\\
	& S_{ij \beta \gamma}(T) = \delta_{\beta \gamma} 
	\int_T \nabla_{\Gamma^m_h} \phi_i \cdot \nabla_{\Gamma^m_h} \phi_j d\sigma,
	\\
	& B_{ij \beta \gamma}(T) = 	
	\delta_{\beta \gamma} \int_T \phi_i (\nabla_{\Gamma^m_h} \phi_j)_\iota 
	\left( \hat{H}^m_h(T)^{-1} (\nabla_{\Gamma^m_h} y_h^m)^T w_h^m \right)^\iota d\sigma,
	\\
	& M_{ij \beta \gamma}(T) = \left( \alpha \rho_h^m(T) \delta_{\beta \gamma} 
		+ ( 1 - \alpha \rho_h^m(T))\tilde{\nu}^m_{h\beta}(p_i) \tilde{\nu}^m_{h\gamma}(p_i) \right) 
		\delta_{ij} \int_{T} \phi_i d\sigma,
	\\
	& D_{ij \beta \gamma}(T) = 
	\left( \delta_{\beta \gamma} - \tilde{\nu}^m_{h \beta}(p_i) \tilde{\nu}^m_{h \gamma}(p_i)\right)
	\left(\hat{H}^m_h(T)^{-1} \right)^{\kappa \iota} \rho_h^m(T)
	\int_T (\nabla_{\Gamma^m_h} \phi_i)_\kappa (\nabla_{\Gamma^m_h} \phi_j)_\iota d\sigma.
\end{align*} 
Here, $\phi_i$ and $\phi_j$ are the nodal basis functions associated with 
the vertices $p_i$ and $p_j$ of the simplex $T$.
Please note that the unit normal vector $\nu_h^m(T)$, the matrix $\hat{H}^m_h(T)$ and the weight function 
$\rho_h^m(T)$ are constant on each simplex.  
The linear systems (\ref{system_of_linear_equations_MCF_original_DeTurck})
and (\ref{system_of_linear_equations_MCF}) can be solved by the
biconjugate gradient stabilized method. 
The vector $w_h^m$ in Algorithm \ref{algo_moving_hypersurface} can be easily computed by inverting a mass matrix.
The initial polyhedral hypersurface $\Gamma_h^0$ is constructed
by mapping the vertices of a triangulation of the reference hypersurface $\M_h$,
which is the unit sphere in Example 1 and 2 and a torus in Example 3, onto
the initial smooth hypersurface. The linear interpolation of the image then gives
$\Gamma_h^0$. 
By this means, the inverse 
$y_h^0 := (\hat{x}_h^0)^{-1} = \sum_{\gamma=1}^{n+1} \sum_{j=1}^N \mathbf{Y}^{j \gamma} \phi_j e_\gamma$ is determined by 
the position vectors $(\mathbf{Y}^{j \gamma})_{\gamma= 1, \ldots, n+1}$ for $j=1, \ldots, N$, of the
vertices of the reference hypersurface $\M_h$.
Please note that the vector $\mathbf{Y}$ is constant in time and that only the basis functions 
change in each time step $m$, that is 
$y_h^m = \sum_{\gamma =1}^{n+1} \sum_{j=1}^N \mathbf{Y}^{j \gamma} \phi_j^m e_\gamma$
for all $m$.
In the following, we will compare the performance of Algorithms
\ref{algo_moving_hypersurface} and \ref{algorithm_harmonic_MCF}
as well as of the benchmark scheme (2.25) in \cite{BGN08}, defined by the equation
\begin{equation*}
	\frac{1}{\tau} \int_{\Gamma_h^m} I_h ((u_h^{m+1} \cdot \overline{\nu}_h^m)
	(\overline{\nu}_h^m \cdot \eta_h))
	+  \int_{\Gamma^m_h} 
		  \nabla_{\Gamma^m_h} u^{m+1}_h : \nabla_{\Gamma^m_h} \eta_h d\sigma
	= \frac{1}{\tau} \int_{\Gamma_h^m} I_h ((\tilde{u}_h^{m} \cdot \overline{\nu}_h^m)
	(\overline{\nu}_h^m \cdot \eta_h)),
\end{equation*} 
where the definition of $\overline{\nu}_h^m \in \mathcal{S}(\Gamma_h^m)^{n+1}$, see (2.7) in \cite{BGN08}, 
slightly differs from our definition
of $\tilde{\nu}_h^m$ in Algorithm \ref{algorithm_harmonic_MCF}.
If we formally set $\alpha = 0$ in Algorithm \ref{algorithm_harmonic_MCF},
the main difference between Algorithm \ref{algorithm_harmonic_MCF} and the BGN-scheme
is the second order term given by the matrix $D$.
In our numerical tests, we will focus on the 
mesh properties of the schemes. A good quantity to evaluate the mesh quality of a polyhedral surface is 
\begin{equation}
	\sigma_{max} := \max_{T \in \mathcal{T}_h} 
	\frac{h(T)}{r(T)},
\label{definition_sigma_max}
\end{equation}
where $h(T)$ denotes the diameter of the simplex $T$ and $r(T)$ is the radius of the
largest ball contained in the simplex. Small values of $\sigma_{max}$ imply that there
are no simplices with sharp angles.

\subsection*{Example 1:}
The initial surface, approximated 
by the polyhedral surface in Figure \ref{initial_surface_MCF_example_1}, 
is given by the local parametrization
\begin{equation*}
	X_0(\theta, \varphi) :=
	\left( 
	\begin{array}{c}	
		\cos \varphi \\
		(0.7 \cos^2 \varphi + 0.3) \cos \theta  \sin \varphi \\
		(0.7 \cos^2 \varphi + 0.3) \sin \theta  \sin \varphi
	\end{array}		
	\right),
	\quad \theta \in [0,2\pi), \varphi \in [0,\pi].
\end{equation*}
The discrete reference hypersurface $\M_h$ is a triangulation of the unit sphere, see
Figure \ref{reference_surface_MCF_example_1}. 
The simulation in Figure \ref{Figure_simulation_MCF_example_1} shows that the mean curvature
flow for this initial surface develops a neck pinch singularity in finite time,
see Figures \ref{Fig_neck_surface_alpha_0_0001_tau_0_0001_3} 
or \ref{Fig_neck_surface_BGN_tau_0_0001_3}.
Please note that under Algorithm \ref{algorithm_harmonic_MCF} the simplices at the poles 
have a greater area compared to the simplices at the neck of the surface,
see Figures \ref{Fig_neck_surface_alpha_0_0001_tau_0_0001_1}, 
\ref{Fig_neck_surface_alpha_0_0001_tau_0_0001_2} and 
\ref{Fig_neck_surface_alpha_0_0001_tau_0_0001_3},
whereas under the BGN-scheme (2.25) in \cite{BGN08} this is not the case.
However, this does not mean that the mesh properties of Algorithm \ref{algorithm_harmonic_MCF}
are not good. On the contrary, the area of the simplices can be easily reduced by local mesh
refinements, whereas the size of the quantity $\sigma_{max}$ almost remains unchanged under
local mesh refinements. It is therefore much more preferable to have an algorithm that
produces meshes with small values of $\sigma_{max}$ rather than meshes with simplices
of the same area size.
The comparison of the simplices at the surface neck, see Figures 
\ref{Figure_enlarged_surface_neck_alpha_0_0001_MCF} and
\ref{Figure_enlarged_surface_neck_BGN_MCF},
gives a first hint that Algorithm \ref{algorithm_harmonic_MCF} is indeed able to produce
good meshes, that is without any sharp angles.
This observation is confirmed in a systematic study of $\sigma_{max}$
for different choices of the parameter $\alpha$. Figure \ref{Figure_mesh_properties_MCF_example_1}
shows that Algorithm \ref{algorithm_harmonic_MCF} clearly outperform the BGN-scheme
for $\alpha \leq 0.1$. For $\alpha = 1.0$ the result of Algorithm \ref{algorithm_harmonic_MCF}
looks similar to the result of the BGN-scheme, see also 
Figure \ref{Figure_simulation_MCF_for_alpha_1_0_example_1}.
In Figure \ref{comparison_alg_3_and_4_example_1}, the mesh properties of Algorithms \ref{algo_moving_hypersurface}
and \ref{algorithm_harmonic_MCF} are compared.
For small values of $\alpha$, both schemes show a similar performance, although Algorithm \ref{algorithm_harmonic_MCF}
behaves slightly better close to the surface singularity. However, this is not always the case,
see for example Figure \ref{comparison_alg_3_and_4_example_2}, where Algorithm \ref{algo_moving_hypersurface}
shows better behaviour at the singularity.
For $\alpha = 1.0$ and $\alpha = 0.1$, Algorithm \ref{algo_moving_hypersurface} generally seems to produce better meshes.
In the following, we will mainly focus on the comparison of Algorithm \ref{algorithm_harmonic_MCF} and 
the BGN-scheme, since Algorithms \ref{algo_moving_hypersurface} and \ref{algorithm_harmonic_MCF} provide similar behaviour. 
Figure \ref{Figure_area_decrease_MCF_example_1} shows the decrease of the discrete surface
area under the mean curvature flow. 
Please note that for small values of $\alpha$ and for times close
to the starting point,
the area of the surface is not monotonically 
decreasing under Algorithm \ref{algorithm_harmonic_MCF}.
This behaviour is due to relatively large tangential motions that occur for small
choices of $\alpha$ if the initial surface parametrization $x_0: \M \rightarrow \Gamma(0)$
is not harmonic. 
Although, Figure \ref{Figure_dependence_of_area_decrease_on_time_step_size} indicates
that the initial increase of the surface area can be reduced by choosing smaller 
time step sizes $\tau$, an improved time discretization might solve the problem
also for larger time step sizes. This is an open problem that should be addressed in further
research. Figure \ref{Figure_dependence_sigma_max_mesh_refinements} illustrates 
the influence of decreasing maximal diameters 
$h = \max_T h(T)$ on the behaviour of the mesh quality $\sigma_{max}$.

\begin{figure}
\begin{center}
~~~~~~
\subfloat[][\centering Reference surface for \mbox{Algorithm \ref{algorithm_harmonic_MCF}.}]
{\includegraphics[width=0.3\textwidth]{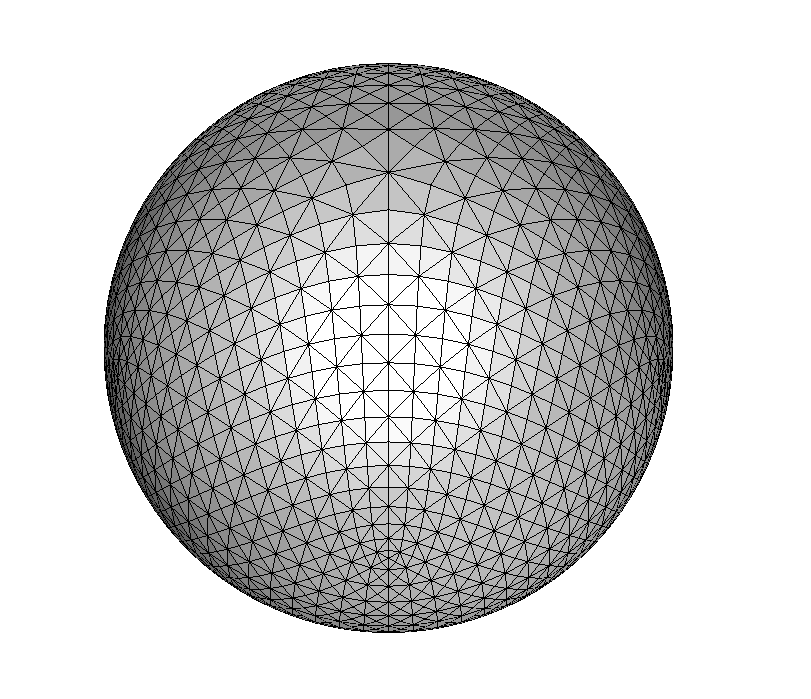}
\label{reference_surface_MCF_example_1}
} 
~~~~~
\subfloat[][\centering Surface at time $t = 0.0$. The surface area is $5.549$.]
{\includegraphics[width=0.4\textwidth]{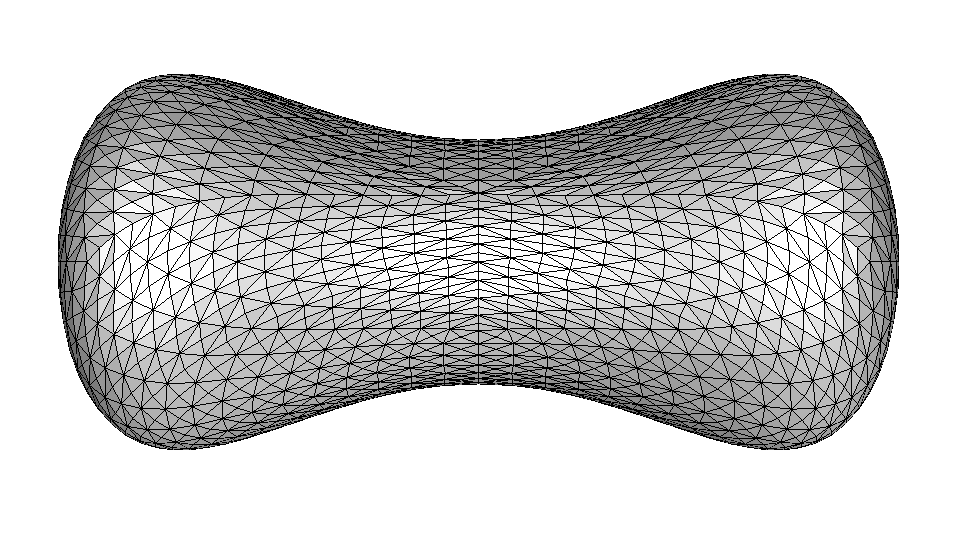}
\label{initial_surface_MCF_example_1}
} \\

\subfloat[][\centering Algorithm \ref{algorithm_harmonic_MCF}
at time $t=0.0301$. The surface area is $3.296$.]
{\includegraphics[width=0.4\textwidth]
{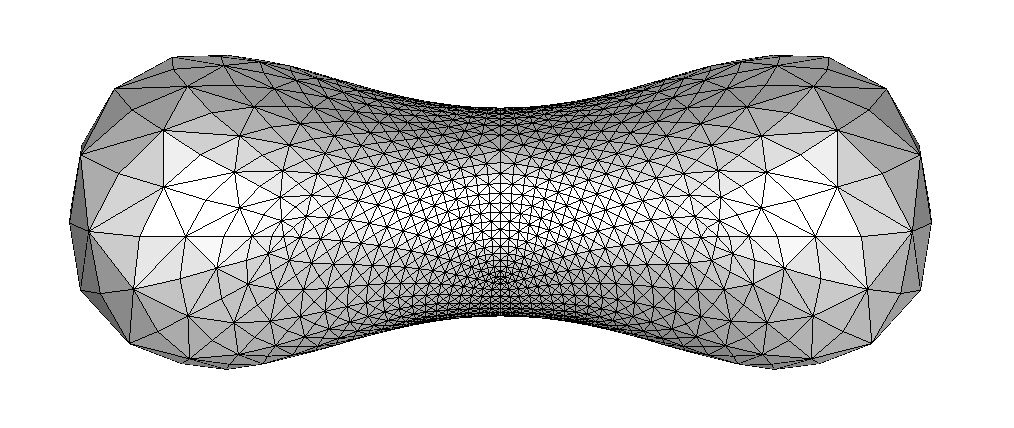}
\label{Fig_neck_surface_alpha_0_0001_tau_0_0001_1}
} 
\subfloat[][\centering BGN-scheme at time $t = 0.0301$. The surface area is $3.111$.]
{\includegraphics[width=0.4\textwidth]
{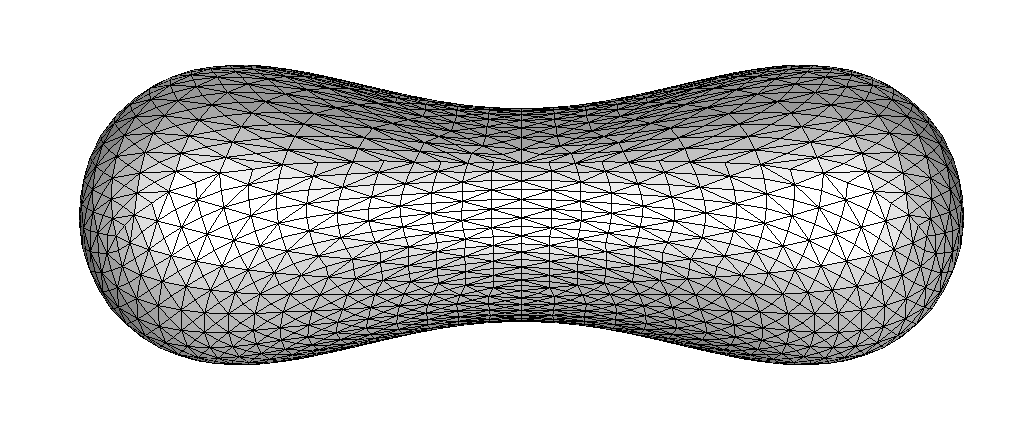}}\\

\subfloat[][\centering Algorithm \ref{algorithm_harmonic_MCF}
at time $t=0.0556$. The surface area is $1.208$.]
{\includegraphics[width=0.4\textwidth]
{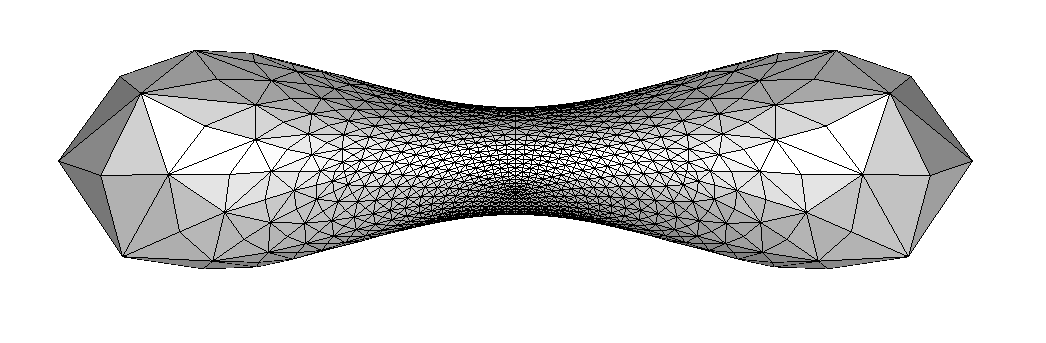}
\label{Fig_neck_surface_alpha_0_0001_tau_0_0001_2}
} 
\subfloat[][\centering BGN-scheme at time $t = 0.0556$. The surface area is $0.893$.]
{\includegraphics[width=0.4\textwidth]
{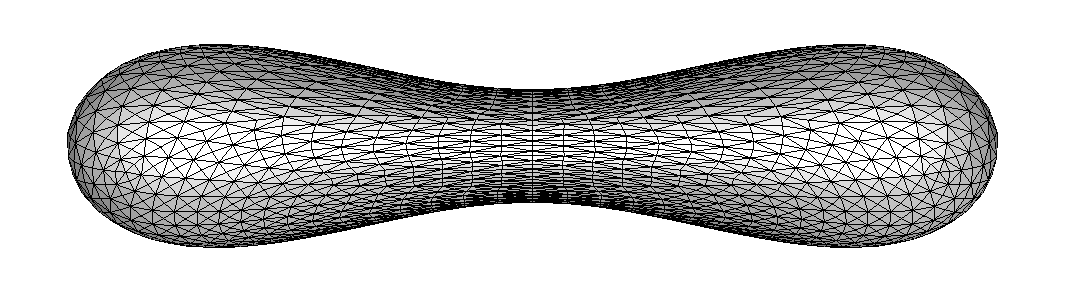}}\\

\subfloat[][\centering Algorithm \ref{algorithm_harmonic_MCF}
at time $t=0.0556$. The surface area is $1.208$. Enlarged section of the surface neck.]
{\includegraphics[width=0.3\textwidth]
{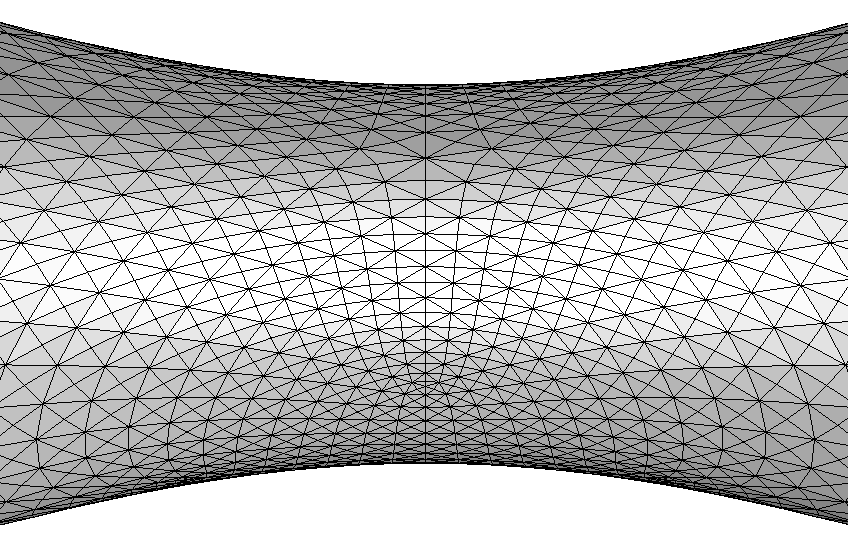}
\label{Figure_enlarged_surface_neck_alpha_0_0001_MCF}
} 
~~~~~~~~~
\subfloat[][\centering BGN-scheme at time $t = 0.0556$. The surface area is $0.893$.
Enlarged section of the surface neck.]
{\includegraphics[width=0.31\textwidth]
{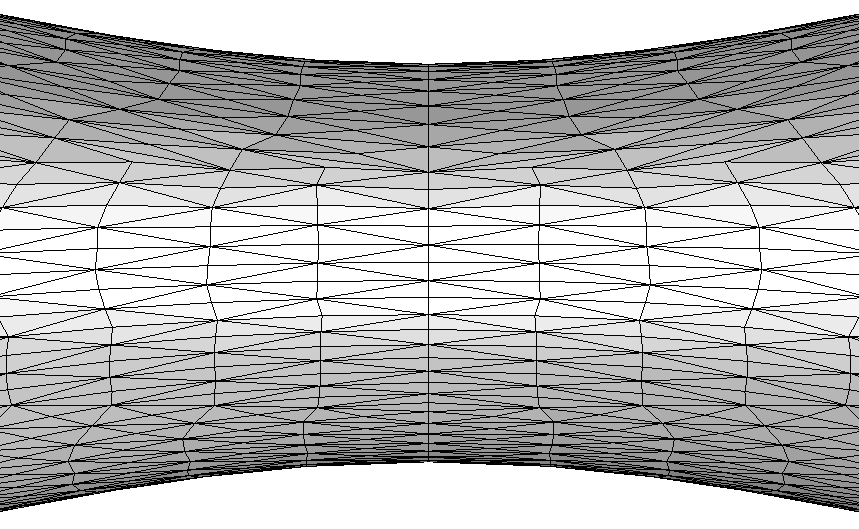}
\label{Figure_enlarged_surface_neck_BGN_MCF}
}\\

\subfloat[][\centering Algorithm \ref{algorithm_harmonic_MCF}
at time $t=0.0603$. The surface area is $0.697$.]
{\includegraphics[width=0.4\textwidth]
{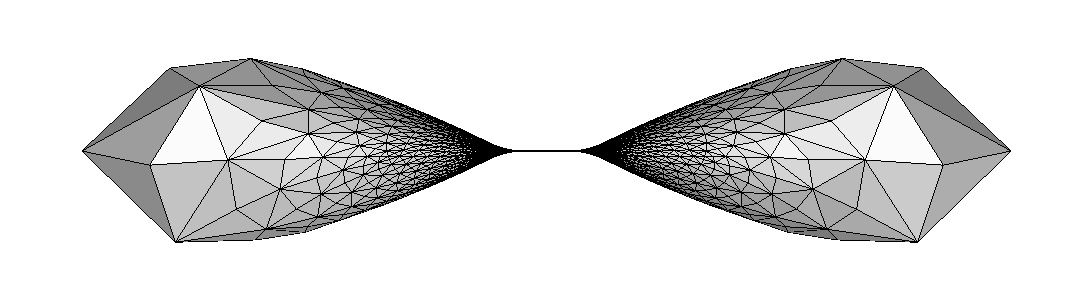}
\label{Fig_neck_surface_alpha_0_0001_tau_0_0001_3}
} 
\subfloat[][\centering BGN-scheme at time $t = 0.0596$. The surface area is $0.389$.]
{\includegraphics[width=0.4\textwidth]
{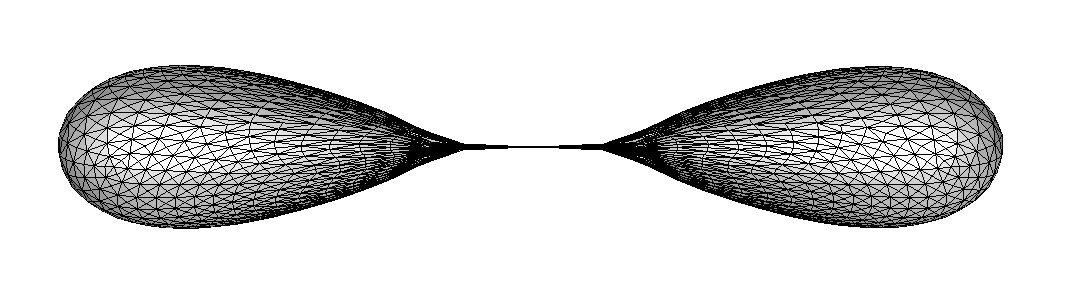}
\label{Fig_neck_surface_BGN_tau_0_0001_3}
}
\caption{Comparison of Algorithm \ref{algorithm_harmonic_MCF} for $\alpha = 10^{-4}$
and the BGN-scheme (2.25) in \cite{BGN08}. The time step size for both schemes was $\tau = 10^{-4}$.
The mesh had $5120$ triangles and $2562$ vertices.
The images are rescaled. See Example 1 of Section \ref{Numerical_results_MCF} for further details.}
\label{Figure_simulation_MCF_example_1}
\end{center}
\end{figure}

\begin{figure}
\begin{center}
\subfloat[][\centering Algorithm \ref{algorithm_harmonic_MCF}
at time $t=0.0301$. The surface area is $3.127$.]
{\includegraphics[width=0.4\textwidth]{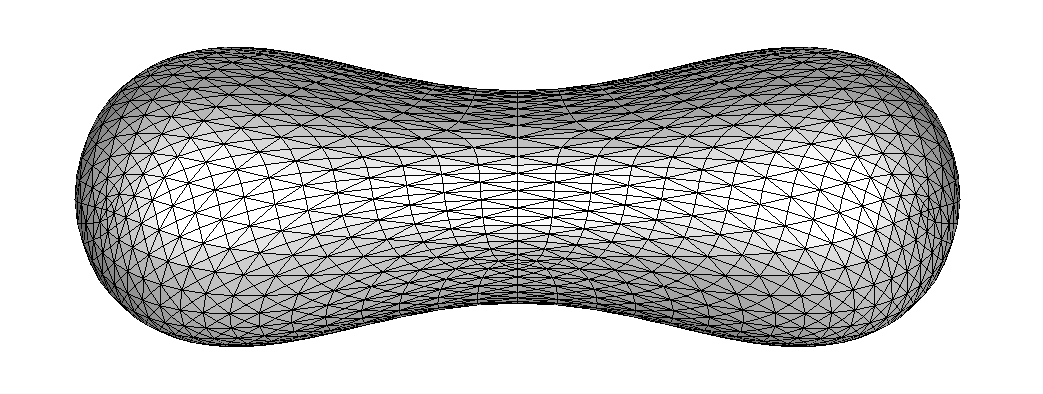}} 
\subfloat[][\centering Algorithm \ref{algorithm_harmonic_MCF}
at time $t=0.0556$. The surface area is $0.923$.]
{\includegraphics[width=0.4\textwidth]{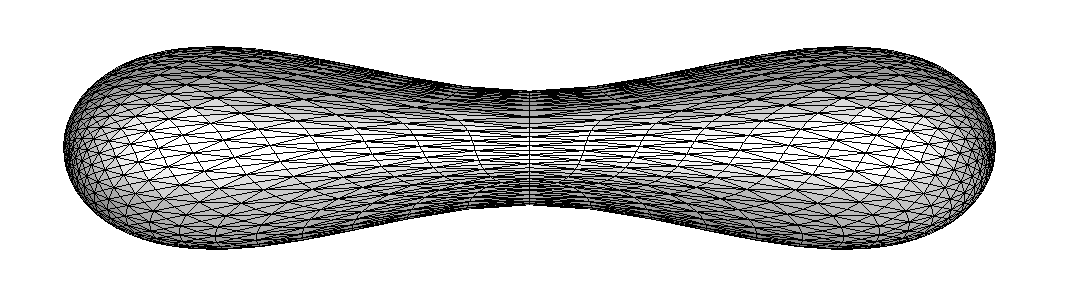}} 
\caption{Results of Algorithm \ref{algorithm_harmonic_MCF} for $\alpha = 1$. 
The time step size was $\tau = 10^{-4}$.
The images are rescaled. See Figures \ref{Fig_neck_surface_alpha_0_0001_tau_0_0001_1} and
\ref{Fig_neck_surface_alpha_0_0001_tau_0_0001_2} 
for the corresponding results with $\alpha = 10^{-4}$.
See Example 1 of Section \ref{Numerical_results_MCF} for further details.
}
\label{Figure_simulation_MCF_for_alpha_1_0_example_1}
\end{center}
\end{figure}

\gdef\gplbacktext{}%
\gdef\gplfronttext{}%
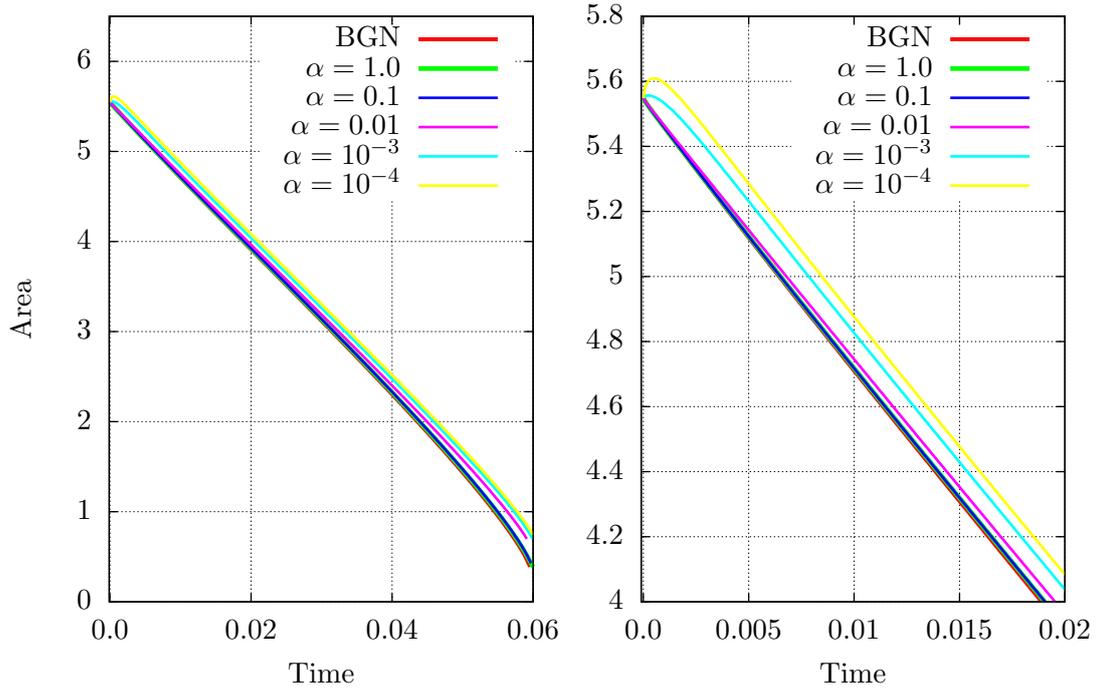
\begin{figure}
\centering
  \begin{picture}(7936.00,4534.00)%
    \gplgaddtomacro\gplbacktext{%
      \csname LTb\endcsname%
      \put(660,110){\makebox(0,0)[r]{\strut{} 0}}%
      \csname LTb\endcsname%
      \put(660,787){\makebox(0,0)[r]{\strut{} 1}}%
      \csname LTb\endcsname%
      \put(660,1464){\makebox(0,0)[r]{\strut{} 2}}%
      \csname LTb\endcsname%
      \put(660,2141){\makebox(0,0)[r]{\strut{} 3}}%
      \csname LTb\endcsname%
      \put(660,2818){\makebox(0,0)[r]{\strut{} 4}}%
      \csname LTb\endcsname%
      \put(660,3495){\makebox(0,0)[r]{\strut{} 5}}%
      \csname LTb\endcsname%
      \put(660,4172){\makebox(0,0)[r]{\strut{} 6}}%
      \csname LTb\endcsname%
      \put(797,-110){\makebox(0,0){\strut{}0.0}}%
      \csname LTb\endcsname%
      \put(1850,-110){\makebox(0,0){\strut{}0.02}}%
      \csname LTb\endcsname%
      \put(2902,-110){\makebox(0,0){\strut{}0.04}}%
      \csname LTb\endcsname%
      \put(3954,-110){\makebox(0,0){\strut{}0.06}}%
      \put(154,2310){\rotatebox{-270}{\makebox(0,0){\strut{}Area}}}%
      \put(2373,-440){\makebox(0,0){\strut{}Time}}%
    }%
    \gplgaddtomacro\gplfronttext{%
      \csname LTb\endcsname%
      \put(2967,4338){\makebox(0,0)[r]{\strut{}BGN}}%
      \csname LTb\endcsname%
      \put(2967,4118){\makebox(0,0)[r]{\strut{}$\alpha = 1.0$}}%
      \csname LTb\endcsname%
      \put(2967,3898){\makebox(0,0)[r]{\strut{}$\alpha = 0.1$}}%
      \csname LTb\endcsname%
      \put(2967,3678){\makebox(0,0)[r]{\strut{}$\alpha = 0.01$}}%
      \csname LTb\endcsname%
      \put(2967,3458){\makebox(0,0)[r]{\strut{}$\alpha =10^{-3}$}}%
      \csname LTb\endcsname%
      \put(2967,3238){\makebox(0,0)[r]{\strut{}$\alpha =10^{-4}$}}%
    }%
    \gplgaddtomacro\gplbacktext{%
      \csname LTb\endcsname%
      \put(4628,110){\makebox(0,0)[r]{\strut{} 4}}%
      \csname LTb\endcsname%
      \put(4628,599){\makebox(0,0)[r]{\strut{} 4.2}}%
      \csname LTb\endcsname%
      \put(4628,1088){\makebox(0,0)[r]{\strut{} 4.4}}%
      \csname LTb\endcsname%
      \put(4628,1577){\makebox(0,0)[r]{\strut{} 4.6}}%
      \csname LTb\endcsname%
      \put(4628,2066){\makebox(0,0)[r]{\strut{} 4.8}}%
      \csname LTb\endcsname%
      \put(4628,2555){\makebox(0,0)[r]{\strut{} 5}}%
      \csname LTb\endcsname%
      \put(4628,3044){\makebox(0,0)[r]{\strut{} 5.2}}%
      \csname LTb\endcsname%
      \put(4628,3533){\makebox(0,0)[r]{\strut{} 5.4}}%
      \csname LTb\endcsname%
      \put(4628,4022){\makebox(0,0)[r]{\strut{} 5.6}}%
      \csname LTb\endcsname%
      \put(4628,4511){\makebox(0,0)[r]{\strut{} 5.8}}%
      \csname LTb\endcsname%
      \put(4776,-110){\makebox(0,0){\strut{}0.0}}%
      \csname LTb\endcsname%
      \put(5562,-110){\makebox(0,0){\strut{}0.005}}%
      \csname LTb\endcsname%
      \put(6348,-110){\makebox(0,0){\strut{}0.01}}%
      \csname LTb\endcsname%
      \put(7135,-110){\makebox(0,0){\strut{}0.015}}%
      \csname LTb\endcsname%
      \put(7921,-110){\makebox(0,0){\strut{}0.02}}%
      \put(6340,-440){\makebox(0,0){\strut{}Time}}%
    }%
    \gplgaddtomacro\gplfronttext{%
      \csname LTb\endcsname%
      \put(6934,4338){\makebox(0,0)[r]{\strut{}BGN}}%
      \csname LTb\endcsname%
      \put(6934,4118){\makebox(0,0)[r]{\strut{}$\alpha = 1.0$}}%
      \csname LTb\endcsname%
      \put(6934,3898){\makebox(0,0)[r]{\strut{}$\alpha = 0.1$}}%
      \csname LTb\endcsname%
      \put(6934,3678){\makebox(0,0)[r]{\strut{}$\alpha = 0.01$}}%
      \csname LTb\endcsname%
      \put(6934,3458){\makebox(0,0)[r]{\strut{}$\alpha =10^{-3}$}}%
      \csname LTb\endcsname%
      \put(6934,3238){\makebox(0,0)[r]{\strut{}$\alpha =10^{-4}$}}%
    }%
    \gplbacktext
    \put(0,0){\includegraphics{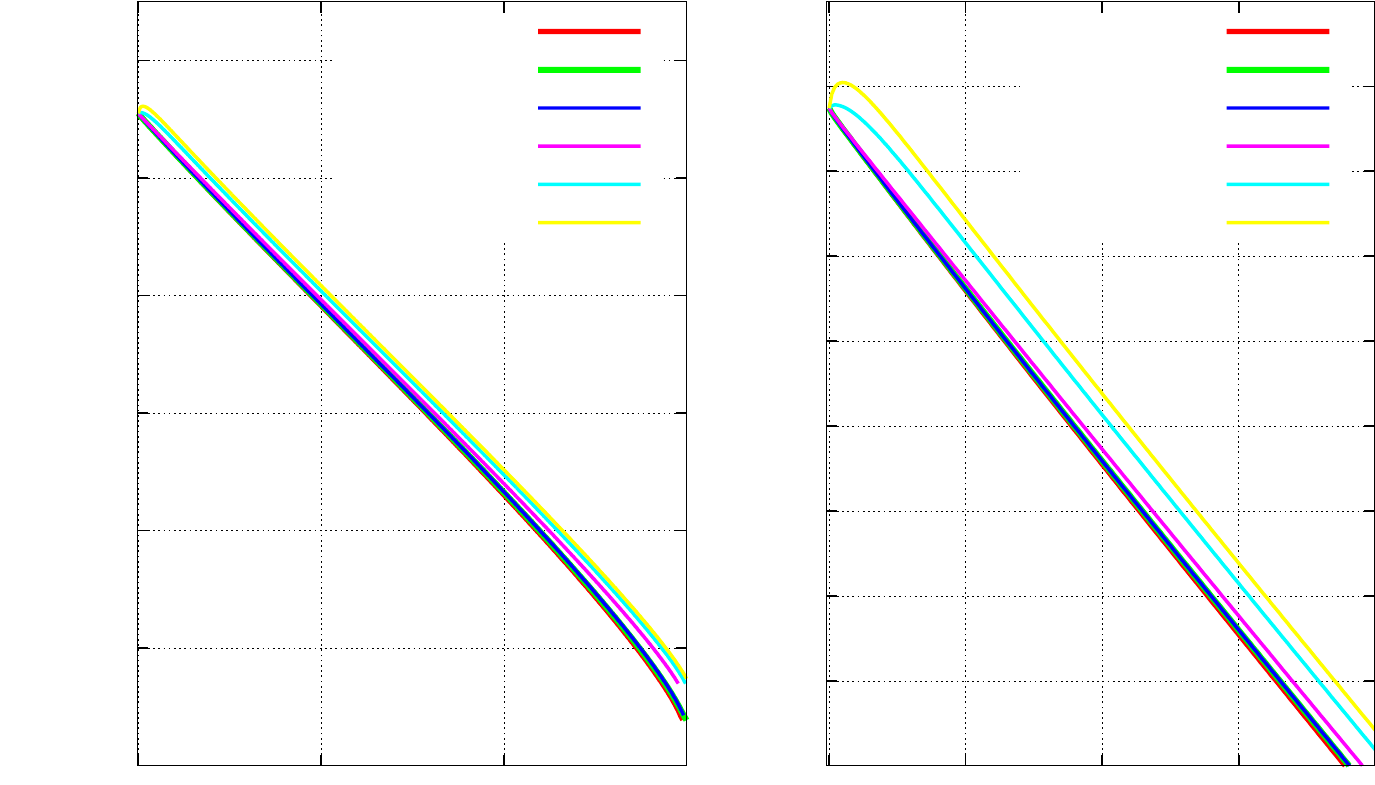}}%
    \gplfronttext
  \end{picture}%
  \vspace*{12mm}
  \caption{The images show the behaviour of the surface area for the BGN-scheme (2.25) in \cite{BGN08} 
  and for Algorithm \ref{algorithm_harmonic_MCF} 
  for different choices of $\alpha$. 
  The initial surface is shown in Figure \ref{initial_surface_MCF_example_1}. 
  The time step size was chosen as $\tau = 10^{-4}$. 
  The right image shows an enlarged section for small times $t$. 
  In general, for small $\alpha$ and small times $t$, the area of the solution of the $\alpha$-scheme is 
  not monotonically decreasing, 
  see for example the solution for $\alpha = 10^{-4}$ in the right picture.
  However, smaller time step sizes $\tau$ lead to a drop of the absolute increase of the area, 
  see Figure \ref{Figure_dependence_of_area_decrease_on_time_step_size}. 
  For further details see Example 1 of Section \ref{Numerical_results_MCF}.
}
\label{Figure_area_decrease_MCF_example_1}
\end{figure}

\gdef\gplbacktext{}%
\gdef\gplfronttext{}%
\begin{figure}
\centering
\begin{picture}(5102.00,3400.00)%
    \gplgaddtomacro\gplbacktext{%
      \csname LTb\endcsname%
      \put(396,382){\makebox(0,0)[r]{\strut{} 4.8}}%
      \csname LTb\endcsname%
      \put(396,927){\makebox(0,0)[r]{\strut{} 5}}%
      \csname LTb\endcsname%
      \put(396,1471){\makebox(0,0)[r]{\strut{} 5.2}}%
      \csname LTb\endcsname%
      \put(396,2016){\makebox(0,0)[r]{\strut{} 5.4}}%
      \csname LTb\endcsname%
      \put(396,2560){\makebox(0,0)[r]{\strut{} 5.6}}%
      \csname LTb\endcsname%
      \put(396,3105){\makebox(0,0)[r]{\strut{} 5.8}}%
      \csname LTb\endcsname%
      \put(528,-110){\makebox(0,0){\strut{}0.0}}%
      \csname LTb\endcsname%
      \put(1668,-110){\makebox(0,0){\strut{}0.0025}}%
      \csname LTb\endcsname%
      \put(2808,-110){\makebox(0,0){\strut{}0.005}}%
      \csname LTb\endcsname%
      \put(3947,-110){\makebox(0,0){\strut{}0.0075}}%
      \csname LTb\endcsname%
      \put(5087,-110){\makebox(0,0){\strut{}0.01}}%
      \put(-374,1743){\rotatebox{-270}{\makebox(0,0){\strut{}Area}}}%
      \put(2807,-440){\makebox(0,0){\strut{}Time}}%
    }%
    \gplgaddtomacro\gplfronttext{%
      \csname LTb\endcsname%
      \put(4100,3204){\makebox(0,0)[r]{\strut{}BGN, $\tau =10^{-4}$}}%
      \csname LTb\endcsname%
      \put(4100,2984){\makebox(0,0)[r]{\strut{}$\alpha =10^{-4}, \tau =10^{-4}$}}%
      \csname LTb\endcsname%
      \put(4100,2764){\makebox(0,0)[r]{\strut{}$\alpha =10^{-4}, \tau =10^{-5}$}}%
      \csname LTb\endcsname%
      \put(4100,2544){\makebox(0,0)[r]{\strut{}$\alpha =10^{-4}, \tau =10^{-6}$}}%
      \csname LTb\endcsname%
      \put(4100,2324){\makebox(0,0)[r]{\strut{}$\alpha =10^{-4}, \tau =10^{-7}$}}%
      \csname LTb\endcsname%
      \put(4100,2104){\makebox(0,0)[r]{\strut{}$\alpha =10^{-4}, \tau =10^{-8}$}}%
    }%
    \gplbacktext
  \put(0,0){\includegraphics{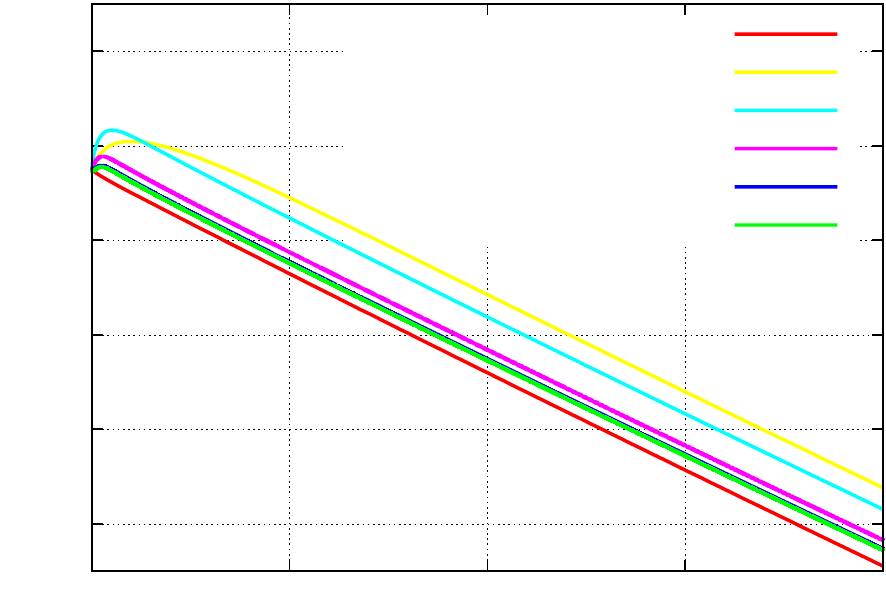}}%
    \gplfronttext
  \end{picture}%
  \vspace*{12mm}
  \caption{The image shows the behaviour of the area under the mean curvature flow 
  for the BGN-scheme (2.25) in \cite{BGN08} and
  for Algorithm \ref{algorithm_harmonic_MCF} with $\alpha = 10^{-4}$ for different time step sizes $\tau$. 
  The initial surface of the simulation is shown in Figure \ref{initial_surface_MCF_example_1}. 
  The area increase of the $\alpha$-scheme clearly depends on the time step size.
  Smaller time step sizes lead to a significant drop of the absolute increase of the surface area.  
  See Example 1 of Section \ref{Numerical_results_MCF} for further details.}
  \label{Figure_dependence_of_area_decrease_on_time_step_size}
\end{figure}
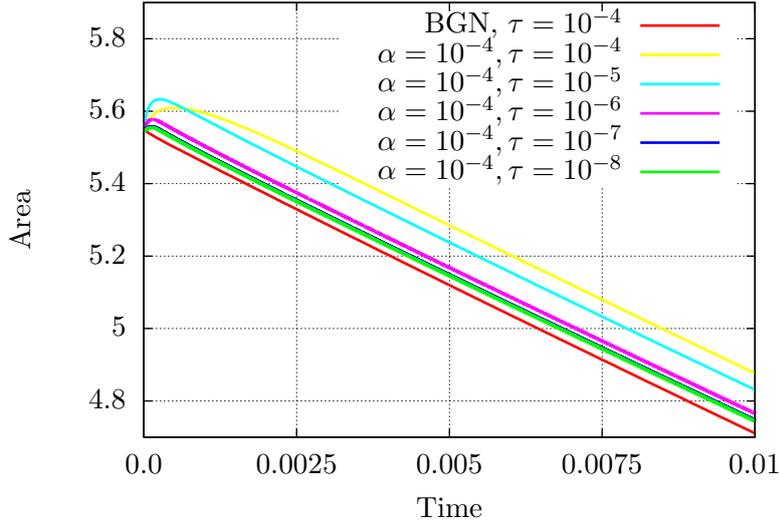

\gdef\gplbacktext{}%
\gdef\gplfronttext{}%
\begin{figure}
\centering
 \begin{picture}(7936.00,4534.00)%
    \gplgaddtomacro\gplbacktext{%
      \csname LTb\endcsname%
      \put(660,110){\makebox(0,0)[r]{\strut{} 5}}%
      \csname LTb\endcsname%
      \put(660,599){\makebox(0,0)[r]{\strut{} 10}}%
      \csname LTb\endcsname%
      \put(660,1088){\makebox(0,0)[r]{\strut{} 15}}%
      \csname LTb\endcsname%
      \put(660,1577){\makebox(0,0)[r]{\strut{} 20}}%
      \csname LTb\endcsname%
      \put(660,2066){\makebox(0,0)[r]{\strut{} 25}}%
      \csname LTb\endcsname%
      \put(660,2555){\makebox(0,0)[r]{\strut{} 30}}%
      \csname LTb\endcsname%
      \put(660,3044){\makebox(0,0)[r]{\strut{} 35}}%
      \csname LTb\endcsname%
      \put(660,3533){\makebox(0,0)[r]{\strut{} 40}}%
      \csname LTb\endcsname%
      \put(660,4022){\makebox(0,0)[r]{\strut{} 45}}%
      \csname LTb\endcsname%
      \put(660,4511){\makebox(0,0)[r]{\strut{} 50}}%
      \csname LTb\endcsname%
      \put(797,-110){\makebox(0,0){\strut{}0.0}}%
      \csname LTb\endcsname%
      \put(1850,-110){\makebox(0,0){\strut{}0.02}}%
      \csname LTb\endcsname%
      \put(2902,-110){\makebox(0,0){\strut{}0.04}}%
      \csname LTb\endcsname%
      \put(3954,-110){\makebox(0,0){\strut{}0.06}}%
      \put(22,2310){\rotatebox{-270}{\makebox(0,0){\strut{}$\sigma_{max}$}}}%
      \put(2373,-440){\makebox(0,0){\strut{}Time}}%
    }%
    \gplgaddtomacro\gplfronttext{%
      \csname LTb\endcsname%
      \put(2967,4338){\makebox(0,0)[r]{\strut{}$\mbox{Alg. 2, } \alpha = 1.0$}}%
      \csname LTb\endcsname%
      \put(2967,4118){\makebox(0,0)[r]{\strut{}$\mbox{Alg. 2, } \alpha = 0.1$}}%
      \csname LTb\endcsname%
      \put(2967,3898){\makebox(0,0)[r]{\strut{}$\mbox{Alg. 2, } \alpha = 0.01$}}%
      \csname LTb\endcsname%
      \put(2967,3678){\makebox(0,0)[r]{\strut{}$\mbox{Alg. 2, } \alpha =10^{-3}$}}%
    }%
    \gplgaddtomacro\gplbacktext{%
      \csname LTb\endcsname%
      \put(4628,110){\makebox(0,0)[r]{\strut{} 5}}%
      \csname LTb\endcsname%
      \put(4628,599){\makebox(0,0)[r]{\strut{} 10}}%
      \csname LTb\endcsname%
      \put(4628,1088){\makebox(0,0)[r]{\strut{} 15}}%
      \csname LTb\endcsname%
      \put(4628,1577){\makebox(0,0)[r]{\strut{} 20}}%
      \csname LTb\endcsname%
      \put(4628,2066){\makebox(0,0)[r]{\strut{} 25}}%
      \csname LTb\endcsname%
      \put(4628,2555){\makebox(0,0)[r]{\strut{} 30}}%
      \csname LTb\endcsname%
      \put(4628,3044){\makebox(0,0)[r]{\strut{} 35}}%
      \csname LTb\endcsname%
      \put(4628,3533){\makebox(0,0)[r]{\strut{} 40}}%
      \csname LTb\endcsname%
      \put(4628,4022){\makebox(0,0)[r]{\strut{} 45}}%
      \csname LTb\endcsname%
      \put(4628,4511){\makebox(0,0)[r]{\strut{} 50}}%
      \csname LTb\endcsname%
      \put(4765,-110){\makebox(0,0){\strut{}0.0}}%
      \csname LTb\endcsname%
      \put(5817,-110){\makebox(0,0){\strut{}0.02}}%
      \csname LTb\endcsname%
      \put(6869,-110){\makebox(0,0){\strut{}0.04}}%
      \csname LTb\endcsname%
      \put(7921,-110){\makebox(0,0){\strut{}0.06}}%
      \put(6340,-440){\makebox(0,0){\strut{}Time}}%
    }%
    \gplgaddtomacro\gplfronttext{%
      \csname LTb\endcsname%
      \put(6934,4338){\makebox(0,0)[r]{\strut{}$\mbox{Alg. 3, } \alpha = 1.0$}}%
      \csname LTb\endcsname%
      \put(6934,4118){\makebox(0,0)[r]{\strut{}$\mbox{Alg. 3, } \alpha = 0.1$}}%
      \csname LTb\endcsname%
      \put(6934,3898){\makebox(0,0)[r]{\strut{}$\mbox{Alg. 3, } \alpha = 0.01$}}%
      \csname LTb\endcsname%
      \put(6934,3678){\makebox(0,0)[r]{\strut{}$\mbox{Alg. 3, } \alpha =10^{-3}$}}%
    }%
    \gplbacktext
    \put(0,0){\includegraphics{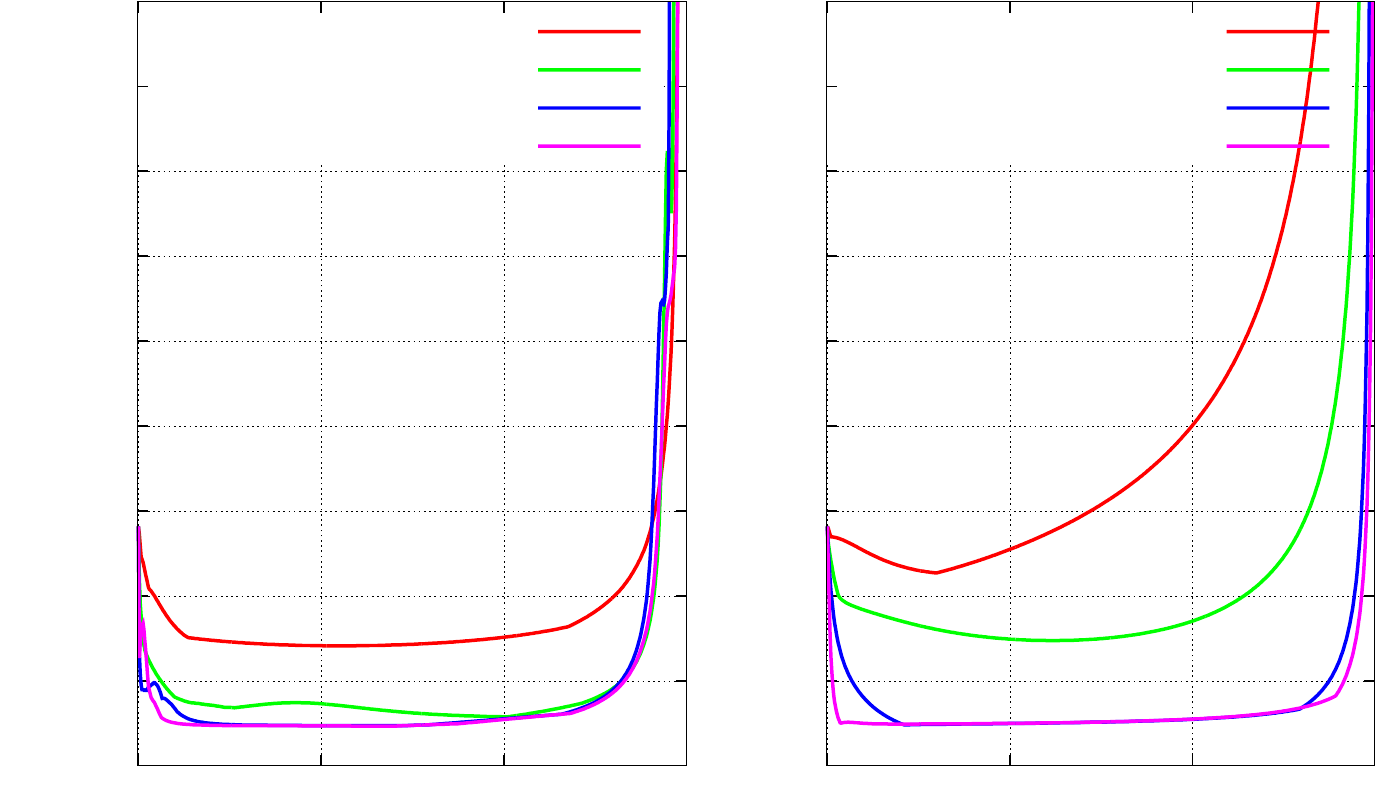}}%
    \gplfronttext
  \end{picture}%
\vspace*{12mm}
  \caption{Comparison of Algorithms \ref{algo_moving_hypersurface} and \ref{algorithm_harmonic_MCF} with respect to the 
  quantity $\sigma_{max}$ for different choices of $\alpha$. 
  The reference surface and the initial surface are shown in Figures \ref{reference_surface_MCF_example_1}
  and \ref{initial_surface_MCF_example_1}. The time step size was $\tau = 10^{-4}$.
  The experiment shows that for $\alpha = 0.01$ and $\alpha= 10^{-3}$ both algorithms
  produce good meshes. For these choices of $\alpha$, Algorithm \ref{algorithm_harmonic_MCF} seems to have a slightly better
  behaviour close to the surface singularity, 
  whereas Algorithm \ref{algo_moving_hypersurface} has a better mesh performance for 
  $\alpha = 1.0$ and $\alpha = 0.1$.  
  See Example 1 of Section \ref{Numerical_results_MCF} for further details.}
  \label{comparison_alg_3_and_4_example_1}
\end{figure}
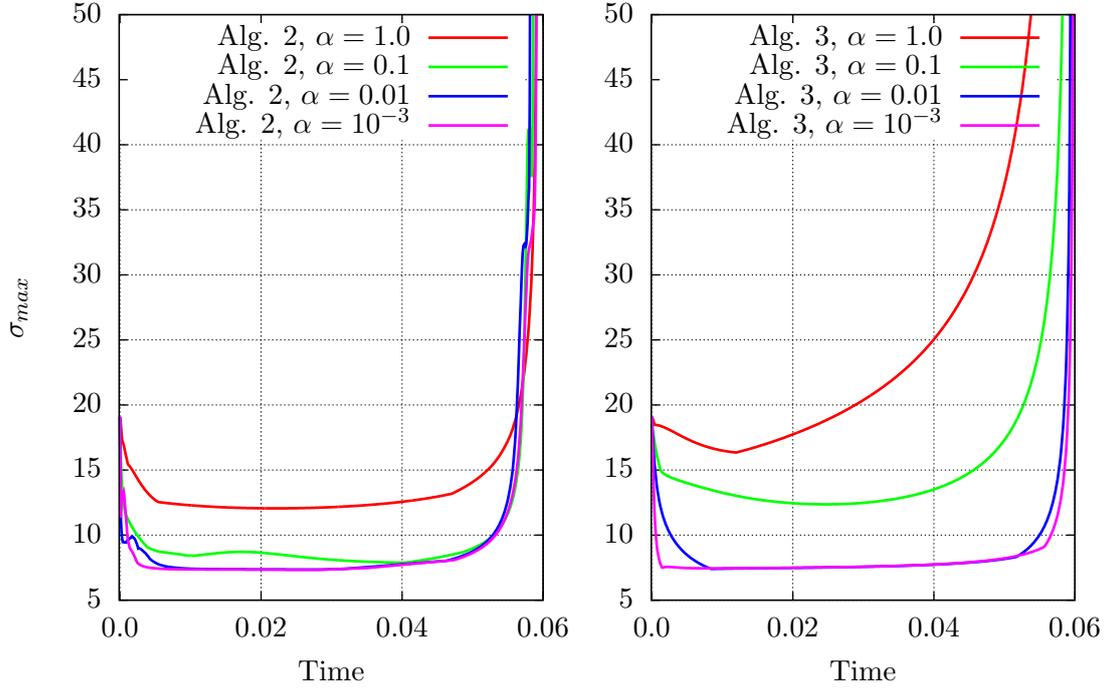

\gdef\gplbacktext{}%
\gdef\gplfronttext{}%
\begin{figure}
\centering
\begin{picture}(5102.00,3400.00)%
    \gplgaddtomacro\gplbacktext{%
      \csname LTb\endcsname%
      \put(396,110){\makebox(0,0)[r]{\strut{} 5}}%
      \csname LTb\endcsname%
      \put(396,473){\makebox(0,0)[r]{\strut{} 10}}%
      \csname LTb\endcsname%
      \put(396,836){\makebox(0,0)[r]{\strut{} 15}}%
      \csname LTb\endcsname%
      \put(396,1199){\makebox(0,0)[r]{\strut{} 20}}%
      \csname LTb\endcsname%
      \put(396,1562){\makebox(0,0)[r]{\strut{} 25}}%
      \csname LTb\endcsname%
      \put(396,1925){\makebox(0,0)[r]{\strut{} 30}}%
      \csname LTb\endcsname%
      \put(396,2288){\makebox(0,0)[r]{\strut{} 35}}%
      \csname LTb\endcsname%
      \put(396,2651){\makebox(0,0)[r]{\strut{} 40}}%
      \csname LTb\endcsname%
      \put(396,3014){\makebox(0,0)[r]{\strut{} 45}}%
      \csname LTb\endcsname%
      \put(396,3377){\makebox(0,0)[r]{\strut{} 50}}%
      \csname LTb\endcsname%
      \put(536,-110){\makebox(0,0){\strut{}0.0}}%
      \csname LTb\endcsname%
      \put(2053,-110){\makebox(0,0){\strut{}0.02}}%
      \csname LTb\endcsname%
      \put(3570,-110){\makebox(0,0){\strut{}0.04}}%
      \csname LTb\endcsname%
      \put(5087,-110){\makebox(0,0){\strut{}0.06}}%
      \put(-242,1743){\rotatebox{-270}{\makebox(0,0){\strut{}$\sigma_{max}$}}}%
      \put(2807,-440){\makebox(0,0){\strut{}Time}}%
    }%
    \gplgaddtomacro\gplfronttext{%
      \csname LTb\endcsname%
      \put(1716,3204){\makebox(0,0)[r]{\strut{}BGN}}%
      \csname LTb\endcsname%
      \put(1716,2984){\makebox(0,0)[r]{\strut{}$\alpha = 1.0$}}%
      \csname LTb\endcsname%
      \put(1716,2764){\makebox(0,0)[r]{\strut{}$\alpha = 0.1$}}%
      \csname LTb\endcsname%
      \put(1716,2544){\makebox(0,0)[r]{\strut{}$\alpha = 0.01$}}%
      \csname LTb\endcsname%
      \put(1716,2324){\makebox(0,0)[r]{\strut{}$\alpha =10^{-3}$}}%
      \csname LTb\endcsname%
      \put(1716,2104){\makebox(0,0)[r]{\strut{}$\alpha =10^{-4}$}}%
    }%
    \gplbacktext
\put(0,0){\includegraphics{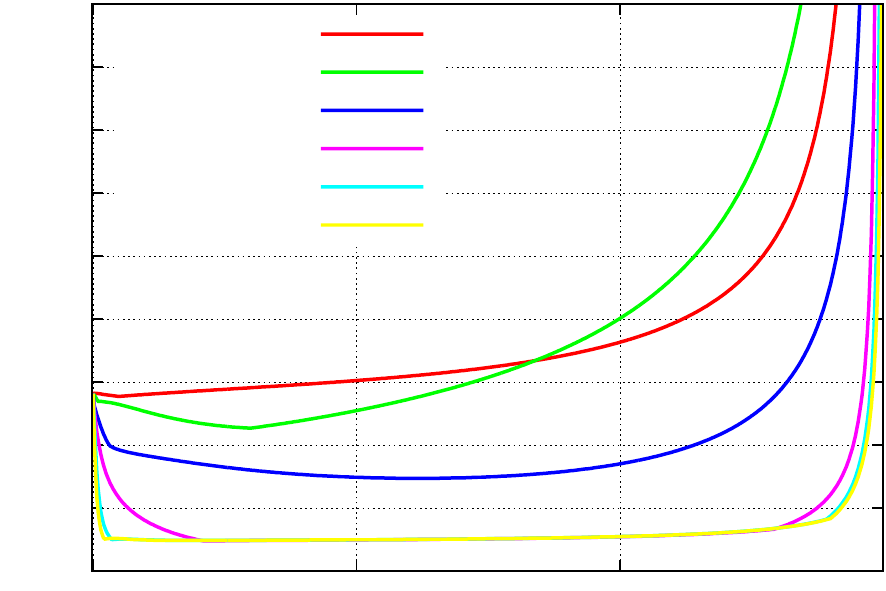}}%
    \gplfronttext
  \end{picture}%
  \vspace*{12mm}
  \caption{The image shows the behaviour of the mesh quality described by the quantity $\sigma_{max}$
  for the BGN-scheme (2.25) in \cite{BGN08} and
  for Algorithm \ref{algorithm_harmonic_MCF} for different choices of $\alpha$. 
  The initial surface is shown in Figure \ref{initial_surface_MCF_example_1}. 
  The reference surface for the $\alpha$-scheme is presented in Figure 
  \ref{reference_surface_MCF_example_1}. In this example, the $\alpha$-schemes 
  with $\alpha \leq 0.1$ outperform the benchmark scheme (2.25) in \cite{BGN08} with respect 
  to the mesh quality.
  See Example 1 of Section \ref{Numerical_results_MCF} for further details.	  
  }
  \label{Figure_mesh_properties_MCF_example_1}
\end{figure}
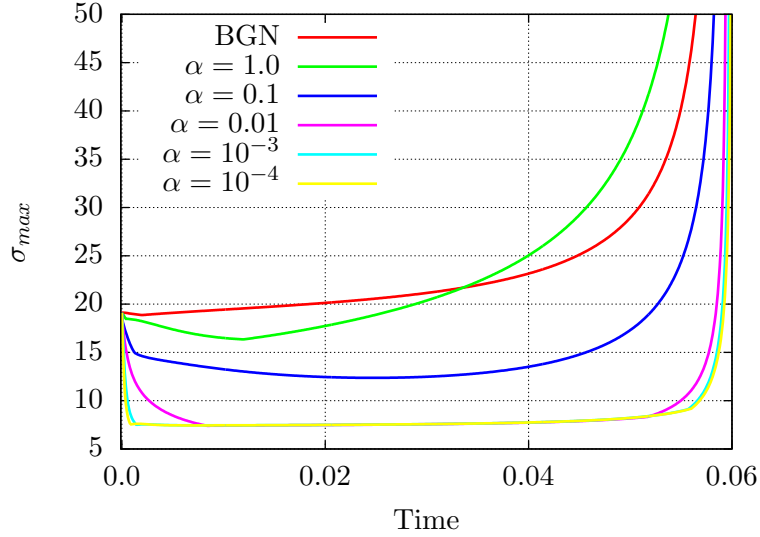

\gdef\gplbacktext{}%
\gdef\gplfronttext{}%
\begin{figure}
\centering
\begin{picture}(5102.00,3400.00)%
    \gplgaddtomacro\gplbacktext{%
      \csname LTb\endcsname%
      \put(396,302){\makebox(0,0)[r]{\strut{} 7}}%
      \csname LTb\endcsname%
      \put(396,687){\makebox(0,0)[r]{\strut{} 8}}%
      \csname LTb\endcsname%
      \put(396,1071){\makebox(0,0)[r]{\strut{} 9}}%
      \csname LTb\endcsname%
      \put(396,1455){\makebox(0,0)[r]{\strut{} 10}}%
      \csname LTb\endcsname%
      \put(396,1840){\makebox(0,0)[r]{\strut{} 11}}%
      \csname LTb\endcsname%
      \put(396,2224){\makebox(0,0)[r]{\strut{} 12}}%
      \csname LTb\endcsname%
      \put(396,2608){\makebox(0,0)[r]{\strut{} 13}}%
      \csname LTb\endcsname%
      \put(396,2993){\makebox(0,0)[r]{\strut{} 14}}%
      \csname LTb\endcsname%
      \put(396,3377){\makebox(0,0)[r]{\strut{} 15}}%
      \csname LTb\endcsname%
      \put(536,-110){\makebox(0,0){\strut{}0.0}}%
      \csname LTb\endcsname%
      \put(2053,-110){\makebox(0,0){\strut{}0.02}}%
      \csname LTb\endcsname%
      \put(3570,-110){\makebox(0,0){\strut{}0.04}}%
      \csname LTb\endcsname%
      \put(5087,-110){\makebox(0,0){\strut{}0.06}}%
      \put(-242,1743){\rotatebox{-270}{\makebox(0,0){\strut{}$\sigma_{max}$}}}%
      \put(2807,-440){\makebox(0,0){\strut{}Time}}%
    }%
    \gplgaddtomacro\gplfronttext{%
      \csname LTb\endcsname%
      \put(3502,3204){\makebox(0,0)[r]{\strut{}$n=3, \alpha = \tau = 10^{-4}$}}%
      \csname LTb\endcsname%
      \put(3502,2984){\makebox(0,0)[r]{\strut{}$n=4, \alpha = \tau = 10^{-4}$}}%
      \csname LTb\endcsname%
      \put(3502,2764){\makebox(0,0)[r]{\strut{}$n=5, \alpha = \tau = 10^{-4}$}}%
      \csname LTb\endcsname%
      \put(3502,2544){\makebox(0,0)[r]{\strut{}$n=6, \alpha = \tau = 10^{-4}$}}%
    }%
    \gplbacktext
    \put(0,0){\includegraphics{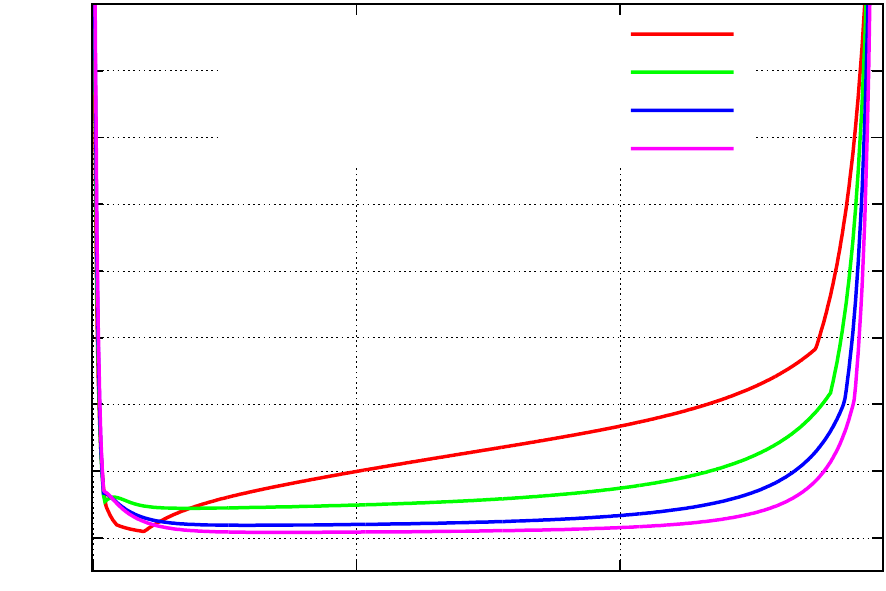}}%
    \gplfronttext
  \end{picture}%
  \vspace*{12mm}
  \caption{The image shows the behaviour of the mesh quantity $\sigma_{max}$
  for Algorithm \ref{algorithm_harmonic_MCF} with $\alpha = 10^{-4}$ and 
  for different global mesh refinements $n$ of the macro triangulation. 
  In each refinement step the simplices are bisected twice so that the maximal diameter of the simplices
  is approximately halved in each step.
  The reference surface and the initial surface for $n=4$ are shown in Figures 
  \ref{reference_surface_MCF_example_1} and \ref{initial_surface_MCF_example_1}. 
  For smaller mesh sizes (that is a higher number of
  global mesh refinements) the positive properties of
  the $\alpha$-scheme with respect to the mesh quantity $\sigma_{max}$ become more pronounced. 
  In Figure \ref{Figure_mesh_properties_MCF_example_1} the
  mesh properties of the BGN-scheme (2.25) in \cite{BGN08} and
  of Algorithm \ref{algorithm_harmonic_MCF} are compared for the global refinement $n=4$
  and different choices of $\alpha$.
  See Example 1 of Section \ref{Numerical_results_MCF} for further details.	  
  }
  \label{Figure_dependence_sigma_max_mesh_refinements}
\end{figure}

\subsection*{Example 2:}
We now change the local surface parametrization to be
\begin{equation*}
	X_0(\theta, \varphi) :=
	\left( 
	\begin{array}{c}	
		\cos \varphi \\
		(0.6 \cos^2 \varphi + 0.4) \cos \theta  \sin \varphi \\
		(0.6 \cos^2 \varphi + 0.4) \sin \theta  \sin \varphi
	\end{array}		
	\right),
	\quad \theta \in [0,2\pi), \varphi \in [0,\pi].
\end{equation*}
The radius at the neck of the initial surface is now equal to $0.4$ instead of $0.3$
as in Example $1$.
In this case, the mean curvature flow does not develop a neck pinch singularity,
but shrinks to a round sphere, see Figure \ref{Figure_simulation_MCF_shrinking_sphere}.
The decrease of the surface area is presented in Figure \ref{Figure_area_decrease_MCF_example_2}.
For a round shrinking sphere the area should decrease linearly in time.
However, for both schemes and fixed time step size $\tau$,
the area does not decrease linearly close to the singularity of the surface.
We, therefore, couple the time step size to the maximal diameter $h$ of the mesh
by $\tau = 0.001 h$ and $\tau = 0.01 h^2$. This leads to a linear decrease of the surface area also
for times close to the surface singularity. In Figure \ref{Figure_mesh_properties_for_fixed_alpha_example_2}
the comparison of the mesh quantity $\sigma_{max}$ is presented. Away from the singularity,
Algorithm \ref{algorithm_harmonic_MCF} outperforms the BGN-scheme.
The BGN-scheme, however, shows better mesh behaviour at the singularity.
Coupling the parameter $\alpha$ to the time step size $\tau$ 
and thus to the mesh size $h$,
a good mesh behaviour at the singularity is also achievable for Algorithm \ref{algorithm_harmonic_MCF}, 
see Figure \ref{Figure_mesh_properties_for_alpha_is_tau_example_2}.
Please note that such a coupling is not necessary for Algorithm \ref{algo_moving_hypersurface},
see left image in Figure \ref{comparison_alg_3_and_4_example_2}.
In this example, Algorithm \ref{algo_moving_hypersurface} clearly shows better mesh behaviour
than Algorithm \ref{algorithm_harmonic_MCF}.
 
\begin{figure}
\begin{center}
~~~~~
\subfloat[][\centering Reference surface for \mbox{Algorithm \ref{algorithm_harmonic_MCF}.}]
{\includegraphics[width=0.3\textwidth]{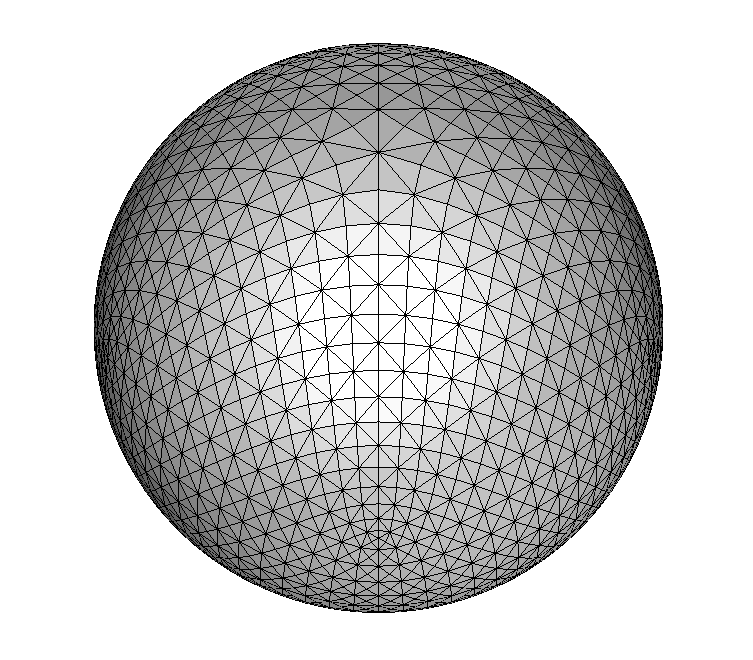}
\label{reference_surface_MCF_example_2}
} 
~~~~~~
\subfloat[][\centering Surface at time $t = 0.0$. The surface area is $6.330$.]
{\includegraphics[width=0.4\textwidth]{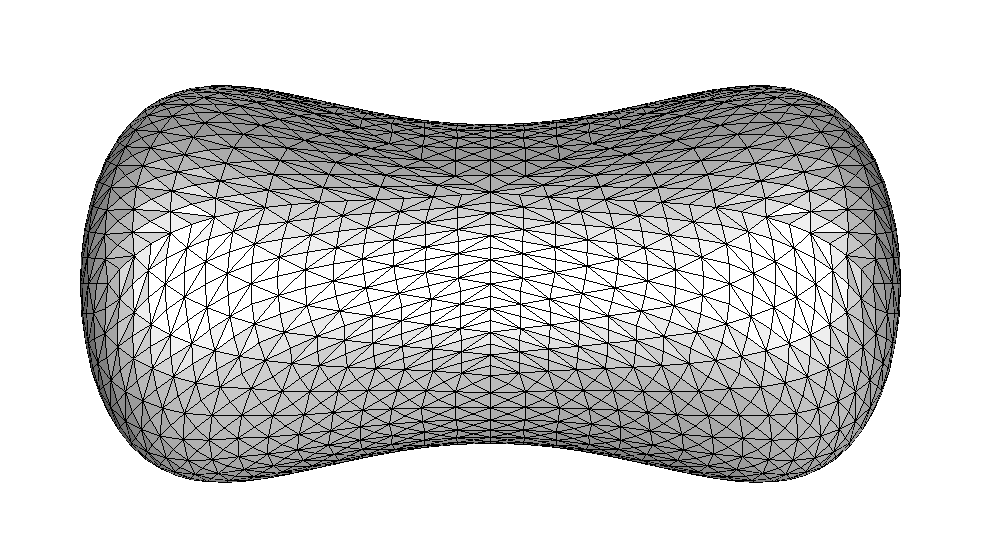}
\label{start_surface_MCF_example_2}
} \\

\subfloat[][\centering Algorithm \ref{algorithm_harmonic_MCF}
at time $t=0.0751$. The surface area is $1.221$.]
{\includegraphics[width=0.4\textwidth]{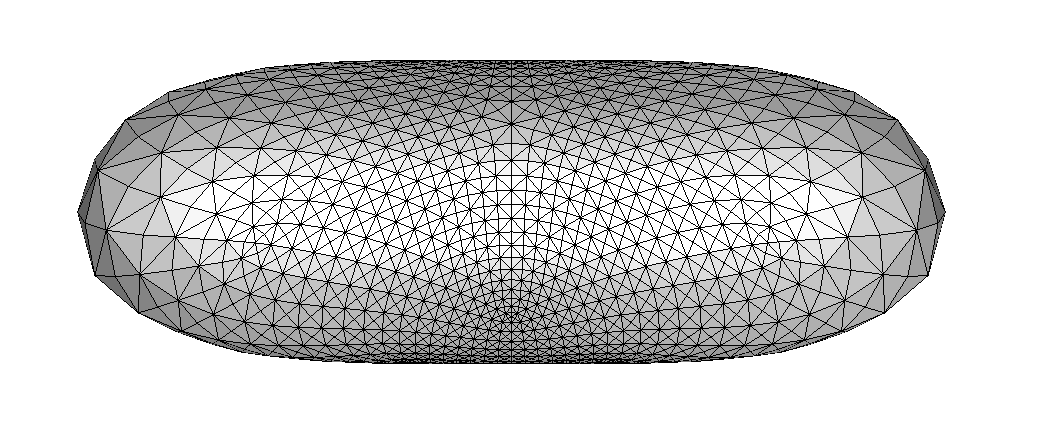}} 
\subfloat[][\centering BGN-scheme at time $t = 0.0750$. The surface area is $1.071$.]
{\includegraphics[width=0.4\textwidth]{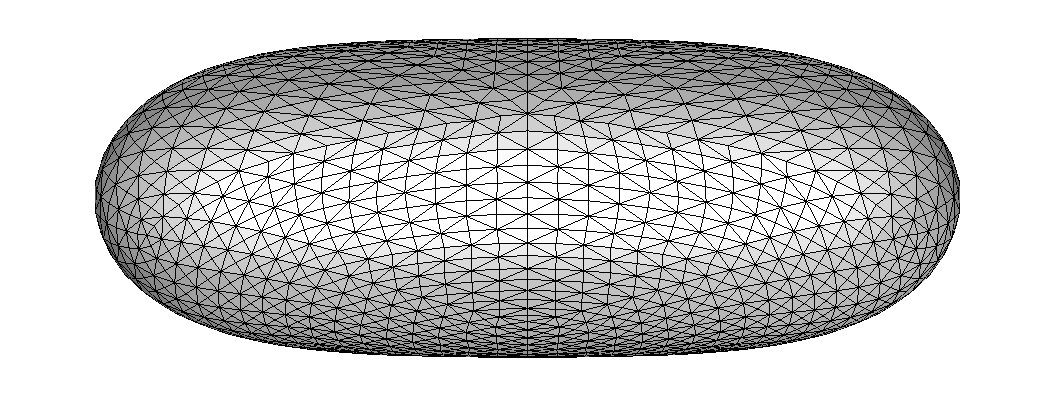}}\\
~
\subfloat[][\centering Algorithm \ref{algorithm_harmonic_MCF}
at time $t=0.0928$. The surface area is $0.00001$.]
{\includegraphics[width=0.3\textwidth]{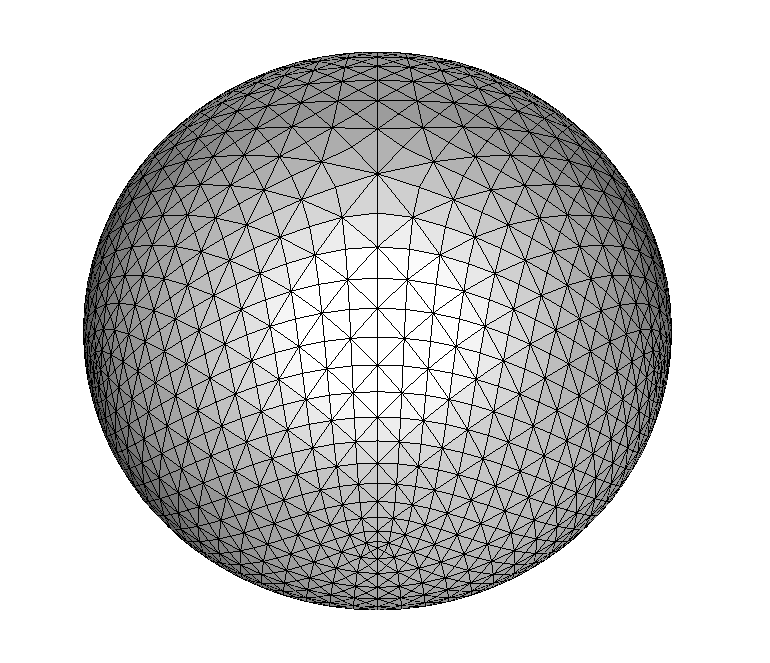}}
~~~~~~~~~~
\subfloat[][\centering BGN-scheme at time $t = 0.0908$. The surface area is $0.00001$.]{\includegraphics[width=0.3\textwidth]{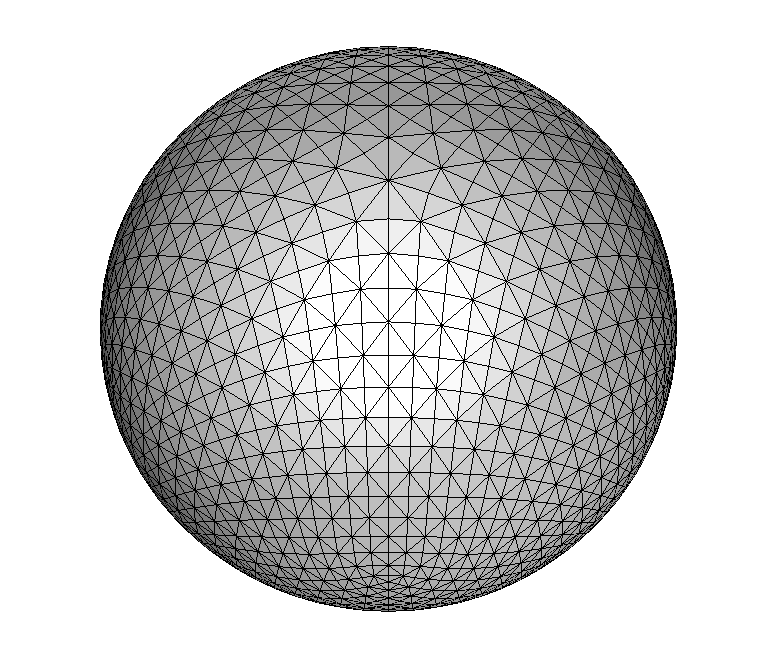}}
\caption{Comparison of Algorithm \ref{algorithm_harmonic_MCF} for $\alpha = \tau$
and the BGN-scheme (2.25) in \cite{BGN08}. The time step size in both schemes was chosen as 
$\tau = 0.01 h^2$, where $h$ is the maximal diameter of the surface triangulation.
The mesh had $5120$ triangles and $2562$ vertices.
The images are rescaled. See Example 2 of Section \ref{Numerical_results_MCF} for further details.}
\label{Figure_simulation_MCF_shrinking_sphere}
\end{center}
\end{figure}
\gdef\gplbacktext{}%
\gdef\gplfronttext{}%
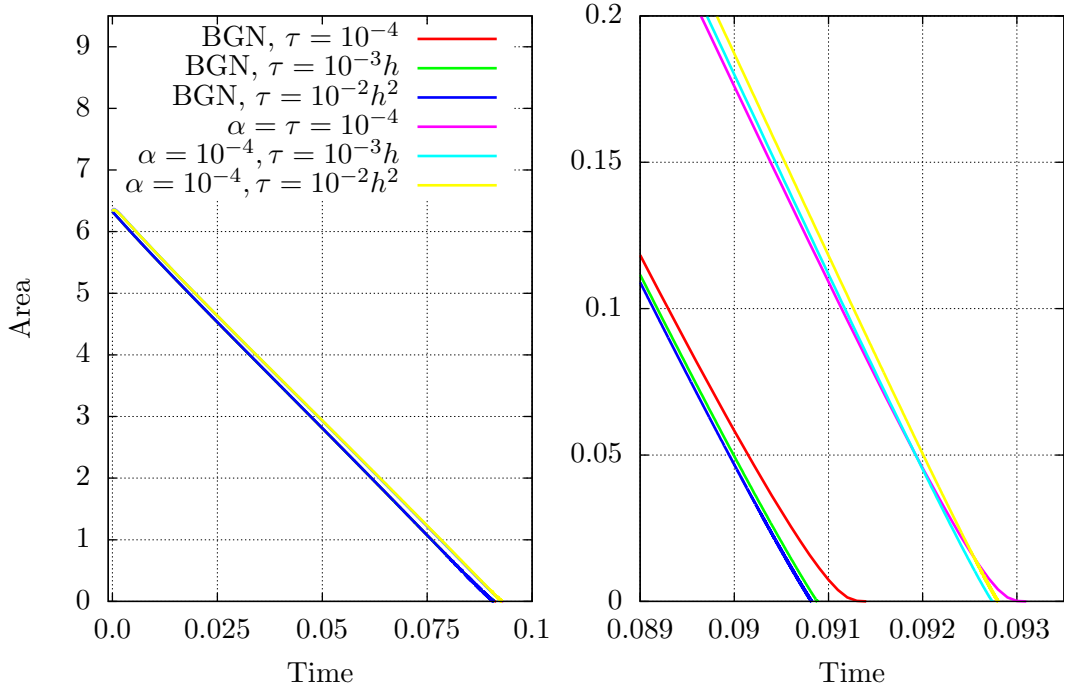
\begin{figure}
\centering
  \begin{picture}(7936.00,4534.00)%
    \gplgaddtomacro\gplbacktext{%
      \csname LTb\endcsname%
      \put(660,110){\makebox(0,0)[r]{\strut{} 0}}%
      \csname LTb\endcsname%
      \put(660,573){\makebox(0,0)[r]{\strut{} 1}}%
      \csname LTb\endcsname%
      \put(660,1037){\makebox(0,0)[r]{\strut{} 2}}%
      \csname LTb\endcsname%
      \put(660,1500){\makebox(0,0)[r]{\strut{} 3}}%
      \csname LTb\endcsname%
      \put(660,1963){\makebox(0,0)[r]{\strut{} 4}}%
      \csname LTb\endcsname%
      \put(660,2426){\makebox(0,0)[r]{\strut{} 5}}%
      \csname LTb\endcsname%
      \put(660,2890){\makebox(0,0)[r]{\strut{} 6}}%
      \csname LTb\endcsname%
      \put(660,3353){\makebox(0,0)[r]{\strut{} 7}}%
      \csname LTb\endcsname%
      \put(660,3816){\makebox(0,0)[r]{\strut{} 8}}%
      \csname LTb\endcsname%
      \put(660,4279){\makebox(0,0)[r]{\strut{} 9}}%
      \csname LTb\endcsname%
      \put(823,-110){\makebox(0,0){\strut{}0.0}}%
      \csname LTb\endcsname%
      \put(1606,-110){\makebox(0,0){\strut{}0.025}}%
      \csname LTb\endcsname%
      \put(2389,-110){\makebox(0,0){\strut{}0.05}}%
      \csname LTb\endcsname%
      \put(3171,-110){\makebox(0,0){\strut{}0.075}}%
      \csname LTb\endcsname%
      \put(3954,-110){\makebox(0,0){\strut{}0.1}}%
      \put(154,2310){\rotatebox{-270}{\makebox(0,0){\strut{}Area}}}%
      \put(2373,-440){\makebox(0,0){\strut{}Time}}%
    }%
    \gplgaddtomacro\gplfronttext{%
      \csname LTb\endcsname%
      \put(2967,4338){\makebox(0,0)[r]{\strut{}BGN, $\tau = 10^{-4}$}}%
      \csname LTb\endcsname%
      \put(2967,4118){\makebox(0,0)[r]{\strut{}BGN, $\tau = 10^{-3} h$}}%
      \csname LTb\endcsname%
      \put(2967,3898){\makebox(0,0)[r]{\strut{}BGN, $\tau = 10^{-2} h^2$}}%
      \csname LTb\endcsname%
      \put(2967,3678){\makebox(0,0)[r]{\strut{}$\alpha = \tau = 10^{-4}$}}%
      \csname LTb\endcsname%
      \put(2967,3458){\makebox(0,0)[r]{\strut{}$\alpha =10^{-4}, \tau = 10^{-3} h$}}%
      \csname LTb\endcsname%
      \put(2967,3238){\makebox(0,0)[r]{\strut{}$\alpha =10^{-4}, \tau = 10^{-2} h^2$}}%
    }%
    \gplgaddtomacro\gplbacktext{%
      \csname LTb\endcsname%
      \put(4628,110){\makebox(0,0)[r]{\strut{} 0}}%
      \csname LTb\endcsname%
      \put(4628,1210){\makebox(0,0)[r]{\strut{} 0.05}}%
      \csname LTb\endcsname%
      \put(4628,2311){\makebox(0,0)[r]{\strut{} 0.1}}%
      \csname LTb\endcsname%
      \put(4628,3411){\makebox(0,0)[r]{\strut{} 0.15}}%
      \csname LTb\endcsname%
      \put(4628,4511){\makebox(0,0)[r]{\strut{} 0.2}}%
      \csname LTb\endcsname%
      \put(4760,-110){\makebox(0,0){\strut{}0.089}}%
      \csname LTb\endcsname%
      \put(5462,-110){\makebox(0,0){\strut{}0.09}}%
      \csname LTb\endcsname%
      \put(6165,-110){\makebox(0,0){\strut{}0.091}}%
      \csname LTb\endcsname%
      \put(6867,-110){\makebox(0,0){\strut{}0.092}}%
      \csname LTb\endcsname%
      \put(7570,-110){\makebox(0,0){\strut{}0.093}}%
      \put(6340,-440){\makebox(0,0){\strut{}Time}}%
    }%
    \gplgaddtomacro\gplfronttext{%
    }%
    \gplbacktext
    \put(0,0){\includegraphics{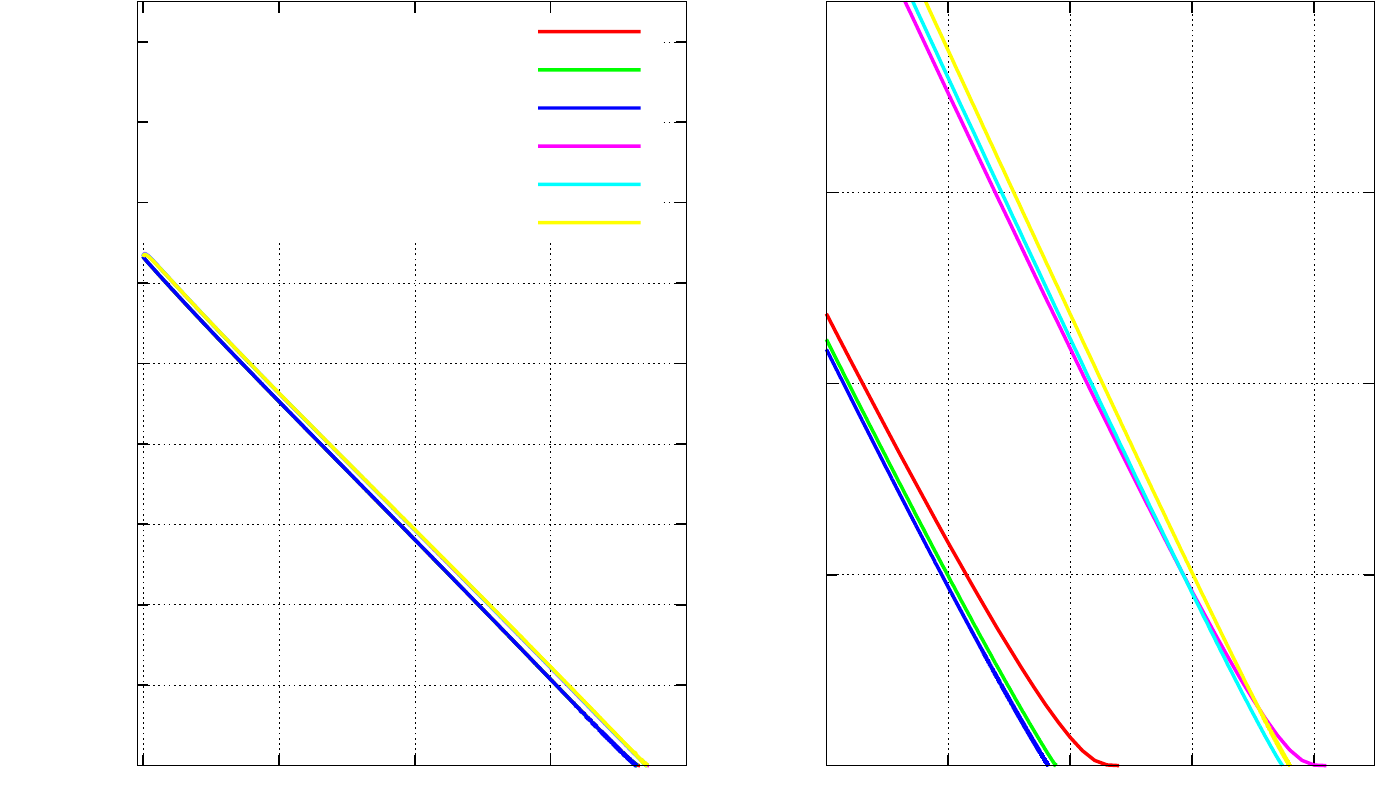}}%
    \gplfronttext
  \end{picture}%
  \vspace*{12mm}
  \caption{The images show the decrease of the surface area for the BGN-scheme (2.25) in \cite{BGN08}
  and for Algorithm \ref{algorithm_harmonic_MCF} with $\alpha = 10^{-4}$.
  The initial surface and the reference surface used 
  in Algorithm \ref{algorithm_harmonic_MCF} are presented 
  in Figures \ref{start_surface_MCF_example_2} and \ref{reference_surface_MCF_example_2}.
  The right image shows an enlarged section for times close to the surface singularity.
  The colour code is valid for both images.
  Different time step sizes were chosen. For a fixed time step size 
  $\tau = 10^{-4}$, the area does not decrease linearly close to the singularity
  (right image). However, by coupling the time step size $\tau$ to the maximal diameter $h$ 
  of the triangulation, we also obtain a linear area decrease close to the singularity. 
  See Example 2 of Section \ref{Numerical_results_MCF} for further details.
  }
  \label{Figure_area_decrease_MCF_example_2}
\end{figure}

\gdef\gplbacktext{}%
\gdef\gplfronttext{}%
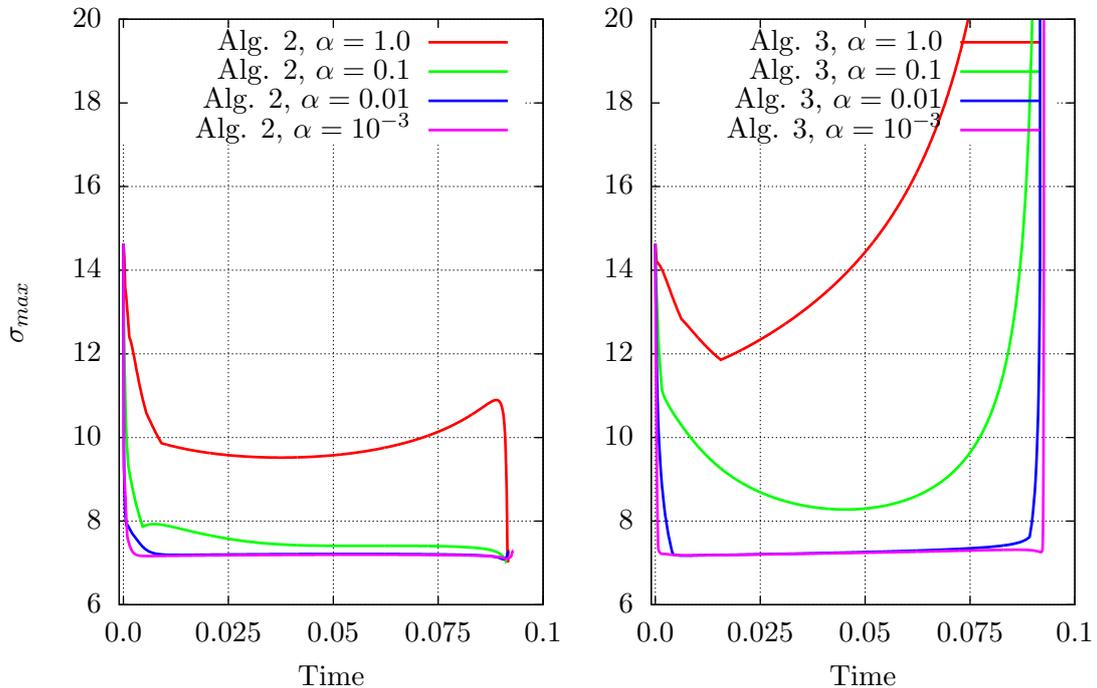
\begin{figure}
\centering
\begin{picture}(7936.00,4534.00)%
    \gplgaddtomacro\gplbacktext{%
      \csname LTb\endcsname%
      \put(660,110){\makebox(0,0)[r]{\strut{} 6}}%
      \csname LTb\endcsname%
      \put(660,739){\makebox(0,0)[r]{\strut{} 8}}%
      \csname LTb\endcsname%
      \put(660,1367){\makebox(0,0)[r]{\strut{} 10}}%
      \csname LTb\endcsname%
      \put(660,1996){\makebox(0,0)[r]{\strut{} 12}}%
      \csname LTb\endcsname%
      \put(660,2625){\makebox(0,0)[r]{\strut{} 14}}%
      \csname LTb\endcsname%
      \put(660,3254){\makebox(0,0)[r]{\strut{} 16}}%
      \csname LTb\endcsname%
      \put(660,3882){\makebox(0,0)[r]{\strut{} 18}}%
      \csname LTb\endcsname%
      \put(660,4511){\makebox(0,0)[r]{\strut{} 20}}%
      \csname LTb\endcsname%
      \put(823,-110){\makebox(0,0){\strut{}0.0}}%
      \csname LTb\endcsname%
      \put(1606,-110){\makebox(0,0){\strut{}0.025}}%
      \csname LTb\endcsname%
      \put(2389,-110){\makebox(0,0){\strut{}0.05}}%
      \csname LTb\endcsname%
      \put(3171,-110){\makebox(0,0){\strut{}0.075}}%
      \csname LTb\endcsname%
      \put(3954,-110){\makebox(0,0){\strut{}0.1}}%
      \put(22,2310){\rotatebox{-270}{\makebox(0,0){\strut{}$\sigma_{max}$}}}%
      \put(2373,-440){\makebox(0,0){\strut{}Time}}%
    }%
    \gplgaddtomacro\gplfronttext{%
      \csname LTb\endcsname%
      \put(2967,4338){\makebox(0,0)[r]{\strut{}$\mbox{Alg. 2, } \alpha = 1.0$}}%
      \csname LTb\endcsname%
      \put(2967,4118){\makebox(0,0)[r]{\strut{}$\mbox{Alg. 2, } \alpha = 0.1$}}%
      \csname LTb\endcsname%
      \put(2967,3898){\makebox(0,0)[r]{\strut{}$\mbox{Alg. 2, } \alpha = 0.01$}}%
      \csname LTb\endcsname%
      \put(2967,3678){\makebox(0,0)[r]{\strut{}$\mbox{Alg. 2, } \alpha =10^{-3}$}}%
    }%
    \gplgaddtomacro\gplbacktext{%
      \csname LTb\endcsname%
      \put(4628,110){\makebox(0,0)[r]{\strut{} 6}}%
      \csname LTb\endcsname%
      \put(4628,739){\makebox(0,0)[r]{\strut{} 8}}%
      \csname LTb\endcsname%
      \put(4628,1367){\makebox(0,0)[r]{\strut{} 10}}%
      \csname LTb\endcsname%
      \put(4628,1996){\makebox(0,0)[r]{\strut{} 12}}%
      \csname LTb\endcsname%
      \put(4628,2625){\makebox(0,0)[r]{\strut{} 14}}%
      \csname LTb\endcsname%
      \put(4628,3254){\makebox(0,0)[r]{\strut{} 16}}%
      \csname LTb\endcsname%
      \put(4628,3882){\makebox(0,0)[r]{\strut{} 18}}%
      \csname LTb\endcsname%
      \put(4628,4511){\makebox(0,0)[r]{\strut{} 20}}%
      \csname LTb\endcsname%
      \put(4791,-110){\makebox(0,0){\strut{}0.0}}%
      \csname LTb\endcsname%
      \put(5574,-110){\makebox(0,0){\strut{}0.025}}%
      \csname LTb\endcsname%
      \put(6356,-110){\makebox(0,0){\strut{}0.05}}%
      \csname LTb\endcsname%
      \put(7139,-110){\makebox(0,0){\strut{}0.075}}%
      \csname LTb\endcsname%
      \put(7921,-110){\makebox(0,0){\strut{}0.1}}%
      \put(6340,-440){\makebox(0,0){\strut{}Time}}%
    }%
    \gplgaddtomacro\gplfronttext{%
      \csname LTb\endcsname%
      \put(6934,4338){\makebox(0,0)[r]{\strut{}$\mbox{Alg. 3, } \alpha = 1.0$}}%
      \csname LTb\endcsname%
      \put(6934,4118){\makebox(0,0)[r]{\strut{}$\mbox{Alg. 3, } \alpha = 0.1$}}%
      \csname LTb\endcsname%
      \put(6934,3898){\makebox(0,0)[r]{\strut{}$\mbox{Alg. 3, } \alpha = 0.01$}}%
      \csname LTb\endcsname%
      \put(6934,3678){\makebox(0,0)[r]{\strut{}$\mbox{Alg. 3, } \alpha =10^{-3}$}}%
    }%
    \gplbacktext
    \put(0,0){\includegraphics{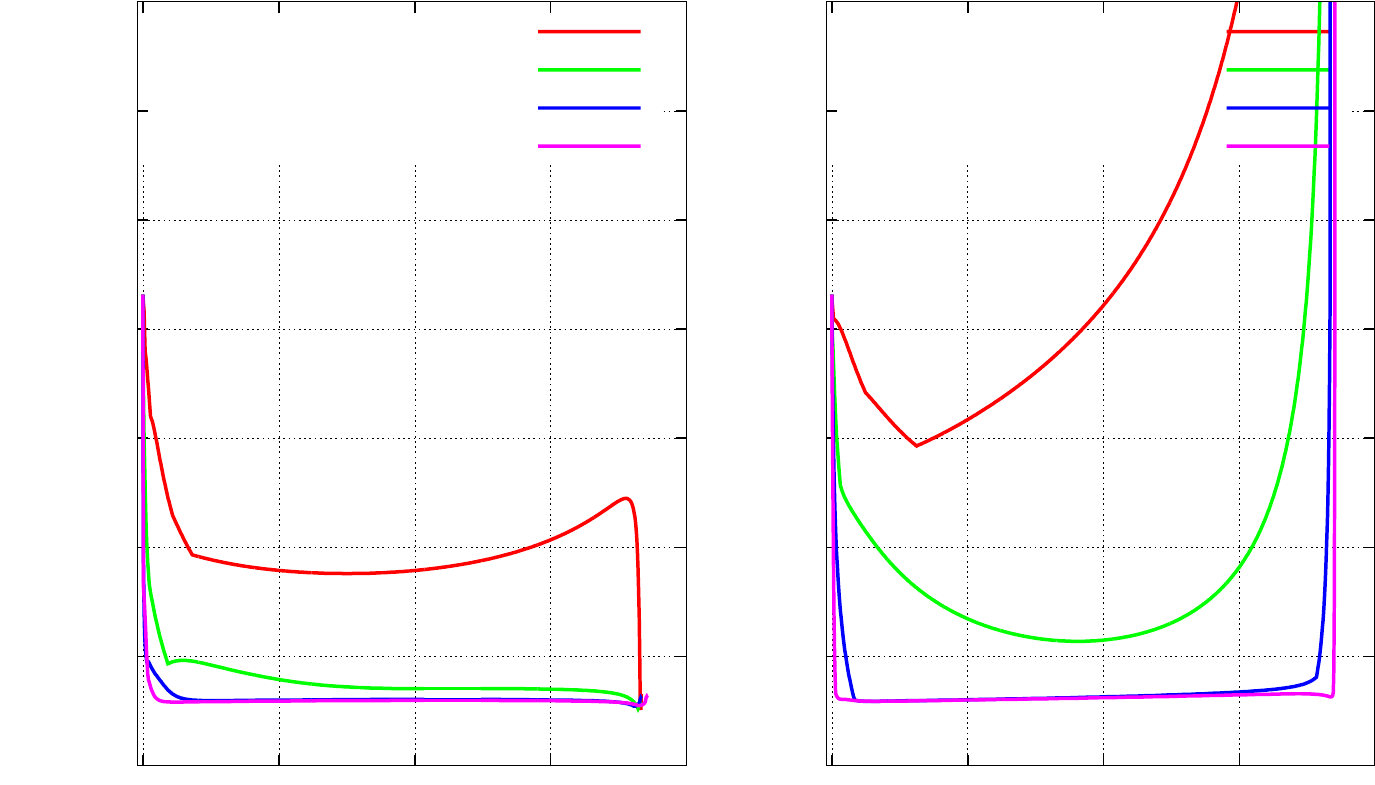}}%
    \gplfronttext
  \end{picture}%
\vspace*{12mm}
  \caption{Comparison of Algorithms \ref{algo_moving_hypersurface} and \ref{algorithm_harmonic_MCF} 
  with respect to the quantity $\sigma_{max}$ for different choices of $\alpha$. 
  The reference surface and the initial surface are shown in Figures \ref{reference_surface_MCF_example_2}
  and \ref{start_surface_MCF_example_2}. The time step size was $\tau = 10^{-4}$.
  For $\alpha = 0.01$ and $\alpha = 10^{-3}$ both algorithms have good mesh properties.
  However, Algorithm \ref{algo_moving_hypersurface} outperforms Algorithm \ref{algorithm_harmonic_MCF} for $\alpha = 1.0$
  and $\alpha = 0.1$ as well as for times close to the surface singularity. 
  For Algorithm \ref{algo_moving_hypersurface} 
  it is not necessary to couple the parameter $\alpha$ to the time step size $\tau$, 
  like in Figure \ref{Figure_mesh_properties_for_alpha_is_tau_example_2},
  in order to obtain a good mesh behaviour at the surface singularity. 
  See Example 2 of Section \ref{Numerical_results_MCF} for further details.}
  \label{comparison_alg_3_and_4_example_2}
\end{figure}

\gdef\gplbacktext{}%
\gdef\gplfronttext{}%
\begin{figure}
\centering
	 \begin{picture}(5102.00,3400.00)%
    \gplgaddtomacro\gplbacktext{%
      \csname LTb\endcsname%
      \put(396,110){\makebox(0,0)[r]{\strut{} 6}}%
      \csname LTb\endcsname%
      \put(396,577){\makebox(0,0)[r]{\strut{} 8}}%
      \csname LTb\endcsname%
      \put(396,1043){\makebox(0,0)[r]{\strut{} 10}}%
      \csname LTb\endcsname%
      \put(396,1510){\makebox(0,0)[r]{\strut{} 12}}%
      \csname LTb\endcsname%
      \put(396,1977){\makebox(0,0)[r]{\strut{} 14}}%
      \csname LTb\endcsname%
      \put(396,2444){\makebox(0,0)[r]{\strut{} 16}}%
      \csname LTb\endcsname%
      \put(396,2910){\makebox(0,0)[r]{\strut{} 18}}%
      \csname LTb\endcsname%
      \put(396,3377){\makebox(0,0)[r]{\strut{} 20}}%
      \csname LTb\endcsname%
      \put(573,-110){\makebox(0,0){\strut{}0.0}}%
      \csname LTb\endcsname%
      \put(1702,-110){\makebox(0,0){\strut{}0.025}}%
      \csname LTb\endcsname%
      \put(2830,-110){\makebox(0,0){\strut{}0.05}}%
      \csname LTb\endcsname%
      \put(3959,-110){\makebox(0,0){\strut{}0.075}}%
      \csname LTb\endcsname%
      \put(5087,-110){\makebox(0,0){\strut{}0.1}}%
      \put(-242,1743){\rotatebox{-270}{\makebox(0,0){\strut{}$\sigma_{max}$}}}%
      \put(2807,-440){\makebox(0,0){\strut{}Time}}%
    }%
    \gplgaddtomacro\gplfronttext{%
      \csname LTb\endcsname%
      \put(3300,3204){\makebox(0,0)[r]{\strut{}BGN, $\tau = 10^{-4}$}}%
      \csname LTb\endcsname%
      \put(3300,2984){\makebox(0,0)[r]{\strut{}BGN, $\tau = 0.001 h$}}%
      \csname LTb\endcsname%
      \put(3300,2764){\makebox(0,0)[r]{\strut{}BGN, $\tau = 0.01 h^2$}}%
      \csname LTb\endcsname%
      \put(3300,2544){\makebox(0,0)[r]{\strut{}$\alpha = \tau = 10^{-4}$}}%
      \csname LTb\endcsname%
      \put(3300,2324){\makebox(0,0)[r]{\strut{}$\alpha =10^{-4}, \tau = 0.001 h$}}%
      \csname LTb\endcsname%
      \put(3300,2104){\makebox(0,0)[r]{\strut{}$\alpha =10^{-4}, \tau = 0.01 h^2$}}%
    }%
    \gplbacktext
    \put(0,0){\includegraphics{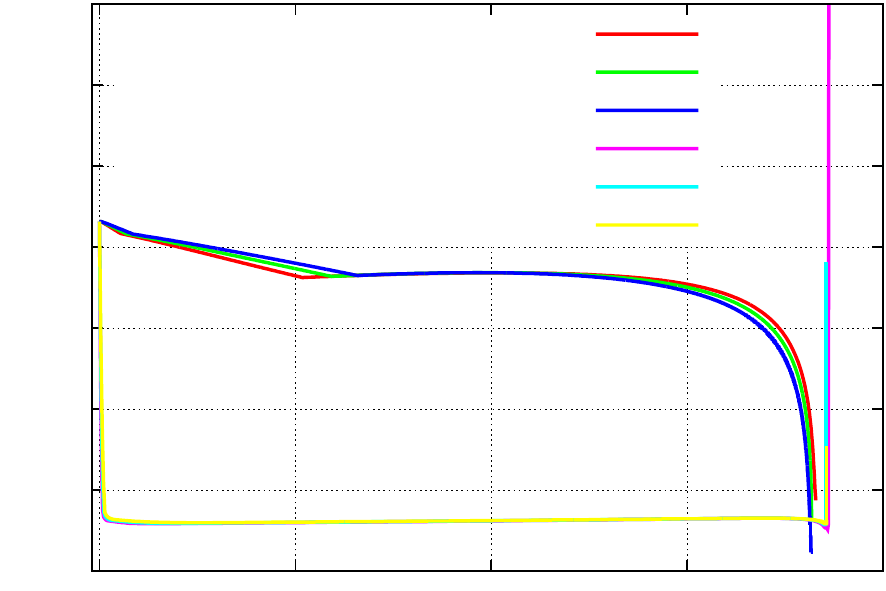}}%
    \gplfronttext
  \end{picture}%
  \vspace*{12mm}
  \caption{The image shows the behaviour of the mesh quantity $\sigma_{max}$ for the BGN-scheme
  (2.25) in \cite{BGN08} and for Algorithm \ref{algorithm_harmonic_MCF} with $\alpha = 10^{-4}$.
  The initial surface and the reference surface for Algorithm \ref{algorithm_harmonic_MCF} are presented 
  in Figures \ref{start_surface_MCF_example_2} and \ref{reference_surface_MCF_example_2}.
  Apart from the singularity, Algorithm \ref{algorithm_harmonic_MCF} shows good mesh properties.
  By coupling $\alpha$ and $\tau$, see Figure \ref{Figure_mesh_properties_for_alpha_is_tau_example_2},
  the mesh properties of Algorithm \ref{algorithm_harmonic_MCF} also remain controlled close to the surface singularity.
  For further details see Example 2 of Section \ref{Numerical_results_MCF}.}
  \label{Figure_mesh_properties_for_fixed_alpha_example_2}
\end{figure}
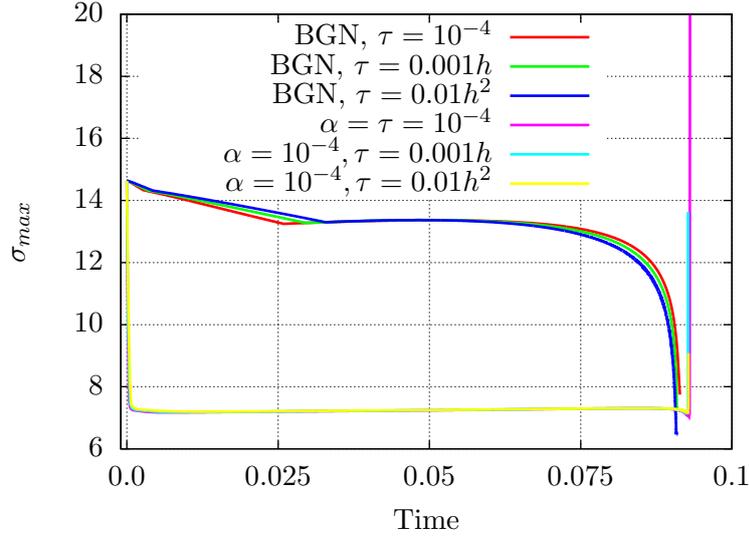

\gdef\gplbacktext{}%
\gdef\gplfronttext{}%
\begin{figure}
\centering
 \begin{picture}(5102.00,3400.00)%
    \gplgaddtomacro\gplbacktext{%
      \csname LTb\endcsname%
      \put(396,110){\makebox(0,0)[r]{\strut{} 6}}%
      \csname LTb\endcsname%
      \put(396,577){\makebox(0,0)[r]{\strut{} 8}}%
      \csname LTb\endcsname%
      \put(396,1043){\makebox(0,0)[r]{\strut{} 10}}%
      \csname LTb\endcsname%
      \put(396,1510){\makebox(0,0)[r]{\strut{} 12}}%
      \csname LTb\endcsname%
      \put(396,1977){\makebox(0,0)[r]{\strut{} 14}}%
      \csname LTb\endcsname%
      \put(396,2444){\makebox(0,0)[r]{\strut{} 16}}%
      \csname LTb\endcsname%
      \put(396,2910){\makebox(0,0)[r]{\strut{} 18}}%
      \csname LTb\endcsname%
      \put(396,3377){\makebox(0,0)[r]{\strut{} 20}}%
      \csname LTb\endcsname%
      \put(573,-110){\makebox(0,0){\strut{}0.0}}%
      \csname LTb\endcsname%
      \put(1702,-110){\makebox(0,0){\strut{}0.025}}%
      \csname LTb\endcsname%
      \put(2830,-110){\makebox(0,0){\strut{}0.05}}%
      \csname LTb\endcsname%
      \put(3959,-110){\makebox(0,0){\strut{}0.075}}%
      \csname LTb\endcsname%
      \put(5087,-110){\makebox(0,0){\strut{}0.1}}%
      \put(-242,1743){\rotatebox{-270}{\makebox(0,0){\strut{}$\sigma_{max}$}}}%
      \put(2807,-440){\makebox(0,0){\strut{}Time}}%
    }%
    \gplgaddtomacro\gplfronttext{%
      \csname LTb\endcsname%
      \put(2772,3204){\makebox(0,0)[r]{\strut{}BGN, $\tau = 10^{-4}$}}%
      \csname LTb\endcsname%
      \put(2772,2984){\makebox(0,0)[r]{\strut{}BGN, $\tau = 0.001 h$}}%
      \csname LTb\endcsname%
      \put(2772,2764){\makebox(0,0)[r]{\strut{}BGN, $\tau = 0.01 h^2$}}%
      \csname LTb\endcsname%
      \put(2772,2544){\makebox(0,0)[r]{\strut{}$\alpha = \tau = 0.001 h$}}%
      \csname LTb\endcsname%
      \put(2772,2324){\makebox(0,0)[r]{\strut{}$\alpha = \tau = 0.01 h^2$}}%
    }%
    \gplbacktext
    \put(0,0){\includegraphics{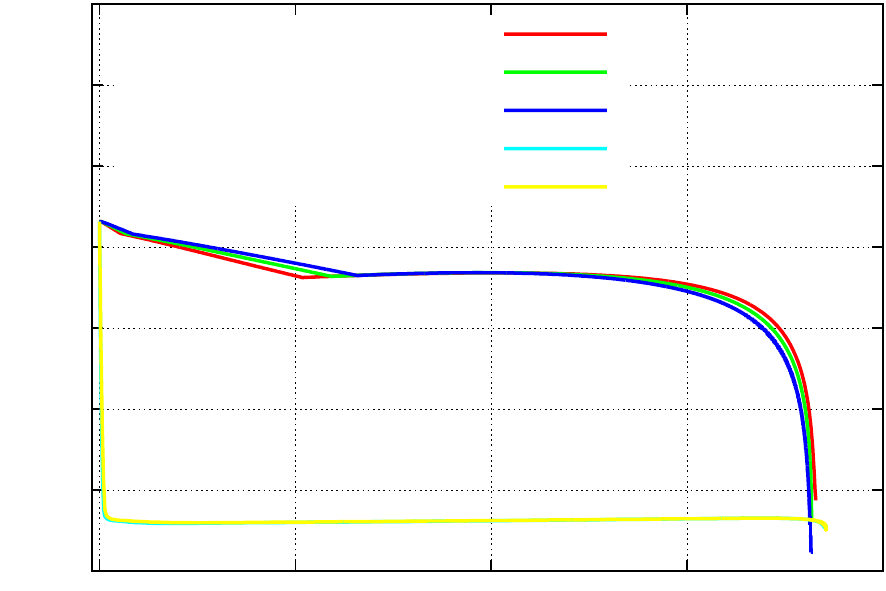}}%
    \gplfronttext
  \end{picture}%
  \vspace*{12mm}
  \caption{
  The image shows the behaviour of the mesh quantity $\sigma_{max}$ for the BGN-scheme
  (2.25) in \cite{BGN08} and for Algorithm \ref{algorithm_harmonic_MCF} with $\alpha = 10^{-4}$.
  The initial surface and the reference surface for Algorithm \ref{algorithm_harmonic_MCF} are presented 
  in Figures \ref{start_surface_MCF_example_2} and \ref{reference_surface_MCF_example_2}.
  The parameter $\alpha$ was coupled to the time step size $\tau$ for the choices $\tau = 0.001 h$ and
  $\tau = 0.01 h^2$. For these choices Algorithm \ref{algorithm_harmonic_MCF} has good mesh properties 
  even close to the singularity of the surface.
  See Example 2 of Section \ref{Numerical_results_MCF} for further details.}
  \label{Figure_mesh_properties_for_alpha_is_tau_example_2}
\end{figure}
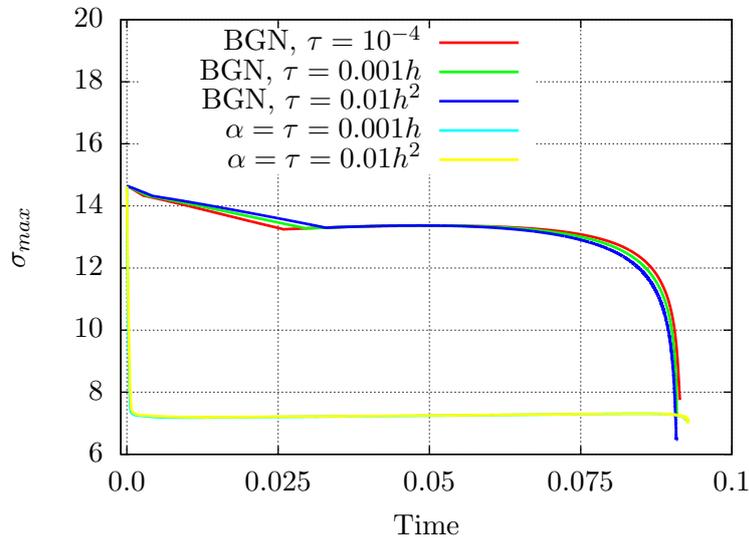

\subsection*{Example 3:}
In the last example, we consider a surface of genus one given by the local parametrization
\begin{equation*}
	X_0(\theta, \phi) :=
	\left(
		\begin{array}{c}
			(r_1 + r_2 \cos \varphi) \cos \theta \\
 			(r_1 + r_2 \cos \varphi) \sin \theta \\
 			 r_2 \sin \varphi + \tfrac{1}{5} \sin(6\theta)
		\end{array}
	\right),
	\quad \theta \in [0,2\pi), \varphi \in [0,2\pi).
\end{equation*}
We simulated the mean curvature flow for the radius $r_1 =1.0$ and 
for different choices of the radius $r_2$, that is for $r_2=0.7$, $r_2 = 0.6$ and $r_2 = 0.65$.
For $r_2 = 0.6$ the surface shrinks to a circle, see Figure \ref{circle_singularity}, 
whereas for $r_2=0.7$ it tries to converge
to a sphere developing a singularity, see Figure \ref{sphere_singularity} and Figure 4.7 in \cite{DDE05}. 
Because of this behaviour there must be a range of radii $r_2$, where 
the formation of the singularity becomes unstable in the sense that 
small changes of the initial triangulation, the mesh size $h$ or
the time step size $\tau$ can lead to the formation of different singularities.
Moreover, the formation of the singularity will even depend on the chosen algorithm.
Because of this instability, a rigorous study of the formation of the singularities
for radii $r_2$ close to $r_2 \approx 0.65$ does not make sense here.
Nevertheless, a comparison of Algorithm \ref{algorithm_harmonic_MCF} to the BGN-scheme (2.25) of \cite{BGN08}
for times far away from the singularity gives an interesting insight into the mesh
behaviour of both algorithms. As shown in Figures \ref{Figure_degenerated_tori_under_BGN} 
and \ref{Figure_mesh_properties_MCF_example_3}, the mesh of the surface degenerates
under the BGN-scheme at a time when the surface is still far away from the singularity. 
In contrast, Algorithm \ref{algorithm_harmonic_MCF} provides good meshes
with small values of $\sigma_{max}$ as long as the surface does not become singular, 
see Figures \ref{Figure_good_mesh_behaviour_tori_under_alpha_scheme}
and \ref{Figure_mesh_properties_MCF_example_3}. Algorithm \ref{algo_moving_hypersurface}
again provides a similar behaviour as Algorithm \ref{algorithm_harmonic_MCF}.

\begin{figure}
\begin{center}
\subfloat[][Reference surface for \mbox{Algorithm \ref{algorithm_harmonic_MCF}.}]
{\includegraphics[width=0.4\textwidth]
{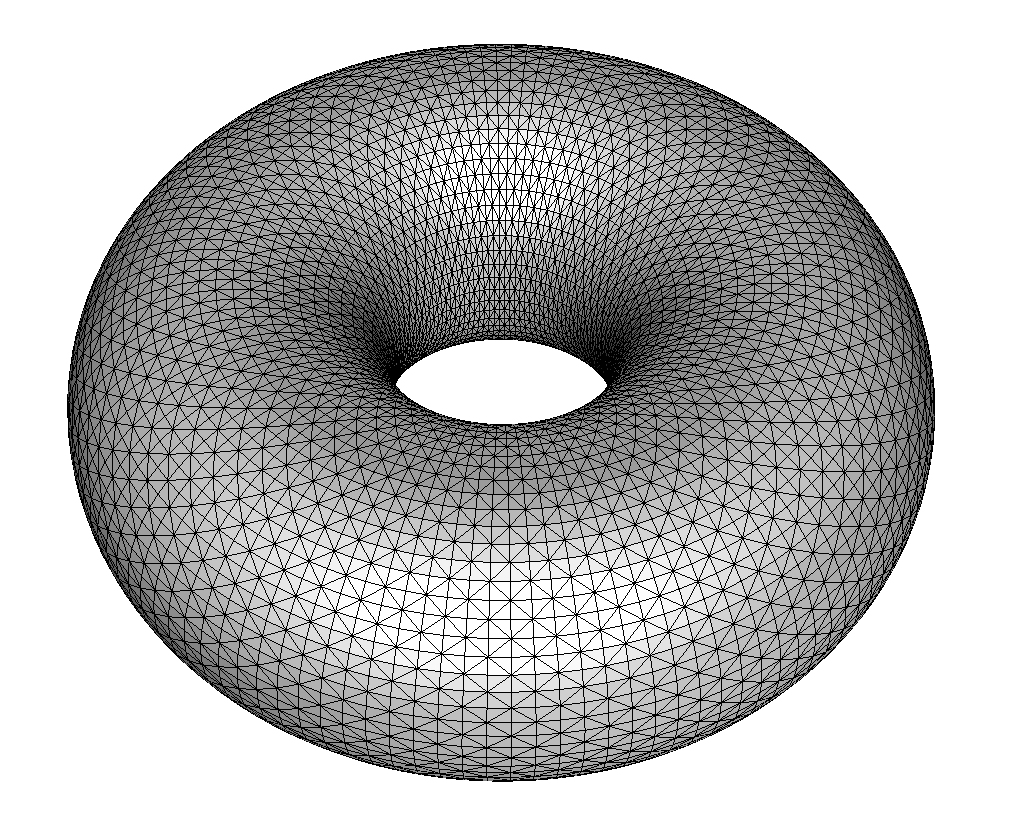}} 
\subfloat[][\centering Surface at time $t = 0.0$. The radii are $r_1=1.0$ and $r_2=0.6$. 
The surface area is $27.56$.]
{\includegraphics[width=0.4\textwidth]
{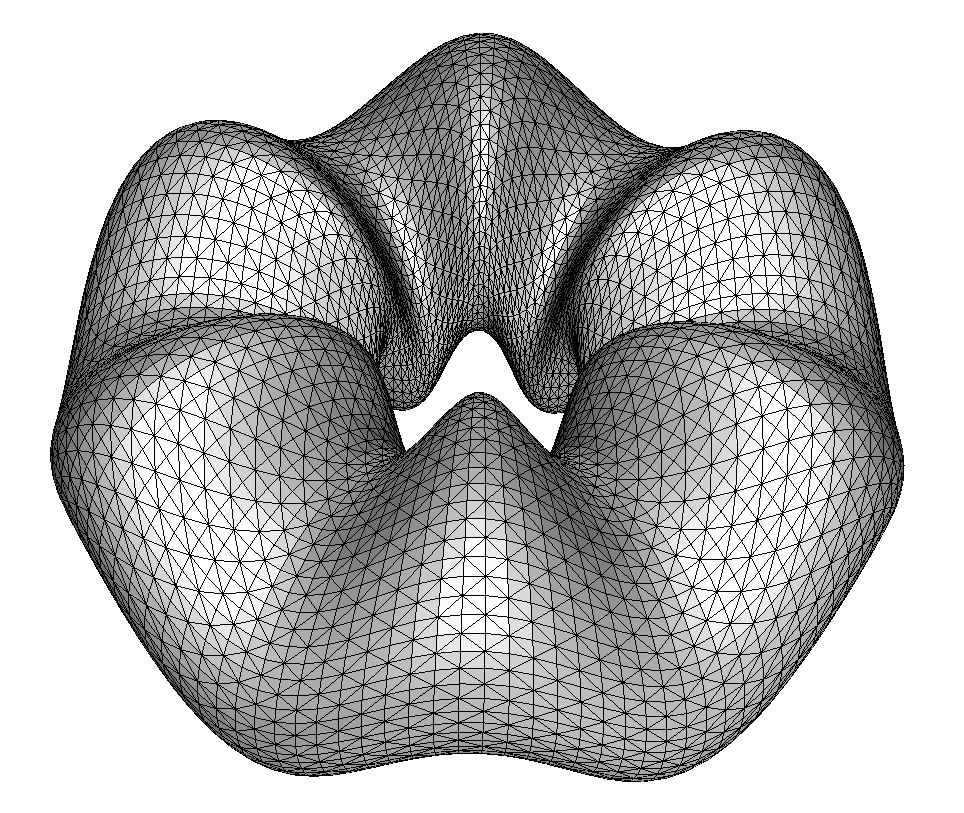}
\label{Fig_Torus_radii_1_0_and_0_6_time_0_0}} \\

\subfloat[][\centering Algorithm \ref{algorithm_harmonic_MCF}
at time $t=0.1$ for the initial surface \ref{Fig_Torus_radii_1_0_and_0_6_time_0_0}. 
The surface area is $14.91$.]{\includegraphics[width=0.4\textwidth]
{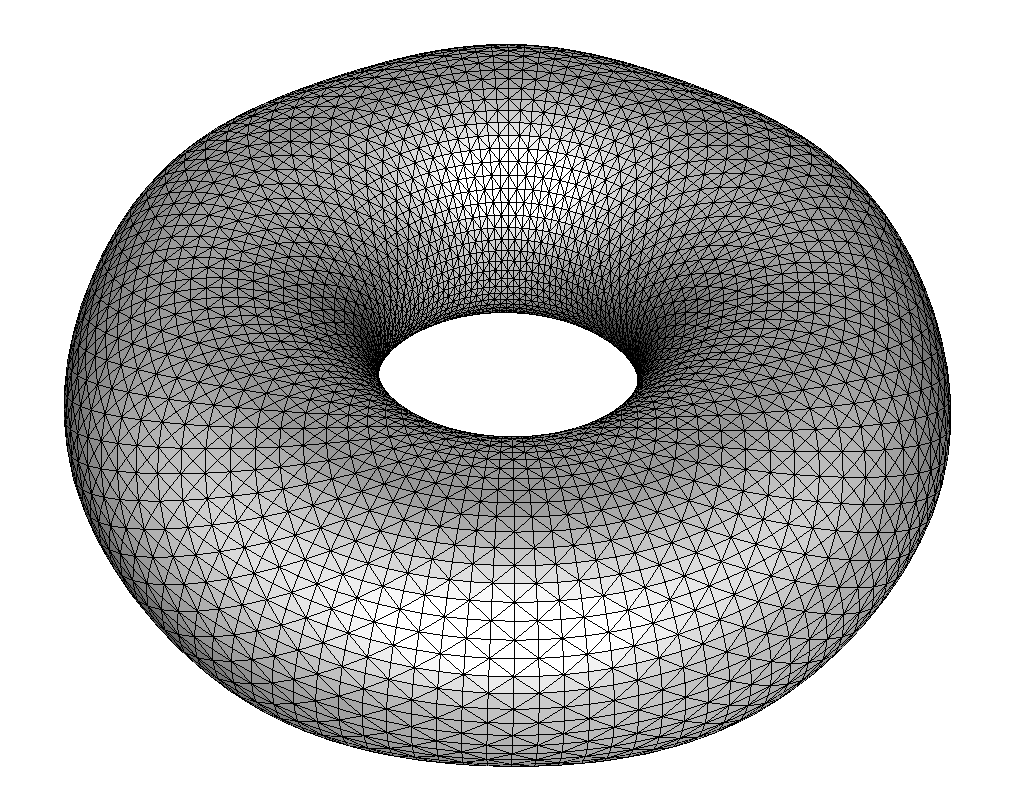}} 
\subfloat[][\centering Algorithm \ref{algorithm_harmonic_MCF}
at time $t=0.2$ for the initial surface \ref{Fig_Torus_radii_1_0_and_0_6_time_0_0}. 
The surface area is $3.22$.]{\includegraphics[width=0.4\textwidth]
{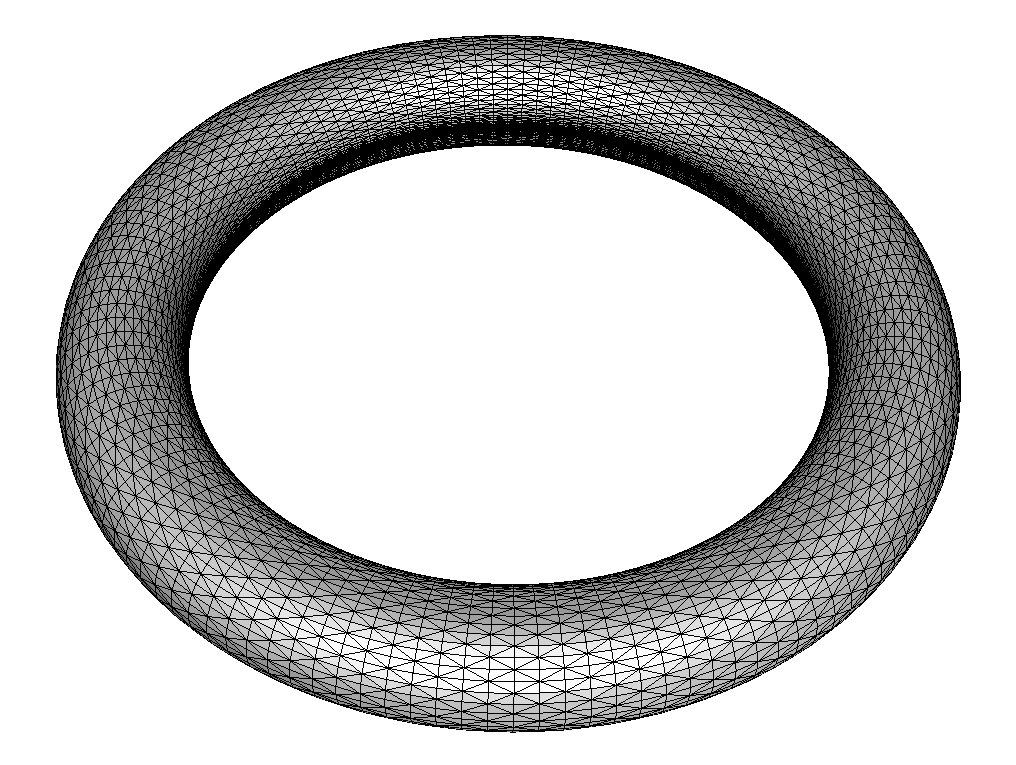}
\label{circle_singularity}}\\

\subfloat[][\centering Surface at time $t = 0.0$. The radii are $r_1=1.0$ and $r_2=0.7$. 
The surface area is $32.28$.]
{\includegraphics[width=0.4\textwidth]
{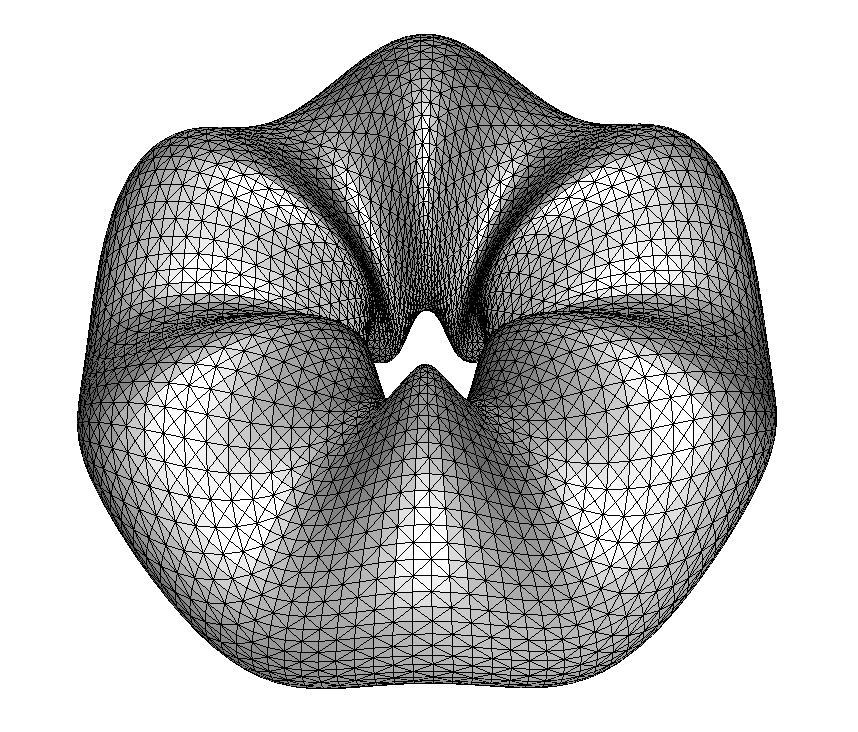}
\label{Fig_Torus_radii_1_0_and_0_7_time_0_0}}
\subfloat[][\centering Algorithm \ref{algorithm_harmonic_MCF}
at time $t=0.08$ for the initial surface \ref{Fig_Torus_radii_1_0_and_0_7_time_0_0}. 
The surface area is $21.11$.]{\includegraphics[width=0.4\textwidth]
{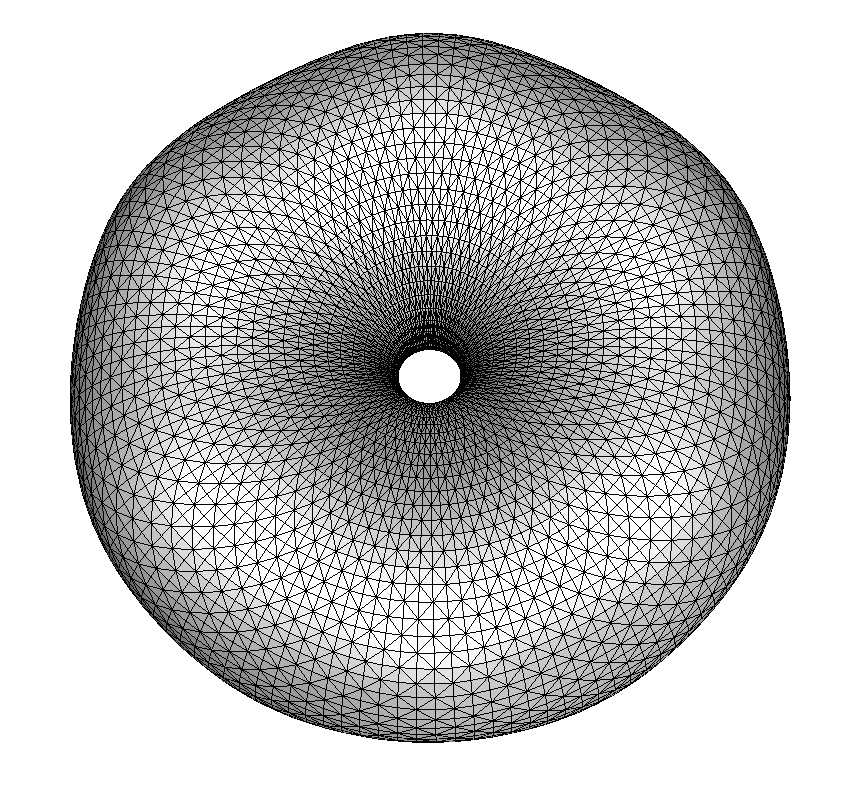}
\label{sphere_singularity}}
\caption{Simulation of the mean curvature flow for undulating tori with different radii $r_2$.
The pictures show the numerical results for Algorithm \ref{algorithm_harmonic_MCF}
with $\alpha = 0.01$. 
The time step size was $\tau = 10^{-4}$. The computational mesh had $16384$ triangles and $8385$ vertices.
The surface singularity strongly depends on the initial radius $r_2$.
In Figure \ref{circle_singularity} the torus converges to a circle, whereas
in \ref{sphere_singularity} it tries to converge to a sphere developing a singularity,
see also Figure 4.7 in \cite{DDE05}. For further details see Example 3 in Section \ref{Numerical_results_MCF}. 
} 
\end{center}
\end{figure}

\begin{figure}
\begin{center}
\subfloat[][\centering Algorithm \ref{algorithm_harmonic_MCF} at time $t = 0.11$. 
The surface area is $16.54$.]
{\includegraphics[width=0.4\textwidth]
{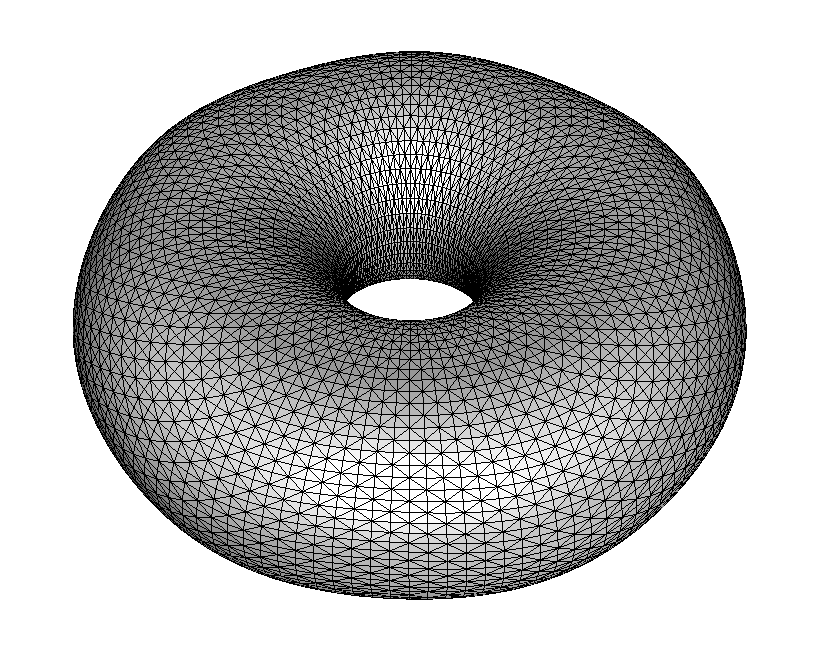}
\label{Figure_good_mesh_behaviour_tori_under_alpha_scheme}
} 
\subfloat[][\centering BGN-scheme at time $t= 0.11$. The surface area is $16.50$.]
{\includegraphics[width=0.38\textwidth]
{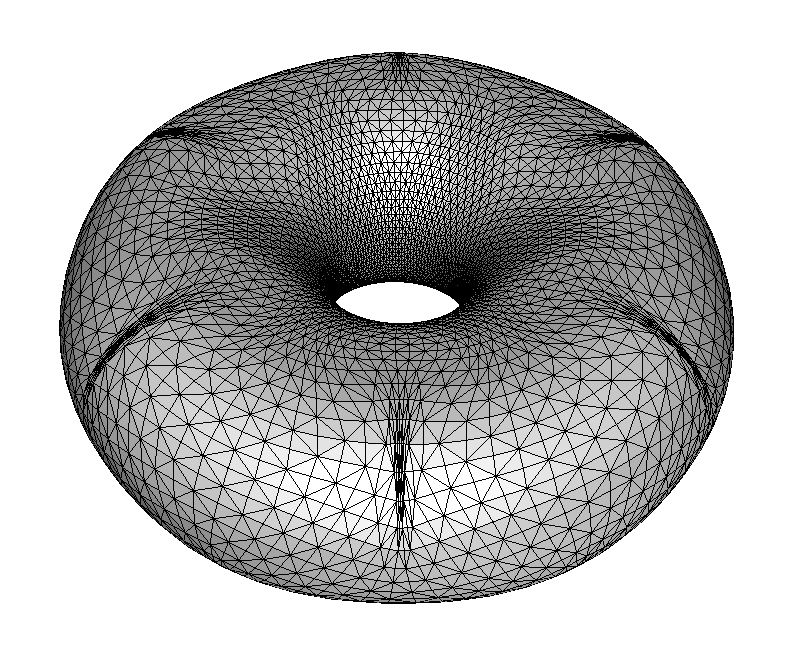}
\label{Figure_degenerated_tori_under_BGN}
}

\caption{Comparison of Algorithm \ref{algorithm_harmonic_MCF} for $\alpha = 1.0$
and the BGN-scheme (2.25) in \cite{BGN08}. 
The radii of the initial surface were $r_1 = 1.0$ and $r_2 = 0.65$, 
see Figures \ref{Fig_Torus_radii_1_0_and_0_6_time_0_0} ($r_1 = 1.0$ and $r_2 = 0.6$) and 
\ref{Fig_Torus_radii_1_0_and_0_7_time_0_0} ($r_1 = 1.0$ and $r_2 = 0.7$) for a visualization
of the initial surface.
The time step size for both schemes was $\tau = 10^{-5}$.
Both meshes had $16384$ triangles and $8385$ vertices.
In this example the mesh degenerates under the BGN-scheme, whereas the mesh
of the $\alpha$-scheme evolves in a controlled way.
See Figure \ref{Figure_mesh_properties_MCF_example_3} for the behaviour of the mesh quantity
$\sigma_{max}$ and Example 3 of Section \ref{Numerical_results_MCF} for further details.
}
\label{Figure_degenerated_mesh_of_the_torus}
\end{center}
\end{figure}

\gdef\gplbacktext{}%
\gdef\gplfronttext{}%
\begin{figure}
\centering
	 \begin{picture}(5102.00,3400.00)%
	 \gplgaddtomacro\gplbacktext{%
      \csname LTb\endcsname%
      \put(396,110){\makebox(0,0)[r]{\strut{} 0}}%
      \csname LTb\endcsname%
      \put(396,637){\makebox(0,0)[r]{\strut{} 5}}%
      \csname LTb\endcsname%
      \put(396,1164){\makebox(0,0)[r]{\strut{} 10}}%
      \csname LTb\endcsname%
      \put(396,1691){\makebox(0,0)[r]{\strut{} 15}}%
      \csname LTb\endcsname%
      \put(396,2218){\makebox(0,0)[r]{\strut{} 20}}%
      \csname LTb\endcsname%
      \put(396,2745){\makebox(0,0)[r]{\strut{} 25}}%
      \csname LTb\endcsname%
      \put(396,3272){\makebox(0,0)[r]{\strut{} 30}}%
      \csname LTb\endcsname%
      \put(528,-110){\makebox(0,0){\strut{} 0}}%
      \csname LTb\endcsname%
      \put(1328,-110){\makebox(0,0){\strut{} 0.05}}%
      \csname LTb\endcsname%
      \put(2128,-110){\makebox(0,0){\strut{} 0.1}}%
      \csname LTb\endcsname%
      \put(2927,-110){\makebox(0,0){\strut{} 0.15}}%
      \csname LTb\endcsname%
      \put(3727,-110){\makebox(0,0){\strut{} 0.2}}%
      \csname LTb\endcsname%
      \put(4527,-110){\makebox(0,0){\strut{} 0.25}}%
      \put(-242,1743){\rotatebox{-270}{\makebox(0,0){\strut{}Area}}}%
      \put(2807,-440){\makebox(0,0){\strut{}Time}}%
    }%
    \gplgaddtomacro\gplfronttext{%
      \csname LTb\endcsname%
      \put(4100,3204){\makebox(0,0)[r]{\strut{}BGN}}%
      \csname LTb\endcsname%
      \put(4100,2984){\makebox(0,0)[r]{\strut{}$\alpha = 1.0$}}%
      \csname LTb\endcsname%
      \put(4100,2764){\makebox(0,0)[r]{\strut{}$\alpha = 0.1$}}%
      \csname LTb\endcsname%
      \put(4100,2544){\makebox(0,0)[r]{\strut{}$\alpha = 0.01$}}%
      \csname LTb\endcsname%
      \put(4100,2324){\makebox(0,0)[r]{\strut{}$\alpha = 0.001$}}%
    }%
    \gplbacktext
	 \put(0,0){\includegraphics{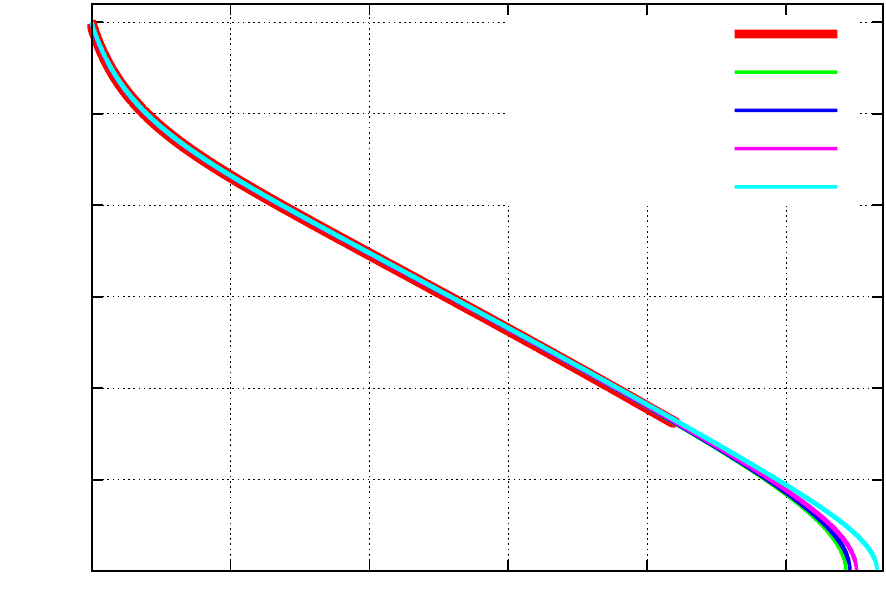}}%
    \gplfronttext
  \end{picture}%
  \vspace*{12mm}
  \caption{
  The image shows the area decrease under the BGN-scheme (2.25) in \cite{BGN08}
  and under Algorithm \ref{algorithm_harmonic_MCF} for different choices of $\alpha$. 
  The radii of the initial surface were $r_1 = 1.0$ and $r_2 = 0.65$.  
  See Figure \ref{Figure_degenerated_mesh_of_the_torus} for a visualization of the surface
  at time $t=0.11$.
  The time step size was $\tau = 10^{-5}$. Please note that the BGN-scheme has
  to be stopped at time $t= 0.21023$, since the mesh fully degenerates, 
  see Figure \ref{Figure_mesh_properties_MCF_example_3}.
  See Example 3 of Section \ref{Numerical_results_MCF} for further details. }
\end{figure}
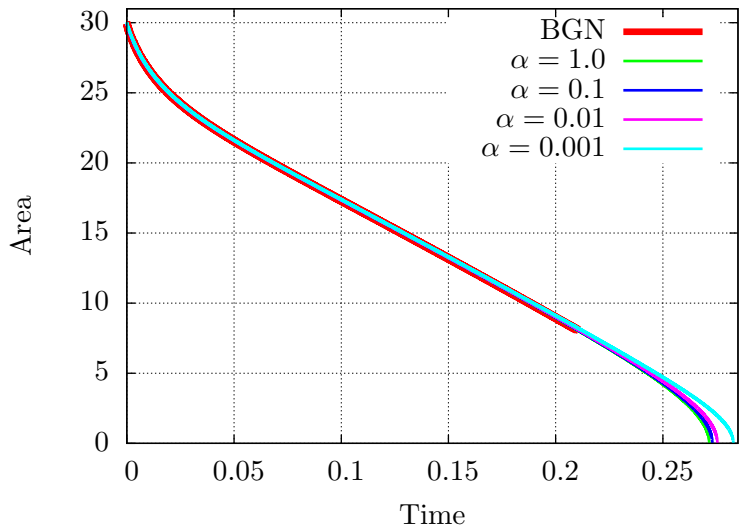

\gdef\gplbacktext{}%
\gdef\gplfronttext{}%
\begin{figure}
\centering
	 \begin{picture}(5102.00,3400.00)%
	 \gplgaddtomacro\gplbacktext{%
      \csname LTb\endcsname%
      \put(396,110){\makebox(0,0)[r]{\strut{} 0}}%
      \csname LTb\endcsname%
      \put(396,763){\makebox(0,0)[r]{\strut{} 20}}%
      \csname LTb\endcsname%
      \put(396,1417){\makebox(0,0)[r]{\strut{} 40}}%
      \csname LTb\endcsname%
      \put(396,2070){\makebox(0,0)[r]{\strut{} 60}}%
      \csname LTb\endcsname%
      \put(396,2724){\makebox(0,0)[r]{\strut{} 80}}%
      \csname LTb\endcsname%
      \put(396,3377){\makebox(0,0)[r]{\strut{} 100}}%
      \csname LTb\endcsname%
      \put(528,-110){\makebox(0,0){\strut{} 0}}%
      \csname LTb\endcsname%
      \put(1342,-110){\makebox(0,0){\strut{} 0.05}}%
      \csname LTb\endcsname%
      \put(2156,-110){\makebox(0,0){\strut{} 0.1}}%
      \csname LTb\endcsname%
      \put(2970,-110){\makebox(0,0){\strut{} 0.15}}%
      \csname LTb\endcsname%
      \put(3784,-110){\makebox(0,0){\strut{} 0.2}}%
      \csname LTb\endcsname%
      \put(4599,-110){\makebox(0,0){\strut{} 0.25}}%
      \put(-374,1743){\rotatebox{-270}{\makebox(0,0){\strut{}$\sigma_{max}$}}}%
      \put(2807,-440){\makebox(0,0){\strut{}Time}}%
    }%
    \gplgaddtomacro\gplfronttext{%
      \csname LTb\endcsname%
      \put(4100,3204){\makebox(0,0)[r]{\strut{}BGN}}%
      \csname LTb\endcsname%
      \put(4100,2984){\makebox(0,0)[r]{\strut{}$\alpha = 1.0$}}%
      \csname LTb\endcsname%
      \put(4100,2764){\makebox(0,0)[r]{\strut{}$\alpha = 0.1$}}%
      \csname LTb\endcsname%
      \put(4100,2544){\makebox(0,0)[r]{\strut{}$\alpha = 0.01$}}%
      \csname LTb\endcsname%
      \put(4100,2324){\makebox(0,0)[r]{\strut{}$\alpha = 0.001$}}%
    }%
    \gplbacktext
	 \put(0,0){\includegraphics{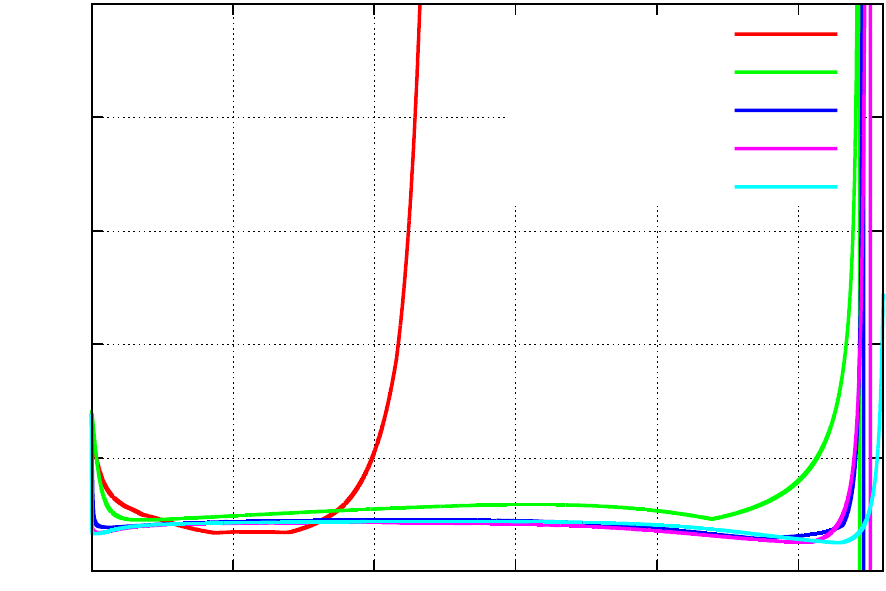}}%
    \gplfronttext
  \end{picture}%
  \vspace*{12mm}
  \caption{Comparison of the behaviour of the mesh quantity $\sigma_{max}$ 
  for the BGN-scheme (2.25) in \cite{BGN08}
  and for Algorithm \ref{algorithm_harmonic_MCF} for different choices of $\alpha$. 
  The time step size was $\tau = 10^{-5}$.
  In this example, the mesh degenerates under the BGN-scheme although the surface
  is still far away from developing a singularity, see Figure \ref{Figure_degenerated_tori_under_BGN}.
  In contrast, the mesh of the $\alpha$-scheme does not degenerate until the surface becomes singular.
  See Example 3 of Section \ref{Numerical_results_MCF} for further details. 
  }
  \label{Figure_mesh_properties_MCF_example_3}
\end{figure}

\section{Generalizations to other geometric flows}
\label{Section_generalization_to_other_geometric_flows}
We are aware that the ideas developed in Sections $2$, $5$ and $6$ 
can be easily generalized
to other geometric evolution equations such as the anisotropic mean curvature flow
and the Willmore flow. For example, suppose that the 
time-dependent embedding $x: \M \times [0,T) \rightarrow \mathbb{R}^{n+1}$
evolves according to 
$$
	\frac{\partial}{\partial t} x = \mathcal{V} \circ x,
$$
where $\mathcal{V}$ denotes the normal velocity of $\Gamma(t) = x(\M,t)$. 
Obviously, this equation of motion can be rewritten in the form
$$
	\frac{\partial}{\partial t} x = \Delta_{g(t)} x + \widetilde{\mathcal{V}} \circ x,
$$
with
$$
	\widetilde{\mathcal{V}} \circ x = (\mathcal{V} + H \nu) \circ x. 
$$
Since $\widetilde{\mathcal{V}}$ is normal to the hypersurface 
$\Gamma(t) \subset \mathbb{R}^{n+1}$, this term remains unaffected
under the operations of Sections $2$, $5$ and $6$.
Hence, the reparametrized equations equivalent to (\ref{reparam_MCF_2}) are given by
\begin{align*}
\label{reparam_curvature_flow}
	\left( \alpha \unit 
	+ (1 - \alpha) (\nu \circ \hat{x}_\alpha) \otimes (\nu \circ \hat{x}_\alpha) \right) 
	\frac{\partial}{\partial t}\hat{x}_\alpha
	&= \Delta_{\hat{g}_\alpha} \hat{x}_\alpha
	- (\nabla_\M \hat{x}_\alpha) v_\alpha + \widetilde{\mathcal{V}} \circ \hat{x}_\alpha 
	\\	
	&= \mathcal{V} \circ \hat{x}_\alpha 
	- (\nabla_\M \hat{x}_\alpha) v_\alpha,
\end{align*}
and the reparametrized equations equivalent to (\ref{reparam_MCF_3}) are
\begin{align*}
	\left( \alpha \hat{\rho} \unit 
	+ (1 - \alpha \hat{\rho}) (\nu \circ \hat{x}_\alpha) \otimes (\nu \circ \hat{x}_\alpha) \right) 
	\frac{\partial}{\partial t}\hat{x}_\alpha
	&= \mathcal{V} \circ \hat{x}_\alpha 
	 + \hat{\rho} (P \circ \hat{x}_\alpha) \Delta_h \hat{x}_\alpha.
\end{align*}
We plan to publish more details including numerical experiments elsewhere.

\section{Discussion}
In this paper, we have introduced reformulations of the curve shortening and mean curvature flows 
based on the DeTurck trick.
The main idea was to reparametrize the flows by solutions to the harmonic map heat flow,
which leads to (strongly) parabolic PDEs
called the curve shortening-DeTurck and mean curvature-DeTurck flows.
The motivation for this approach is that the reparametrization should give rise to tangential
motions that might be advantageous in numerical simulations, in particular with respect
to the mesh quality. It has turned out in our numerical tests that this is indeed the case.
By a straightforward discretization of the reparametrized evolution equations in space
and time, we have obtained algorithms with very good mesh properties.
For the tangential motions being able to redistribute the mesh vertices efficiently,
it was necessary to introduce a variable time scale $\alpha$ 
on which the tangential motions take place. 

The here presented built-in approach for generating good meshes is 
clearly more preferable to approaches 
where the evolution of the flow has to be stopped in order to improve the mesh, see, for example, in \cite{St14}.
We have therefore compared our schemes to algorithms that are in the spirit of a built-in approach.
Namely, we have considered the schemes of Barrett, Garcke and N\"urnberg,
introduced in \cite{BGN08} and \cite{BGN11}.
For the computation of the curve shortening
and mean curvature flows, these schemes represent the present benchmarks with respect 
to the quality of the generated meshes. Our numerical tests show that the algorithms
developed on the basis of the DeTurck trick can
outperform the BGN-schemes with respect to the mesh quality 
-- at least away from surface singularities.
However, we do not regard the better mesh behaviour
as the main advantage of our approach,
but the fact that the tangential motions in our scheme have an analogue in the continuous case
and that our schemes arise from the straightforward discretizations of non-degenerate PDEs.
This might be crucial for the numerical analysis of the schemes.

In the following we must clearly distinguish between a hypersurface that evolves according to the mean curvature flow in a purely geometric sense,
and its parametrization which might not evolve according to the corresponding PDE.
For example, the solutions to our schemes will certainly not approximate the PDE-solution to the mean curvature flow (with vanishing tangential velocity) in the continuous case.
However, we hope that under sufficient conditions they are good approximations to the PDE-solution of the mean curvature-DeTurck flow, 
and hence also approximate the geometric evolution of the mean curvature flow. 
The fact that we have concrete and unique PDE-candidates to which our discrete solutions might converge is very important if one is interested in the error analysis of our schemes.
The scheme for the mean curvature flow in \cite{BGN08} is based on the discretization of the system
\begin{equation}
	\frac{\partial x}{\partial t} \cdot (\nu \circ x) = - H \circ x,
	\qquad (H \nu) \circ x = \Delta_{g(t)} x,
	\label{BGN_PDE}
\end{equation}
see (1.7) in \cite{BGN08}. As the authors clearly state this system has a whole family of solutions,
since the tangential component of the velocity of $x$
is not prescribed. The question, which then arises, is: Which of these solutions is approximated by the
discrete solution of the BGN-scheme? Or, is there at least a candidate which might be considered in the error analysis of the scheme?
Firstly, it is not possible that the solution to the BGN-scheme approximates the solution to the mean curvature flow (\ref{MCF_equation}),
since this solution has vanishing tangential velocity, whereas the desirable mesh properties of the BGN-schemes
are exactly due to non-vanishing tangential motions. 
Another possible answer to this question might be given in
Section 4.1 of \cite{BGN08}. There, the authors point out that the discrete solution
$u^{m+1}_h: \Gamma^m_h \rightarrow \Gamma^{m+1}_h$ to their scheme is
an approximate discrete conformal map. 
One could therefore speculate that the discrete solution approximates a parametrization,
which evolves according to the mean curvature flow
along the normal direction and which in addition satisfy a kind of harmonic map equation in each time step.
Apart from the fact that this would couple a PDE of
parabolic type to an equation of elliptic type, this cannot be the answer. 
The reason is that the only surface in the continuous case is 
the surface $\Gamma(t)$ and the only Riemannian metric is the metric induced by the Euclidean metric of the ambient space. Hence, this would lead to harmonic maps
from $\Gamma(t)$ onto itself with respect to the same metric. 
The identity map clearly solves the corresponding harmonic map equation. 
However, it is not clear how this solution can 
induce any tangential motions. So, the question remains: 
Which solution to (\ref{BGN_PDE}) is approximated by the BGN-scheme? Or, is there a non-degenerate PDE with a unique solution that is approximated by the BGN-scheme?
To have a candidate to which the discrete solutions might converge is a clear advantage of our approach. This statement holds regardless of the fact whether Algorithms
\ref{algo_moving_hypersurface} or \ref{algorithm_harmonic_MCF} do converge or not. For the curve-shortening flow, the situation is slightly different. The weak formulation of
the BGN-scheme in \cite{BGN11} also depends on a background metric. This metric is actually the same metric which we have used for the curve shortening-DeTurck flow.
This strengthens the idea that the BGN-scheme in \cite{BGN11} might approximate the limit of the curve shortening-DeTurck flow (\ref{weak_formulation_reparam_CSF}) 
for $\alpha \searrow 0$ provided that this limit exists.
Whether this view can be made rigorous is an open problem.

In (2.16a) and (2.16b) of \cite{BGN08b}, 
the authors introduced numerical schemes for the mean curvature
and other geometric flows, which allow to reduce or induce tangential motions.
Although, these schemes might seem to have some similarity with the schemes developed in this paper,
there are two main differences. Firstly, in contrast to (2.16a) and (2.16b) of \cite{BGN08b},
our schemes are based on the straightforward discretization
of some reparametrized PDEs.
Secondly, the tangential motions in our scheme are uniquely determined by the DeTurck trick,
and in some sense by the fixed background metric $h$,
whereas in (2.16a) and (2.16b) of \cite{BGN08b} there are no background metrics at all.   

Compared to the scheme (2.25) in \cite{BGN08} 
the implementation of Algorithm \ref{algorithm_harmonic_MCF} 
requires the assemblage of the additional stiffness matrix $D$ in each time step, whereas
for Algorithm \ref{algo_moving_hypersurface} an additional matrix associated to
the tangential part in (\ref{reparam_MCF_2}) has to be assembled,
see Section \ref{Numerical_results_MCF} for further details. 
In both algorithms, the map $y_h^0$ has to be computed
in the first time step, which is, however, trivial if the initial surface is given by an
embedding of the reference surface.

From a numerical point of view, the main difference between the approaches in Sections 
\ref{reparametrizations_by_the_original_DeTurck_trick}
and \ref{reparametrizations_via_a_generalized_DeTurck_trick} is that the original DeTurck
trick leads to a term, which behaves like a first order term, whereas our variant of the DeTurck trick introduced in Section 
\ref{reparametrizations_via_a_generalized_DeTurck_trick} leads to a term that can be written in divergence form. 
Whether Algorithm \ref{algo_moving_hypersurface} or \ref{algorithm_harmonic_MCF} is
generally  more advantageous
for computations as well as for numerical analysis is an open question.

In our numerical tests, Algorithms \ref{algo_moving_hypersurface} and \ref{algorithm_harmonic_MCF} 
turned out to be numerically stable if the surface was not singular. 
Nevertheless, it is not clear whether they are
unconditionally stable in general. This is still an open problem. In contrast, the BGN-scheme in \cite{BGN08} is known to be unconditionally stable.
This does, however, not imply that this scheme is always able to prevent mesh degenerations, see Figure \ref{Figure_degenerated_mesh_of_the_torus}.
It is not unlikely that the time discretization in Algorithms \ref{algo_moving_hypersurface} and \ref{algorithm_harmonic_MCF} has to be improved in order to prove 
stability. A further question that should be addressed in future research is
how the parameter $\alpha$ can be chosen to obtain an optimal behaviour with respect to
the mesh quality.

\subsection*{Acknowledgements}
We would like to thank Klaus Deckelnick and Gerhard Dziuk to call our attention 
to the fact that their work in \cite{DD94} can be linked to the DeTurck trick
as described above.
The second author would also like to thank the Alexander von Humboldt Foundation,
Germany, for their financial support by a Feodor Lynen Research Fellowship 
in collaboration with the University of Warwick, UK.

\end{document}